\documentclass[11pt]{amsart}
\usepackage{amsmath} 
\usepackage{amssymb}
\usepackage{amscd} 
\usepackage{amsxtra} 
\usepackage{epic,eepic} 
\usepackage{graphicx} 
\usepackage{calligra}
\usepackage{mathrsfs}
\usepackage{booktabs}
\usepackage{overpic}
\usepackage{pdflscape}

\usepackage{hyperref}
\usepackage{amsrefs} 

\usepackage[margin=3cm]{geometry} 
\usepackage{booktabs}
\usepackage{url}
\usepackage[all]{xy}

\newtheorem{maintheorem}{Theorem}

\newtheorem{theorem}{Theorem}[subsection]
\newtheorem{lemma}[theorem]{Lemma}
\newtheorem{proposition}[theorem]{Proposition}
\newtheorem{proposition-definition}[theorem]{Proposition-Definition} 
\newtheorem{corollary}[theorem]{Corollary}
\newtheorem{assumption}[theorem]{Assumption} 

\theoremstyle{definition} 
\newtheorem{definition}[theorem]{Definition}
\newtheorem{remark}[theorem]{Remark}
\newtheorem{example}[theorem]{Example}
\newtheorem{notation}[theorem]{Notation}
\newtheorem{recapitulation}[theorem]{Recapitulation} 

\newcommand{\Proj}{\mathbb P}

\newcommand{\bt}{\mathbf{t}}
\newcommand{\bq}{\mathbf{q}}
\newcommand{\hbq}{\hat{\bq}}
\newcommand{\ba}{\mathbf{a}} 
\newcommand{\hba}{\hat{\ba}} 

\DeclareMathOperator{\SL}{SL}
\DeclareMathOperator{\PSL}{PSL}

\DeclareMathOperator{\NE}{NE}
\DeclareMathOperator{\Id}{Id}
\newcommand{\cT}{\mathcal{T}}
\DeclareMathOperator{\sHom}{\mathscr{H}\text{\kern -3pt {\calligra\large om}}\,}

\newcommand{\Ybar}{\ov{Y}}
\newcommand{\cXbar}{\ov{\cX}}
\newcommand{\fry}{\mathfrak{y}}
\newcommand{\frd}{\mathfrak{d}}
\newcommand{\frD}{\mathfrak{D}}
\newcommand{\cMbar}{\ov{\cM}}
\newcommand{\cMA}[1]{\cM_{{\rm A},#1}}
\newcommand{\DA}[1]{D_{{\rm A},#1}}
\newcommand{\cMB}{\cM_{\rm B}}
\newcommand{\cMBbig}{\cMB^{\rm big}}
\newcommand{\Dbig}{D^{\rm big}} 
\newcommand{\cR}{\mathcal{R}}
\DeclareMathOperator{\Hess}{Hess}
\newcommand{\FA}[1]{\sfF_{{\rm A}, #1}}
\newcommand{\cFA}[1]{\cF_{{\rm A}, #1}}
\newcommand{\tcFA}[1]{\widetilde{\cF}_{{\rm A},#1}}
\newcommand{\VA}[1]{V_{{\rm A},#1}}
\newcommand{\nablaA}[1]{\nabla^{{\rm A}, #1}}
\newcommand{\pairingA}[1]{(\cdot,\cdot)_{{\rm A}, #1}}
\newcommand{\FB}{\sfF_{\rm B}}
\newcommand{\FBbig}{\FB^{\rm big}}
\newcommand{\cFB}{\cF_{\rm B}}
\newcommand{\cFBbig}{\cF_{\rm B}^{\rm big}}
\newcommand{\cFGKZ}{\cF_{\rm GKZ}^\times}
\newcommand{\cFbar}{\widebar{\cF}}
\newcommand{\nablaB}{\nabla^{\rm B}}
\newcommand{\pairingB}{(\cdot,\cdot)_{\rm B}}
\DeclareMathOperator{\mir}{mir}
\DeclareMathOperator{\Mir}{Mir}
\newcommand{\tss}{\rm ss}

\DeclareMathOperator{\Gr}{Gr}

\newcommand{\orb}{{\rm orb}}
\newcommand{\LR}{{\rm LR}}
\newcommand{\alg}{{\rm alg}} 
\newcommand{\cc}{{\rm cc}} 

\newcommand{\HX}{H_{\cX}}
\newcommand{\HXbar}{H_{\cXbar}}
\newcommand{\HY}{H_Y}
\newcommand{\HYbar}{H_{\Ybar}}

\newcommand{\cI}{\mathcal{I}}
\newcommand{\cIX}{\mathcal{I}(\cX)}
\newcommand{\cIXbar}{\mathcal{I}(\cXbar)}

\newcommand{\logDTEP}{$\log$-TEP\ }
\newcommand{\logDtrTLEP}{$\log$-trTLEP\ }
\newcommand{\logDtrTLEPns}{$\log$-trTLEP}
\DeclareMathOperator{\Liesp}{\mathfrak{sp}}
\renewcommand{\vec}{{\text{\rm vec}}}
\newcommand{\aroundinfinity}{\Proj^1\setminus\{0\}}
\newcommand{\bP}{\mathbf{P}}
\newcommand{\bF}{\mathbf{F}}

\newcommand{\Haff}{H_{\text{\rm aff}}}
\newcommand{\Hvec}{H_{\text{\rm vec}}}

\newcommand{\cB}{\mathcal{B}}
\newcommand{\hZ}{\widehat{Z}} 
\newcommand{\tcA}{\widetilde{\cA}} 

\DeclareMathOperator{\gr}{\mathrm{gr}}

\newcommand{\C}{\mathbb C}
\newcommand{\R}{\mathbb R}
\newcommand{\Z}{\mathbb Z}
\newcommand{\Q}{\mathbb Q}
\newcommand{\N}{\mathbb N} 
\newcommand{\F}{\mathbb F} 
\newcommand{\A}{\mathbb A} 
\newcommand{\hA}{\widehat{\A}} 
\newcommand{\hE}{\widehat{E}}
\newcommand{\HH}{\mathbb{H}} 
\newcommand{\LL}{\boldsymbol{\mathsf L}}  
\newcommand{\LLo}{\LL^\circ}
\newcommand{\bx}{\boldsymbol{x}} 
\newcommand{\bY}{\boldsymbol{Y}} 
\newcommand{\bR}{\boldsymbol{R}} 
\newcommand{\bOmega}{\boldsymbol{\Omega}}
\newcommand{\bOmegao}{\bOmega_{\circ}}
\newcommand{\bTheta}{\boldsymbol{\Theta}} 
\newcommand{\bThetao}{\bTheta_{\circ}} 
\newcommand{\unit}{\boldsymbol{1}} 
\newcommand{\Nabla}{\boldsymbol{\nabla}} 
\newcommand{\cFCY}{\cF_{\rm CY}} 

\newcommand{\ev}{\operatorname{ev}}

\newcommand{\Hom}{\operatorname{Hom}}
\newcommand{\Ext}{\operatorname{Ext}}

\newcommand{\Ker}{\operatorname{Ker}}
\newcommand{\Cok}{\operatorname{Cok}}
\newcommand{\Image}{\operatorname{Im}}
\newcommand{\rank}{\operatorname{rank}} 
\newcommand{\id}{\operatorname{id}}
\newcommand{\End}{\operatorname{End}}
\newcommand{\Res}{\operatorname{Res}}
\newcommand{\Spf}{\operatorname{Spf}}   
\newcommand{\Sym}{\operatorname{Sym}}
\newcommand{\pr}{\operatorname{pr}} 
\newcommand{\KS}{\operatorname{KS}}
\newcommand{\Cont}{\operatorname{Cont}}
\newcommand{\Aut}{\operatorname{Aut}} 
\newcommand{\con}{{\operatorname{con}}}
\newcommand{\pt}{{\operatorname{pt}}} 
\newcommand{\inv}{\operatorname{inv}} 
\newcommand{\ch}{\operatorname{ch}} 
\newcommand{\tch}{\widetilde{\ch}}
\newcommand{\VC}{\operatorname{VC}}
 
\newcommand{\Fun}{\operatorname{Fun}} 

\newcommand{\Eaff}{E^\circ} 
\newcommand{\tor}{\mathtt{t}}
\newcommand{\cA}{\mathcal{A}}
\newcommand{\cC}{\mathcal{C}} 
\newcommand{\cO}{\mathcal{O}}
\newcommand{\bcO}{\boldsymbol{\cO}}
\newcommand{\cD}{\mathcal{D}}
\newcommand{\cL}{\mathcal{L}}
\newcommand{\hcL}{\widehat{\cL}}
\newcommand{\cX}{\mathcal{X}}
\newcommand{\cZ}{\mathcal{Z}} 
\newcommand{\cE}{\mathcal{E}} 
\newcommand{\cM}{\mathcal{M}} 
\newcommand{\cF}{\mathcal{F}} 
\newcommand{\cG}{\mathcal{G}} 
\newcommand{\cW}{\mathcal{W}}

\newcommand{\tnabla}{\widetilde{\nabla}}

\newcommand{\sfF}{\mathsf{F}}
\newcommand{\hGamma}{\widehat{\Gamma}} 
\newcommand{\hq}{\hat{q}} 
\newcommand{\ha}{\hat{a}} 

\newcommand{\tp}{\tilde{p}} 
\newcommand{\tx}{\tilde{x}}
\newcommand{\tit}{\tilde{t}} 
\newcommand{\tL}{\widetilde{L}} 
\newcommand{\tA}{\widetilde{A}} 
\newcommand{\tB}{\widetilde{B}} 
\newcommand{\tC}{\widetilde{C}} 
\newcommand{\tE}{\widetilde{E}} 

\newcommand{\hC}{\widehat{C}} 
\newcommand{\hcW}{\widehat{\cW}}
\newcommand{\hotimes}{\mathbin{\widehat\otimes}}

\newcommand{\sfP}{\mathsf{P}}
\newcommand{\sx}{\mathsf{x}}
\newcommand{\sy}{\mathsf{y}}

\newcommand{\frt}{\mathfrak{t}} 

\newcommand{\Fock}{\mathfrak{Fock}}
\newcommand{\ov}{\overline}
\newcommand{\surj}{\twoheadrightarrow} 

\newcommand{\vPi}{\overrightarrow{\Pi}}

\newcommand{\iu}{\mathtt{i}} 

\newcommand{\ceil}[1]{\lceil #1\rceil}

\renewcommand{\projlim}{\varprojlim}
\renewcommand{\injlim}{\varinjlim} 

\def\pair#1#2{\langle #1,#2\rangle}
\def\parfrac#1#2{\frac{\partial{#1}}{\partial #2}}
\def\corr#1{\left\langle #1 \right\rangle} 
\def\fracp#1{\left\langle #1 \right\rangle}

\newcommand{\widebar}{\overline} 

\newcommand{\CY}{\rm CY}
\newcommand{\cMCY}{\cM_{\CY}}
\newcommand{\DCY}{D_{\CY}}
\font\rsfs=rsfs10
\newcommand{\marsfs}[1]{\mbox{\rsfs #1}}
\newcommand{\wave}{\marsfs{C}}
\newcommand{\sFCY}{\sfF_{\CY}}

\title[Gromov--Witten Invariants of Local $\Proj^2$ and Modular Forms]{Gromov--Witten Invariants of Local $\Proj^2$\\ and Modular Forms}
\author{Tom Coates}
\address{Department of Mathematics, Imperial College London, 
180 Queen's Gate, London SW7 2AZ, United Kingdom}
\email{t.coates@imperial.ac.uk}
\author{Hiroshi Iritani}
\address{Department of Mathematics, Graduate School of Science, 
Kyoto University, 
Kitashirakawa-Oiwake-cho, Sakyo-ku, Kyoto, 606-8502, Japan}
\email{iritani@math.kyoto-u.ac.jp}
  
\begin{document}

\begin{abstract}
We construct a sheaf of Fock spaces over the moduli space 
of elliptic curves $E_y$ with $\Gamma_1(3)$-level structure, 
arising from geometric quantization of $H^1(E_y)$, and 
a global section of this Fock sheaf.  The global section coincides, near appropriate limit points, 
with the Gromov--Witten potentials of local $\Proj^2$ 
and of the orbifold $[\C^3/\mu_3]$. This proves 
that the Gromov--Witten potentials of local $\Proj^2$ are 
quasi-modular functions for the group $\Gamma_1(3)$, as predicted by 
Aganagic--Bouchard--Klemm,
and proves the Crepant Resolution Conjecture for $[\C^3/\mu_3]$ 
in all genera. 
\end{abstract}

\maketitle 

\let\oldtocsection=\tocsection
\let\oldtocsubsection=\tocsubsection

\renewcommand{\tocsection}[2]{\hspace{0em}\oldtocsection{#1}{#2}}
\renewcommand{\tocsubsection}[2]{\hspace{1em}\oldtocsubsection{#1}{#2}}
\tableofcontents 

\raggedbottom

\section{Introduction}

Let $Y$ be the total space $K_{\Proj^2}$ of the canonical line bundle 
of $\Proj^2$, and let $\cX$ denote the orbifold 
$\big[\C^3/\mu_3\big]$, 
where the group $\mu_3$ of third roots of unity acts 
with weights $(1,1,1)$.  
Let 
\begin{align*}
  F_Y^g= -\frac{t^3}{18} \delta_{g,0} 
  - \frac{t}{12} \delta_{g,1} 
  + \sum_{d=0}^\infty n_{g,d} e^{dt}, 
  &&
     F_\cX^g =\sum_{k=1}^\infty  
     n^{\orb}_{g,k}
     \frac{\frt^k}{k!} 
\end{align*}
denote the genus-$g$ Gromov--Witten potentials of $Y$ and $\cX$ 
respectively.  Here $n_{g,d}$ is the genus-$g$, degree-$d$ Gromov--Witten invariant of $Y$, and $n^{\orb}_{g,k}$ is the genus-$g$ Gromov--Witten invariant of $\cX$ with $k$ insertions of the age-$1$ orbifold class.
 We regard $F_Y^g$ as a function of $t\in H^2(Y)$, and $F^g_{\cX}$ as a function of $\frt \in H^2_\orb(\cX)$.  The main result of this paper is:

\begin{maintheorem}[see Corollary~\ref{cor:analytic_HH}, 
Theorem~\ref{thm:quasi-modularity}, 
Theorem~\ref{thm:CRC_explicit} 
for precise statements]
\label{thm:quasi_modularity_introd} 
Introduce modular parameters $\tau$,~$\tau_{\orb}$ by 
\begin{align*}
  \tau & = -\frac{1}{2} - \frac{3}{2\pi\iu} \parfrac{^2 F^0_Y}{t^2} 
  & \tau_{\orb} & = 3 \parfrac{^2 F^0_{\cX}}{\frt^2} \\ 
  & = -\frac{1}{2} + \frac{t}{2\pi\iu} + O(e^t) & & = \frt + O(\frt^4)
\end{align*}
Then: 
\begin{itemize} 
\item[(1)] when regarded as a function of $\tau$, 
$F^g_Y$ extends to a holomorphic function
on the upper half-plane $\HH$; 

\item[(2)] when regarded as a function of $\tau_{\orb}$, $F^g_{\cX}$ extends to a holomorphic function on the disc $|\tau_{\orb}| < r$, where $r= \Gamma(\frac{1}{3})^3/\Gamma(\frac{2}{3})^3$;

\item[(3)] for $g\ge 2$, $F^g_Y$ is a quasi-modular function with respect 
to the congruence subgroup$:$ 
\[
\Gamma_1(3) = 
\left\{\begin{pmatrix} a & b \\ c & d \end{pmatrix}  
\in \SL(2,\Z) :  a \equiv d \equiv 1,\  c \equiv 0 \bmod 3 \right\}; 
\]
\item[(4)] $($\emph{Crepant Resolution Conjecture}$)$ 
$\{F^g_Y\}$ and 
$\{F^g_\cX\}$ are related by an explicit Feynman diagram 
expansion, which takes the following form for $g\ge 2$$:$ 
\[
F^g_\cX = F^g_Y + (\text{polynomial expressions in 
$\{\partial_t^k F_Y^h: 0\le h<g, 1\le k\le 3g-3\}$})
\]
where $\tau_{\orb} = r \cdot \frac{3 \tau + 1-\xi}{3 \tau + 1-\xi^2}$.
\end{itemize} 
\end{maintheorem}

\noindent This proves conjectures of Aganagic--Bouchard--Klemm~\cite{ABK}.

\subsection{Geometric Quantization and the Fock Sheaf}
Aganagic--Bouchard--Klemm's prediction was based on 
Witten's discovery~\cite{Witten:background} 
that a topological string partition function 
\begin{equation} 
\label{eq:wave_function} 
Z = \exp\left(\sum_{g=0}^\infty \hbar^{g-1} F^g\right) 
\end{equation} 
can be understood as a `wave function' of a quantum-mechanical system  
that arises from geometric quantization of the state space of 
the theory. In the present setting, the state space is given by 
$H^1(E_y)$, where
\begin{equation}
  \label{eq:Ey_introduction}
  E_y = \text{compactification of } \left \{(x_1,x_2) \in (\C^\times)^2 :   \textstyle     x_1 + x_2 + \frac{y}{x_1 x_2} + 1 =0 \right\} 
\end{equation}
is the family of elliptic curves, parametrised by $y \in \ov{\HH/\Gamma_1(3)} \cong \Proj(3,1)$, that corresponds to $Y$ under mirror symmetry.
Quasi-modularity follows by `quantizing' monodromies of this family. 
The aim of the present paper is to verify this physics picture 
for local $\Proj^2$ mathematically. 

Let us begin with the genus-zero part of the story. 
The graph of $dF_Y^0$ defines a Lagrangian submanifold 
$\cL$ 
in the cotangent bundle $T^*H^2(Y) \cong H^2(Y) \oplus H^4(Y)$ 
of $H^2(Y)$, 
where we regard $H^4(Y) = H^4_c(Y)$ as the dual of $H^2(Y)$ 
via the intersection pairing: 
\[
\cL = \Gamma(dF_Y^0) = 
\left\{ (t, p) \in H^2(Y)\oplus H^4(Y) : p = \parfrac{F^0_Y}{t} 
\right\}.  
\]
The Crepant Resolution Conjecture at genus zero -- proved in this case by \cite{CIT:wall-crossings, Coates:local} -- 
says that the graphs of $dF_Y^0$ and $dF_\cX^0$ 
coincide under an affine symplectic transformation 
$U \colon T^* H^2(Y) \to T^* H^2_\orb(\cX)$:
see Figure~\ref{fig:genus_zero_CRC}. 
The family $\{T_t \cL\}$ of tangent spaces to 
$\cL$ defines a \emph{variation of Hodge structure} 
(VHS) of weight 1; 
under mirror symmetry, 
this is identified with the VHS on $H^1(E_y)$ of the mirror family. 
The mirror curve $E_y$ is parameterized 
by $y\in \Proj(3,1) \cong  \ov{\HH/\Gamma_1(3)}$ and 
the Gromov--Witten potentials $F^0_Y$, $F^0_\cX$ 
describe the local behaviour of the VHS near $y=0$ 
(the large-radius limit) and $y=\infty$ (the orbifold point) respectively. 
An important observation here is that the directions
of the `$y$-axes' $H^4(Y)$ and $H^4_\orb(\cX)$ 
do not coincide. In higher genus,  
these $y$-axes play the role of a polarization in geometric 
quantization. 

\begin{figure}[ht]
\includegraphics[bb=200 600 400 720]{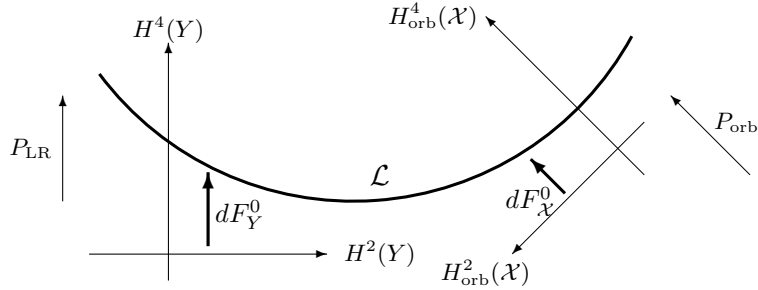}
\caption{Genus-zero Crepant Resolution Conjecture 
for $\cX=[\C^3/\mu_3]$}
\label{fig:genus_zero_CRC} 
\end{figure}

Geometric quantization (see~\cite{Kirillov:quantization}) 
associates to a symplectic vector space $H$ a Hilbert space, called the Fock space,
which is an irreducible representation of the Heisenberg algebra 
associated with $H$. To construct such a representation, 
we need the data of a \emph{polarization}, that is, a Lagrangian subspace 
$P$ of $H$. 
Given a polarization $P$, the Fock space is a space $\Fun(H/P)$ 
of functions on $H/P$ of an appropriate class
($C^\infty$, $L^2$, Schwartz, etc).  It carries an action of the Heisenberg algebra given by 
`canonical quantization'. 
For instance, if $H$ is a 2-dimensional symplectic vector space with 
Darboux co-ordinates $(p,x)$, and if we choose $P$ to be the subspace 
$\langle \partial/\partial p\rangle$, then the corresponding 
Fock space is a space of functions of $x$, and the Heisenberg algebra 
acts by operators $x$ and $\hbar \partial/\partial x$. 
If we have two different polarizations $P_1$, $P_2$, the corresponding 
Fock spaces are canonically isomorphic 
as projective representations of the Heisenberg algebra: 
\[
T(P_1,P_2) \colon \Fun(H/P_1) \xrightarrow{\cong} \Fun(H/P_2) 
\]
by the Stone--von~Neumann theorem.
Such an isomorphism is given by an integral transformation of Fourier type. 

We are interested in Fock space elements of the form 
\eqref{eq:wave_function},
which can be viewed as asymptotic series in $\hbar$. 
Agangagic--Bouchard--Klemm \cite{ABK} 
described the isomorphism $T(P_1,P_2)$ for such 
asymptotic functions using a sum over Feynman diagrams. 
Using their Feynman rule, we construct in \S\ref{sec:Fock_sheaf_fd} a 
sheaf
$\Fock_{\CY}$ 
of Fock spaces\footnote{This is a sheaf of sets, not of vector spaces, 
as functions of the form 
\eqref{eq:wave_function} are not closed under addition.}  over the base $\Proj(3,1)$ of the mirror family.  Note that we need to construct a Fock \emph{sheaf} here, instead of a 
Fock \emph{space}, because there is no globally defined, single-valued, flat polarization over the moduli space $\Proj(3,1) \cong \ov{\HH/\Gamma_1(3)}$, because the mirror family has non-trivial monodromies.  Roughly speaking, we:
\begin{itemize} 
\item[(a)] choose an open covering $\{U_\alpha\}$ of $\Proj(3,1)$ 
by sufficiently small open sets $U_\alpha$;   
\item[(b)] choose a Gauss--Manin flat polarization $P_\alpha 
\subset H^2(E_y)$ 
such that $P_\alpha \pitchfork \cL$ over $U_\alpha$, 
i.e.~$P_\alpha \oplus H^{1,0}(E_y) = H^1(E_y)$ for each $y\in U_\alpha$;  
\item[(c)]  
define $\Gamma(U_\alpha, \Fock_{\CY})$ 
to be the space of asymptotic series 
$\exp(\sum_{g=1}^\infty \hbar^{g-1} F^g)$, where 
$F^g$ is a holomorphic function on $U_\alpha$; 
\end{itemize}
and then patch local Fock spaces $\Gamma(U_\alpha,\Fock_{\CY})$ 
over overlaps $U_\alpha \cap U_\beta$ using the Feynman rule 
of Aganagic--Bouchard--Klemm. Theorem~\ref{thm:quasi_modularity_introd} above is a consequence of
the following more fundamental result:
 
\begin{maintheorem}[Theorem \ref{thm:wave_CY}, 
Theorem \ref{thm:conifold_estimate}] \label{thm:wave_CY_introduction}
There exists a global section $\wave_{\CY}$ 
of the Fock sheaf $\Fock_{\CY}$ such that:
\begin{itemize} 
\item[(1)] in a neighbourhood of the large-radius limit point 
$y=0$ and with respect to the polarization $P_\LR = H^4(Y)$, 
$\wave_{\CY}$ is represented by the Gromov--Witten potentials  
of $Y$; 
\item[(2)] in a neighbourhood of the orbifold point $y=\infty$ 
and with respect to the polarization $P_\orb =H^4_\orb(\cX)$, 
$\wave_{\CY}$ is represented by the Gromov--Witten potentials 
of $\cX$;
\item[(3)] in a neighbourhood of the conifold point $y=-\frac{1}{27}$ 
and with respect to a polarization $P_\con$, 
$\wave_{\CY}$ is represented by a collection $\{F^g_\con\}$ 
of functions such that $F^g_\con$ has poles of order $2g-2$ 
at $y_1=-\frac{1}{27}$ for $g\ge 2$.  
\end{itemize} 
\end{maintheorem}

There are various possible choices of polarization, which are summarized in  Table~\ref{tab:polarizations}.  
The VHS of the mirror family has singularities at the large-radius 
$(y=0)$, conifold $(y=-\frac{1}{27})$, and orbifold $(y=\infty)$ 
points. Near these points, there are unique flat polarizations 
$P_\LR$, $P_\con$, $P_\orb$ characterized by invariance under 
local monodromy (Notation~\ref{nota:opposite_fd}, Proposition~\ref{prop:opposite_cusps_flat}).  These polarizations become multi-valued when 
they are analytically continued. 

On the other hand, we can also consider polarizations which are 
\emph{not} Gauss-Manin flat, but are \emph{single-valued}.  Expressing the global section $\wave_{\CY}$ with respect to a single-valued polarization yields $\exp(\sum_{g=1}^\infty \hbar^{g-1} F^g)$, where the correlation functions $F^g$ are single-valued  on $\Proj(3,1)$.
The polarization $P_\cc$ defined by 
$H^{0,1}(E_y) = \ov{H^{1,0}(E_y)}$, which we call the complex-conjugate polarization, is single-valued, and
coincides with $P_\LR$ at $y=0$, 
$P_\con$ at $y=-\frac{1}{27}$, and $P_\orb$ at $y=\infty$.  It varies non-holomorphically 
along $\Proj(3,1)$, and correlation functions for $\wave_{\CY}$ with respect to $P_\cc$ satisfy the 
Bershadsky--Cecotti--Ooguri--Vafa 
holomorphic anomaly equation~\cite{BCOV}  (Proposition \ref{prop:HAE}).  We can also obtain single-valued polarizations by considering the algebraic structure of the bundle $H^1(E_y)$ over $\Proj(3,1)$: 
$\bigcup_y H^1(E_y) \cong \cO(1) \oplus \cO(-1)$.  The algebraic polarization $P_\alg$ is a single-valued holomorphically-varying 
non-flat polarization, corresponding to $\cO(-1)$ over $\Proj(3,1)$; 
correlation functions for $P_\alg$ can be 
thought of as the `holomorphic ambiguity' in the holomorphic anomaly 
equation.  Correlation functions for $P_\alg$ are rational functions, 
and it follows that the Gromov--Witten potentials $F^g_Y$, 
$F^g_\cX$ belong to certain polynomial rings 
(see Theorem~\ref{thm:finite_generation}). 

\begin{table}[ht] 
\begin{tabular}{cccl} \toprule  
  \multicolumn{1}{c}{polarization} & 
  \multicolumn{1}{c}{flat/curved} & 
  \multicolumn{1}{c}{global behaviour} &
  \multicolumn{1}{c}{correlation functions} \\ \midrule 
  $P_\LR$ & flat & multi-valued & $F^g_Y$, quasi-modular \\ 
  $P_\orb$ & flat & multi-valued &  $F^g_\cX$ \\ 
  $P_\con$ & flat & multi-valued &  quasi-modular \\ 
  $P_\cc$ & curved & single-valued & almost-holomorphic modular \\
  $P_\alg$ & curved & single-valued & holomorphic modular 
                                      (rational functions) \\ \bottomrule \\[-0.5ex]
\end{tabular} 
\caption{Various Polarizations} 
\label{tab:polarizations}
\end{table} 

\begin{remark}
  Polarizations are called `opposite line bundles' in the main body of the text.  
\end{remark}

\begin{remark} 
Lho--Pandharipande 
\cite{Lho-Pandharipande:SQ_HAE, Lho-Pandharipande:CRC} 
also proved a similar finite generation result, 
and a version of the Crepant Resolution Conjecture for $Y$ and $\cX$. 
We give a proof of their version of Crepant Resolution 
Conjecture using our method below (Theorem~\ref{thm:finite_generation})
but we learned its elegant formulation from them.
\end{remark} 

\begin{remark}
It was conjectured by Huang--Klemm~\cite{Huang--Klemm, Huang--Klemm--Quackenbush} that 
the correlation function $F^g_\con$ with respect to $P_\con$ 
should satisfy a certain `gap condition' -- see \eqref{eq:gap}. 
We do not have a proof of this conjecture, but verify it up to 
genus $g=7$.  See~\S\ref{sec:calculation}. 
\end{remark}

\subsection{Summary of the Argument} 
In outline: we pass from $Y$ and $\cX$ to their toric compactifications 
$\Ybar = \Proj_{\Proj^2}(\cO(-3) \oplus \cO)$ and 
$\cXbar = \Proj(1,1,1,3)$.  
These have generically semisimple quantum cohomology, 
which is not true for $Y$ or $\cX$.
We determine the Gromov--Witten potentials of $\Ybar$ and $\cXbar$ 
using the Givental--Teleman formula; this requires semisimplicity.  
We relate the two potentials via mirror symmetry for $\Ybar$ and $\cXbar$. 
The Gromov--Witten potentials of $\Ybar$ and $\cXbar$ glue 
together to give a single-valued section $\wave_{\rm B}$ of 
an infinite-dimensional version of the Fock sheaf, 
which we constructed in~\cite{Coates--Iritani:Fock}; 
this is a higher-genus version of the Crepant Resolution Conjecture 
for $\Ybar$ and $\cXbar$.  
The finite-dimensional version of the Fock sheaf, and the
global section $\wave_{\CY}$, 
emerge from their infinite-dimensional counterparts 
by taking a certain `conformal limit' or `local limit'.  
In this limit, the volume of the fiber of $\Ybar \to \Proj^2$ 
becomes infinitely large, and the Gromov--Witten theory of 
$\Ybar$ reduces to that of $Y$. 

Let us explain some more details. 
We consider the Landau--Ginzburg model 
that is mirror to the small quantum cohomology of $\Ybar$.  
This is given by 
  \begin{align} \label{eq:LG_introduction}
    W_{y_1,y_2}  =\left( x_1+x_2+\frac{y_1}{x_1 x_2} + 1\right) x_3 
+\frac{y_2}{x_3} && (x_1,x_2,x_3) \in \big(\C^\times\big)^3
  \end{align}
where $(y_1,y_2) \in (\C^\times)^2$.  
This family of Laurent polynomials extends 
over a partial compactification $\cMB$ of $(\C^\times)^2$, 
where the limits 
\begin{align*}
  (y_1,y_2) \to (0,0) && (y_1^{-1/3}, y_1^{1/3}y_2) \to (0, 0)
\end{align*}
correspond respectively to the large-radius limit for~$\Ybar$ 
and the large-radius limit for~$\cXbar$: see Figure~\ref{fig:B-model_moduli}. 
\begin{figure}[ht]
\includegraphics[bb=200 590 400 710]{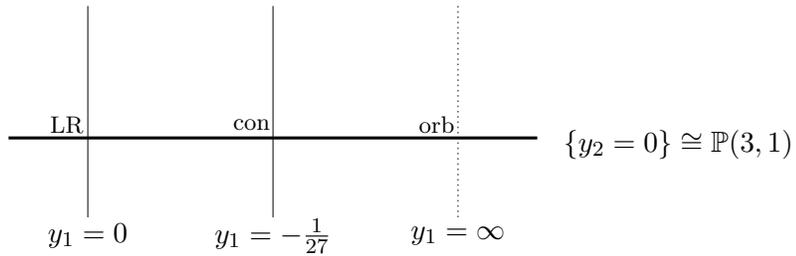} 
\caption{The B-model moduli space $\cMB$: this is the base 
space of the family $W_{y_1,y_2}$ of Landau--Ginzburg potentials.} 
\label{fig:B-model_moduli}
\end{figure}

The Landau--Ginzburg mirror determines an infinite-dimensional 
symplectic vector bundle over $\cMB$, with fiber over 
$(y_1,y_2) \in (\C^\times)^2$ equal to 
  \[
    H^3\big(\Omega_{(\C^\times)^3}^\bullet(\!(z)\!), zd+dW_{y_1,y_2}\wedge\big)
  \]
This vector bundle carries a flat Gauss--Manin connection, 
which has logarithmic singularities along $y_1=0$, $y_2=0$, 
and $y_1 = {-\frac{1}{27}}$, and has a Lagrangian subbundle
  \[
   \FB = H^3\big(\Omega_{(\C^\times)^3}^\bullet[\![z]\!], zd+dW_{y_1,y_2}\wedge\big)
  \]
Such structures have been studied by K.~Saito~\cite{SaitoK:higher_residue} in the context of 
singularity theory. 
By transporting the Lagrangian subspaces $\FB$ in the fibers 
to a fixed fiber\footnote
{We are hiding some technical details here.  To obtain a moving subspace realization, we need to analytify $\FB$ in the $z$-direction. 
The analytification is denoted by $\cFB$ in the main body of the text.},
using the Gauss--Manin connection, we obtain a moving family of 
semi-infinite Lagrangian subspaces.  
This is an example of a \emph{variation of semi-infinite Hodge 
structure} (VSHS)~\cite{Barannikov}. 

On the other side of mirror symmetry, we consider 
the descendant potentials $\cF^g_{\Ybar}$ and $\cF^g_{\cXbar}$, which are generating functions for genus-$g$ 
Gromov--Witten invariants of $\Ybar$ and $\cXbar$ with descendants.  
$\cF^g_{\Ybar}$ and $\cF^g_{\cXbar}$ are functions on an infinite-dimensional space, and 
the graph of the differential $d\cF^0_{\Ybar}$ 
defines a Lagrangian submanifold 
\[
\mathfrak{L}_{\Ybar} \subset H^\bullet(\Ybar) \otimes \C(\!(z^{-1})\!) 
\]
in Givental's symplectic space $H^\bullet(\Ybar) \otimes \C(\!(z^{-1})\!)$. 
Under mirror symmetry, the tangent spaces to the Givental cone $\mathfrak{L}_{\Ybar}$ are identified with the VSHS determined by the Landau--Ginzburg mirror, near the large-radius limit point for $\Ybar$.  Analogous statements hold for $\cXbar$.

The Landau--Ginzburg model \eqref{eq:LG_introduction} is a mirror to the \emph{small} quantum cohomology of $\Ybar$, rather than the full big quantum cohomology, and therefore the VSHS that it determines is not miniversal.  In \S\ref{sec:enlarge_base} we construct a miniversal unfolding of this semi-infinite variation, over a six-dimensional base $\cMBbig$, that is a mirror to the big quantum cohomology of $\Ybar$.  The base $\cMBbig$ is a thickening of $\cMB \setminus \{{y_1 = -\frac{1}{27}}\}$.  It carries an infinite-dimensional version $\Fock_{\rm B}$ of the Fock sheaf, which we constructed in~\cite{Coates--Iritani:Fock} and review in \S \ref{sec:Fock_sheaf} below, and furthermore there is a distinguished global section $\wave_{\rm B}$ of this Fock sheaf.  The global section $\wave_{\rm B}$ is constructed using Givental's formula for higher-genus potentials~\cite{Givental:semisimple}; see~\cite[\S7.2]{Coates--Iritani:Fock}.  It coincides under mirror symmetry, near the large-radius limit points for $\Ybar$ and $\cXbar$, with the total descendant potentials $\cZ_{\Ybar} = \exp(\sum_{g \geq 0} \hbar^{g-1} \cF^g_{\Ybar})$ of $\Ybar$ and $\cZ_{\cXbar}$ of $\cXbar$.  This follows from Teleman's theorem~\cite{Teleman} \cite[Theorem~7.15]{Coates--Iritani:Fock}.

Now we take the local limit.  
Observe that the Landau--Ginzburg potential $W_{y_1,y_2}$ 
with $y_2$ set to zero defines the family \eqref{eq:Ey_introduction} 
of elliptic curves $E_{y_1}$; thus the divisor $(y_2=0)$ 
in $\cMB$ can be identified with the base $\Proj(3,1) \cong 
\ov{\HH/\Gamma_1(3)}$ of the mirror family of $Y$. 
A key step in the argument is the construction of a 
``restriction map'' from the infinite-dimensional Fock sheaf 
$\Fock_{\rm B}$ over $\cMBbig$ to the finite-dimensional Fock sheaf 
$\Fock_{\CY}$ over $\Proj(3,1)$. 
This requires care, as the VSHS 
associated with $W_{y_1,y_2}$ has logarithmic singularities 
along $y_2=0$. 
We also need a result comparing polarizations for the VSHS 
with polarizations for the VHS. 

\begin{maintheorem}[see 
Propositions~\ref{pro:correspondence_of_opposites},~\ref{pro:extending_trivialization},~\ref{pro:existence_of_compatible_opposites} and~\ref{pro:LR_conifold_orbifold}, Notation~\ref{nota:opposite_fd}, Theorem~\ref{thm:Fock_restriction}]
Let $y_1 \in \Proj(3,1)$. 
There is a one-to-one correspondence between:
\begin{itemize} 
\item[(a)] flat polarizations near $y_1$  for the VSHS associated with the Landau--Ginzburg mirror that are compatible with the Deligne extension;
\item[(b)] flat polarizations near $y_1$ for the VHS associated with the mirror family $\{E_y\}$ of elliptic curves.
\end{itemize} 
Let $\Fock_{\rm B}$ denote the infinite-dimensional Fock sheaf 
over $\cMBbig$ and let $\Fock_{\CY}$ denote the finite-dimensional 
Fock sheaf over $\Proj(3,1)$. 
Write $\Proj(3,1)^\circ = \Proj(3,1) \setminus \{-\tfrac{1}{27}\}$ for the complemement of the conifold point, and let $i\colon \Proj(3,1)^\circ \hookrightarrow 
 \cMBbig$  denote the inclusion map. 
There is a restriction map: 
\[
i^{-1} \Fock_{\rm B} \longrightarrow \Fock_{\CY}\big|_{\Proj(3,1)^\circ}
\]
\end{maintheorem}

By applying the restriction map to the global section $\wave_{\rm B}$ 
of $\Fock_{\rm B}$, we obtain a section $\wave_{\CY}$ of 
$\Fock_{\CY}$ over $\Proj(3,1)^\circ$. 
It is then easy to check that $\wave_{\CY}$ corresponds to the 
Gromov--Witten potentials of $Y$ and $\cX$, respectively, near 
$y_1=0$ and $y_1=\infty$.  
In \S \ref{sec:conifold_estimate}, we show that 
the genus-$g$ potential of $\wave_{\CY}$ has poles of order 
$2g-2$ at the conifold point $y_1=-\frac{1}{27}$, 
by analysing the pole order of the ingredients in 
Givental's formula for higher-genus potentials. 
This proves Theorem~\ref{thm:wave_CY_introduction}. 

\begin{remark} 
  In the main body of the text, we consider various versions of VSHS but do not use the term `VSHS' itself, instead using the equivalent notions of  TEP structures, \logDTEP structures, and $\log$-cTEP structures.
\end{remark}

\begin{remark}[Related work]  Higher-genus Gromov--Witten invariants of local $\Proj^2$ and $\big[\C^3/\mu_3\big]$ have been studied by many authors.   In string theory, Alim--Scheidegger--Yau--Zhou~\cite{ASYZ}, Huang--Klemm~\cite{Huang--Klemm}, and Huang--Klemm--Quackenbush~\cite{Huang--Klemm--Quackenbush} have emphasized the importance of special geometry and the holomorphic anomaly equations.  On the mathematics side, Bouchard--Cavalieri have computed Gromov--Witten invariants of $\big[\C^3/\mu_3\big]$ at genus~$2$ and~$3$ using Hodge and Hurwitz--Hodge integrals~\cite{Bouchard--Cavalieri}.  Lho--Pandharipande have recently established the holomorphic anomaly equation for local $\Proj^2$, in the precise form predicted by physicists, and used this to prove a higher-genus Crepant Resolution Conjecture for $\big[\C^3/\mu_3\big]$~\cite{Lho-Pandharipande:SQ_HAE, Lho-Pandharipande:CRC}.  Another approach goes via the Remodelling Conjecture of Bouchard--Klemm--Mari\~no--Pasquetti~\cite{BKMP} and Eynard--Orantin recursion~\cite{Eynard--Orantin}.  Fang--Liu--Zong~\cite{Fang--Liu--Zong} have established the Remodelling Conjecture for all toric Calabi--Yau 3-orbifolds, and this should lead to a proof of modularity and the holomorphic anomaly equation in our setting. 
\end{remark}

\subsection*{Acknowledgements}  The authors thank Mina Aganagic, Vincent Bouchard, Andrea Brini, Albrecht Klemm, Hyenho Lho,  Rahul Pandharipande, Yongbin Ruan, and Hsian-Hua Tseng for useful conversations during the course of this work.  H.I.~thanks Atsushi Kanazawa for inviting Hyenho Lho to Kyoto, and providing an occasion to discuss the Crepant Resolution Conjecture.  We are grateful to Lho and Pandharipande for sharing a draft of their paper~\cite{Lho-Pandharipande:CRC} before publication.  We also thank the anonymous referee for their careful reading of the paper and insightful comments.  This research was partially supported by a Royal Society University Research Fellowship; the Leverhulme Trust; ERC Starting Investigator Grant number~240123; ERC Consolidator Grant number~682602; EPSRC Mathematics Platform grant EP/I019111/1; Inoue Research Award for Young Scientists; EPSRC grant EP/E022162/1; and JSPS Kakenhi Grants Number 19740039, 22740042, 23224002, 24224001, 25400069, 26610008, 
16K05127, 16H06335, 16H06337, and 17H06127. 
This material is based upon work supported by the National Science 
Foundation under Grant No.~DMS-1440140 while H.I.~was in residence 
at the Mathematical Sciences Research Institute in Berkeley, California, 
during the Spring semester of 2018. It completes a project that was begun during the workshop `New Topological Structures in Physics' at the Mathematical Sciences Research Institute in 2006.  We thank the organisers  and participants of the workshop for many inspiring discussions.

\bigskip

\bigskip

\noindent 
{\bf List of Notation} 

\renewcommand{\arraystretch}{1.15} 

\begin{table}[h] 
\begin{tabular}{ll}  
\toprule 
$\VA{X}\subset H_X$
& open subset of the form \eqref{eq:LRL_nbhd}.  
\\ 
$\cMA{X}^\times=\frac{\VA{X}}{2\pi\iu H^2(X,\Z)}$ 
& the base of the A-model TEP structure; see Example~\ref{ex:AmodelTEP}.  
\\ 
$(\cMA{X}, \DA{X})$ 
& the base of the A-model \logDTEP structure; 
\\  
&  $\cMA{X}^\times = \cMA{X} \setminus \DA{X}$; 
see Example~\ref{ex:log_TEP_A}, 
Theorem~\ref{thm:big_mirror_symmetry_Ybar}. 
\\ 
$(\tcFA{X}, \nablaA{X}, \pairingA{X})$ 
& 
see Example~\ref{ex:AmodelTEP}. 
\\
$(\cFA{X}^\times,\nablaA{X},\pairingA{X})$ 
& A-model TEP structure; 
see Example~\ref{ex:AmodelTEP}, 
Theorems~\ref{thm:mirrorsymmetryforYbar},~\ref{thm:mirrorsymmetryforXbar}.  
\\ 
 $(\cFA{X}, \nablaA{X}, \pairingA{X})$ 
& A-model \logDTEP structure; 
see Theorem~\ref{thm:big_mirror_symmetry_Ybar}.  
\\ 
$(\FA{X},\nablaA{X}, \pairingA{X})$ 
& A-model $\log$-cTEP structure; see Example~\ref{ex:log_cTEP_A}. 
\\ 
$\bP_{\rm A}$ 
& canonical opposite module for the A-model TEP structure; 
\\ 
& see Example \ref{ex:opposite_A}. 
\\ 
$\sfP_{\rm A}$ 
& canonical opposite module for the A-model 
$\log$-cTEP structure; 
\\
& see Example~\ref{ex:opposite_A_log-cTEP}. 
\\ 
$\Fock_{{\rm A},X}$ & A-model Fock sheaf;  
see Definition~\ref{def:GW-wave}.  
\\
\midrule
$(\cMB, D)$
& the base of the B-model \logDTEP structure; 
\\ 
& 
see \eqref{eq:MBsm_coordinates} and 
Proposition-Definition~\ref{pro-def:2d_log_TEP}.
\\ 
$\cMB^\times$ 
& $\cM_{\rm B} \setminus D$; 
the base of the B-model TEP structure; see \eqref{eq:MBsm_coordinates}, 
Definition~\ref{def:BmodelconformalTEP}. 
\\ 
$\cMB^\circ$ 
& $\cMB\setminus \{y_1 = -1/27\}$; see Theorem~\ref{thm:unfolding}. 
\\ 
$(\cMBbig,\Dbig)$
& the base of the big B-model \logDTEP structure; 
see Theorem~\ref{thm:unfolding}; 
\\
&  
this contains $\cMB^\circ$ but not $\cMB$. 
\\
$(\cMCY,\DCY)$ 
& $\left(\Proj(3,1), \{0,-\frac{1}{27}\}\right)$; 
the base of $\cFCY$, $\sFCY$ 
(\S\ref{sec:six}) and 
$\Fock_{\CY}$ (\S\ref{sec:Fock_fd}).
\\
$\cMCY^\circ$ 
& $\Proj(3,1) \setminus \{-\frac{1}{27}\}$; the base of $\Fock_{\CY}^\circ$ (\S\ref{sec:conformal_limit_Fock}).  
\\  
 $(\cFB^\times,\nablaB,\pairingB )$ 
& the B-model TEP structure 
(Definition~\ref{def:BmodelconformalTEP}) 
with base $\cMB$.  
\\ 
$(\cFGKZ, \nabla^{\rm GKZ})$
& the GKZ system, isomorphic to $(\cFB^\times,\nablaB)$; 
see \S\ref{sec:GKZ}. 
\\ 
$(\cFB,\nablaB,\pairingB)$ 
& the B-model \logDTEP structure 
with base $(\cMB,D)$; 
\\
& see Proposition-Definition~\ref{pro-def:2d_log_TEP}.  
\\
$(\cFBbig, \nablaB,\pairingB)$ 
& the big B-model \logDTEP structure; 
see Theorem~\ref{thm:unfolding}. 
\\ 
$(\FBbig, \nabla,(\cdot,\cdot))$ 
& the big B-model $\log$-cTEP structure;
see Example~\ref{ex:log_cTEP_B}.  
\\ 
$(\cFCY,\nabla, (\cdot,\cdot))$ & the restriction of $\cFB$ to $\cMCY$;
see \S\ref{sec:six}. \\ 
$H$,~$\widebar{H}$,~$\Hvec$ & 
vector bundles (of rank $6$, $3$, $2$) on $\cMCY$ obtained from $\cFCY$;
see \S\ref{sec:six}--\ref{sec:two}.
\\ 
$\bP_\LR$,~$\bP_\con$,~$\bP_\orb$ & unique 
Deligne-extension-compatible opposite modules 
for $\cFB$ \\ 
& near $y=0$, $y=-\frac{1}{27}$, and $y=\infty$ respectively; 
see Proposition~\ref{pro:LR_conifold_orbifold}.
\\
$\Fock_{\rm B}$ & the B-model Fock sheaf over $\cMBbig$; 
see Definition~\ref{def:Fock_B}. \\ 
$\Fock_{\CY}$ & the finite-dimensional Fock sheaf over $\cMCY
= \Proj(3,1)$; see Definition~\ref{def:Fock_CY}. 
\\
\midrule
$\cF^\divideontimes$
& restriction of a sheaf $\cF$ 
over $\cM \times \C$ to $\cM\times \C^\times$; 
see Notation \ref{notation:x}. 
\\ \bottomrule
\end{tabular}
\end{table}

\vfill 

\renewcommand{\arraystretch}{1} 
\normalsize 

\clearpage 

\section{Notation and Preliminaries} 

\subsection{Bases for Cohomology and Orbifold Cohomology}
\label{sec:bases}
Let $X$ denote one of $\cX$, $\cXbar$, $Y$ and $\Ybar$. 
We fix bases $\{\phi_0,\phi_1,\dots,\phi_N\}$ 
for the (orbifold) cohomology $H_X$ of $X$ such that 
\begin{itemize} 
\item $\phi_0$ is the identity class 
\item writing $r$ for the dimension of the (untwisted) degree two cohomology 
group $H^2(X)$ -- so that $r=0$, $1$, $1$, $2$ respectively for $X=\cX$, $\cXbar$, $Y$ and $\Ybar$ -- the classes $\phi_1,\dots,\phi_r$ form a nef integral basis of $H^2(X)$; 
\item if $X$ is compact, $\{\phi^0,\dots,\phi^N\}$ is a basis dual to 
$\{\phi_0,\dots,\phi_N\}$ with respect to the (orbifold) 
Poincar\'{e} pairing. 
\end{itemize} 
More specifically we choose the following explicit bases. 
Let $\HX$ denote the Chen--Ruan orbifold
cohomology $H^\bullet_\orb(\cX;\C)$. We fix the basis:
\begin{align*}
  \phi_0 = \unit_0 &&
  \phi_1 = \unit_{\frac{1}{3}} &&  
  \phi_2 = \unit_{\frac{2}{3}} 
\end{align*}
for $\HX$, where $\unit_k$, $k \in \{0,\frac{1}{3},\frac{2}{3}\}$,
denotes the fundamental class of the component of the inertia stack
$\cIX$ corresponding to the element $\exp(2 \pi \iu k) \in \mu_3$.
The age of $\unit_k$ is $3k$. 

Let $\HXbar$ denote the Chen--Ruan orbifold cohomology
$H^\bullet_\orb(\cXbar;\C)$.  Let $h \in H^2(\cXbar;\C)$ denote the first
Chern class of the line bundle $\cO(1) \to \cXbar$, and regard
elements of $H^\bullet(\cXbar;\C)$ as orbifold cohomology classes via
the canonical inclusion of $\cXbar$ into the inertia stack $\cIXbar$.
We fix the basis:
\begin{align*}
  \phi_0 = \unit_0 &&
  \phi_1 = h &&
  \phi_2 = h^2 &&
  \phi_3 = h^3 &&
  \phi_4 = \unit_{\frac{1}{3}} &&  
  \phi_5 = \unit_{\frac{2}{3}} 
\end{align*}
for $\HXbar$, where $\unit_0$, $\unit_{\frac{1}{3}}$,
$\unit_{\frac{2}{3}}$ denote the fundamental classes of the components
of the inertia stack $\cIXbar$, ordered so that the age of $\unit_k$
is $3k$.
The orbifold Poincar\'{e} pairing on $\cXbar$ satisfies
\[
(1,h^3)= (h,h^2) = \left(\unit_{\frac{1}{3}},\unit_{\frac{2}{3}}\right)=\frac{1}{3} 
\]
and that all other pairings among basis elements are zero.

Let $\HY = H^\bullet(Y;\C)$.  Let $\unit \in \HY$ denote the unit
class, let $\pi\colon Y \to \Proj^2$ denote the projection, and let $h \in
H^2(Y;\C)$ denote the first Chern class of the line bundle $\pi^\star
\cO(1) \to Y$.  We fix the basis:
\begin{align*}
  \phi_0 = \unit &&
  \phi_1 = h &&
  \phi_2 = h^2
\end{align*}
for $\HY$. 

Let $\HYbar = H^\bullet(\Ybar;\C)$.  Let $h_1, h_2 \in H^2(\Ybar)$ be
such that, regarding $\Ybar$ as the projective compactification of the
line bundle $\cO(-3) \to \Proj^2$, the zero section
is Poincar\'e dual to $h_2-3h_1$, the infinity section is dual to
$h_2$ and the fiber is dual to $h_1$.  
With these conventions, $h_1$ and $h_2$ are rays of the
K\"ahler cone for $\Ybar$.  We fix the basis:
\begin{align*}
  \phi_0 = 1 &&
  \phi_1 = h_1 &&
  \phi_2 = h_2 &&
  \phi_3 = h_1^2 &&
  \phi_4 = h_1(h_2-3h_1) &&
  \phi_5 = h_1^2 h_2
\end{align*}
for $\HYbar$.

\subsection{Gromov--Witten Invariants}
\label{sec:GW}

Let $X$ denote one of $\cX$, $\cXbar$, $Y$, $\Ybar$.  
Let $X_{g,n,d}$ denote the moduli space of $n$-pointed
genus-$g$ stable maps to $X$ of degree $d \in H_2(X;\Q)$.  If $X$ is a
smooth algebraic variety (so $X = Y$ or $X = \Ybar$) then there are
evaluation maps:
\begin{align*}
  & \ev_k \colon X_{g,n,d} \to X
  & k \in \{1,2,\ldots,n\} \\
  \intertext{If $X$ is an orbifold (so $X = \cX$ or $X = \cXbar$) then there are
    evaluation maps to the rigidified cyclotomic inertia stack:}
  & \ev_k \colon X_{g,n,d} \to \overline{\cI}(X)
  & k \in \{1,2,\ldots,n\} \\
  \intertext{and a canonical isomorphism $H^\bullet(\overline{\cI} X;\Q) \cong H^\bullet(\cI X;\Q)$, so we get cohomological pullbacks}
  & \ev_k^\star \colon H_X \to H^\bullet(X_{g,n,d};\C)
  & k \in \{1,2,\ldots,n\} 
\end{align*}
that behave like the pullbacks via evaluation maps: see
\cite{AGV} or \cite{CCLT}*{\S2.2.2}.  Write:
\begin{align}
  \label{eq:primary_correlator}
  \corr{\alpha_1,\ldots,\alpha_n}^{X}_{g,n,d}
  =
  \int_{[X_{g,n,d}]^{\text{vir}}}
  \prod_{k=1}^{k=n} \ev_k^\star(\alpha_k) 
\end{align}
where $\alpha_1,\ldots,\alpha_n \in H_X$; the integral denotes cap product with
the virtual fundamental class \citelist{\cite{Behrend--Fantechi} 
\cite{Li--Tian}} followed by push-forward (in homology) along the
map from $X_{g,n,d}$ to a point; 
if $X$ is non-compact (i.e.~$X=\cX$ or $X=Y$), 
we require that $d\neq 0$ or that at least one of the classes 
$\alpha_1,\dots,\alpha_n$ has a compact support, 
so that the integral \eqref{eq:primary_correlator} is well-defined\footnote
{Here we use the property that the evaluation maps for $\cX$ and $Y$ are proper; this will also appear in \S\ref{subsec:QC}.}. 
The right-hand
side of \eqref{eq:primary_correlator} is a rational number 
when $\alpha_1,\dots,\alpha_n$ are rational, called a
\emph{Gromov--Witten invariant} of $X$.

Let $\psi_1,\ldots,\psi_n \in H^2\big(X_{g,n,d};\Q\big)$ denote the
universal cotangent line classes \cite{AGV}*{\S8.3}.  Write:
\begin{align}
  \label{eq:descendant_correlator}
  \corr{\alpha_1 \psi_1^{i_1},\ldots, \alpha_n
    \psi_n^{i_n}}^{X}_{g,n,d}
  =
  \int_{[X_{g,n,d}]^{\text{vir}}}
  \prod_{k=1}^{k=n} \ev_k^\star(\alpha_k) \cup \psi_k^{i_k}
\end{align}
where $\alpha_1,\ldots, \alpha_n \in H_X$; $i_1, \ldots, i_n$ are non-negative
integers; the integral denotes cap product with the virtual
fundamental class followed by push-forward to a point; 
and as before we insist that $d \ne 0$ or that one of $\alpha_1,\dots,\alpha_n$ 
has compact support.  
The right-hand side
of \eqref{eq:descendant_correlator} is a rational number 
when $\alpha_1,\dots,\alpha_n$ are rational, called a
\emph{gravitational descendant} of $X$.

Consider now the morphism $p_m \colon X_{g,m+n,d} \to \cMbar_{g,m}$
that forgets the map, forgets the last $n$ marked points, forgets any
stack structure at the marked points (if $X$ is an orbifold), and then
stabilises the resulting prestable curve.  Let $\psi_{m|i} \in
H^2(X_{g,n+m,d};\Q)$ denote the pullback along $p_m$ of the $i$th
universal cotangent line class on $\cMbar_{g,m}$.  Write:
\begin{multline}
  \label{eq:ancestor_correlator}
  \corr{\alpha_1 \bar{\psi}_1^{i_1},\ldots, \alpha_m
    \bar{\psi}^{i_m}:\beta_1,\ldots,\beta_n}^{X}_{g,m+n,d} \\
  =
  \int_{[X_{g,m+n,d}]^{\text{vir}}}
  \prod_{k=1}^{k=m} 
  \Big(\ev_k^\star(\alpha_k) \cup \psi_{m|k}^{i_k}\Big)
  \cdot
  \prod_{l=m+1}^{l=m+n} \ev_{l}^\star (\beta_{l-m})  
\end{multline}
where $\alpha_1,\ldots,\alpha_m \in H_X$; 
$\beta_1,\ldots,\beta_n \in H_X$;
$i_1,\ldots,i_m$ are non-negative integers; the integral denotes cap
product with the virtual fundamental class followed by push-forward to
a point; and as before we insist that $d \ne 0$ or that one of $\alpha_1,\dots,\alpha_n$ 
has compact support.
We insist also that $m \geq 3$ if $g=0$ and that $m \geq 1$ if $g=1$,
so that the map $p_m$ is well-defined.  The right-hand side of
\eqref{eq:ancestor_correlator} is a rational number, called an
\emph{ancestor invariant} of $X$.

\subsection{Gromov--Witten Potentials}

Let $X$ denote one of $\cX$, $\cXbar$, $Y$, $\Ybar$.  Let $r$ be the
rank of $H_2(X)$.  In \S\ref{sec:bases} we fixed a basis
$\phi_0,\ldots,\phi_N$ for $H_X$ such that $\phi_1,\ldots,\phi_r$ is a
nef basis for $H^2(X;\C) \subset H_X$.  For $d \in H_2(X;\Q)$, write:
\[
Q^d = Q_1^{d_1} \cdots Q_r^{d_r}
\]
where $d_i = d\cdot \phi_i$.  Let $t^0,\ldots,t^N$ be the co-ordinates
on $H_X$ defined by the basis $\phi_0,\ldots,\phi_N$, so that 
$t \in H_X$ satisfies $t = t^0 \phi_0 + \ldots + t^N \phi_N$.  
The \emph{genus-$g$ Gromov--Witten potential} is:
\begin{equation} 
\label{eq:GW_potential} 
F^g_X = \sum_{d \in \NE(X)} 
\sum_{n = 0}^\infty \frac{Q^d}{n!}
\corr{\vphantom{\big\vert}t, \ldots,t}^X_{g,n,d}
\end{equation} 
where the first sum is over the set $\NE(X)$ of degrees of effective
curves in $X$.  This is a generating function for genus-$g$
Gromov--Witten invariants.  The genus-$g$ Gromov--Witten potential is
\emph{a priori} a formal power series in variables $Q_i$ and $t^j$:
\begin{align*}
  & F^g_X \in \C[\![Q_1,\ldots,Q_r]\!][\![t^0,\ldots,t^N]\!] \\
  \intertext{but the Divisor Equation \cite{AGV}*{Theorem~8.3.1}
    implies that:}
  & F^g_X \in \C[\![t^0,Q_1 e^{t^1},\ldots,Q_r e^{t^r},t^{r+1}\ldots,t^N]\!] \\
  \intertext{It thus makes sense to set $Q_1 = \cdots = Q_r = 1$, obtaining an
    element:}
  & \left. F^g_X \right|_{Q_1 = \cdots = Q_r = 1} \in \C[\![t^0,e^{t^1},\ldots,e^{t^r},t^{r+1}\ldots,t^N]\!]
\end{align*}
There is an open region $\VA{X} \subset H_X$ of the form:
\begin{align}
  \label{eq:LRL_nbhd}
  \begin{cases}
    |t^i| < \epsilon_i & \text{$i=0$ or $r<i\leq N$} \\
    \Re t^i <-M_i & 1 \leq i \leq r
  \end{cases}
\end{align}
such that all of the power series $\left.F^g_X \right|_{Q_1 = \cdots =
  Q_r = 1}$, $g \geq 0$, converge on $\VA{X}$
\cite{Coates--Iritani:convergence}. In the rest of this paper we will
write $F^g_X$ for the analytic function $\left. F^g_X \right|_{Q_1 =
  \cdots = Q_r = 1}$ defined on $\VA{X}$, so that:
\begin{align*}
  & F^g_X(t) = \sum_{d \in \NE(X)} \sum_{n = 0}^\infty 
  \frac{e^{d\cdot t^{(2)}}}{n!}  
  \corr{\vphantom{\big\vert}t', \ldots,t'}^X_{g,n,d}
  & t \in \VA{X}
\end{align*}
where we write $t^{(2)} = \sum_{i=1}^r t^i \phi_i$ for  
the degree two part of $t$ and $t' = t -t^{(2)}$. 
We refer to the limit point:
\begin{align*}
  \begin{cases}
    t^i =0  & \text{$i=0$ or $r<i\leq N$} \\
    \Re t^i \to -\infty & 1 \leq i \leq r
  \end{cases}
\end{align*}
as the \emph{large-radius limit point} for $X$.  

\subsection{Quantum Cohomology} 
\label{subsec:QC} 
Let $X$ be one of $\cX$, $\cXbar$, $Y$, $\Ybar$. 
When $X$ is compact, i.e.~$X$ is either $\Ybar$ or $\cXbar$, 
we define the quantum product $*$ on $H_X$ by the formula: 
\begin{equation} 
\label{eq:quantum_product} 
(\phi_i * \phi_j, \phi_k)_X = 
\parfrac{^3 F^0_X}{t^i \partial t^j\partial t^k} (t) 
\bigg|_{Q_1 = \cdots = Q_r= 1} 
\end{equation}  
where the pairing $(\cdot,\cdot)_X$ on the left-hand side is the 
(orbifold) Poincar\'{e} pairing. 
The product $*$ defines a family of commutative ring structures 
on $H_X$ parameterized by $t \in \VA{X}$, called the 
\emph{quantum cohomology} of $X$. 
When $X$ is not compact, i.e.~$X$ is either $\cX$ or $Y$, 
we define the quantum product by using the push-forward by 
the last marked point 
\[
\phi_i * \phi_j = 
\sum_{d\in \NE(X)} \sum_{n=0}^\infty 
\frac{e^{d \cdot t^{(2)}} }{n!} 
(\ev_{n+3})_{\star}\left(\ev_1^\star (\phi_i) \ev_2^\star (\phi_j) \prod_{k=1}^n 
\ev_{k+2}^\star (t') \cap [X_{0,n+3,d}]^{\rm vir}\right) 
\]
Here we write $t^{(2)} = \sum_{i=1}^r t^i \phi_i$ 
for the degree two part of $t$ and $t' = t - t^{(2)}$. 
This makes sense because the evaluation map $\ev_{n+3}$ is proper. 
The quantum products for $\cX$ and $Y$ can be 
obtained as the limits of the quantum products for $\cXbar$ 
and $\Ybar$ respectively. 
We have 
\begin{align*} 
\lim_{\Re(t^1) \to -\infty} \iota^\star(\phi_i *_t^{\cXbar} \phi_j ) 
& = \iota^\star (\phi_i) *_{\iota^\star (t)}^{\cX} \iota^\star( \phi_j )\\ 
\lim_{\Re(t^2) \to -\infty} 
\iota^\star(\phi_i *_t^{\Ybar} \phi_j ) 
& = \iota^\star(\phi_i)*_{\iota^\star( t)}^{Y} \iota^\star(\phi_j) 
\end{align*} 
where $\iota$ denotes the natural inclusion of $\cX$ into $\cXbar$ 
or $Y$ into $\Ybar$ and $*^X_t$ denotes the quantum product 
of $X$ at the parameter $t$. 
In particular, the quantum products for $\cX$ and $Y$ are also convergent 
on regions of the form \eqref{eq:LRL_nbhd}. 

\subsection{Dubrovin Connection, Fundamental Solution and 
$J$-Function} 
Let $X$ be one of $\cX$, $\cXbar$, $Y$, $\Ybar$. 
Write $c_1(X) = \rho^1 \phi_1 + \cdots
  + \rho^r \phi_r$.  
Define the \emph{Euler vector field} $E$ on $H_X$ by: 
  \begin{equation}
    \label{eq:Euler_field}
    E = {t^0 \parfrac{}{t^0}} + \sum_{i=1}^r \rho^i \parfrac{}{t^i} +
    \sum_{i=r+1}^N \big(1 - \textstyle\frac{1}{2}{\deg \phi_i} \big) 
    t^i \parfrac{}{t^i}
  \end{equation}
  and the \emph{grading operator} $\mu \colon H_X \to H_X$ by:
  \[
  \mu(\phi_i) = \left(\textstyle\frac{1}{2} \deg \phi_i -
    \textstyle\frac{1}{2} \dim_{\C} X\right) \phi_i
  \]
  Let $\pi \colon \VA{X} \times \C \to \VA{X}$ 
denote projection to the first factor. 
The \emph{Dubrovin connection}\footnote
{The sign of $z$ is often flipped in the literature: see e.g.~\cite{Iritani:integral}.} 
is a meromorphic flat connection $\nabla$ on 
$\pi^\star \big(T\VA{X}\big) \cong 
H_X\times (\VA{X}\times \C)$, 
defined by: 
\begin{align} 
\label{eq:Dubrovin_connection} 
\begin{split} 
\nabla_{\parfrac{}{t^i}} &=  \parfrac{}{t^i} 
- \frac{1}{z}\big(\phi_i {\ast}\big)
\qquad \qquad 0 \leq i \leq N \\
\nabla_{z \parfrac{}{z}} &=  z \parfrac{}{z} + \frac{1}{z} 
\big({E\ast}\big) + \mu 
\qquad  \text{where $z$ is the co-ordinate on $\C$} 
\end{split} 
\end{align}
The Dubrovin connection defines the A-model TEP structure 
in Example \ref{ex:AmodelTEP} below. 

The Dubrovin connection admits the following fundamental 
solution $L(t, -z)$ 
\citelist{\cite{Givental:equivariantGW}*{Corollary 6.2}
\cite{Iritani:integral}*{Proposition 2.4}}. 
Suppose that $X$ is compact, i.e.~$X$ is either $\cXbar$ or $\Ybar$. 
Then the fundamental solution is an $\End(H_X)$-valued function 
of $(t,z) \in \VA{X}\times \C^\times$ defined by 
\begin{equation}
\label{eq:GW_fundsol}
L(t,-z) \alpha = e^{t^{(2)}/z}  
\alpha + \sum_{\substack{
d\in \NE(X), n \ge 0 \\ (n,d) \neq (0,0)}} \sum_{i=0}^N 
\frac{e^{d\cdot t^{(2)}}}{n!} 
\corr{\phi^i,t',\dots,t', \frac{e^
{t^{(2)}/z}\alpha}{z-\psi}}_{0,n+2,d} \phi_i 
\end{equation} 
which satisfies the differential equation: 
\begin{equation} 
\label{eq:L_diffeq}
\nabla_{\parfrac{}{t^i}} \left( L (t,-z) \alpha \right)= 0 \qquad 
i=0,\dots, N 
\end{equation} 
and preserves the (orbifold) Poincar\'{e} pairing 
\begin{align} 
\label{eq:L_unitarity}
(L(t,- z) \alpha, L(t,z)\beta)_X = (\alpha,\beta)_X && \text{for all $\alpha,\beta \in H_X$.} 
\end{align}
Givental's $J$-function is defined to be 
\begin{align} 
\label{eq:J_function} 
\begin{split} 
J(t, -z) & = L(t,-z)^{-1} \unit \\
& = e^{-t^{(2)}/z}\left( \unit - \frac{t'}{z} + 
\sum_{\substack{d\in \NE(X), n \ge 0 \\ (n,d) \neq (0,0), (1,0)}}
\sum_{i=0}^N \frac{e^{d\cdot t^{(2)}}}{n!} 
\corr{t',\dots,t', \frac{\phi^i}{z(z+\psi)}}_{0,n+1,d} \phi_i\right). 
\end{split} 
\end{align} 
When $X$ is non-compact, the fundamental solution $L(t,z)$ 
and the $J$-function $J(t,z)$ are defined similarly, replacing $\{\phi^i\}$ above with the dual basis 
of $\{\phi_i\}$ in the compactly-supported cohomology group.  See \cite{Iritani:Ruan}*{\S 2.5} for more details. 

\subsection{Descendant Potentials and Ancestor Potentials} \label{subsec:descendant_and_ancestor}

Let $X$ be one of $\cX$, $\cXbar$, $Y$, $\Ybar$.  Let
$\phi_0,\ldots,\phi_N$ be the basis for $H_X$ defined in
\S\ref{sec:bases}.  Let $(t_0,t_1,t_2,\ldots)$ be an infinite sequence of
elements of $H_X$, and write $t_n= t_n^0 \phi_0 + \cdots + t_n^N
\phi_N$.  Set $\bt(z) = \sum_{n=0}^\infty t_n z^n 
\in H_X[\![z]\!]$. 
The \emph{genus-$g$ descendant potential} of $X$ is: 
\begin{equation} 
\label{eq:genus_g_descendant_potential}
\cF^g_X = \sum_{d \in \NE(X)} \sum_{n = 0}^{\infty} 
\frac{Q^d}{n!} \corr{\bt(\psi_1),\dots,\bt(\psi_n)}_{g,n,d}^X. 
\end{equation} 
This is a formal power series\footnote{\label{footnote:powerseries} 
See \cite{Coates--Iritani:convergence}*{\S2.5} for a precise 
statement.}  in variables $Q_i$, $1 \leq i \leq r$, and
$t^j_n$, $0 \leq j \leq N$, $0 \leq n < \infty$; it is a
generating function for genus-$g$ gravitational descendants of $X$.
The \emph{total descendant potential} is:
\begin{equation*} 
\cZ_X = \exp \Bigg(\sum_{g=0}^\infty \hbar^{g-1} \cF^g_X\Bigg). 
\end{equation*}
This is a formal power series\footnotemark[\value{footnote}] in
variables $\hbar$, $\hbar^{-1}$, $Q_i$, $1 \leq i \leq r$, and
$t^j_n$, $0 \leq j \leq N$, $0 \leq n < \infty$; it is a
generating function for all gravitational descendants of $X$. 

Let $t \in H_X$, let $(a_0, a_1, a_2, \ldots)$ be an infinite sequence
of elements of $H_X$, and write:
\begin{align*}
  t = t^0 \phi_0 + \cdots + t^N \phi_N
  &&
  a_n = a_n^0 \phi_0 + \cdots + a_n^N \phi_N  
  && 
  \ba(z) = \sum_{i=0}^\infty a_n z^n \in H_X[\![z]\!] 
\end{align*}
The \emph{genus-$g$ ancestor potential} of $X$ is:
\begin{equation}
  \label{eq:ancestor}
  \bar{\cF}^g_{X} = 
  \sum_{d \in \NE(X)}
  \sum_{n=0}^\infty
  \sum_{m=0}^\infty
  \frac{Q^d}{n!m!}
 \corr{\ba(\bar\psi_1),\ldots, \ba(\bar\psi_m): 
    \overbrace{t,\ldots, t}^n}^X_{g,m+n,d}
\end{equation}
This is a formal power series\footnotemark[\value{footnote}] in
variables $Q_i$, $1 \leq i \leq r$; $t^j$, $0 \leq j \leq
N$; and $a^k_n$, $0 \leq k \leq N$, $0 \leq n < \infty$.  It
is a generating function for genus-$g$ ancestor invariants of~$X$.
The \emph{total ancestor potential} is:
\begin{equation*}
  \cA_{X} = \exp \Bigg(\sum_{g=0}^\infty \hbar^{g-1} \bar{\cF}^g_{X}\Bigg)
\end{equation*}
This is a formal power series\footnotemark[\value{footnote}] in
variables $\hbar$; $\hbar^{-1}$; $Q_i$, $1 \leq i \leq r$; 
$t^j$,
$0 \leq j \leq N$; and $a^k_n$, $0 \leq k \leq N$, 
$0 \leq n < \infty$.  It is a generating function for all ancestor
invariants of~$X$.  

\subsection{TEP Structures and \logDTEP Structures} 

\begin{definition}
  \label{def:TEP} 
  Let $\cM$ be a complex manifold.  Let $z$ denote the standard
  co-ordinate on $\C$, let $(-)\colon \cM \times \C \to \cM \times
  \C$ be the map sending $(t, z)$ to $(t,-z)$, and let $\pi \colon
  \cM\times \C \to \cM$ be the projection.  A \emph{TEP structure}
  $\big(\cF, \nabla, (\cdot,\cdot)_{\cF}\big)$ with base $\cM$ consists of a
  locally free $\cO_{\cM\times \C}$-module $\cF$ of rank $N+1$, a
  meromorphic flat connection:
  \[ 
  \nabla \colon \cF \to 
  \big(\pi^\star\Omega_{\cM}^1 \oplus \cO_{\cM\times \C} z^{-1}dz\big)
  \otimes_{\cO_{\cM}} \cF(\cM\times \{0\}) 
  \]  
  and a non-degenerate pairing:
  \[
  (\cdot,\cdot)_{\cF} \colon 
  (-)^\star \cF \otimes_{\cO_{\cM\times \C}} \cF \to \cO_{\cM\times \C}
  \]
  which satisfies:
  \begin{align}
\label{eq:TEP-properties} 
    \begin{split} 
      \big((-)^\star s_1,s_2\big)_{\cF} 
      & = (-)^\star \big((-)^\star s_2, s_1\big)_{\cF}   \\ 
      d \big((-)^\star s_1, s_2\big)_{\cF}  
      & = \big((-)^\star \nabla s_1, s_2\big)_{\cF} 
      + \big((-)^\star s_1,\nabla s_2\big)_{\cF}  
    \end{split} 
  \end{align}
  for local sections $s_1\in \cF((-)^\star V)$, $s_2 \in \cF(V)$, 
  where $V \subset \cM\times \C$ is an open subset. 
Here $\cF(\cM\times \{0\})$ denotes the
  sheaf of sections of $\cF$ with poles of order at most $1$ along the
  divisor $\cM\times \{0\} \subset \cM \times \C$.
\end{definition} 

\begin{definition} 
\label{def:logDTEP}
Let $D \subset \cM$ be a normal crossing divisor. 
A \emph{\logDTEP structure} with base $(\cM,D)$ 
is a tuple $(\cF, \nabla, (\cdot,\cdot)_{\cF})$ 
consisting of a locally free sheaf $\cF$ of rank $N+1$ over $\cM$, 
a meromorphic flat connection $\nabla$ 
\[
\nabla \colon \cF \to \Omega^1_{\cM\times \C}(\log Z) 
\otimes_{\cO_{\cM\times \C}}
\cF(\cM\times \{0\}) 
\]
where $Z = (\cM \times \{0\}) \cup (D \times \C)$ is a normal crossing 
divisor in $\cM \times \C$,  
and a non-degenerate pairing 
\[
(\cdot,\cdot)_{\cF} \colon 
(-)^\star \cF \otimes_{\cO_{\cM\times \C}} \cF \to \cO_{\cM\times \C}
\]
which satisfies the same properties \eqref{eq:TEP-properties} as 
TEP structure.  
Here $\Omega^1_{\cM\times \C}(\log Z)$ 
denotes the sheaf of differential forms with 
logarithmic singularities along $Z$. 
\end{definition} 

\begin{remark} 
  The notion of TEP structure is due to Hertling \cite{Hertling:ttstar}: `TEP' stands
  for Twister, Extension, and Pairing.  This gives us a
  co-ordinate-free language in which to discuss mirror symmetry.  More
  precisely, a TEP structure in our sense is what Hertling would call
  a TEP($0$)-structure; for us all TEP structures have weight zero.  A
  \logDTEP structure is a TEP structure with logarithmic
  singularities; cf.~Reichelt's notion of \logDtrTLEP structure
  \cite{Reichelt}*{Definition~1.8}.  When $D=\varnothing$, a \logDTEP structure is the
  same thing as a TEP structure.
\end{remark} 

\begin{definition} 
\label{def:miniversal} 
A \logDTEP structure $(\cF,\nabla,(\cdot,\cdot))$ with base $(\cM,D)$ 
is said to be \emph{miniversal} if for every point $x\in \cM$, there 
exists a section $\xi$ of $\cF|_{z=0}$ on a neighbourhood $U_x$ 
of $x$ such that the map 
\begin{align*} 
\Theta_\cM(\log D) & \longrightarrow \cF|_{z=0} \\
X & \longmapsto z \nabla_X \xi
\end{align*} 
is an isomorphism over $U_x$. Here $\Theta_\cM(\log D)$ denotes 
the sheaf of logarithmic vector fields, that is, the subsheaf of $\Theta_\cM$ 
consisting of vector fields tangent to the divisor $D$. (When $\cM$ has an orbifold singularity at $x$, we take 
$U_x$ above to be a uniformizing chart near $x$.) 
By taking $D = \varnothing$, this also defines miniversality for TEP structures. 
\end{definition} 

\begin{example}[A-model TEP structure] 
  \label{ex:AmodelTEP}
An important class of examples of miniversal TEP structures is
provided by the quantum cohomology of a smooth algebraic variety 
or orbifold $X$.  We will need this only when $X$ is one of $\cX$,
$\cXbar$, $Y$, $\Ybar$, but the definition here makes sense whenever the 
genus-zero Gromov--Witten potential $F^0_X$ defines an analytic 
function on a region $\VA{X} \subset H_X$ of the form 
\eqref{eq:LRL_nbhd}.  
The Dubrovin connection \eqref{eq:Dubrovin_connection} 
defines a TEP structure $(\tcFA{X}, 
\nablaA{X}, \pairingA{X})$ with base $\VA{X}$, 
where 
\begin{itemize} 
\item $\tcFA{X}$ is the locally free sheaf corresponding to the 
trivial $H_X$-bundle over $\VA{X} \times \C$; 
\item $\nablaA{X}$ is the Dubrovin connection; 
\item $\pairingA{X}$ is the pairing induced by the orbifold 
Poincar\'e pairing. 
\end{itemize} 
When $X$ is a smooth variety, the Divisor Equation 
implies that the Dubrovin connection descends 
to $\cMA{X}^\times  \times \C$, where
\[
\cMA{X}^\times := \VA{X}/2\pi \iu H^2(X,\Z).     
\]
and $2\pi\iu H^2(X,\Z)$ acts on $\VA{X}$ by translation. 
When $X$ is an orbifold, and we interpret $H^2(X,\Z)$ 
as the sheaf cohomology\footnote{An element of $H^2(X,\Z)$ 
corresponds to an isomorphism class of 
a topological orbi-line bundle on $X$.}
of the topological stack $X$, we again have 
that the Dubrovin connection descends to 
$\cMA{X}^\times\times \C$. 
In this case, $2\pi \iu H^2(X,\Z)$ acts on the vector bundle 
$H_X \times (\VA{X} \times \C) \to (\VA{X}\times \C)$ 
by the so-called \emph{Galois action}, which is also nontrivial 
in the fibre direction. 
We refer the reader to 
\cite[Proposition 2.3]{Iritani:integral} for details; see also Example~\ref{ex:log_TEP_A}. 
The TEP structure $(\tcFA{X},\nablaA{X}, \pairingA{X})$ described above descends, via the Galois action, to a TEP structure $(\cFA{X}^\times, \nablaA{X}, \pairingA{X})$ with base $\cMA{X}^\times$.  This is the
\emph{A-model TEP structure}. 
\end{example}

\begin{example}[A-model \logDTEP structure] 
\label{ex:log_TEP_A} 
The quotient space $\cMA{X}^\times$ 
has a natural partial compactification defined by our choice of nef basis for
$H^2(X)$; this compactification, which we denote by 
$\cMA{X}$, adds a normal crossing divisor $\DA{X}$ at infinity. 
The A-model TEP structure extends to the partial 
compactification to give a miniversal \logDTEP structure 
\[
\left( \cFA{X}, \nablaA{X}, \pairingA{X}\right) 
\]
with base $\big(\cMA{X},\DA{X}\big)$, 
called the \emph{A-model \logDTEP structure}: 
see \cite{Iritani:Ruan}*{\S 2.2}. 
Concretely, this amounts to the following.  
Suppose first that $X$ is a smooth variety. 
Recall that we have fixed 
a basis $\phi_0,\ldots,\phi_N$ for $H_X$ such that $\phi_0 \in H^0_X$
is the unit class and that $\phi_1,\ldots,\phi_r$ is a nef basis for $H^2(X)$ 
in \S\ref{sec:bases}.  
This defines co-ordinates $t^0,\ldots,t^N$ on $H_X$.  Set
$q_i = e^{t^i}$, $1 \leq i \leq r$, and consider $\C^{N+1} = \C \times
\C^r \times \C^{N-r}$ with co-ordinates
$(t^0,q_1,\ldots,q_r,t^{r+1},\ldots,t^N)$.  
The partial compactification $\cMA{X}$ 
is a neighbourhood of the origin in $\C^{N+1}$. 
The locally free sheaf $\cFA{X}$ is given by the trivial 
$H_X$-bundle over $\cMA{X}\times \C$. 
The divisor $\DA{X}$ is the locus $q_1q_2\cdots q_r =
0$, the pairing is as in Example~\ref{ex:AmodelTEP}, and the
meromorphic flat connection is:
\begin{eqnarray}
\label{eq:log_Dubrovin_connection} 
\begin{split} 
\nabla_{\parfrac{}{t^i}} & = \textstyle \parfrac{}{t^i} 
  - \frac{1}{z}\big(\phi_i {\ast}\big)
&& \text{$i=0$ or $r < i \leq N$} \\
\nabla_{q_i \parfrac{}{q_i}} & = \textstyle q_i \parfrac{}{q_i} 
  - \frac{1}{z}\big(\phi_i {\ast}\big)
&& \text{$1 \leq i \leq r$} \\
\nabla_{z \parfrac{}{z}} & = \textstyle z \parfrac{}{z} + \frac{1}{z} 
\big({E\ast}\big) + \mu \qquad 
&& 
\end{split} 
\end{eqnarray}
where (as before) $z$ is the standard co-ordinate on $\C$ and $E$ is
the Euler vector field:
\[
 E = {t^0 \parfrac{}{t^0}} + \sum_{i=1}^r \rho^i q_i\parfrac{}{q_i} +
    \sum_{i=r+1}^N \big(1 - \textstyle\frac{1}{2}{\deg \phi_i} \big) 
    \displaystyle 
    t^i \parfrac{}{t^i}
\]
When $X$ is an orbifold, $\cMA{X}$ has orbifold 
singularities along the divisor $\DA{X}$ 
and $\cFA{X}$ is defined as an orbi-sheaf over $\cMA{X}$. 
We shall describe the structure explicitly for $X= \cXbar
=\Proj(1,1,1,3)$ (and this is the only case we need). 
In this case we have co-ordinates $t^0,\dots,t^5$ on $H_X$ 
dual to the basis 
$\unit_0,h,h^2,h^3,\unit_{\frac{1}{3}}, \unit_{\frac{2}{3}}$ 
from \S\ref{sec:bases}. 
Set $q = e^{t^1}$ and 
consider the space $\C^6$ with co-ordinates 
$(t^0, \sqrt[3]{q} = e^{t^1/3}, t^2,t^3,t^4,t^5)$. 
By the Divisor Equation, 
the Dubrovin connection \eqref{eq:Dubrovin_connection} 
for $X= \cXbar$ 
induces a meromorphic flat connection of the form 
\eqref{eq:log_Dubrovin_connection}
on the trivial $H_X$-bundle 
over $V \times \C$, where $V$ is a small open neighbourhood 
of the origin in the $\C^6$. 
Let $\mu_3$ act\footnote{
The group $\mu_3$ here arises as $H^2(X,\Z)/H^2(|X|,\Z)$, where 
$|X|$ denotes the coarse moduli space of $X=\cXbar$. } 
on the trivial bundle 
$H_X \times (V \times \C) \to (V \times \C)$ 
by 
\begin{align*} 
\xi \cdot 
\left(\alpha, (t^0, \sqrt[3]{q}, t^2,t^3,t^4, t^5), z\right) 
= 
\left(G(\xi) \alpha, (t^0, \xi^{-1} \sqrt[3]{q}, t^2, t^3, \xi t^4, 
\xi^{-1} t^5),z \right)  
\end{align*} 
where $G(\xi)$ is the endomorphism of $H_X$ represented by 
the matrix 
\[
\left( 
\begin{array}{c|cc} 
I & & \\  \hline 
& \xi & 0 \\ 
& 0 & \xi^{-1} 
\end{array} 
\right) 
\] 
in the basis $\unit_0,h,h^2,h^3,\unit_{\frac{1}{3}}, \unit_{\frac{2}{3}}$ and 
$I$ is the identity matrix of size $4$. 
The $\mu_3$-action here preserves the Dubrovin connection and the 
orbifold Poincar\'e pairing. The base of the A-model \logDTEP structure 
is given by:
\[
(\cMA{X},\DA{X}) = \left(
\left[V/\mu_3\right], \left[\{\sqrt[3]{q}=0\}/\mu_3\right]\right)
\]
$\cFA{X}$ is the orbi-sheaf corresponding to the orbi-vector 
bundle:
\[
[(H_X \times (V\times \C))/\mu_3] \to [V/\mu_3]  \times \C
\] 
$\nablaA{X}$ is the meromorphic flat connection 
induced by the Dubrovin connection, 
and $\pairingA{X}$ is the pairing on $\cFA{X}$ 
induced by the orbifold Poincar\'e pairing. 
\end{example} 

\begin{notation} 
As found in the notation 
$\cMA{X}^\times = \cMA{X}\setminus \DA{X}$, 
$\cFA{X}^\times$, we often put a cross ``$\times$'' to denote 
spaces (or sheaves) obtained by deleting normal crossing divisors 
from other spaces (or by restricting to the complement of these divisors). 
\end{notation} 

\subsection{From TEP Structures to trTLEP Structures via  
Opposite Modules}
\label{sec:trTLEP}

Hertling has defined the notion of a trTLEP structure with
base $\cM$.  This consists of a TEP structure $\big(\cF, \nabla,
(\cdot,\cdot)_{\cF}\big)$ with base $\cM$ together with certain
extension data for $\cF$,~$\nabla$,~and $(\cdot,\cdot)_{\cF}$ across
$\cM \times \{\infty\} \subset \cM \times\Proj^1$. 
We encode these extension data using a subsheaf of 
$\pi_\star(\cF|_{\cM\times \C^\times})$ of semi-infinite rank called an \emph{opposite module}
(Definition~\ref{def:opposite}). 
A TEP structure equipped with an opposite module is 
equivalent to a trTLEP structure, so the reader who prefers sheaves of finite rank 
can translate statements about opposite modules into statements about trTLEP structures. 
We will use both languages since 
opposite modules fit well with Givental quantization.

\begin{definition}[Hertling \cite{Hertling:ttstar}*{\S 5.2}]
  \label{def:logDtrTLEP}
  Let $\cM$ be a complex manifold and let $(-)\colon \cM \times
  \Proj^1 \to \cM \times \Proj^1$ be the map sending $(t, z)$ to
  $(t,-z)$.  A trTLEP structure $\big(\cE,\nabla,(\cdot,\cdot)_{\cE}
  \big)$ with base $\cM$ consists of:
  \begin{itemize}
  \item a locally free sheaf $\cE$ on $\cM \times \Proj^1$ such that 
  $\cE|_{\{y\}\times \Proj^1}$ is a free $\cO_{\Proj^1}$-module 
  for each $y\in \cM$; 
  \item a meromorphic flat connection $\nabla$ on $\cE$ with poles 
  along $Z= \cM\times \{0\} \cup \cM\times \{\infty\}$: 
  \[
  \nabla \colon \cE \to \Omega^1_{\cM\times \Proj^1}(\log Z) 
  \otimes \cE(\cM\times \{0\}) ;
  \]
  \item a non-degenerate pairing:
  \[
  (\cdot,\cdot)_{\cE} \colon (-)^\star \cE \otimes_{\cO_{\cM\times \Proj^1}} \cE
  \to \cO_{\cM\times \Proj^1}
  \]
  which satisfies:
  \begin{align*} 
    \begin{split} 
      \big((-)^\star s_1,s_2\big)_{\cE} 
      & = (-)^\star \big((-)^\star s_2, s_1\big)_{\cE}   \\ 
      d \big((-)^\star s_1, s_2\big)_{\cE}  
      & = \big((-)^\star \nabla s_1, s_2\big)_{\cE} 
      + \big((-)^\star s_1,\nabla s_2\big)_{\cE}  
    \end{split} 
  \end{align*}
  for $s_1,s_2 \in \cE$. 
  \end{itemize}
  Note that $\nabla$ has logarithmic singularities 
  along $\cM\times \{\infty\}$, and that
  the restriction of a trTLEP structure $\big(\cE,\nabla,(\cdot,\cdot)_{\cE}
  \big)$ to $\cM \times \C$ is a TEP structure.
\end{definition}

\begin{remark} 
  The `L' in `trTLEP structure' stands for logarithmic (along $\cM\times\{\infty\})$ and the `tr' stands for trivial (along $\{y\}\times \Proj^1$). 
  Our trTLEP structure is what Hertling would call a trTLEP(0) structure: for
  us all trTLEP structures are of weight zero.
\end{remark}

\begin{notation}
  \label{notation:x}
  Let $\cF$ be a sheaf on $\cM \times \C$.  We write $\cF^\divideontimes$ for
  the restriction $\cF\big|_{\cM \times \C^\times}$. 
\end{notation}

\begin{definition}
  \label{def:symplectic}
  Let $\big(\cF, \nabla, (\cdot,\cdot)_{\cF}\big)$ be a TEP structure
  with base $\cM$.  The pairing $(\cdot,\cdot)_{\cF}$ induces a
  symplectic pairing:
  \begin{align*}
    \Omega: \pi_\star \cF^\divideontimes \otimes_{\cO_\cM}
    \pi_\star \cF^\divideontimes & \longrightarrow \cO_{\cM} \\
    s_1 \otimes s_2 & \longmapsto \Res_{z=0} \big((-)^\star s_1,s_2\big)_{\cF}
    \, dz
  \end{align*}
  The connection $\nabla$ induces an operator 
  \[
  \nabla \colon \pi_\star \cF^\divideontimes \to (\Omega^1_{\cM} 
\oplus \cO_{\cM} dz)  
\otimes_{\cO_{\cM}} \pi_\star \cF^\divideontimes 
  \]
which preserves the symplectic pairing $\Omega$.
\end{definition}

\begin{definition} 
  \label{def:opposite} 
Let $\big(\cF, \nabla, (\cdot,\cdot)_{\cF}\big)$ 
be a TEP structure with base $\cM$.
Recall that $\pi_\star \cF^\divideontimes$ is a 
$\pi_\star (\cO_{\cM \times \C^\times})$-module. 
This contains a locally free $\pi_\star(\cO_{\cM\times \C})$-module 
$\bF := \pi_\star \cF$ as a subsheaf.  
Let $\bP$ be a locally free 
$\pi_\star (\cO_{\cM \times(\aroundinfinity)})$-submodule 
of $\pi_\star \cF^\divideontimes$. 
We say that:
  \begin{enumerate}
  \item $\bP$ is opposite to $\bF$ if $\pi_\star \cF^\divideontimes =
    \bF \oplus \bP$;
  \item $\bP$ is isotropic if $\Omega(s_1,s_2) = 0$ for all $s_1, s_2 \in \bP$;
  \item $\bP$ is parallel if $\nabla_X \bP \subset \bP$ for all $X \in
    T\cM$;
  \item $\bP$ is homogeneous if $\nabla_{z \partial_{z}} \bP \subset \bP$.
  \end{enumerate}
  An \emph{opposite module} for $\big(\cF, \nabla,
  (\cdot,\cdot)_{\cF}\big)$ is a locally free $\pi_\star (\cO_{\cM
    \times (\aroundinfinity)})$-submodule $\bP$ of $\pi_\star
  \cF^\divideontimes$ such that $\bP$ is opposite to $\bF$, isotropic,
  parallel, and homogeneous.
\end{definition} 

\begin{example}[the A-model trTLEP structure and 
canonical opposite module] 
  \label{ex:opposite_A}
  Recall that the A-model TEP structure $\cFA{X}^\times$ 
  with base $\cMA{X}^\times$ is 
  given as the quotient of the trivial $H_X$-bundle 
  over $\VA{X} \times \C$ 
  by the Galois action (see Example~\ref{ex:AmodelTEP}). 
  The Dubrovin connection on the trivial $H_X$-bundle 
  over $\VA{X}\times \C$ extends to 
  the trivial $H_X$-bundle over 
  $\VA{X} \times \Proj^1$ with only logarithmic poles 
  along $\VA{X} \times \{\infty\}$, 
  and yields a trTLEP structure with base $\VA{X}$. 
  This trTLEP structure descends, via the Galois action, to give a trTLEP structure with base $\cMA{X}^\times$ called the \emph{A-model 
    trTLEP structure}.  This is an extension of the A-model 
  TEP structure. 
  
  The corresponding opposite module can be described as follows. 
  Consider the sheaf  
  \[
  \widetilde{\bP}_{\rm A} = 
  z^{-1} H_X \otimes \pi_\star (\cO_{\VA{X}\times (\aroundinfinity)}) 
  \subset H_X\otimes \pi_\star(\cO_{\VA{X}\times \C^\times}) 
  \]
  over $\VA{X}$, where $\pi \colon \VA{X} \times \Proj^1
  \to \VA{X}$ is the projection. 
  The sheaf $\widetilde{\bP}_{\rm A}$ gives an opposite module 
  for the TEP structure $(\tcFA{X},\nablaA{X},\pairingA{X})$ 
  introduced in Example \ref{ex:AmodelTEP}.  
  It descends to an opposite module $\bP_{\rm A}$ 
  of the A-model TEP structure via the Galois action. 
  We call $\bP_{\rm A}$ the \emph{canonical opposite module} 
  of the A-model TEP structure. 
  Alternatively, $z \bP_{\rm A}$ can be described as the push-forward 
  along $\pi$ of the restriction of the A-model trTLEP structure 
  to $\cMA{X}^\times \times (\aroundinfinity)$. 
\end{example}

\begin{remark} 
  The subsheaf $\bF = \pi_\star \cF$ of $\pi_\star \cF^\divideontimes$ in the
  above definition gives a \emph{variation of semi-infinite Hodge
    structure} (VSHS) in the sense of Barannikov \cite{Barannikov}.  It is
  maximally isotropic with respect to $\Omega$ and satisfies the
  Griffiths transversality condition $\nabla_X \bF \subset z^{-1} \bF$ for $X
  \in T\cM$.  It also satisfies $\nabla_{z^2\partial_z} \bF \subset
  \bF$.  See \cite{CIT:wall-crossings, Iritani:Ruan} for an
  exposition.
\end{remark} 

We now recall how an opposite module $\bP$ for a TEP structure
$\big(\cF, \nabla, (\cdot,\cdot)_{\cF}\big)$ with base $\cM$
determines a trTLEP structure with base $\cM$.  To give an extension
of the locally free sheaf $\cF^\divideontimes$ on $\cM \times \C^\times$ to a
locally free sheaf on $\cM \times \C$ is the same thing as to give a
locally free $\pi_\star(\cO_{\cM \times \C})$-submodule $\bF$ of
$\pi_\star \cF^\divideontimes$ such that $\pi_\star \cF^\divideontimes = \bF
\otimes_{\pi_\star (\cO_{\cM \times \C})} \pi_\star(\cO_{\cM \times
  \C^\times})$.  The submodule $\bF$ consists of those sections which
extend holomorphically to $z=0$; in the situation at hand the
extension is given by the TEP structure $\cF$ itself, so $\bF =
\pi_\star \cF$.  To give an extension of $\cF^\divideontimes$ to a locally
free sheaf over $\cM \times (\aroundinfinity)$ is the same thing as to
give a locally free $\pi_\star(\cO_{\cM \times
  (\aroundinfinity)})$-submodule $\bF'$ of $\pi_\star \cF^\divideontimes$ such
that $\pi_\star \cF^\divideontimes = \bF' \otimes_{\pi_\star(\cO_{\cM \times
    (\aroundinfinity)})} \pi_\star \cO_{\cM \times \C^\times}$.  The
submodule $\bF'$ consists of those sections which extend
holomorphically to $z=\infty$; in the situation at hand we take $\bF'
= z \bP$.  Thus the opposite module $\bP$ determines an extension of
the locally free sheaf $\cF$ on $\cM \times \C$ to a locally free
sheaf $\cE$ on $\cM \times \Proj^1$.  The restriction
$\cE|_{\{y\}\times \Proj^1}$ is a free $\cO_{\Proj^1}$-module because
$\bP_y$ is opposite to $\bF_y$: the space of global sections of
$\cE|_{\{y\}\times \Proj^1}$ is $z\bP_y \cap \bF_y$, and the
projection $z\bP_y \cap \bF_y \xrightarrow{\scriptscriptstyle \, \sim
  \,} z\bP_y/\bP_y$ gives a trivialization of $\cE|_{\{y\}\times
  \Proj^1}$ (see \cite{Iritani:Ruan}*{Lemma 3.8}).  The pairing
$(\cdot,\cdot)_{\cF}$ on $\cF$ extends holomorphically and
non-degenerately across $z = \infty$ to a pairing on $\cE$ because
$\bP$ is isotropic.  The connection $\nabla$ on $\cF$ induces a
connection on $\cE$ with logarithmic singularity along $z=\infty$
because $\bP$ is homogeneous and parallel.  Thus an opposite module $\bP$ for the
TEP structure $\big(\cF, \nabla, (\cdot,\cdot)_{\cF}\big)$ determines
a trTLEP structure $(\cE,\nabla,(\cdot,\cdot)_{\cE})$.  Conversely, a
trTLEP structure $(\cE,\nabla,(\cdot,\cdot)_{\cE})$ determines an
opposite module $\bP = z^{-1} \pi_\star(\cE|_{\cM\times
  (\Proj^1\setminus \{0\})})$ of the underlying TEP structure.  We
have thus proved:
\begin{proposition} 
\label{pro:opposite_trTLEP}  There is a bijective correspondence between opposite modules for a TEP structure  
and trTLEP structures which extend that TEP structure. 
\end{proposition} 

Let $\bP$ be an opposite module for a TEP structure 
$(\cF,\nabla,(\cdot,\cdot)_\cF)$ with base $\cM$. 
It defines a locally free sheaf $z\bP/\bP$ of rank $N+1 = \rank \cF$ 
on $\cM$. 
This is identified with the restriction to $z = \infty$ of the 
corresponding trTLEP structure $\cE$, and is equipped with a flat connection 
\[
\nabla \colon z\bP/\bP \to \Omega_{\cM}^1 \otimes (z\bP/\bP)
\]
since $\nabla_X$ with $X\in T\cM$ preserves $\bP$. 
Therefore $z\bP/\bP$ defines a flat vector bundle over $\cM$. 
The trivialization $\cE|_{\{y\}\times \Proj^1} 
\cong \cO_{\Proj^1} \otimes (z\bP_y/\bP_y)$ 
discussed before Proposition \ref{pro:opposite_trTLEP} 
yields an isomorphism: 
\begin{equation} 
\label{eq:flat_trivialization} 
\cF \cong \pi^\star( z\bP/\bP) 
\end{equation} 
\begin{definition} 
\label{def:flat_trivialization}
We call the isomorphism \eqref{eq:flat_trivialization} the 
\emph{flat trivialization} associated to the opposite module $\bP$. 
Over a simply-connected base, we can take a flat frame of $z\bP/\bP$ 
that yields a trivialization of  $\cF$. 
This is also called a \emph{flat trivialization}. 
\end{definition} 

\begin{remark} 
The flat trivialization gives rise to a Frobenius-type structure. 
See Hertling \cite{Hertling:ttstar}*{Theorem~5.7} 
and Coates--Iritani--Tseng \cite{CIT:wall-crossings}*{Proposition~2.11}. 
\end{remark} 

\begin{example}
The flat trivialization associated to the canonical opposite module 
$\bP_{\rm A}$ in
Example~\ref{ex:opposite_A} corresponds to the standard trivialization of
$\tcFA{X}$ in Example~\ref{ex:AmodelTEP}.
\end{example}

\section{The Mirror Landau--Ginzburg Model for $\Ybar$ and $\cXbar$}

Mirror symmetry associates to each toric variety a
\emph{Landau--Ginzburg model} \citelist{\cite{Givental:toric}
  \cite{Hori--Vafa}}.  In this context, a Landau--Ginzburg model
consists of:
\begin{itemize}
\item a holomorphic family $\pi\colon Z\rightarrow \cMB^\times$ of
  algebraic tori;
\item a function $W\colon Z\rightarrow \C$, called the {\it
    superpotential};
\item a section $\omega$ of the relative canonical sheaf
  $K_{Z/\cMB^\times}$ which gives a holomorphic volume form $\omega_q$
  on each fibre $Z_q=\pi^{-1}(q)$.  
\end{itemize}
The base space $\cMB^\times$ of the family is called the {\it B-model
  moduli space}.  In this section we define the
Landau--Ginzburg model that corresponds to $\Ybar$ under mirror
symmetry (\S\ref{sec:LG_Ybar}) and use it to construct a TEP
structure, called the \emph{B-model TEP structure}
(\S\ref{sec:B_model_TEP}).  We formulate mirror symmetry for $\Ybar$
as an equivalence of TEP structures (\S\ref{sec:mirror_symmetry_TEP})
between the A-model TEP structure -- or rather its restriction to the
small quantum cohomology locus $H^2(\Ybar) \subset H^\bullet(\Ybar)$ --
and the B-model TEP structure defined from the Landau--Ginzburg model.
We then give an alternative construction of the B-model TEP structure,
in terms of the so-called \emph{GKZ system}, which is useful in
computations (\S\ref{sec:GKZ}).  The B-model TEP structure is defined
over a non-compact base $\cMB^\times$, but computations with the GKZ system
allow us to define an extension of the B-model TEP structure over a
toric partial compactification $\cMB$ of $\cMB^\times$, such
that the extension has logarithmic singularities along the
partially-compactifying divisor (\S\ref{sec:2d_log_TEP}).

\begin{remark}
  The Landau--Ginzburg model that we consider in this section provides a mirror to the \emph{small} quantum cohomology of $\Ybar$: an open subset in the base $\cMB^\times$
  corresponds to a relatively open subset in the small quantum
  cohomology locus $H^2(\Ybar) \subset \HYbar$.  We will construct a mirror to
  \emph{big} quantum cohomology, over a larger base $\cMB$, in
  \S\ref{sec:enlarge_base} below.  
\end{remark}

\subsection{The Mirror Landau--Ginzburg Model} 

\label{sec:LG_Ybar}

The toric variety $\Ybar$ is the GIT quotient of $\C^5$ by
$(\C^\times)^2$ where $(\C^\times)^2$ acts via the inclusion
\begin{equation}
  \label{eq:Ybar_weights}
  (\C^\times)^2 \hookrightarrow (\C^\times)^5, \quad 
  (s,t) \mapsto (s,s,s,s^{-3}t,t). 
\end{equation}
Consider the map $\pi$ given by restricting the dual of this inclusion
\begin{align*}
  \pi\colon (\C^\times)^5 & \longrightarrow (\C^\times)^2 \\
  (w_1,\dots,w_5) &\longmapsto (w_1w_2w_3w_4^{-3},w_4w_5) 
\end{align*}
to the following open subset of $(\C^\times)^2$:
\begin{equation}
  \label{eq:open_subset}
  \Big\{(y_1,y_2) \in (\C^\times)^2 : \textstyle y_1 \neq {-\frac{1}{27}}
  \Big\}
\end{equation}
The superpotential $W$ is:
\[
W=w_1+w_2+w_3+w_4+w_5
\]
and the holomorphic volume form $\omega_y$ on the fibre
$Z_y=\pi^{-1}(y_1,y_2)$ is:
\[
\omega_y =\frac{d\log w_1\wedge \cdots \wedge d\log w_5}
{d\log y_1\wedge d\log y_2}
\] 
We delete the locus $y_1 = {-\frac{1}{27}}$ in \eqref{eq:open_subset}
because critical points of $W|_{Z_y}$ escape to infinity there: see
\cite[\S3.1]{CIT:wall-crossings}.

We now consider a partial compactification $\cMB^\times$ of the open subset
\eqref{eq:open_subset} and extend the Landau--Ginzburg model
considered above to a Landau--Ginzburg model over this larger base.
Consider the secondary fan (Figure~\ref{fig:Ybarsecondaryfan}) for the
toric variety $\Ybar$; this records the weight data
\eqref{eq:Ybar_weights} defining the toric variety $\Ybar$.
\begin{figure}[t!]
\centering 
\includegraphics[bb= 210 644 384 726]{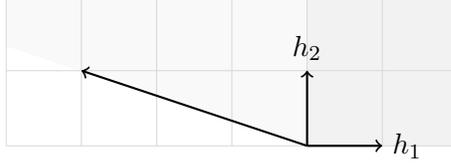} 
\caption{The secondary fan for ${\Ybar}$.}
\label{fig:Ybarsecondaryfan}
\end{figure}
The toric orbifold $\cMB$ associated to the secondary fan
gives a partial compactification of the open set
\eqref{eq:open_subset}.  The two cones in the secondary fan define
toric co-ordinate patches on $\cMB$.  Let $y_1, y_2$ be
the co-ordinates dual respectively to $h_1$ and to $h_2$, and let
$\fry_1$ and $\fry_2$ be the co-ordinates dual respectively to
$h_2-3h_1$ and to $h_2$.  The two co-ordinate systems are related by:
\begin{equation}
  \label{eq:MBsm_coordinates}
  \begin{aligned}
    &\fry_1=y_1^{-1/3}& \hspace{3.5cm}& y_1 = \fry_1^{-3} \\
    &\fry_2= y_1^{1/3} y_2 && y_2 = \fry_1 \fry_2
  \end{aligned}
\end{equation}
Note that $\cMB$ is an orbifold with a $\Z/3\Z$ quotient
singularity at $(\fry_1,\fry_2)=0$, and $(\fry_1,\fry_2)$ is a
uniformizing system near this orbifold point.

We define the base $\cMB^\times$ of our new Landau--Ginzburg model to be
\[
\cMB^\times := \cMB \setminus \overline{\left\{(y_1,y_2) \in \C^2 : \text{$y_1 y_2 = 0$ or $y_1 = {-\tfrac{1}{27}}$}  \right\}} 
\]
Taking $w_1,w_2,w_5$ as co-ordinates on the
fibre $Z_y \subset Z$, we see that:
\begin{equation}
  \label{eq:fibre_coordinates}
  \begin{aligned}
    W_y & =w_1+w_2+\frac{y_1y_2^3}{w_1w_2w_5^3}+\frac{y_2}{w_5}+w_5 \\
    &=w_1+w_2+\frac{\fry_2^3}{w_1w_2w_5^3}+\frac{\fry_1\fry_2}{w_5}+w_5 \\
    \omega_y &= d \log w_1 \wedge d \log w_2 \wedge d \log w_5 
  \end{aligned}
\end{equation}
We can therefore extend the family of tori $\pi$, the superpotential
$W_y$, and the section $\omega$ across the locus $\{\fry_1=0\}$.  These
extensions define a new Landau--Ginzburg model with base $\cMB^\times$.

\begin{notation}
  We refer to the point $(y_1,y_2) = (0,0)$ as the \emph{large-radius
    limit point} and to the point $(\fry_1,\fry_2) = (0,0)$ as the
  \emph{orbifold point}.  We refer to the locus $y_1 =
  {-\frac{1}{27}}$ as the \emph{conifold locus}.
\end{notation}

\begin{remark}
  The right-hand cone in Figure~\ref{fig:Ybarsecondaryfan} is
  canonically identified with the K\"ahler cone of $\Ybar$, and under
  this identification the cohomology classes $h_1$, $h_2$ defined in
  \S\ref{sec:bases} are as pictured.
\end{remark}

\begin{remark}
  The Landau--Ginzburg model described here is discussed in more
  detail in \cite[\S\S2--3]{CIT:wall-crossings}.
\end{remark}

\subsection{The B-Model TEP structure}
\label{sec:B_model_TEP}

We now use the Landau--Ginzburg model $\big(\pi:Z \to \cMB^\times,W,\omega\big)$ to
define a TEP structure, called the B-model TEP structure.  This is
almost the same as the discussions in
\citelist{\cite{CIT:wall-crossings}*{\S2.5.1}
  \cite{Iritani:integral}*{\S3.3}}, with the main difference\footnote
{A minor difference is that the sign of $z$ is flipped compared to 
\cite{CIT:wall-crossings,Iritani:integral}.} being
that there the (equivalent) language of variations of semi-infinite
Hodge structure is used.  Consider the locally free sheaf $\cR$ over
$\cMB^\times \times \C^\times$ with fibre over $(y,z)$ equal to the
relative cohomology group $H^3 \big(Z_y, \{x \in Z_y : 
\Re(W_y(x)/z)\gg 0\}\big)$.  This sheaf carries a flat Gauss--Manin connection
$\nabla^{\rm GM}$, and there is a distinguished global section of
$\cR$ given by:
\begin{equation*}
(y,z) \longmapsto \exp\big({-W_y}/z\big) \, \omega_y .
\end{equation*}
Let $\cO_{Z\times \C}$ denote the analytic structure sheaf.  
Consider the $\cO_{\cMB^\times \times \C}$-module 
$\cF^\times$ consisting of sections of $\cR$
of the form:
\begin{align*}
\Big[ f(x,z) \exp\big({-W}(x)/z\big) \, \omega\Big] 
&& \text{where $f(x,z) \in 
(\pi\times \id)_\star\cO_{Z\times \C}$} 
\end{align*}
such that, for each $z \in \C$, the function $x \mapsto f(x,z)$ is
algebraic on each fibre $Z_y$.  
The sheaf $\cF^\times$ is a locally free
extension of $\cR$ to $\cMB^\times \times \C$
\cite[Proposition~3.14]{Iritani:integral}.  The B-model TEP structure
will, roughly speaking, be the twist of $\cF^\times$ by a factor of
$z^{-3/2}$: this twist will ensure that the pairing on the B-model TEP
structure behaves correctly.

\begin{lemma}[see \protect{\cite[Lemma~2.19]{CIT:wall-crossings}}]
  The intersection pairing:
  \[
  I: H^3\big(Z_y, \{x \in Z_y : \Re(W_y(x)/z)\ll 0\}\big)
  \otimes
  H^3\big(Z_y, \{x \in Z_y : \Re(W_y(x)/z)\gg 0\}\big)
  \to 
  \C
  \]
  induces a pairing:
  \[
  I:(-)^\star \cF^\times \otimes \cF^\times 
\to (2 \pi \iu z)^3 \cO_{\cMB^\times \times \C}
  \]
\end{lemma}
\begin{proof}
  Observe that, on the one hand:
  \[
  I\bigg(\Big[f(x,{-z}) e^{W(x)/z} \, \omega\Big], 
  \Big[g(x,z) e^{{-W}(x)/z} 
\, \omega\Big] \bigg) \in \cO_{\cM \times \C^\times}
  \]
  and on the other hand:
  \begin{multline*}
    I\bigg(\Big[f(x,{-z}) e^{W(x)/z} \, \omega\Big], 
    \Big[g(x,z) e^{{-W}(x)/z} \, \omega\Big] \bigg) = \\
    \sum_{\text{critical points $\sigma$}} 
    \Bigg(
    \int_{\Gamma_-(\sigma)} f(x,{-z}) e^{W(x)/z} \, \omega
    \Bigg) \cdot 
    \Bigg(
    \int_{\Gamma_+(\sigma)} g(x,z) e^{{-W}(x)/z} \, \omega
    \Bigg)
  \end{multline*}
  where 
  \begin{equation}
    \label{eq:Lefschetz}
    \begin{aligned}
      & \Gamma_+(\sigma) \in H_3\big(Z_y, \{x \in Z_y : \Re(W_y(x)/z)\gg
      0\}\big) \\
      & \Gamma_-(\sigma) \in H_3\big(Z_y, \{x \in Z_y : \Re(W_y(x)/z)\ll
      0\}\big)
    \end{aligned}
  \end{equation}
  are the Lefschetz thimbles given by upward (for $\Gamma_+$) and
  downward (for $\Gamma_-$) gradient flow of the function $x \mapsto \Re\big(
  \frac{W(x)}{z}\big)$ from the critical point $\sigma \in Z_y$ of
  $W|_{Z_y}$.  Stationary phase approximation gives that, as $z \to 0$
  in some angular sector:
  \begin{multline}
    \label{eq:stationary_phase}
    I\bigg(\Big[f(x,{-z}) e^{W(x)/z} \, \omega\Big], 
    \Big[g(x,z) e^{{-W}(x)/z} \, \omega\Big] \bigg) \sim \\
    \sum_{\text{critical points $\sigma$}} 
    (-2 \pi z)^{3/2} \Bigg(
    \frac{f(\sigma)}{\sqrt{\Hess_\sigma(W)}} + O(z)
    \Bigg) 
    \cdot
    (2 \pi z)^{3/2}
    \Bigg(
    \frac{g(\sigma)}{\sqrt{\Hess_\sigma(W)}} + O(z)
    \Bigg) 
  \end{multline}
  Thus the function 
  \[
  I\bigg(\Big[f(x,{-z}) e^{W(x)/z} \, \omega\Big], 
  \Big[g(x,z) e^{{-W}(x)/z} \, \omega\Big] \bigg) 
\in \cO_{\cM \times \C^\times}
  \]
  is in fact regular at $z=0$ and lies in $(2 \pi \iu z)^3 \cO_{\cMB^\times \times \C}$.
\end{proof} 

\begin{definition}
  \label{def:BmodelconformalTEP}
  The \emph{B-model TEP structure} $\big(\cFB^\times, \nablaB,
  \pairingB\big)$ consists of:
  \begin{itemize}
  \item the locally free $\cO_{\cMB^\times \times \C}$-module 
$\cFB^\times :=\cF^\times$;
  \item the flat connection:
    \[
    \nablaB \colon \cFB^\times \to
    \big(\pi^\star\Omega_{\cMB^\times}^1 \oplus 
\cO_{\cMB^\times\times \C} z^{-1}dz\big)
    \otimes_{\cO_{\cMB^\times\times \C}} \cFB^\times(\cMB^\times\times \{0\})
    \]
    defined by:
    \[
    \nablaB := \nabla^{\rm GM} - \frac{3}{2}\frac{dz}{z}
    \]
 \item the pairing:
    \[
    \pairingB := \frac{1}{(2 \pi \iu z)^3} I(\cdot,\cdot)
    \]
  \end{itemize}
\end{definition}

It is proven in \cite{Iritani:integral}*{\S3.3} that the B-model TEP
structure is, in fact, a TEP structure.

\begin{remark}
  The connection $\nablaB$ is compatible with the pairing $\pairingB$,
  whereas the connection $\nabla^{\rm GM}$ is compatible with the
  pairing $I(\cdot,\cdot)$.
\end{remark}

\subsection{Mirror Symmetry as an Isomorphism of TEP Structures}
\label{sec:mirror_symmetry_TEP}

Let $X$ denote $\cXbar$ or $\Ybar$.  Let
$\big(\cFA{X}^\times,\nablaA{X},\pairingA{X}\big)$ be the A-model TEP
structure for $X$, as defined in Example~\ref{ex:AmodelTEP}. 
This is a TEP structure with base $\cMA{X}^\times = 
\VA{X}/2\pi\iu H^2(X,\Z)$, where $\VA{X}\subset H_X$ is 
an open subset of the form \eqref{eq:LRL_nbhd}. 
Recall also that $\cMA{X}^\times = \cMA{X} \setminus \DA{X}$, 
where $(\cMA{X},\DA{X})$ is the base of the A-model \logDTEP 
structure in Example~\ref{ex:log_TEP_A}. 
With notation as in Example~\ref{ex:log_TEP_A}, we have 
\begin{align*} 
\cMA{\Ybar} & = \left\{
(t^0,q_1,q_2,t^3,t^4,t^5) \in \C^6: |t^i|<\epsilon_i, |q_i|<\epsilon_i
\right\} 
&& 
\DA{\Ybar}  = \{q_1q_2 = 0\}\\ 
\intertext{for $X= \Ybar$ and} 
\cMA{\cXbar} & = 
\Big[ \left \{(t^0, \sqrt[3]{q},t^2,t^3,t^4,t^5) \in \C^6: 
|t^i|<\epsilon_i, |\sqrt[3]{q}|<\epsilon_1\right\}\big/\mu_3\Big] 
&& 
\DA{\cXbar} = \{q=0\} 
\end{align*} 
for $X = \cXbar$. 

\begin{theorem}[Mirror Symmetry for $\Ybar$] 
\label{thm:mirrorsymmetryforYbar}
Let $(y_1,y_2)$ be the co-ordinates defined in \S\ref{sec:LG_Ybar}.  
Let $h_1, h_2 \in
H^2(\Ybar)$ be as in \S\ref{sec:bases}. 
There are real numbers $\epsilon_1, \epsilon_2 > 0$ 
such that if:
  \[
  U^\times = \big\{(y_1,y_2) \in \cMB^\times : |y_i| < \epsilon_i \big\}
  \]
and the map 
$\mir_{\Ybar} \colon U^\times \to \cMA{X}^\times$ 
is 
\[
\mir_{\Ybar}(y_1,y_2) = (0,y_1 e^{-3g(y_1)}, 
y_2 e^{g(y_1)}, 0,0,0), \qquad 
g(y_1) = \sum_{d=1}^\infty \frac{(3d-1)!}{(d!)^3}(-1)^{d+1} y_1^d   
\]
then there is an isomorphism of TEP structures 
    \begin{equation*}
      \left(\cFB^\times,\nablaB,\pairingB\right) \Big|_{U^\times \times \C} 
\cong 
      \mir_{\Ybar}^\star\bigg(\cFA{\Ybar},\nablaA{\Ybar},\pairingA{\Ybar}\bigg) 
    \end{equation*}
    where on the left we have the B-model TEP structure and on the
    right we have the A-model TEP structure.
\end{theorem} 

\begin{proof}
  This is an example of \cite[Conjecture~2.21]{CIT:wall-crossings},
  and is proven in \cite[\S3.2]{CIT:wall-crossings}.
\end{proof}

\begin{theorem}[Mirror symmetry for $\cXbar$]
  \label{thm:mirrorsymmetryforXbar}
  Let $(\fry_1,\fry_2)$ be the co-ordinates defined in
  \S\ref{sec:LG_Ybar}.  Let $h \in H^2(\cXbar)$
  be the first Chern class of $\cO(1)$ as in \S\ref{sec:bases}. 
  There are real numbers $\epsilon_1,\epsilon_2>0$ such that if: 
  \[
  U^\times = \big\{(\fry_1,\fry_2) \in \cMB^\times : 
|\fry_i|<\epsilon_i \big\}
  \] 
and the map 
$\mir_{\cXbar} \colon U^\times \to \cMA{\cXbar}^\times$ 
is 
\[
\mir_{\cXbar}(\fry_1,\fry_2) = (0,\fry_2,0,0,\frt(\fry_1),0), 
\qquad 
\frt(\fry_1) = \sum_{n=0}^\infty (-1)^n 
\frac{\prod_{j=0}^{n-1}(\frac{1}{3}+j)^3}{(3n+1)!}
\fry_1^{3n+1} 
\]
then there is an isomorphism of TEP structures:
\begin{equation*}
\bigg(\cFB^\times,\nablaB,\pairingB\bigg) \Big|_{U^\times \times \C} 
\cong  \mir_{\cXbar}^\star
\bigg(\cFA{\cXbar},\nablaA{\cXbar},\pairingA{\cXbar}\bigg)
\end{equation*}
where on the left we have the B-model TEP structure and on the right 
we have the A-model TEP structure.
\end{theorem}

\begin{proof}
This is an example of \cite[Conjecture~2.21]{CIT:wall-crossings},
and is proven in \cite[\S3.4]{CIT:wall-crossings} 
along $\fry_1=0$. 
A proof including the $\fry_1$-direction is given in 
\cite[Proposition 4.8]{Iritani:integral}. 
\end{proof}

\begin{remark} 
The map $\mir_{\Ybar}(y_1,y_2) = (0,q_1,q_2,0,0,0)$ 
is determined by the $I$-function \cite{Givental:toric} 
of $\Ybar$ 
via the asymptotics $I_{\Ybar}(y_1,y_2,z) 
= 1 + (h_1 \log q_1+h_2 \log q_2) /z + o(z^{-1})$: 
see \eqref{eq:IYbar}. 
On the other hand, $\mir_{\cXbar}(\fry_1,\fry_2)$ is determined 
by the extended $I$-function \cite{CCIT:mirror_theorem} 
of $\cXbar$ 
\[
I_{\cXbar}(\fry_1,\fry_2,z) = \sum_{k_1,k_2\ge 0} 
\fry_1^{k_1} \fry_2^{k_2+3h/z} 
\frac{\prod_{c\le 0, \fracp{c}=\fracp{\frac{k_2-k_1}{3}}} (h+cz)^3}
{\prod_{c\le \frac{k_2-k_1}{3}, \fracp{c}=\fracp{\frac{k_2-k_1}{3}}}(h+cz)^3}
\frac{1}{k_1! z^{k_1} \prod_{c=1}^{k_2}(3h + cz)} 
\unit_{\fracp{\frac{k_1-k_2}{3}}} 
\]
via the expansion $I_{\cXbar}(\fry_1,\fry_2,z)
= 1 + (\frt \unit_{\frac{1}{3}} + (\log q) h)/z + o(1/z)$. 
\end{remark} 

\subsection{The GKZ System and the B-Model TEP Structure}
\label{sec:GKZ}

We now give an alternative construction of the B-model TEP structure,
which is very convenient for calculations.  This
construction is in terms of the so-called GKZ system, due to
Gelfand--Kapranov--Zelevinsky \cite{GKZ}.

\begin{definition}
Let $\cD^z\subset \End_\C(\cO_{\cMB^\times\times \C})$ 
denote the subsheaf of the sheaf of differential
  operators on $\cMB^\times \times \C$ generated, as a sheaf of rings, by
  $\cO_{\cMB^\times \times \C}$ and $\{ z X : \text{$X$ is a vector field
    on $\cMB^\times$}\}$, where $z$ is the standard co-ordinate on $\C$. 
\end{definition}

\begin{remark}
\label{rem:Dmoduleisconnection} 
Let $\cE$ be a $\cD^z$-module.  The action of $(1/z) \cdot zX$ defines
a map
\[
\nabla_X \colon \cE \to z^{-1} \cE = \cE(\cMB^\times \times \{0\}) 
\]
When $\cE$ is a coherent $\cO_{\cMB^\times\times \C}$-module, 
one may view $\nabla_X$ as a flat connection 
in the direction of $\cMB^\times$ with poles along $z=0$. 
\end{remark}

\begin{definition}
  By \emph{the GKZ system} we mean the $\cD^z$-module $\cFGKZ$ on
  $\cMB^\times \times \C$ defined as follows.  Recall from
  \eqref{eq:MBsm_coordinates} that $\cMB^\times$ is covered by two
  co-ordinate patches $(y_1,y_2)$ and $(\fry_1,\fry_2)$ related by:
  \begin{align*}
    &\fry_1=y_1^{-1/3}&& y_1 = \fry_1^{-3} \\
    &\fry_2= y_1^{1/3} y_2 && y_2 = \fry_1 \fry_2
  \end{align*}
  Define charts $U_\LR^\times$ and $U_\orb^\times$ on $\cMB^\times$ by:
\begin{equation} 
\label{eq:charts}  
 U_\LR^\times = \Big\{(y_1,y_2) \in (\C^\times)^2: 
y_1 \ne {-\tfrac{1}{27}} \Big\} \qquad
U_\orb^\times = \Big\{(\fry_1,\fry_2)\in \C\times \C^\times : \fry_1^3 \ne {-27} \Big\}
\end{equation} 
Let $z$ denote the standard co-ordinate on $\C$.  Consider the 
left ideal $\cI_\LR\subset \cD^z|_{U_\LR^\times\times \C}$ generated by:
\begin{equation}
    \label{eq:GKZ_chart_1}
    \begin{aligned}
      & D_2(D_2-3D_1) - y_2, \\
      & D_1^3 - y_1(D_2-3D_1)(D_2-3D_1+z)(D_2-3D_1+2z)
    \end{aligned}
\end{equation}
where $D_1 = -z y_1 \partial_{y_1}$, 
$D_2 = -z y_2 \partial_{y_2}$.
Consider the left ideal $\cI_\orb\subset \cD^z|_{U_\orb^\times\times \C}$
generated by:
  \begin{equation}
    \label{eq:GKZ_chart_2}
    \begin{aligned}
      & \frD_2 \frd_1 - \fry_2, \\
      & (\frD_2-\fry_1\frd_1)^3 - 27 (\frd_1)^3
    \end{aligned}
  \end{equation} 
where $\frd_1 = -z \partial_{\fry_1}$, 
$\frD_2 = -z \fry_2\partial_{\fry_2}$. 
The ideals $\cI_\LR$ and $\cI_\orb$ coincide on the overlap 
$(U_\LR^\times \cap U_\orb^\times)\times \C$ and define a left ideal 
$\cI\subset \cD^z$ over $\cMB^\times \times \C$. 
The GKZ system $\cFGKZ$ 
is defined to be the $\cD^z$-module $\cD^z/\cI$. 
\end{definition}

\begin{definition}[Grading operator] 
  \label{def:GKZ_grading}
Define the \emph{Euler vector field} $E$ on $\cMB$ by: 
\begin{equation}
    \label{eq:Euler_GKZ}
    E = 2 y_2 \partial_{y_2}
    = 2 \fry_2 \partial_{\fry_2}
\end{equation} 
This matches up with the Euler vector field \eqref{eq:Euler_field} 
on the A-model under the mirror maps $\mir_{\cXbar}$, $\mir_{\Ybar}$. 
Consider the endomorphism $\Gr \in \End_\C(\cD^z)$ defined by 
the commutator: 
\begin{equation} 
\label{eq:GKZ_grading}
\Gr(P) = [z\partial_z + E, P] 
\end{equation} 
This preserves the GKZ ideal $\cI$ and induces an endomorphism $\Gr
\in \End_\C(\cFGKZ)$ of the GKZ system, called the \emph{grading
  operator}.
\end{definition}

Setting:
\begin{equation}
  \label{eq:GKZz}
  \begin{aligned}
    \nabla_{z \partial_{z}} = \Gr - E - \frac{3}{2} 
    & = \Gr + 2 z^{-1} D_2 - \frac{3}{2} \\
    & = \Gr + 2 z^{-1} \frD_2 - \frac{3}{2} 
  \end{aligned}
\end{equation}
defines a meromorphic connection on $\cFGKZ$ in the direction of
$z$.  Combining this with the connection defined in
Remark~\ref{rem:Dmoduleisconnection}, we obtain a meromorphic flat
connection on $\cFGKZ$:
\begin{equation}
  \label{eq:meromorphic_flat_connection_on_GKZ}
  \nabla \colon \cFGKZ \to 
  \Big(\pi^\star\Omega_{\cMB^\times}^1 \oplus \cO_{\cMB^\times \times \C}
  \textstyle \frac{dz}{z}\Big) 
  \otimes_{\cO_{\cMB^\times}} \cFGKZ\big(\cMB^\times\times \{0\}\big) . 
\end{equation}

\begin{remark}
  The GKZ system is a version of what is sometimes referred to as the
  \emph{Horn system}, homogenized by including the variable $z$.
\end{remark}

\begin{remark}
  Recall from Definition~\ref{def:TEP} that we consider TEP structures
  in the category of complex manifolds and holomorphic maps.  The
  A-model TEP structure is naturally a holomorphic object, as the
  structure constants of quantum cohomology are transcendental rather
  than algebraic functions.  The GKZ system and the B-model TEP
  structure can most naturally be defined in the algebraic category
  but, for simplicity of exposition, in this paper we will regard them
  as holomorphic objects.
\end{remark}

\subsubsection{The GKZ System is Isomorphic to the B-Model TEP
  Structure}

The B-model TEP structure $\big(\cFB^\times, \nablaB, \pairingB\big)$ defines
another $\cD^z$-module $\big(\cFB^\times, \nablaB\big)$ on $\cMB^\times \times
\C$, which we call the \emph{B-model $\cD^z$-module}.  Recall that
there is a distinguished global section of $\cFB^\times$:
\begin{equation}
  \label{eq:distinguished_section}
  (y,z) \longmapsto \exp\big({-W}_y/z\big) \, \omega_y .
\end{equation}
Oscillating integrals:
\[
\int_{\Gamma_+(\sigma)} e^{{-W}_y/z} \, \omega_y
\]
over the Gauss--Manin-flat cycles (Lefschetz thimbles)
$\Gamma_+(\sigma)$ defined in \eqref{eq:Lefschetz} are annihilated by
the differential operators \eqref{eq:GKZ_chart_1}, \eqref{eq:GKZ_chart_2},
where we take:
\begin{align*}
  & D_1 = - z \nablaB_{y_1 \partial_{y_1}}
& \frd_1 = -z \nablaB_{\partial_{\fry_1}} \\
  & D_2 = -z \nablaB_{y_2 \partial_{y_2}} 
& \frD_2 = - z \nablaB_{\fry_2 \partial_{\fry_2}}
\end{align*}
It is proven in \cite{Iritani:integral}*{\S 4} that we have a $\cD^z$-module 
isomorphism: 
\[
\varphi \colon (\cFGKZ, \nabla^{\rm GKZ}) 
\overset{\cong}{\longrightarrow} (\cFB^\times, \nablaB) 
\]
defined by sending the distinguished section $1 \in \cFGKZ$
to the distinguished section \eqref{eq:distinguished_section} of $\cFB^\times$. 

\subsubsection{The Pairing on the GKZ System}
We can use the $\cD^z$-module isomorphism between the GKZ system and
the B-model $\cD^z$-module to define a pairing on the GKZ system:
\[
(\cdot,\cdot)_{\rm GKZ} \colon 
({-})^\star \cFGKZ \otimes_{\cO_{\cMB^\times \times \C}} \cFGKZ
\to \cO_{\cMB^\times \times \C}
\]
by pulling back the pairing $\pairingB$ on the B-model $\cD^z$-module along
the isomorphism $\varphi$.  
This pairing can be computed using mirror symmetry: 
the isomorphisms in Theorem~\ref{thm:mirrorsymmetryforYbar}
and Theorem~\ref{thm:mirrorsymmetryforXbar} 
intertwine the pairings $\pairingB$ and $\pairingA{-}$; 
moreover the pairing $\pairingA{-}$ can be computed through 
Givental's $I$-function. 
For example if $f(z,y_1,y_2,D_1,D_2)$ and $g(z,y_1,y_2,D_1,D_2)$ 
are elements of the GKZ system defined near $(y_1,y_2) = (0,0)$,  
then their pairing can be written in terms of the A-model 
pairing: 
\begin{multline*}
\Big( (-)^\star 
f(z,y_1,y_2,D_1,D_2),g(z,y_1,y_2,D_1,D_2)\Big)_{\rm GKZ}  \\  
 = \Big( f(-z,y_1,y_2,  z\nabla_1, z \nabla_2) \unit,  
g(z,y_1,y_2, -z \nabla_1, -z\nabla_2) 
\unit \Big )_{{\rm A}, \Ybar} 
\end{multline*} 
where $\nabla_i 
= (\mir_{\Ybar}^\star \nablaA{\Ybar})_{y_i\partial_{y_i}}$ 
is the Dubrovin connection pulled back by the mirror map 
$\mir_{\Ybar}$. 
By applying the inverse $L(t,-z)^{-1}$ of the fundamental solution 
\eqref{eq:GW_fundsol} to the sections of the A-model TEP structure 
in the right-hand side and using the properties \eqref{eq:L_diffeq}, 
\eqref{eq:L_unitarity} of $L(t,-z)$ and the definition 
\eqref{eq:J_function} of the $J$-function, we find that 
the pairing equals 
\begin{multline*}
  \Big(f(-z,y_1,y_2,  zy_1\partial_{y_1}, 
  zy_2 \partial_{y_2}) 
  J(\mir_{\Ybar}(y_1,y_2), z), \\
  g(z,y_1,y_2, -z y_1\partial_{y_1}, -z y_2 \partial_{y_2}) 
  J(\mir_{\Ybar}(y_1,y_2), -z) \Big )_{\Ybar} 
\end{multline*}
The mirror theorem of Givental \cite{Givental:toric} 
says that $J(\mir_{\Ybar}(y_1,y_2), -z)$ equals the 
cohomology-valued power series $I_{\Ybar}(y_1,y_2, -z)$: 
\begin{equation}
  \label{eq:IYbar}
  I_{\Ybar}(y_1,y_2, -z) = 
  \sum_{d_1,d_2 \geq 0}
  \frac{y_1^{d_1-h_1/z} y_2^{d_2-h_2/z}}
  {\prod_{m=1}^{d_1} (h_1-m z)^3
    \prod_{m=1}^{d_2} (h_2-m z)
  }
  \frac
  {\prod_{m =-\infty}^0 (h_2 - 3 h_1 - m z)}
  {\prod_{m =-\infty}^{d_2-3d_1}(h_2 - 3 h_1- m z)}
\end{equation}
Here we expand the right-hand side as a Taylor series in the
(nilpotent) cohomology classes $h_1$, $h_2$ from
\S\ref{sec:bases}; note that all but finitely many terms
in the infinite products on the right-hand side cancel. 
Hence we obtain 
\begin{multline} 
\label{eq:GKZ_pairing}
\Big((-)^\star f(z,y_1,y_2,D_1,D_2), 
g(z,y_1,y_2,D_1,D_2)\Big)_{\rm GKZ}  \\  
= \Big(f({-z},y_1,y_2, z y_1 \partial_{y_1},
z y_2\partial_{y_2}) I_{\Ybar}(y_1,y_2,z),
  g(z,y_1,y_2, -z y_1 \partial_{y_1}, -zy_2\partial_{y_2}) 
I_{\Ybar}(y_1,y_2,-z)\Big)_{\Ybar} 
\end{multline} 
Equations \eqref{eq:GKZ_pairing} and \eqref{eq:IYbar} together make
clear that the pairing:
\[
\left((-)^\star f(z,y_1,y_2,D_1,D_2), g(z,y_1,y_2,D_1,D_2) 
\right)_{\rm GKZ} 
\]
extends holomorphically across the locus $y_1y_2 = 0$ if 
$f$ and $g$ depend polynomially on $(y_1,y_2)$. 

\subsection{The B-Model \logDTEP Structure}
\label{sec:2d_log_TEP}

Recall that the B-model TEP structure has base $\cMB^\times$, 
which is the open subset of the toric variety $\cMB$ 
obtained by deleting the divisor $D = \overline{( y_1y_2 = 0 ) } 
\cup (y_1 ={-\frac{1}{27}})$ from $\cMB$. 
Here we construct a logarithmic extension of the B-model TEP structure 
across $D$, which we call the \emph{B-model \logDTEP structure}. 
This is a \logDTEP structure in the sense of 
Definition \ref{def:logDTEP}. 

\begin{proposition} 
\label{pro:GKZ}
The flat connection and the pairing of the GKZ system are 
described explicitly as follows. 
\begin{enumerate}
\item[(a)] In the chart near $(y_1,y_2) = (0,0)$ with $y_1 \ne
    {-\frac{1}{27}}$, writing $D_1 = -z y_1 \partial_{y_1}$, 
   $D_2 = -z y_2 \partial_{y_2}$, 
the GKZ system has basis:
    \begin{equation}
      \label{eq:GKZ_basis_LR_chart}
      1, D_2, D_2^2, D_2^3, D_1, (1+27y_1)D_1^2
    \end{equation}
    With respect to this basis, we have:
    \begin{align*}
      & D_1 = 
      \begin{pmatrix}
        0 & 
        {-\frac{1}{3}} y_2 
        & \frac{1}{3}z y_2
        & 18y_2^2\big(y_1-\frac{1}{54}\big)
        & 0 
        & 6zy_1 y_2 \\
        0
        & 0
        & {-\frac{1}{3}} y_2
        & 9z y_1 y_2
        & 0
        & 2z^2 y_1 \\
        0
        & \frac{1}{3}
        & 0
        & \frac{1}{3} y_2(1-27y_1)
        & 0
        & -3 z y_1 \\
        0
        & 0
        & \frac{1}{3}
        & 0
        & 0
        & y_1 \\
        1
        & 0
        & 0
        & -z y_2(1+27y_1) 
        & 0
        & {-3} y_1(3y_2+2z^2) \\
        0
        & 0
        & 0
        & 3y_2
        & \frac{1}{1+27y_1}
        & 0
      \end{pmatrix} \\
      & D_2 =
      \begin{pmatrix}
        0
        & 0
        & 0
        & y_2(54 y_1 y_2 - y_2 + z^2)
        & {-\frac{1}{3} y_2}
        & \frac{1}{9} z y_2 (1+27y_1)\\
        1
        & 0
        & 0
        & -z y_2(2-27y_1)
        & 0
        & {-\frac{1}{9} y_2} (1+27y_1)\\
        0
        & 1
        & 0
        & y_2(2-27y_1)
        & \frac{1}{3}
        & 0 \\
        0
        & 0
        & 1
        & 0
        & 0
        & \frac{1}{9}(1+27y_1)\\
        0
        & 0
        & 0
        & -3 z y_2 (1+27y_1)
        & 0
        & {- \frac{1}{3}}y_2(1+27y_1)\\
        0
        & 0
        & 0
        & 9 y_2
        & 0
        & 0
      \end{pmatrix}
    \end{align*}
    and the Gram matrix of the pairing is:
    \[
    \begin{pmatrix}
      0 &
      0 &
      0 &
      9 &
      0 &
      0 \\
      0 &
      0 &
      9 &
      0 &
      0 &
      (1+27y_1) \\
      0 &
      9 &
      0 &
      27y_2 &
      3 &
      0 \\
      9 &
      0 &
      27y_2 &
      0 &
      0 &
      y_2(1+27y_1) \\
      0 &
      0 &
      3 &
      0 &
      0 &
      9 y_1 \\
      0 &
      1+27y_1 &
      0 &
      y_2(1+27y_1) &
      9 y_1 &
      0 \\
    \end{pmatrix}
    \]
  \item[(b)] In the chart near $(y_1,y_2) = ({-\frac{1}{27}},0)$, 
writing $t =y_1+\frac{1}{27}$ and $D_t = z t \partial_{t}$, 
the following relations define the GKZ system: 
   \begin{align*}
   9ty & = D_2(9 t D_2-(27t-1)D_t) \\ 
729 t^2 D_t^3 & = 
      \big[9tD_2-(27t-1)(D_t+2z)\big] \\ 
      & \quad \times
      \big[9t(D_2+z)-(27t-1)(D_t +z)\big] 
      \big[9t(D_2+2z)-(27t-1)D_t\big]
    \end{align*}
    and the GKZ system has basis:
    \begin{equation}
      \label{eq:GKZ_basis_conifold_chart}
      1, D_2, D_2^2, D_2^3, \big(1-\textstyle\frac{1}{27t}\big)D_t,
      \textstyle\frac{1}{27t}\big(   (27t-1)^2 D_t^2 + (27t-1)D_t \big) 
    \end{equation}
    (This is the same basis as Part (a).)\phantom{.}  With respect to
    this basis, we have:
    \begin{align*}
      & D_t = 
      \begin{pmatrix}
        0 & 
        \frac{9 t y_2}{1-27 t} & 
        \frac{-9 t z y_2}{1-27 t} & 
        \frac{27 t (18 t-1) y_2^2}{27 t-1} & 
        0 & 
        6 t z y_2 \\
        0 & 
        0 & 
        \frac{9 t y_2}{1-27 t} & 
        9 t z y_2 & 
        0 & 
        2 t z^2 \\
        0 & 
        \frac{9 t}{27 t-1} & 
        0 & 
        \frac{9 (2-27 t) t y_2}{27 t-1} & 
        0 & 
        -3 t z \\
        0 & 
        0 & 
        \frac{9 t}{27 t-1} & 
        0 & 
        0 & 
        t \\
        \frac{27t}{27t-1} & 
        0 & 
        0 &
        \frac{-729 t^2 z y_2}{27 t-1} & 
        0 & 
        -3  t \left(2 z^2+3 y_2\right) \\
        0 & 
        0 & 
        0 & 
        \frac{81 t  y_2}{27t-1} & 
        \frac{1}{27  t-1} & 
        0 
      \end{pmatrix} \\ 
      & D_2 =
      \begin{pmatrix}
        0
        & 0
        & 0
        & y_2(54 t y_2 - 3 y_2 + z^2)
        & {-\frac{1}{3} y_2}
        & 3 z y_2 t \\
        1
        & 0
        & 0
        & -z y_2(3-27t)
        & 0
        & {-3 y_2 t} \\
        0
        & 1
        & 0
        & y_2(3-27t)
        & \frac{1}{3}
        & 0 \\
        0
        & 0
        & 1
        & 0
        & 0
        & 3t\\
        0
        & 0
        & 0
        & -81 z y_2 t 
        & 0
        & {- 9}y_2 t\\
        0
        & 0
        & 0
        & 9 y_2
        & 0
        & 0
      \end{pmatrix}
    \end{align*}
    and the Gram matrix of the pairing is:
    \[
    \begin{pmatrix}
      0 &
      0 &
      0 &
      9 &
      0 &
      0 \\
      0 &
      0 &
      9 &
      0 &
      0 &
      27t \\
      0 &
      9 &
      0 &
      27y_2 &
      3 &
      0 \\
      9 &
      0 &
      27y_2 &
      0 &
      0 &
      27 t y_2 \\
      0 &
      0 &
      3 &
      0 &
      0 &
      9 t-\frac{1}{3} \\
      0 &
      27t &
      0 &
      27 t y_2 &
      9 t - \frac{1}{3} &
      0 \\
    \end{pmatrix}
    \]
   
  \item[(c)] In the chart near $(\fry_1,\fry_2) = (0,0)$ with
    $\fry_1^3 \ne {-27}$, writing $\frd_1 = - z \partial_{\fry_1}$,
    $\frD_2 = -z \fry_2\partial_{\fry_2}$, 
   the GKZ system has basis:
    \begin{equation}
      \label{eq:GKZ_basis_orbifold_chart}
      1, \frD_2, \frD_2^2, \frD_2^3, \frd_1, \frd_1^2
    \end{equation}
    With respect to this basis, we have:
    \begin{align*}
      & \frd_1 = 
      \begin{pmatrix}
        0 & 
        \fry_2 &
        -z \fry_2 & 
        z^2 \fry_2 & 
        0 & 
        -\frac{3 z \fry_1 \fry_2}{\fry_1^3+27} \\
        0 &
        0 &
        \fry_2 &
        -2 z \fry_2 &
        0 &
        {-\frac{3 \fry_1 \fry_2}{\fry_1^3+27}}\\
        0 &
        0 &
        0 &
        \fry_2 &
        0 &
        0 \\
        0 &
        0 &
        0 &
        0 &
        0 &
        \frac{1}{\fry_1^3+27}\\
        1 &
        0 &
        0 &
        0 &
        0 &
        {-\frac{z^2 \fry_1}{\fry_1^3+27}}\\
        0 &
        0 &
        0 &
        0 &
        1 &
        \frac{3 z \fry_1}{\fry_1^3+27}
      \end{pmatrix} \\
      & \frD_2 = 
      \begin{pmatrix}
        0 &
        0 &
        0 &
        {-2} z^2 \fry_1 \fry_2&
        \fry_2 &
        0 \\
        1&
        0 &
        0 &
        0 &
        0 &
        0 \\
        0 &
        1 &
        0 &
        3\fry_1 \fry_2 &
        0 &
        0 \\
        0 &
        0 &
        1 &
        0 &
        0 &
        0 \\
        0 &
        0 &
        0 &
        -3z\fry_1^2 \fry_2 &
        0 &
        \fry_2\\
        0 &
        0 &
        0 &
        \fry_2(\fry_1^3+27) &
        0 &
        0 
      \end{pmatrix}
    \end{align*}
    and the Gram matrix of the pairing is:
    \[
    \begin{pmatrix}
      0 &
      0 &
      0 &
      9 &
      0 &
      0 \\
      0 &
      0 &
      9 &
      0 &
      0 &
      0 \\
      0 &
      9 &
      0 &
      27 \fry_1 \fry_2 &
      0 &
      0 \\
      9 &
      0 &
      27 \fry_1 \fry_2 &
      0 &
      0 &
      0 \\
      0 &
      0 &
      0 &
      0 &
      0 &
      \frac{9}{27+\fry_1^3} \\
      0 &
      0 &
      0 &
      0 &
      \frac{9}{27+\fry_1^3} &
      0 \\
    \end{pmatrix}
    \]
  \end{enumerate} 
\end{proposition}

\begin{proof}
  We will prove only part (a).  Part (b) follows trivially from part
  (a), and part (c) is very similar.  Consider first the subsheaf of
  the GKZ system spanned, over $\cO_{U_\LR^\times \times \C}$, by:
  \[
  1, D_2, D_2^2, D_2^3, D_1, (1+27y_1)D_1^2
  \]
  This subsheaf is locally free over $U_\LR^\times \times \C$; to see this,
  it suffices to show that the Gram matrix of the pairing is as
  claimed, for this matrix is invertible for all $y_1$, $y_2$.

  To compute the Gram matrix, observe that the pairing is homogeneous
  of degree ${-3}$ with respect to the grading \eqref{eq:GKZ_grading}
  and that, in view of the discussion immediately above, only
  non-negative powers of $y_1$, $y_2$, and $z$ can occur.  Thus the
  Gram matrix takes the form:
  \[
  \begin{pmatrix}
    0 & 0 & 0 & * & 0 & 0 \\
    0 & 0 & * & * z & 0 & * \\
    0 & * & * z & * z^2 + * y_2 & * & * z \\
    * & * z & * z^2 + * y_2 & * z^3 + * z y_2 & * z & * z^2 + * y_2 \\
    0 & 0 & * & * z & 0 & * \\
    0 & * & * z & * z^2 + * y_2 & * & * z 
  \end{pmatrix}
  \]
  where each asterisk denotes an unknown function of $y_1$.  Consider
  now the matrix entry $((-)^\star D_2,D_1^2)_{\rm GKZ}$.  
Combining equation \eqref{eq:GKZ_pairing} with the equality:
  \[
  I_{\Ybar}(y_1,y_2,-z)\Big|_{y_2=0} = 
  e^{-h_1/z} e^{-h_2/z} 
 \left(1 + \sum_{k>0}
  y_1^k
 \frac
 {\prod_{-3k<m \leq 0}(h_2 - 3 h_1 - m z)}
 {\prod_{1 \leq m \leq k} (h_1-mz)^3}
 \right)
 \]
 yields:
 \begin{multline*}
   ((-)^\star D_2,D_1^2)_{\rm GKZ} = \\
   \left(
     h_2  + \textstyle \sum_{k>0}
     y_1^k h_2
     \frac
     {\prod_{-3k<m \leq 0}(h_2 - 3 h_1 +m z)}
     {\prod_{1 \leq m \leq k} (h_1 +mz )^3},
     h_1^2 + 
     \sum_{l>0}
     y_1^l (h_1 -l z)^2
     \frac
     {\prod_{-3l<m \leq 0}(h_2 - 3 h_1-m z)}
     {\prod_{1 \leq m \leq l} (h_1-mz)^3}
   \right)_{{\rm A},\Ybar}
 \end{multline*}
and hence:
\begin{align*}
  & ((-)^\star D_2,D_1^2)_{\rm GKZ} \\
  & = \int_{\Ybar}
   \left(
     h_2  + \textstyle \sum_{k>0}
     y_1^k h_2
     \frac
     {\prod_{-3k<m \leq 0}(h_2 - 3 h_1 + m z)}
     {\prod_{1 \leq m \leq k} (h_1+mz)^3}\right)
   \left(\textstyle
     h_1^2 + 
     \sum_{l>0}
     y_1^l (h_1-l z)^2
     \frac
     {\prod_{-3l<m \leq 0}(h_2 - 3 h_1 - m z)}
     {\prod_{1 \leq m \leq l} (h_1 -mz)^3}
   \right) \\
   & = \int_{\Ybar} h_2 h_1^2 = 1
 \end{align*}
 where for the second equality we used the relation
 $h_2(h_2-3h_1) = 0$ in $H^\bullet(Y)$.  Thus:
 \[
 \big((-)^\star D_2,(1+27y_1) D_1^2 \big)_{\rm GKZ} = 1+27y_1
 \]

 The same reasoning allows us to fill in almost all terms in the Gram
 matrix that are not divisible by $y_2$:
 \[
 \begin{pmatrix}
   0 & 0 & 0 & 9 & 0 & 0 \\
   0 & 0 & 9 & 0 & 0 & 1+27y_1 \\
   0 & 9 & 0 & * y_2 & 3 & 0 \\
   9 & 0 & * y_2 & * z y_2 & 0 & * y_2 \\
   0 & 0 & 3 & 0 & 0 & * \\
   0 & 1+27y_1 & 0 & * y_2 & * & 0 
 \end{pmatrix}
 \]
 Furthermore the symmetry:
 \[
 \big((-)^\star s_1, s_2 \big)_{\rm B} = 
 (-)^\star \big((-)^\star s_2,s_1\big)_{\rm B}
 \]
 gives a corresponding symmetry of the GKZ pairing, which in
 particular implies that 
 \[
 ((-)^\star D_2^3,D_2^3)_{\rm GKZ} = 0.
\]
All remaining
 terms in the Gram matrix are therefore independent of~$z$.  These can
 be calculated using the principal term of the stationary phase
 approximation \eqref{eq:stationary_phase}, where we see the residue
 pairing:
 \[
 \bigg(\Big[f(x,{-z}) e^{W(x)/z} \, \omega\Big], 
 \Big[g(x,z) e^{{-W}(x)/z} \, \omega\Big] \bigg)_{\rm B} =
 \sum_{\text{critical points $\sigma$}} 
 \frac{f(\sigma,0) g(\sigma,0)}{\Hess_\sigma(W)} + O(z)
 \]
 Thus:
 \begin{align*}
   \big((-)^\star D_2^2,D_2^3\big)_{\rm GKZ} &=
   \sum_{\text{critical points $\sigma$}} 
   \frac{\big(y_2 \frac{\partial W}{\partial y_2}(\sigma)\big)^2 
     \big(y_2 \frac{\partial W}{\partial y_2}(\sigma)\big)^3}{\Hess_\sigma(W)} + O(z) \\
   &= \sum_{\text{critical points $\sigma$}}  {\big({3y_1 y_2^3 \over w_1 w_2 w_5^3} + {y_2 \over w_5}\big)^5
     \over w_1^2 w^5} + O(z) \\
   &= 27 y_2 + O(z)
 \end{align*}
 where we use co-ordinates $(w_1,w_2,w_5)$ on the fibre of the
 Landau--Ginzburg model as in \eqref{eq:fibre_coordinates}, and at the
 last step we used the critical point equations:
 \begin{align*}
   w_1 - {y_1 y_2^3 \over w_1 w_2 w_5^3} =0 &&
   w_2 - {y_1 y_2^3 \over w_1 w_2 w_5^3} =0 &&
   w_5 - {3y_1 y_2^3 \over w_1 w_2 w_5^3} - {y_2 \over w_5} = 0
 \end{align*}
 On the other hand we know that 
$((-)^\star D_2^2,D_2^3\big)_{\rm GKZ}$ is
 independent of $z$, so:
 \[
 \big((-)^\star D_2^2,D_2^3\big)_{\rm GKZ} = 27 y_2
 \]
 The same reasoning yields $(D_2^3,D_2^2)_{\rm GKZ} = 27 y_2$ and:
 \begin{align*}
   & \big((-)^\star D_2^3,(1+27y_1)D_1^2\big)_{\rm GKZ} = y_2(1+27y_1) 
   && \big((-)^\star D_1,(1+27y_1)D_1^2\big)_{\rm GKZ} = 9 y_1 \\
   & \big((-)^\star (1+27y_1)D_1^2,D_2^3\big)_{\rm GKZ} = y_2(1+27y_1) 
   && \big((-)^\star (1+27y_1)D_1^2, D_1\big)_{\rm GKZ} = 9 y_1 
 \end{align*}
 This completes the calculation of the Gram matrix.

 We now compute the connection matrices, i.e. the matrices for the
 action of $D_1$ and $D_2$ on the elements:
 \[
 1, D_2, D_2^2, D_2^3, D_1, (1+27y_1)D_1^2
 \]
 This is routine, involving repeated application of the equations
 \eqref{eq:GKZ_chart_1}; one can do this systematically using
 Gr\"obner basis methods as in \cite{Guest}.  In particular we
 discover that the subsheaf of the GKZ system spanned over $\cO_{U_\LR^\times
   \times \C}$ by the above elements is closed under the action of
 $D_1$ and $D_2$.  It follows that this subsheaf is in fact the entire
 GKZ system over $U_\LR^\times \times \C$, and hence that
 \eqref{eq:GKZ_basis_LR_chart} is a basis for the GKZ system over
 $U_\LR^\times \times \C$, as claimed.
\end{proof}

With these explicit connection matrices in hand, we now construct a
logarithmic extension of the B-model TEP structure to all of
$\cMB$.

\begin{definition}[\cite{Deligne}*{Proposition 5.2}]
\label{def:Deligne}
Let $(\cG^\times, \nabla)$ be a locally free sheaf with flat connection on
$\cM \setminus D$, where $D$ is a normal crossing divisor in $\cM$.  
Let $\cG$ be a locally free  
extension of $\cG^\times$ to $\cM$ such that $\nabla$ is extended to a 
meromorphic flat connection on $\cG^\times$ with logarithmic singularities
along $D$.  We say that $\big(\cG, \nabla \big)$
is \emph{the Deligne extension} of $(\cG^\times,\nabla)$ across $D$ if the
residue endomorphisms of $\nabla$ along $D$ are nilpotent. 
Let $(\cF^\times,\nabla, (\cdot,\cdot))$ be a TEP structure with base $\cM\setminus D$. 
We say that a \logDTEP structure 
$(\cF, \nabla, (\cdot,\cdot))$ 
with base $(\cM,D)$ is \emph{the Deligne extension} of 
$(\cF^\times,\nabla,(\cdot,\cdot))$ 
if $(\cF, \nabla,(\cdot,\cdot))$ restricts to 
$(\cF^\times,\nabla,(\cdot,\cdot))$ over $\cM\setminus D$ 
and for each $z\in \C^\times$,  
$(\cF,\nabla)|_{\cM\times \{z\}}$ is the Deligne 
extension of 
$(\cF^\times,\nabla)|_{(\cM\setminus D) \times \{z\}}$ 
\end{definition}

\begin{remark}
\label{rem:Deligne_extension}
Deligne \cite{Deligne} called this logarithmic extension 
``prolongement canonique". 
The Deligne extension of a flat connection on $\cM\setminus D$ 
across $D$ exists if and only if 
the local monodromy around $D$ is unipotent, 
and is unique if it exists. 
When the local monodromy around $D$ is not unipotent, 
a logarithmic extension is given by the choice of a determination 
of logarithm, i.e.~a section of $\C \to \C/\Z$ 
\cite{Deligne}*{Proposition 5.4}. 
\end{remark}
\begin{proposition-definition}
  \label{pro-def:2d_log_TEP}
  Recall from \S\ref{sec:LG_Ybar} that the toric variety
  $\cMB$ is covered by two toric co-ordinate patches, with
  co-ordinate systems $(y_1,y_2)$ and $(\fry_1,\fry_2)$.  
Let $U_\LR$ and $U_\orb$ denote 
the following co-ordinate patches of $\cMB$ 
(see equation~\eqref{eq:charts} for $U_\LR^\times$ and $U_\orb^\times$) 
\[
U_\LR = \{(y_1,y_2)\in \C^2\} \quad \quad  
U_\orb = \{(\fry_1,\fry_2)\in \C^2: \fry_1^3 \neq -27\} 
\] 
Specifying that the following generators of $\cFB^\times=\cFGKZ$:
  \begin{align*}
    & 1, D_2, D_2^2, D_2^3, D_1, (1+27y_1)D_1^2 &
    \text{over $U_\LR^\times \times \C$, as in \eqref{eq:GKZ_basis_LR_chart}} \\
    & 1, \frD_2, \frD_2^2, \frD_2^3, \frd_1, \frd_1^2 &
    \text{over $U_\orb^\times \times \C$, as in \eqref{eq:GKZ_basis_orbifold_chart}}
  \end{align*}
  form locally free bases for $\cFB$ over (respectively)
  $U_\LR \times \C$ and $U_\orb \times \C$ defines
  a locally free sheaf $\cFB$ over $\cMB \times \C$.
  The sheaf $\cFB$ carries a meromorphic flat connection $\nablaB$
  and a pairing $\pairingB$ and the triple 
  $(\cFB,\nablaB, \pairingB)$ 
  forms a \logDTEP structure with base $(\cMB,D)$ 
  in the sense of Definition \ref{def:logDTEP}, 
  where 
\begin{equation} 
\label{eq:log_divisor} 
D = \overline{(y_1y_2=0)} \cup (y_1 = -1/27). 
\end{equation} 
We call the triple $(\cFB,\nablaB, \pairingB)$ 
the \emph{B-model \logDTEP structure}. 
The restriction of the B-model \logDTEP structure to 
$\cMB^\times \times \C$ is canonically isomorphic to the B-model TEP structure, 
and the B-model \logDTEP structure is the Deligne extension of 
the B-model TEP structure. 
\end{proposition-definition}

\begin{proof}
  We need to check that the generators specified give locally free
  bases for $\cFB^\times$ over (respectively) $U_\LR^\times \times \C$ and $U_\orb^\times
  \times \C$, that the connection matrices with respect to these bases
  have logarithmic singularities along the divisor $D \times \C$, that
  the residue endomorphisms of the connection along $D$ are nilpotent, and
  that the pairing extends holomorphically across $D$.  These
  statements follow easily from Proposition~\ref{pro:GKZ}.
\end{proof}

\begin{remark} 
\label{rem:orbisheaf} 
The locally free sheaf $\cFB$ should be understood 
as an orbi-vector bundle on the orbifold chart, cf.~Example \ref{ex:log_TEP_A}. 
In other words, on the chart $U_\orb$, $\cFB$ is 
a $\mu_3$-equivariant sheaf equipped with $\mu_3$-invariant 
connection and pairing. The $\mu_3$-action is given on the frame by 
$(1,\frD_2,\frD_2^2,\frD_2^3,\frd_1,\frd_2^2) 
\mapsto (1,\frD_2,\frD_2^2,\frD_2^3,e^{2\pi\iu/3} \frd_1, 
e^{4\pi\iu/3} \frd_2^2)$. 
\end{remark}

\section{The Conformal Limit}
\label{sec:conformal_limit}
Let $\cMCY = \Proj(3,1)$, and let $\DCY$ be the divisor
$\{0,{-\frac{1}{27}}\} \subset \cMCY$.  A key ingredient in
Aganagic--Bouchard--Klemm's modularity argument is the family of
elliptic curves:
\begin{equation}
  \label{eq:mirror_family}
  \big\{[X:Y:Z] \in \Proj^2 : X^3+Y^3+Z^3+y^{-1/3} XYZ=0\big\}
\end{equation}
parametrized by $y \in \cMCY \setminus \DCY$, and the corresponding
variation of Hodge structure.  This variation of Hodge structure is a
two-dimensional vector bundle over $\cMCY \setminus \DCY$ equipped
with a flat connection and a Hodge filtration.  We will see in this
section how this finite-dimensional variation of Hodge structure
arises from the B-model TEP structure, by taking the \emph{conformal
  limit} $y_2 \to 0$ of the Deligne extension $\cFB$.

\subsection{A Vector Bundle of Rank 6 on $\cMCY$ with a Logarithmic
  Connection}
\label{sec:six}
The closure of the locus $\{y_2=0\}$ in $\cMB$ is a copy
of $\cMCY$.  Consider the restriction
\[
\cFCY:=\cFB|_{\cMCY\times \C}
\]
of the B-model \logDTEP structure $\cFB$ (Proposition-Definition
\ref{pro-def:2d_log_TEP}) to $\cMCY \times \C \subset
\cMB\times \C$.  The sheaf $\cFCY$ has the structure of a
\logDTEP structure with base $(\cMCY,\DCY)$ together with the
endomorphism $N \colon \cFCY\to z^{-1} \cFCY$ 
defined as the residue of $\nablaB$
along the divisor $\cMCY \times \C \subset \cMB\times \C$
and the grading operator $\Gr$.  More precisely we have:
\begin{itemize} 
\item a meromorphic flat connection with poles along 
$Z = (\DCY \times \C) \cup (\cMCY\times \{0\})$
\[
\nabla \colon \cFCY \to \Omega^1_{\cMCY\times \C}(\log Z) \otimes
\cFCY(\cMCY\times \{0\})
\]
defined by 
\begin{align*} 
\nabla\left( s|_{y_2=0}\right) 
& = \left.\left( 
\nablaB_{y_1\parfrac{}{y_1} - \frac{1}{3} y_2 \parfrac{}{y_2}} s
\right)\right|_{y_2=0} \frac{dy_1}{y_1} 
+ \left.\left(\nablaB_{z\parfrac{}{z}} s \right)\right|_{y_2=0} 
\frac{dz}{z}  \\ 
& = \left.\left( 
\nablaB_{\fry_1\parfrac{}{\fry_1}} s
\right)\right|_{y_2=0} \frac{d\fry_1}{\fry_1} 
+ \left.\left(\nablaB_{z\parfrac{}{z}} s \right)\right|_{y_2=0} 
\frac{dz}{z}
\end{align*}
for a local section $s$ of $\cFB$; 
\item a flat non-degenerate pairing 
\[
(\cdot,\cdot): ({-})^\star \cFCY \otimes \cFCY \to \cO_{\cMCY \times
  \C}
\]
induced by $\pairingB$;

\item the residue endomorphism 
$N\colon \cFCY \to z^{-1} \cFCY$: 
\[
N = \nablaB_{y_2 \partial_{y_2}}\bigr|_{y_2=0} 
= \nablaB_{\fry_2 \partial_{\fry_2}}\bigr|_{\fry_2 =0} 
= -z^{-1} D_2 = -z^{-1} \frD_2  
\] 
which is flat for $\nabla$ and satisfies $((-)^\star N s_1, s_2) 
= - ((-)^\star s_1, N s_2)$ for $s_1, s_2 \in \cFCY$; 

\item the grading operator $\Gr \colon \cFCY \to \cFCY$ 
induced from the grading operator \eqref{eq:GKZ_grading} of 
the GKZ system: this is related to $\nabla_{z\partial_z}$ 
by 
\begin{align}
  \label{eq:restricted_GKZz}
  \nabla_{z \partial_{z}} = \Gr - 2 N - \frac{3}{2} && 
\text{(cf.~equation~\eqref{eq:GKZz})}.
\end{align}
\end{itemize} 
\begin{remark} 
Let $\nabla$ be a flat connection on $\cM$ with logarithmic singularities 
along a smooth divisor $D \subset \cM$. 
In order to obtain a flat connection along $D$ from $\nabla$, 
we need to choose a splitting of the sequence $0 \to \Omega^1_D \to 
\Omega^1_{\cM}(\log D)|_D \xrightarrow{\Res} \cO_D \to 0$ 
(see Example~\ref{ex:no_canonical_map} below) otherwise 
the induced connection along $D$ is defined only 
`up to the residue endomorphism'.  
This choice is not canonical in general, and we chose a particular 
splitting when defining the connection $\nabla$ on $\cFCY$. 
The splitting does not play an important role in this section, 
but will appear again in 
\S\S\ref{sec:correlation_functions_conformal_limit}--\ref
{sec:opposites_conformal_limit} 
and will be important there. 
\end{remark} 
The triple $(\cFCY, \nabla, (\cdot,\cdot))$ is a \logDTEP structure
with base $(\cMCY,\DCY)$ in the sense of Definition \ref{def:logDTEP}.
The grading operator $\Gr$ on $\cFCY$ is $\pi^{-1}\cO_{\cMCY}$-linear
since the variable $y_1$ of the base is of degree zero. Thus it serves
as another connection in the $z$-direction.  The GKZ description also
passes to $\cFCY$: on the chart $\cMCY\setminus
\{0,-\frac{1}{27},\infty\}$, it is defined by the relations
\begin{align*} 
& D_2(D_2-3D_1)=0, \\ 
& D_1^3 - y_1 (D_2-3D_1) (D_2- 3D_1 +z) 
(D_2 - 3D_1+2z) =0 
\end{align*} 
where $D_1 = -z y_1 \partial{}{y_1}$ 
is as before and $D_2 = [-zy_2\partial_{y_2}]_{y_2=0} 
= -z N$ is now an $\cO_{\cMCY\times\C}$-linear 
endomorphism commuting with $D_1$.  It is extended
across the three points $\{0,-\frac{1}{27},\infty\}$ by the bases
specified in Proposition-Definition \ref{pro-def:2d_log_TEP}.

\subsubsection{The Rank 6 Vector Bundle $H$} 
Consider now the push-forward $\pi_\star (\cFCY^\divideontimes)$ where
$\pi\colon \cMCY \times \C^\times \to \cMCY$ is the projection; see
Notation~\ref{notation:x} for the notation here.  Consider the
subsheaf of $\pi_\star ( \cFCY^\divideontimes)$ consisting of homogeneous
sections of degree~$1$ with respect to $\Gr$; this subsheaf is locally
free of rank~6 over $\cMCY$, and thus defines a rank-$6$ vector bundle
$H \to \cMCY$.  The vector bundle $H$ carries the following
structures:
\begin{itemize}
\item a logarithmic flat connection $\nabla\colon \cO(H) \to 
\Omega^1_{\cMCY}(\log \DCY) \otimes \cO(H)$, 
induced from the meromorphic flat connection on $\cFCY$;
\item a $\nabla$-flat endomorphism $N\in \End(H)$ of vector bundles, 
induced by the residue endomorphism $N \colon \cFCY\to z^{-1} \cFCY$; 
\item an $\cO_{\cMCY}$-bilinear symplectic pairing 
$\Omega \colon \cO(H)\otimes \cO(H) \to \cO_{\cMCY}$, 
induced by the pairing $(\cdot,\cdot)$ on $\cFCY$. 
\end{itemize}
The pairing $(\cdot,\cdot)$ induces a symplectic pairing $\Omega$ 
on $\cO(H)$ because, when restricted to $\cO(H)$, $(\cdot,\cdot)$
takes values in $z^{-1} \cO_{\cMCY}$; we set:
\begin{align*}
  \Omega(s_1,s_2) = \Res_{z=0} \big((-)^\star s_1,s_2) \, dz
  &&
  \text{for $s_1$, $s_2 \in \cO(H)$.}
\end{align*}
The connection $\nabla$ on $H$ preserves the symplectic form, 
and $N:H \to H$ is infinitesimally symplectic, i.e.~$\Omega(Nv, w) + 
\Omega(v, Nw ) =0$. 
In view of Proposition-Definition \ref{pro-def:2d_log_TEP}, 
local frames of $H$ over the manifold chart 
$\cMCY \setminus \{y_1= \infty\}$ and the orbifold chart 
$\cMCY\setminus \{y_1 = 0\}$ 
are given respectively by:  
\begin{align}
\label{eq:basis_for_H}
\begin{aligned} 
& \text{$-z$, $D_2$, $z^{-1}D_2^2$, $z^{-2}D_2^3$, 
$D_1-\tfrac{1}{3}D_2$, 
$z^{-1}(1+27 y_1)(D_1-\tfrac{1}{3}D_2)^2$} 
\\
& 
\text{$-z$, $\frD_2$, $z^{-1}\frD_2^2$, $z^{-2}\frD_2^3$, 
$\frd_1$, $z^{-1} (1+\tfrac{1}{27}\fry_1^3) \frd_1^2$} 
\end{aligned} 
\quad 
\begin{aligned} 
& \text{on $\cMCY \setminus \{y_1=\infty\}$} \\ 
& \text{on $\cMCY\setminus \{y_1 = 0\}$} 
\end{aligned} 
\end{align}
These two bases are related by the transition matrix 
\[
\left( 
\begin{array}{c|cc}
I &  & \\ \hline 
 & -\frac{1}{3}\fry_1 & -\frac{1}{9}\fry_1 (1+27 \fry_1^{-3}) \\ 
&  0& 3\fry_1^{-1} 
\end{array} 
\right) 
\] 
where $I$ is the identity matrix of rank 4. 
This implies that
$\cO(H) \cong \cO^{\oplus 4} \oplus \cO(1) \oplus \cO(-1)$ 
as a bundle on $\cMCY = \Proj(3,1)$. 

\subsubsection{The Hodge Filtration}
The vector bundle $H$ carries a `Hodge filtration' given 
by pole order at $z=0$:
\[
F^p = \Bigg[\pi_\star \Big(z^{p-2} \cFCY \Big)\Bigg]_{\deg
  1}
\]
where $\pi\colon \cMCY\times \C \to \cMCY$ is the projection 
and the subscript indicates that we take the subsheaf consisting of
homogeneous elements of degree~1.  This is a decreasing filtration 
by subbundles: 
\[
0 \subset F^3 \subset F^2 \subset F^1 \subset F^0 = H
\]
such that one has: 
\begin{align*}
  \nabla_v F^p \subset F^{p-1}
  \qquad 
  N F^p \subset F^{p-1} 
  \qquad 
  \Omega(F^p, F^{4-p}) =0 
\end{align*}
for any vector field $v \in \Theta_{\cMCY}(\log\DCY\})$. 
Explicit bases of the subbundles $F^p$ 
on the manifold chart $\cMCY\setminus \{y_1=\infty\}$ 
are given by: 
\begin{align*} 
F^3 : & && {-z}  \\
F^2 : & && {-z}, D_2, D_1-\tfrac{1}{3}D_2 \\
F^1 : & && {-z}, D_2, D_1 - \tfrac{1}{3}D_2, z^{-1} D_2^2, 
z^{-1} (1+27y_1)(D_1-\tfrac{1}{3}D_2)^2\\ 
F^0 : & && {-z}, D_2, D_1 - \tfrac{1}{3}D_2, z^{-1} D_2^2, 
z^{-1} (1+27y_1)(D_1-\tfrac{1}{3}D_2)^2, z^{-2} D_2^3
\end{align*} 
There is a `primitive section' $\zeta\in F^3$ of $H$, represented by 
$-z$ in the GKZ system.  
This satisfies $N^3 \zeta \ne 0$, and $N^3 \zeta$ is flat. 

\subsubsection{The Kernel and the Image of $N$} 
\label{subsubsec:ker_im_N}
The endomorphism $N$ is flat for $\nabla$, and therefore 
the kernel and image of $N$ are preserved by $\nabla$. 
By examining the action of $N=-z^{-1}D_2$ on the basis 
\eqref{eq:basis_for_H}, 
we know that both $\Ker N$ and $\Image N$ are of rank $3$ 
and have the following explicit bases (on the manifold 
chart $\cMCY\setminus \{y_1= \infty\}$): 
\begin{align*} 
\Ker N: && & z^{-2} D_2^3, \ D_1- \tfrac{1}{3}D_2,\ 
z^{-1} (1+27y_1) (D_1- \tfrac{1}{3}D_2)^2  \\
\Image N : &&  D_2, \ z^{-1}D_2^2, \ & z^{-2} D_2^3 
\end{align*}

\subsection{A Vector Bundle of Rank 2 on $\cMCY$ with a
  Logarithmic Connection}
\label{sec:two}

We now pass from $H$, which is a six-dimensional symplectic
vector bundle over $\cMCY$, to a two-dimensional symplectic
vector bundle $\Hvec$ over $\cMCY$. 
The bundle $\Hvec$ is obtained from $H$ via the infinitesimally 
symplectic endomorphism $N$. 
A similar construction appears in the work of Konishi--Minabe 
\cite{Konishi--Minabe}*{\S 8} in the A-model.

\subsubsection{The Rank 3 Vector Bundle $\widebar{H}=\Cok N$ 
and Quantum D-Module of $K_{\Proj^2}$} 
Consider the cokernel $\widebar{H}$ of the map $N:H \to H$.  
This carries a flat connection $\nabla$ with logarithmic poles along $\DCY$ induced by $\nabla$ on $H$. 
Write $\theta = \nabla_{y_1\partial_{y_1}} = -z^{-1} D_1$ 
for the operator\footnote{As $z^{-1}D_2$ acts trivially on 
$\cO(\widebar{H})$, we have $\theta= -z^{-1}
(D_1-\frac{1}{3}D_2)= -\frac{1}{3}\fry_1\frd_1$.} 
acting on $\cO(\widebar{H})$. 
Local frames for $\widebar{H}$ on the manifold chart 
$\cMCY\setminus \{y_1= \infty\}$ and 
the orbifold chart 
$\cMCY\setminus \{y_1 = 0\}$ are given respectively by: 
\begin{equation} 
\label{eq:frame_Hbar} 
\zeta =[-z], \quad \theta\zeta= [D_1- \tfrac{1}{3}D_2], \quad
-(1+27y_1) \theta^2 \zeta
=  [z^{-1} (1+27 y_1) (D_1- \tfrac{1}{3}D_2)^2] 
\end{equation} 
and 
\[
\zeta =[-z],\quad  -3\fry_1^{-1} \theta \zeta = [\frd_1], \quad 
\tfrac{1}{3}\fry_1(1+27\fry_1^{-3})
\theta(\theta+\tfrac{1}{3})\zeta 
= [-z^{-1} (1+\tfrac{1}{27}\fry_1^3) \frd_1^2] 
\]
We have $\cO(\widebar{H}) \cong \cO \oplus \cO(1) \oplus \cO(-1)$. 
The differential operator  
\begin{equation}
  \label{eq:QDE_for_Y}
\theta^3 - y_1 (-3\theta)(-3\theta-1)(-3\theta-2) 
\end{equation}
annihilates the 
primitive section $\zeta \in \cO(\widebar{H})$. 
Hence the D-module $(\cO(\widebar{H}),\nabla)$ is isomorphic to 
the quantum D-module for $Y = K_{\Proj^2}$; 
equation \eqref{eq:QDE_for_Y} is the Picard--Fuchs equation 
for the family of elliptic curves \eqref{eq:mirror_family} mirror to $K_{\Proj^2}$. 
With respect to the frame $\{\zeta, \theta\zeta, 
(1+27y_1)\theta^2\zeta\}$ \eqref{eq:frame_Hbar} in the manifold chart, 
the action of $\theta$ is represented by the matrix: 
\[
\theta = 
\begin{pmatrix} 
0 & 0 & 0 \\ 
1 & 0 & -6y_1 \\ 
0 & \frac{1}{1+27y_1} & 0 
\end{pmatrix} 
\]
\subsubsection{Affine Subbundle $\Haff$ of $\widebar{H}$}
\label{subsubsec:Haff} 
Any local function $\psi(y_1)$ annihilated by the 
differential operator \eqref{eq:QDE_for_Y} 
gives a D-module homomorphism 
$\psi^\sharp \colon \cO(H) \to \cO_{\cMCY}$ 
sending $\zeta$ to the function $\psi(y_1)$. 
The constant function $1$ is a solution to the equation \eqref{eq:QDE_for_Y} 
and thus defines a homomorphism $1^\sharp \colon \cO(H)\to \cO_{\cMCY}$. 
Consider the slice (affine subbundle) $\Haff$ of $\widebar{H}$ 
given by:
\[
\Haff = \{ v \in \widebar{H} : 1^\sharp(v) =1 \}
\]
(cf.~the dilaton shift in equation~\eqref{eq:Dilaton_shift}).  
Elements of $\Haff$ take the form:
\begin{align}
\label{eq:coordinates_Haff} 
\zeta + x \cdot \theta \zeta - p \cdot  (1+27y_1) \theta^2 \zeta 
&& x, \ p \in \C;
\end{align}
on the manifold chart $\cMCY \setminus \{y_1 = \infty\}$. 
As we see in \S \ref{subsubsec:Hvec} below, 
each fibre of the affine bundle $\Haff$ is naturally 
equipped with an affine symplectic structure. 
The affine bundle $\Haff$ is preserved by the connection $\nabla$ on 
$\widebar{H}$. 

\subsubsection{Rank 2 Vector Subbundle $\Hvec$ of $\widebar{H}$ 
Parallel to $\Haff$} 
\label{subsubsec:Hvec}
Consider the canonical projection $\Ker N \to \Cok N = \widebar{H}$. 
This induces an embedding of vector bundles
$\Ker N/(\Image N \cap \Ker N) \to \widebar{H}$. 
Let $\Hvec$ denote the image of $\Ker N/(\Image N \cap \Ker N)$ in
$\widebar{H}$.  
From the description of $\Haff \subset 
\widebar{H}$ in \S \ref{subsubsec:Haff} and 
$\Ker N$ in \S \ref{subsubsec:ker_im_N}, 
there is a canonical identification between the tangent space 
to the affine space $\Haff|_y$ and the fibre $\Hvec|_y$. 
In other words, $\Hvec$ is a vector subbundle of $\widebar{H}$ 
parallel to $\Haff$.
The bundle $\Hvec$ carries a flat connection $\nabla$ with 
logarithmic poles along $\DCY$ 
and one has $\cO(\Hvec) \cong \cO(1) \oplus\cO(-1)$. 

The symplectic structure on $H$ descends to a symplectic 
structure on $\Hvec$. 
Given a finite-dimensional symplectic vector space $(H,\Omega)$ and an
infinitesimal symplectic transformation $N \in \Liesp(H)$, the
symplectic orthogonal $(\Ker N)^\perp$ coincides with $\Image N$:
since $\Omega(Nv,w) = {-\Omega(v,Nw)}$ we have that $\Image N \subset
(\Ker N)^\perp$, and the two spaces have the same dimension.  The
symplectic pairing $\Omega$ thus induces a symplectic pairing on the
quotient space $\Ker N / (\Image N \cap \Ker N)$.  Applying this
construction to the (six-dimensional, flat) symplectic vector bundle
$H$ and the bundle map $N:H \to H$ yields a (two-dimensional, flat)
symplectic vector bundle $\Hvec = \Ker N / (\Image N \cap \Ker N)$.  
The symplectic pairing is given by:
\begin{equation} 
\label{eq:symplectic_pairing}
  \Omega\big(\theta \zeta, -(1+27y_1)\theta^2\zeta \big) 
= - \frac{1}{3} 
= \Omega \big([\frd_1], [-z^{-1}(1+\tfrac{1}{27}\fry_1^3) 
\frd_1^2] \big) 
\end{equation}
and therefore the symplectic form on $\Haff$ is given by 
$\frac{1}{3} dp \wedge dx$ in the co-ordinates 
\eqref{eq:coordinates_Haff}. 

\subsection{ Opposite Filtrations on $H$, $\overline{H}$, and $\Hvec$}
\label{sec:opposite_filtrations}

The Hodge filtration $F^\bullet$ on $H$ induces a filtration:
\[
0 \subset \widebar{F}^3 \subset \widebar{F}^2 \subset \widebar{F}^1 = \widebar{H}
\]
on $\widebar{H}$, 
where $\widebar{F}^k := F^k/(\Image N \cap
F^k)$. They are spanned by the bases 
\begin{align*} 
& \widebar{F}^3 : \quad  \zeta = [-z]\\
& \widebar{F}^2 : \quad  \zeta = [-z], 
\theta\zeta = [D_1-\tfrac{1}{3}D_2] \\ 
& \widebar{F}^1 : \quad \zeta =[-z], 
\theta\zeta = [D_1 - \tfrac{1}{3}D_2], 
-(1+27y_1)\theta^2 \zeta = 
[z^{-1} (1+27 y_1)(D_1- \tfrac{1}{3} D_2)^2] 
\end{align*} 
on the manifold chart $\cMCY\setminus \{y_1 = \infty\}$. 
This restricts to a filtration:
\[
0 = F^3_\vec \subset F^2_\vec \subset F^1_\vec = \Hvec
\]
on $\Hvec$, where $F^k_\vec = \Hvec \cap \widebar{F}^k$.  

\subsubsection{Opposite Filtration}
Opposite filtrations are decreasing filtrations which are complementary 
to the Hodge filtration. We study a well-behaved class of opposite 
filtrations which yield trivializations of $\cFCY|_{\{y\}\times \C}$ 
(i.e.~extensions of $\cFCY|_{\{y\}\times \C}$ to a free 
$\cO_{\Proj^1}$-module) with good properties. 
See \citelist{\cite{SaitoM}*{\S 3}\cite{Hertling:book}*{\S 7}} 
for a closely related discussion. 

\begin{proposition}
  \label{pro:correspondence_of_opposites}
  Let $y \in \cMCY$.  
Let $z$ denote the standard co-ordinate on
  $\{y\} \times \Proj^1 \cong \Proj^1$. There is a one-to-one
  correspondence between:
  \begin{itemize}
  \item[(A)] subspaces $P$ of $\Hvec|_y$ such that $F^2_\vec \oplus P
    = \Hvec|_y$;
  \item[(B)] filtrations:
    \[
    0 = 
    \widebar{U}_0 \subset 
    \widebar{U}_1 \subset 
    \widebar{U}_2 \subset 
    \widebar{U}_3 =
    \widebar{H}|_y
    \]
    satisfying $\widebar{F}^p \oplus \widebar{U}_{p-1} =
    \widebar{H}|_y$ and $\widebar{U}_2 = \Hvec|_y$.
  \item[(C)] filtrations:
    \[
    0 \subset U_0 \subset U_1 \subset U_2 \subset U_3=H|_y
    \]
    satisfying $F^p \oplus U_{p-1} = H|_y$, $N(U_p) \subset U_{p-1}$,
    $U_0^\perp = U_2$, and $U_1^\perp = U_1$; 
  \item[(D)] extensions of $\cFCY|_{\{y\}\times \C} = 
\cFB\big|_{\{y\} \times \C}$ to a locally free
    sheaf $\cE$ on $\{y\} \times \Proj^1$ such that:
    \begin{itemize}
    \item the corresponding holomorphic vector bundle on $\{y\} \times
      \Proj^1$ is trivial;
    \item the pairing $\pairingB$ extends holomorphically across
      $z=\infty$ and is non-degenerate there;
    \item the connection $\nabla$ has a logarithmic pole at
      $z=\infty$; 
    \item the map $N$ defined in \S\ref{sec:six} extends
      holomorphically across $z=\infty$ and vanishes there. 
    \end{itemize}
  \end{itemize}
  This correspondence satisfies:
  \begin{align*}
    & \widebar{U}_k = U_k/(\Image N \cap U_k) \\
    & P = \widebar{U}_1 = U_1/(\Image N \cap U_1)
  \end{align*}
\end{proposition}

\begin{proof} \ 
  \subsubsection*{($\text{A} \iff \text{B}$)} To give a subspace $P$ as in (A) is
  exactly the same as to give a filtration $\widebar{U}_\bullet$ as in
  (B) such that $\widebar{U}_1 = P$.

  \subsubsection*{($\text{B} \implies \text{C}$)} Suppose that
  $\widebar{U}_\bullet$ is a filtration as in (B).  Set:
  \begin{align*}
    & U_0 = \langle N^3 \zeta \rangle \\
    & U_1 = \big\{s \in \Ker N : s + \Image N \in \widebar{U}_1\big\}
    + \langle N^2 \zeta \rangle \\
    & U_2 = \Ker N + \Image N 
  \end{align*}
  where recall that $\zeta = -z$ and $N=-z^{-1}D_2$. 
  It is clear that $F^1 \oplus U_0 = H$, that $F^3 \oplus U_2 = H$,
  that $N(U_p) \subset U_{p-1}$, that $U_0^\perp = U_2$, and that
  $\widebar{U}_k = U_k/(\Image N \cap U_k)$.  It remains to show that
  $F^2 \oplus U_1 = H$ and that $U_1^\perp = U_1$.  The space $U_1$ is
  certainly isotropic, and:
  \[
  \dim U_1 = \dim \widebar{U}_1 + \dim (\Ker N \cap \Image N) + 1 = 3
  \]
  so $U_1$ is maximal isotropic: $U_1^\perp = U_1$.  Both $F^2$ and
  $U_1$ have dimension~$3$, so to show that $F^2 \oplus U_1 = H$ it
  suffices to show that $F^2 + U_1 = H$.  Let $v \in H$ be arbitrary,
  and let $\widebar{v}$ denote the equivalence class of $v$ in
  $\widebar{H}$.  Since $\widebar{F}^2 \oplus \widebar{U_1} =
  \widebar{H}$, there exist $\widebar{f} \in \widebar{F}^2$ and
  $\widebar{u} \in \widebar{U}_1$ such that $\widebar{v} = \widebar{f}
  + \widebar{u}$.  Let $f \in F^2$ and $u \in U_1$ be lifts of
  $\widebar{u}$ and $\widebar{f}$ respectively.  Then $v - f - u \in
  \Image N = \langle N \zeta, N^2 \zeta, N^3\zeta \rangle$.  Since $N
  \zeta \in F^2$ and $N^2 \zeta$,~$N^3 \zeta \in U_1$, it follows that
  $v \in F^2 + U_1$.  Thus if $\widebar{U}_\bullet$ is a filtration as
  in (B), we can define $U_\bullet$ as above to obtain a filtration as
  in (C) which satisfies $\widebar{U}_k = U_k/(\Image N \cap U_k)$.

  \subsubsection*{($\text{C} \implies \text{B}$)} Suppose that
  we are given a filtration $U_\bullet$ as in (C).  The filtration
  $U_\bullet$ is opposite to $F^\bullet$, and counting dimensions
  gives:
  \begin{align*}
    \dim U_0 = 1 &&
    \dim U_1 = 3 &&
    \dim U_2 = 5 
  \end{align*}
  The elements $\zeta \in U_3$, $N \zeta \in
  U_2$, $N^2 \zeta \in U_1$, and $N^3 \zeta \in U_0$ are non-zero 
and linearly independent; in particular $U_0 = \langle N^3 \zeta \rangle$.  We
  have $U_1 = \langle N^3 \zeta, N^2 \zeta, e_1 \rangle$ for some
  $e_1$.  Since $N e_1 \in U_0$ is a scalar multiple of $N^3 \zeta$ we
  may, by replacing $e_1$ with a linear combination of $e_1$ and $N^2
  \zeta$, without loss of generality assume that $N e_1 = 0$.  We have
  $U_2 = \langle N^3 \zeta, N^2 \zeta, e_1, N \zeta, e_2 \rangle$ for
  some $e_2$.  Replacing $e_2$ with a linear combination of $e_2$,
  $N\zeta$, and $N^2 \zeta$ we may without loss of generality assume
  that $N e_2 \in U_1$ is a scalar multiple of $e_1$.  Thus, with
  respect to the basis $N^3 \zeta$,~$N^2 \zeta$,~$e_1$,~$N
  \zeta$,~$e_2$,~$\zeta$ for $\cO(H)$, the matrix of $N$ has the form:
  \[
  \begin{pmatrix}
    0 & 1 & 0 & 0 & 0 & 0 \\
    0 & 0 & 0 & 1 & 0 & 0 \\
    0 & 0 & 0 & 0 & \ast & 0 \\
    0 & 0 & 0 & 0 & 0 & 1 \\
    0 & 0 & 0 & 0 & 0 & 0 \\
    0 & 0 & 0 & 0 & 0 & 0 
  \end{pmatrix}
  \]
  We know that the kernel of $N$ is three-dimensional
  (\S\ref{subsubsec:ker_im_N}), so $\ast$ must be zero and $N e_2 =0$.

  Set $\widebar{U}_k = U_k/(\Image N \cap U_k)$.  We find:
  \begin{align*}
    \widebar{U}_0 = 0 &&
    \widebar{U}_1 = \langle [e_1] \rangle &&
    \widebar{U}_2 = \langle [e_1], [e_2] \rangle  &&
    \widebar{U}_3 = \widebar{U}
  \end{align*}
  Now $Ne_1 = Ne_2 = 0$, so $\widebar{U}_2 \subset \Hvec$, and both
  spaces are two-dimensional, so $\widebar{U}_2 = \Hvec$.  Since $F^p
  \oplus U_{p-1} = H$, it follows that $\widebar{F}^p +
  \widebar{U}_{p-1} = \widebar{H}$.  For dimensional reasons we have
  $\widebar{F}^p \oplus \widebar{U}_{p-1} = \widebar{H}$.  Thus given
  a filtration $U_\bullet$ as in (C), setting $\widebar{U}_k =
  U_k/(\Image N \cap U_k)$ determines a filtration as in (B).

  \subsubsection*{($\text{C} \implies \text{D}$)} We construct the
  extension (D) using an appropriate opposite module, as in
  \S\ref{sec:trTLEP} but taking the base $\cM$ there to be the point
  $\{y\}$.  Let $\pi:\{y\} \times \Proj^1 \to \{y\}$ be the
  projection map and note, for comparison with \S\ref{sec:trTLEP},
  that:
  \begin{align*}
    \pi_\star(\cO_{\{y\} \times (\aroundinfinity)}) =
    \cO_{\Proj^1}\big(\aroundinfinity\big) && \text{and} &&
    \pi_\star(\cO_{\{y\} \times \C}) = \cO_{\Proj^1}(\C) \\
    \intertext{To match with \S\ref{sec:trTLEP}, write:}
    \cF^\divideontimes = \cFCY\big|_{\{y\} \times \C^\times}
    && \text{and} &&
    \bF = \pi_\star
    \big(\cFCY\big|_{\{y\} \times \C}\big)
  \end{align*}
  We construct the opposite module using the Rees construction.
  
  Recall that $\cO(H)$ is the submodule of $\pi_\star \cF^\divideontimes$ 
  consisting of degree one sections. 
  Define $\bP$ to be the
  $\cO_{\Proj^1}\big(\aroundinfinity\big)$-submodule of $\pi_\star
  \cF^\divideontimes$ spanned by:
  \[
  z^{-2} U_3 + z^{-1} U_2 + U_1 + z U_0
  \]
  The submodule $\bP$ is homogeneous.  Recall that $\bF$ is the
  $\cO_{\Proj^1}(\C)$-submodule of $\pi_\star \cF^\divideontimes$ spanned by:
  \[
  z^{-1} F^3 + F^2 + z F^1 + z^2 F^0
  \]
  The fact that $F^p \oplus U_{p-1} = H|_y$ implies that $\pi_\star
  \cF^\divideontimes = \bF \oplus \bP$.  The facts that $U_0^\perp = U_2$ and
  $U_1^\perp = U_1$ imply that $\bP$ is isotropic.  Thus $\bP$ is an
  opposite module.  The discussion in \S\ref{sec:trTLEP} produces from
  $\bP$ an extension of $\cFCY\big|_{\{y\} \times \C}$ to a
  locally free sheaf $\cE$ on $\{y\} \times \Proj^1$ such that: 
  \begin{itemize}
  \item the corresponding holomorphic vector bundle on $\{y\} \times
    \Proj^1$ is trivial;
  \item the pairing $\pairingB$ extends holomorphically across
    $z=\infty$ and is non-degenerate there;
  \item the connection $\nabla$ has a logarithmic pole at
    $z=\infty$.
  \end{itemize} 
  Recall that $z\bP$ consists of sections of $\cE$ over $\aroundinfinity$. 
  The fact that $N(U_p) \subset U_{p-1}$ implies that the map $N$
  extends holomorphically across $z=\infty$ and vanishes there.  Thus
  a filtration $U_\bullet$ as in (C) determines an extension as in
  (D).
    
  \subsubsection*{($\text{D} \implies \text{C}$)}
  Consider again the discussion in \S\ref{sec:trTLEP} with the base
  $\cM$ there taken to be the point $\{y\}$.  An extension $\cE$ 
as in (D) determines an opposite module 
$\bP= z^{-1} \pi_\star\big(\cE|_{\{y\}\times (\aroundinfinity)}\big)$. 
Set: 
\[
U_p = z^{p-1} \bP \cap H 
\]
This defines an increasing filtration. 
Recall that we have 
\[
F^p = z^{p-2} \bF \cap H
\]
The grading operator $\Gr$ preserves $z\bP \cap \bF \cong \bF/z\bF$ 
and is semisimple there, 
and therefore $\bP$ is generated by homogeneous elements 
over $\cO_{\Proj^1}(\aroundinfinity)$. 
Thus the decomposition $z^{p-2} \bF \oplus z^{p-2} \bP 
= \pi_\star \cF^\divideontimes$ restricted to degree one part 
implies $U_{p-1} \oplus F^p = H$. 
The fact that $\bP$ is isotropic implies that 
$\Omega(U_p, U_{2-p})=0$ and thus one has 
$U_p = U_{2-p}^\perp$ for dimension reasons.  
Furthermore $N(U_p) \subset U_{p-1}$ follows from the 
fact that $N$ extends across $z=\infty$ and vanishes there. 
Thus an extension as in (D) determines a filtration $U_\bullet$ as
in (C).
\end{proof}

\subsubsection{Opposite Filtration at the Cusps $y=0,-\frac{1}{27},\infty$} 
At the large-radius point $y=0$, the conifold point $y=-\frac{1}{27}$ and 
the orbifold point $y=\infty$, we have distinguished free 
extensions of $\cFCY|_{\{y\}\times \C }$ to $\{y\}\times \Proj^1$ 
characterized by local monodromy around them. 
By Proposition \ref{pro:correspondence_of_opposites}, 
each of them corresponds to a line $P$ in the fibre of $\Hvec$ at $y$. 

\begin{proposition} 
\label{pro:opposite_cusps} 
Let $y$ be one of the three points $\{0,-\frac{1}{27},\infty\}$ 
in $\cMCY = \Proj(3,1)$. 
There exists a unique extension of $\cFCY|_{\{y\}\times \C}$ 
to a locally free $\cO_{\{y\}\times \Proj^1}$-module $\cE$ 
such that the condition (D) of Proposition \ref{pro:correspondence_of_opposites} 
holds and that, in addition:
\begin{itemize} 
\item when $y$ is the large radius limit point or the conifold point, 
the residue endomorphism $\cFCY|_{\{y\}\times \C} 
\to z^{-1}\cFCY|_{\{y\}\times \C}$ of the connection $\nabla$ at $y$ extends regularly across $z=\infty$ 
and vanishes there; 

\item when $y$ is the orbifold point, the action of  $\Aut(y) = \mu_3$ on $\cFCY|_{\{y\}\times \C}$ 
extends across $z=\infty$.  Here $\mu_3$ acts trivially on the base $\{y\}\times \C$. 
\end{itemize} 
The free extensions of $\cFCY|_{\{y\}\times \C}$ to $\{y\}\times \Proj^1$ 
are given explicitly by the following bases: 
\begin{align*} 
& \text{$1$, $ D_2$, $D_2^2$, $D_2^3$, 
      $D_1$, $ D_1^2$} & & \text{(large radius limit point $y=0$)} \\
& \text{$1$, $D_2$, $D_2^2$, $D_2^3$, 
      $D_1$, $(1+27y_1) D_1^2$} && \text{(conifold point $y=-\tfrac{1}{27}$)} \\
& \text{$1$, $\frD_2$, $\frD_2^2$, $\frD_2^3$, 
 $\frd_1$, $\frd_1^2$} && \text{(orbifold point $y=\infty$)} 
\end{align*} 
Let $P_{\LR}$, $P_{\con}$, $P_{\orb}$ denote the corresponding 
subspace $P$ of $\Hvec|_y$ at the large radius, conifold and orbifold points 
under the correspondence between (A) and (D) in Proposition 
\ref{pro:correspondence_of_opposites}. They are given by: 
\begin{align*} 
P_{\LR} = \langle \theta^2 \zeta \rangle, \qquad 
P_{\con} = \langle (1+27y_1)\theta^2 \zeta \rangle \qquad 
P_{\orb} = \langle z^{-1} \frd_1^2 \rangle = 
\langle \fry_1^{-2} \theta(\theta+\tfrac{1}{3}) \zeta \rangle  
\end{align*} 
\end{proposition}

\begin{proof} We discuss the three cases $y=0, -\frac{1}{27}, \infty$ 
separately. 

($y=0$, existence) 
Take the frame of $\cFCY|_{\{0\}\times \C}$ 
described in the proposition. 
Recall that $\cFCY$ is the restriction of the B-model \logDTEP 
structure $\cFB$ to $\cMCY$. 
The connection $\nablaB$ defines two residue endomorphisms $N_i 
\colon \cFCY|_{\{0\}\times \C} \to z^{-1}\cFCY|_{\{0\}\times\C}$ about 
the divisors $y_i=0$, $i=1$,~$2$. 
The map $N_2$ equals $N$ in \S \ref{sec:six}. 
By Proposition \ref{pro:GKZ}, $N_i$
are represented by the matrices: 
\begin{align*}
    -\frac{1}{z}  
    \begin{pmatrix}
        0 & 
        0 &
        0 &
        0 & 
        0 &
        0 \\
        0
        & 0
        & 0
        & 0
        & 0
        & 0 \\
        0
        & \frac{1}{3}
        & 0
        & 0
        & 0
        & 0 \\
        0
        & 0
        & \frac{1}{3}
        & 0
        & 0
        & 0 \\
        1
        & 0
        & 0
        & 0
        & 0
        & 0 \\
        0
        & 0
        & 0
        & 0
        & 1
        & 0
      \end{pmatrix} 
      &&
      -\frac{1}{z}
      \begin{pmatrix}
        0
        & 0
        & 0
        & 0
        & 0
        & 0 \\
        1
        & 0
        & 0
        & 0
        & 0
        & 0 \\
        0
        & 1
        & 0
        & 0
        & \frac{1}{3}
        & 0 \\
        0
        & 0
        & 1
        & 0
        & 0
        & \frac{1}{9}\\
        0
        & 0
        & 0
        & 0
        & 0
        & 0 \\
        0
        & 0
        & 0
        & 0
        & 0
        & 0
      \end{pmatrix}
\end{align*}
respectively for $i=1$ and $i=2$. 
These are regular at $z=\infty$ and vanish there. 
The connection $\nabla_{z\partial_z}$ equals 
$\Gr - 2 N_2 -\frac{3}{2}$ along $\{0\}\times \C$ 
(see equation~\eqref{eq:GKZz}). 
This is regular at $z=\infty$ since the frame is homogeneous. 
The Gram matrix of the pairing $\pairingB$ along $\{0\}\times \C$ 
is independent of $z$ by Proposition \ref{pro:GKZ}. 
The frame thus gives an extension of $\cFCY|_{\{0\}\times \C}$ 
to $\{0\}\times \Proj^1$ satisfying the conditions in the proposition. 

($y=0$, uniqueness) 
Suppose that we have an extension of 
$\cFCY|_{\{0\}\times \C}$ 
to a free $\cO_{\{0\}\times \Proj^1}$-module $\cE$ 
satisfying the conditions in the proposition. 
Set 
\[
V := \Gamma(\Proj^1, \cE) \subset \Gamma(\C, 
\cFCY|_{\{0\}\times \C})
\] 
Recall that $\Gr$ acts on $\cFCY|_{\{0\}\times \C}$. 
It preserves the space $V$ since 
$\Gr = \nabla_{z\partial_z} + 2N_2 + \frac{3}{2}$ 
is regular at $z=\infty$. 
Therefore $V$ is graded. 
We have the graded isomorphism $\cFCY|_{(0,0)} \cong V$. 
Under this isomorphism $1 \in \cFCY|_{(0,0)}$ corresponds to a 
degree-zero global section of $\cFCY|_{\{0\}\times\C}$ which restricts to $1$ at $z=0$,
but $1$ is the only such global section and therefore $1\in V$. 
The operators $zN_1$, $zN_2$ are regular at both $z=0$ and $z=\infty$
and thus they act on $V$. 
Therefore $\C[zN_1,zN_2] \cdot 1 \subset V$. 
On the other hand, $-zN_i$ is given by the multiplication by 
$D_i$ in the GKZ system, and thus 
$\C[zN_1,zN_2] \cdot 1$ contains a 6-dimensional subspace 
spanned by $1,D_2, D_2^2, D_2^3, D_1, D_2^2$. 
Hence $V = \C[z N_1, zN_2]\cdot 1$. The conclusion follows. 

($y=-\frac{1}{27}$, existence) This is essentially identical to ($y=0$, existence).  
We use Proposition \ref{pro:GKZ} again. 
 
($y=-\frac{1}{27}$, uniqueness) 
Suppose that we have an extension of 
$\cFCY|_{\{-\frac{1}{27}\}\times \C}$ 
to a free $\cO_{\Proj^1}$-module $\cE$ 
satisfying the conditions in the proposition.  
Set $V =\Gamma(\Proj^1,\cE) \subset 
\Gamma(\C, \cFCY|_{\{-\frac{1}{27}\}\times \C})$ 
as before. 
For the same reason as in ($y=0$, uniqueness), $V$ is graded 
and is preserved by the operators $zN_2$, $zN_t$, 
where $N=N_2, N_t$ are the residue endomorphisms 
along the divisors $y_2=0$ and $t=y_1 + \frac{1}{27} = 0$ respectively. 
Therefore, under the graded isomorphism 
$\cFCY|_{(-\frac{1}{27},0)}\cong V$, 
the homogeneous basis $1, D_2, D_2^2, D_2^3, 
D_1, (1+27y_1) D_1^2$ of $\cFCY|_{(-\frac{1}{27},0)}$ 
lifts to a basis of $V$ of the form: 
  \[
  \text{$1$, $D_2$, $D_2^2$, $D_2^3$, $D_1+\alpha z 1$, $zN_t(D_1+\alpha z1)$}
  \]
for some $\alpha$, where we used the fact that 
$z N_t (D_1+\alpha z 1)= (1+27y_1) D_1^2$. 
We have $-z N_2(D_1+\alpha z 1) = \frac{1}{3} D_2^2 + 
\alpha z D_2$ by 
Proposition \ref{pro:GKZ} and it has to lie in $V$. 
Therefore $\alpha=0$. The result follows.

($y=\infty$, existence) This is essentially identical to ($y=0$, existence). 
We use Proposition \ref{pro:GKZ} again. 

($y=\infty$, uniqueness) Once again, 
suppose that we have an extension of 
$\cFCY|_{\{\infty\}\times \C}$ 
to a free $\cO_{\{\infty\}\times \Proj^1}$-module $\cE$ 
satisfying the conditions in the proposition. 
Set $V =\Gamma(\Proj^1,\cE) \subset 
\Gamma(\C, \cFCY|_{\{\infty\}\times \C})$. 
As before $V$ is graded and preserved by the residue 
endomorphism $zN = \frD_2$. 
Therefore a homogeneous basis 
$1$, $\frD_2$, $\frD_2^2$, $\frD_2^3$, 
$\frd_1$, $\frd_1^2$ of $\cFCY|_{(\infty,0)}$ 
lifts to a basis of $V$ of the form: 
  \[
  \text{$1$, $\frD_2$, $\frD_2^2$, $\frD_2^3$, $\frd_1 + \alpha' z 1$, 
    $\frd_1^2 + \beta' z \frd_1 + \gamma' z \frD_2 + \delta' z^2 1$}
  \]
for appropriate scalars $\alpha'$,~$\beta'$,~$\gamma'$,~$\delta'$.
The space $V$ is invariant under the $\Z/3Z$-action; 
$\Z/3Z$ acts by $\frd_1 \mapsto e^{2\pi\iu/3} \frd_1$, 
$\frD_2 \mapsto \frD_2$.  
Thus $\alpha'=0$, as otherwise $V$ contains both $1$ and $z$, 
contradicting the fact that $V \cong \cFCY|_{(\infty,0)}$.  
Similarly $\beta'=\gamma'=\delta'=0$.  This completes
  the proof.
\end{proof}

\section{Enlarging the Base of the B-Model \logDTEP Structure}
\label{sec:enlarge_base}

In this section we enlarge the base of the B-model \logDTEP structure
(see \S\ref{sec:2d_log_TEP}) in such a way that the enlarged \logDTEP
structure, which we call the \emph{big B-model \logDTEP structure}, is
miniversal (Definition \ref{def:miniversal}).  The process of
enlarging the base, described below, should be an example of a
universal unfolding of \logDTEP structure.  We prove:

\begin{theorem} 
\label{thm:unfolding} 
Let $(\cFB, \nablaB,\pairingB)$ be the B-model \logDTEP structure 
with base $(\cMB,D)$ in \S \ref{sec:2d_log_TEP}. 
Let $\cMB^\circ := \cMB \setminus \{y_1=-1/27\}$ be the 
complement of the conifold locus. 
We have 
\begin{itemize} 
\item a 6-dimensional complex manifold $\cMBbig$ 
\item a closed embedding $\iota \colon \cMB^\circ \to \cMBbig$

\item a divisor $\Dbig$ in $\cMBbig$ such that $\iota^{-1} \Dbig =
D \cap \cMB^\circ$; 

\item a miniversal \logDTEP structure 
$(\cFBbig, \nablaB, \pairingB)$ with base $(\cMBbig, \Dbig)$ 
such that 
$\iota^\star(\cFBbig, \nablaB, \pairingB)$ is isomorphic 
to $(\cFB, \nablaB, \pairingB)|_{\cMB^\circ\times \C}$. 
\end{itemize} 
We call the triple $(\cFBbig, \nablaB, \pairingB)$ the 
\emph{big B-model \logDTEP structure}. 
\end{theorem} 

We construct the enlarged base for the B-model TEP structure using
Reichelt's universal unfolding for \logDtrTLEP structures.  The
argument is in three steps, as follows.  In the first step we
construct, for each $y \in \cMB$, a \logDtrTLEP structure
on a neighbourhood $U_y$ of $y$.  In the second step we delete the
conifold locus $y_1 = {-\frac{1}{27}}$ (because Reichelt's generation
condition fails there) and apply Reichelt's unfolding result to
construct a miniversal \logDtrTLEP structure on $U_y \times V_y$,
where $V_y$ is a neighbourhood of the origin in $\C^4$, such that the
restriction to $U_y \times \{0\}$ is the \logDtrTLEP structure
constructed in the first step.  In the third step we show that the
\logDTEP structures that underly the \logDtrTLEP structures
constructed in step two are compatible on chart overlaps, and thus
assemble to give a global miniversal \logDTEP structure over a
six-dimensional base $\cMBbig$.  (The \logDtrTLEP structures themselves
are in general not compatible with each other on chart
overlaps.)\phantom{.} The six-dimensional base $\cMBbig$ contains
$\cMB^\circ$ as a subset.

\subsection{Step 1: Constructing \logDtrTLEP Structure Locally} 
We begin with a general method to construct 
a \logDtrTLEP structure near a unipotent monodromy 
point of a \logDTEP structure. 
As we discussed in \S \ref{sec:trTLEP}, 
an opposite module for a TEP structure gives rise 
to a trTLEP structure and a flat trivialization 
(Definition \ref{def:flat_trivialization}). 
Suppose that a \logDTEP structure with base $(\cM,D)$ 
is the Deligne extension of a TEP structure with base 
$\cM\setminus D$ (Definition \ref{def:Deligne}). 
In this case, a flat trivialization of the TEP structure given by 
an opposite module does not necessarily extend to 
the \logDTEP structure. 
We introduce below the notion of 
``compatibility with Deligne extension" for an opposite module.  
This describes a certain special situation 
where the flat trivialization extends to a trivialization of 
the \logDTEP structure and yields a \logDtrTLEP structure. 
The resulting \logDtrTLEP structure is very special: the residue 
endomorphisms are nilpotent and vanish at $z=\infty$.  
We then show that opposite modules near $p$ 
for the \logDTEP structure which is compatible with the Deligne 
extension is uniquely determined by a trivialization of 
the \logDTEP structure over $\{p\}\times \C$ 
satisfying certain conditions. 
Finally we apply this method to the B-model \logDTEP structure 
and construct a \logDtrTLEP structure locally on $\cMB$.

\begin{definition} 
\label{def:compatible_with_Deligne}
Let $\cM$ be a complex manifold with normal crossing 
divisor $D$. 
Let $(\cF, \nabla, (\cdot,\cdot))$ be a \logDTEP 
structure with base $(\cM,D)$ which 
is the Deligne extension of a TEP structure 
$(\cF^\times,\nabla,(\cdot,\cdot))$ with base $\cM \setminus D$ 
(Definition \ref{def:Deligne}). 
Let $p$ be a point in $\cM$ and let 
$U_p$ be a contractible open neighbourhood of $p$ 
such that every (nonempty) irreducible component 
of $D\cap U_p$ contains $p$. 
An opposite module $\bP$ for $(\cF^\times,\nabla,(\cdot,\cdot))$ 
defined over $U_p\setminus D$ 
is said to be \emph{compatible with the Deligne extension} 
$(\cF,\nabla,(\cdot,\cdot))$ 
if the following conditions are satisfied: 
\begin{itemize} 
\item[(1)] the flat trivialization of $\cF^\times|_{(U_p \setminus D)\times \C}$ 
associated to $\bP$ (Definition \ref{def:flat_trivialization}) 
has no monodromy around $D$ 
and thus defines a locally free extension $\cE$ of $\cF$ to 
$U_p\times \Proj^1$ such that the corresponding vector bundle 
over $U_p\times \Proj^1$ is trivial; 

\item[(2)] the connection $\nabla$ defines a meromorphic flat connection 
on $\cE$ with: 
\[
\nabla \colon \cE \to \Omega^1_{U_p\times \Proj^1}(\log Z) 
\otimes \cE(U_p\times \{0\}) 
\]
where $Z = (D \times \Proj^1) \cup (U_p \times \{0\}) 
\cup (U_p \times \{\infty\})$;  

\item[(3)] the pairing $(\cdot,\cdot)$ extends holomorphically 
across $(U_p \times \{\infty\}) \cup ((D\cap U_p) \times \Proj^1)$ 
and is non-degenerate there; 

\item[(4)] the residue endomorphisms of $\nabla$ along 
$(D\cap U_p)\times (\Proj^1\setminus\{0\})$ are nilpotent and 
vanish at $(D\cap U_p) \times \{\infty\}$. 
\end{itemize} 
Condition (4) implies that $(\cE, \nabla, (\cdot,\cdot))$ 
coincides with the Deligne extension $(\cF,\nabla,(\cdot,\cdot))$ over 
$U_p \times \C$, because the Deligne extension is the unique 
logarithmic extension such that the residue endomorphisms are 
nilpotent. 
\end{definition} 

\begin{remark} 
\label{rem:compatibility_orbifoldcase}
When the base $\cM$ has an orbifold singularity at $p$ 
and the Deligne extension $\cF$ is an orbi-sheaf 
(e.g.~the B-model \logDTEP structure, see Remark 
\ref{rem:orbisheaf}), 
we define the compatibility with the Deligne extension near $p$ 
by replacing $U_p$ with the uniformizing chart and 
requiring the same conditions (1)--(4) in 
Definition \ref{def:compatible_with_Deligne} over the uniformizing chart. 
The locally free sheaf $\cE$ on $U_p \times \Proj^1$ in (1) becomes 
$\Aut(p)$-equivariant, where $\Aut(p)$ is the finite automorphism group 
at $p$ which acts on $U_p$. The connection $\nabla$ and the pairing 
$(\cdot,\cdot)$ on $\cE$ are invariant under the $\Aut(p)$-action. 
\end{remark} 

\begin{remark} 
Compatibility with the Deligne extension has been 
discussed in the context of the Crepant Resolution Conjecture 
and mirror symmetry: see 
\cite{CIT:wall-crossings}*{Theorem 3.5} and 
\cite{Iritani:Ruan}*{\S 3.5} where 
a characterization of the A-model opposite module 
is given at certain cusps in the B-model moduli space. 
\end{remark} 

It is convenient to rephrase the above conditions (1)--(4) 
in Definition \ref{def:compatible_with_Deligne} in terms 
of an explicit trivialization. 
Choose local co-ordinates $(x_1,\dots,x_r, y_1,\dots,y_s)$ of 
$\cM$ centred at $p\in \cM$ such that the divisor 
$D\cap U_p$ can be written as $x_1x_2 \cdots x_r =0$.
(We set $r=0$ if $p\notin D$.) 
Then an opposite module $\bP$ 
compatible with the Deligne extension 
yields a trivialization of $\cF|_{U_p\times \C}$ 
with the following properties: 
\begin{itemize}
\item 
the connection in the trivialization takes the form: 
\begin{equation} 
\label{eq:logconn_flat_triv}
d + \frac{1}{z} \left ( 
\sum_{i=1}^r 
A_i(x,y) \frac{dx_i}{x_i} + 
\sum_{i=1}^s B_i(x,y) dy_i  
+ (C_0(x,y) + z C_1(x,y)) \frac{dz}{z} 
\right ) 
\end{equation} 
where $A_i$, $B_i$, $C_0$, $C_1$ are matrix-valued holomorphic functions 
on $U_p$ such that the residue endomorphisms 
$A_i|_{x_i=0}$ are nilpotent; 

\item 
the Gram matrix of the pairing $(\cdot,\cdot)$ 
is constant with respect to the trivialization. 
\end{itemize} 
This trivialization extends the flat trivialization of 
$\cF^\times|_{(U_p \setminus D)\times \C}$ associated to $\bP$, 
and we refer to it as a \emph{flat trivialization} of $\cF$ 
associated to $\bP$. 
Conditions (1)--(3) in Definition~\ref{def:compatible_with_Deligne} 
imply that an opposite module $\bP$ compatible with 
the Deligne extension yields a \logDtrTLEP
structure with base $U_p$ in the sense of Reichelt 
\cite{Reichelt}*{Definition~1.8}. 
Note however that Reichelt's notion of \logDtrTLEP structure is more general 
than our notion of `compatibility with the Deligne extension':  
the connection $\nabla$ of a \logDtrTLEP structure has a 
form similar to \eqref{eq:logconn_flat_triv} but the term 
$A_i$ there can depend linearly on $z$, i.e.~$A_i = A_{i0}(x,y) 
+ A_{i1}(x,y) z$. 

\begin{remark} 
\label{rem:C1} 
In view of the proof of 
Proposition~\ref{pro:Birkhoff_logarithmic_connection}, 
slightly more is true about the connection 
\eqref{eq:logconn_flat_triv}: 
$A_i|_{x_i=0}$ 
is independent of $x_1,\dots,x_{i-1},x_{i+1},\dots,x_r,y_1,\dots,y_s$ 
and $C_1(x,y)$ is independent of $x$ and $y$. 
These follow automatically from the flatness of the connection. 
\end{remark} 

The existence of an opposite module over $U_p\setminus D$ 
which is compatible with the Deligne extension 
is reduced to the existence of a trivialization 
of $\cF$ over $\{p\}\times \C$ (or equivalently, an extension 
of $\cF|_{\{p\}\times \C}$ to a free 
$\cO_{\{p\}\times \Proj^1}$-module) 
satisfying certain properties. 

\begin{proposition} 
\label{pro:extending_trivialization}
Let $D$ be a normal crossing divisor in $\cM$ and 
let $(\cF, \nabla,(\cdot,\cdot))$ be a \logDTEP 
structure with base $(\cM,D)$ which is the Deligne 
extension of a TEP structure 
$(\cF^\times,\nabla,(\cdot,\cdot))$ with base $\cM\setminus D$. 
Let $p$ be a point in $\cM$ and take local co-ordinates  
$(x_1,\dots,x_r, y_1,\dots,y_s)$ centred at $p$ such 
that $D$ can be written as $x_1 x_2 \cdots x_r = 0$ near $p$.  
(We take $r=0$ when $p\notin D$.)
There is a one-to-one 
correspondence between the following: 
\begin{itemize} 
\item[(a)] an extension of $\cF|_{\{p\}\times \C}$ 
to a free $\cO_{\{p\}\times \Proj^1}$-module such that  
\begin{itemize} 
\item the residue endomorphisms 
$\nabla_{x_i \partial_{x_i}} |_{p}  
\colon \cF|_{\{p\}\times \C} \to z^{-1}
\cF|_{\{p\}\times \C}$, $i=1,\dots,r$ extend  
regularly across $z=\infty$ and vanish there; 

\item the connection $\nabla$ on $\cF|_{\{p\}\times \C}$ 
has a logarithmic pole at $z=\infty$, i.e.~$\nabla_{z\partial_z}$ 
is regular at $z=\infty$; 

\item the pairing $(\cdot,\cdot)$ on $\{p\}\times \C$ 
extends regularly across $z=\infty$ and is non-degenerate there; 

\item when $p$ is an orbifold point, 
the $\Aut(p)$-action on $\cF|_{\{p\}\times \C}$ extends across $z=\infty$. 
\end{itemize} 

\item[(b)] an opposite module $\bP$ for $(\cF^\times,\nabla,(\cdot,\cdot))$, 
defined near $p$, which
is compatible with the Deligne extension 
$(\cF,\nabla,(\cdot,\cdot))$.  
\end{itemize} 
\end{proposition}

\begin{proof} 
Let $U_p$ be a contractible open neighbourhood of $p$ in $\cM$ 
such that every irreducible component of $D\cap U_p$ contains $p$. 
(If $p$ is an orbifold point, we take $U_p$ to be a uniformizing 
chart.)
In view of the discussion after 
Definition~\ref{def:compatible_with_Deligne}, 
an opposite module $\bP$, defined over $U_p\setminus D$, 
which is compatible with the Deligne extension yields a flat trivialization 
of $\cF$ over $U_p\times \C$ such that 
\begin{itemize} 
\item[(i)] the connection $\nabla$ in the trivialization 
takes the form: 
\begin{equation}
\label{eq:logconn_normalform}
d + \frac{1}{z} \left( 
\sum_{i=1}^r A_i(x,y) \frac{dx_i}{x_i} 
+ \sum_{i=1}^s B_i (x,y) dy_i 
+ (C_0(x,y) + z C_1(x,y)) \frac{dz}{z} 
\right) 
\end{equation} 
where $A_i$, $B_i$, $C_0$, $C_1$ are matrix-valued 
holomorphic functions on $U_p$ and 
$A_i|_{x_i=0}$ is nilpotent;  
\item[(ii)] the pairing $(\cdot,\cdot)$ is constant with respect 
to the trivialization. 
\end{itemize} 
Restricting the trivialization to $\{p\}\times\C$, 
we obtain an extension of $\cF|_{\{p\}\times \C}$ 
to a free $\cO_{\{p\}\times \Proj^1}$-module 
satisfying the conditions in (a). 
When $p$ is an orbifold point, recall from 
Remark~\ref{rem:compatibility_orbifoldcase} that 
$\cF|_{U_p\times \C}$ extends, via the trivialization, 
to an $\Aut(p)$-equivariant free $\cO_{U_p\times \Proj^1}$-module $\cE$.

Conversely, suppose that 
we have an extension of $\cF|_{\{p\}\times \C}$ 
to a free $\cO_{\{p\}\times \Proj^1}$-module satisfying the conditions in (a). 
We take a trivialization of $\cF|_{\{p\}\times \C}$ which 
yields the free extension. 
We shall show that there exists a unique trivialization of $\cF$ 
over $U_p\times \C$ extending the trivialization 
over $\{p\}\times \C$ 
and satisfying the properties (i)--(ii) listed above. 

To see the existence of a trivialization over $U_p\times \C$, 
we first extend the given trivialization of $\cF$ 
along $\{p\}\times \C$ 
to $U_p\times \C$ in an arbitrary way (shrinking $U_p$ if necessary). 
The connection $\nabla$ in the trivialization takes the form 
\[
d + \frac{1}{z} 
\left ( \sum_{i=1}^r A_i(x,y,z) \frac{dx_i}{x_i} 
+ \sum_{i=1}^s B_i(x,y,z) dy_i 
+ C(x,y,z) \frac{dz}{z} \right). 
\] 
where $A_i, B_i, C$ are matrix-valued holomorphic functions 
on $U_p \times \C$, 
$A_i(0,0,z)$, $1\le i\le r$ 
are nilpotent and independent of $z$ 
and $C(0,0,z)$ depends linearly on $z$, i.e. 
$C(0,0,z) = C_0(0,0) + z C_1(0,0)$. 
By Propositions \ref{pro:fundamental_solution_logarithmic_connection} 
and \ref{pro:Birkhoff_logarithmic_connection}, after shrinking $U_p$ 
if necessary, there exists a gauge transformation $L_+$ 
defined on $U_p\times \C$ such that 
$L_+|_{\{p\} \times \C} = \id$ and  that 
this connection is transformed to a connection of the 
form \eqref{eq:logconn_normalform} by $L_+$. 
By Proposition \ref{pro:constant_pairing}, the 
Gram matrix of the pairing $(\cdot,\cdot)_\cF$ is constant 
over $U_p\times \C$ after the gauge transformation. 

Next we show the uniqueness of such a trivialization. 
Suppose we have a gauge transformation 
$G$ such that $G|_{\{p\} \times \C} =\id$ and that 
$G$ transforms the connection \eqref{eq:logconn_normalform} 
to a connection of the same form: 
\begin{equation}
\label{eq:normalform2}
 d + \frac{1}{z} 
\left(\sum_{i=1}^r A'_i(x,y) \frac{dx_i}{x_i} 
+  \sum_{i=1}^s B'_i(x,y) dy_i 
+ (C_0'(x,y) + z C_1'(x,y)) \frac{dz}{z} \right) 
\end{equation} 
where $A_i'(0,y)$, $1\le i\le r$ are nilpotent. 
By Proposition \ref{pro:fundamental_solution_logarithmic_connection}, 
the connections \eqref{eq:logconn_normalform} and \eqref{eq:normalform2} 
admit respectively unique 
``fundamental solutions in the $U$-direction" of the form: 
\[
\tilde{L}(x,y,z)e^{\sum_{i=1}^r A_i(0,0) \log x_i /z}, \quad 
\tilde{L'}(x,y,z)e^{\sum_{i=1}^4 A'_i(0,0)\log x_i /z} 
\] 
satisfying the initial conditions 
$\tilde{L}(0,0,z) = \tilde{L'}(0,0,z)=\id$. 
Then we have 
\[
\tilde{L'}(x,y,z)e^{-\sum_{i=1}^r A'_i(0,0)\log x_i/z} = G(x,y,z) 
\tilde{L}(x,y,z)e^{-\sum_{i=1}^r A_i(0,0) \log x_i/z}  
\]
Since the trivialization along $\{p\}\times \C$ is fixed, 
the residue endomorphisms are the same $A_i(0,0) = A'_i(0,0)$. 
Since the connections \eqref{eq:logconn_normalform}, \eqref{eq:normalform2} 
in the $U$-direction are trivial along $z=\infty$,  
$\tilde{L}$ and $\tilde{L'}$ are regular on 
$U_p \times (\Proj^1\setminus \{0\})$ and 
$\tilde{L}|_{z=\infty} = \tilde{L'}|_{z=\infty}=\id$. 
Therefore $G$ has to be the identity on $U_p \times \C$. 

When $p$ is an orbifold point, we additionally need to check that 
the opposite module corresponding to the trivialization of 
$\cF|_{U_p\times \C}$ is well-defined on 
the quotient $(U_p\setminus D)/\Aut(p)$. 
(The trivialization itself may not descend to the 
quotient.) 
It suffices to show that each $g\in \Aut(p)$ acts 
on the trivializing frame by a constant matrix 
(independent of $z$). 
This follows from the uniqueness statement: 
let $s_0,\dots,s_N$ be the trivializing frame 
of $\cF|_{U_p\times \C}$ 
and define a matrix-valued function $M$ on $U_p\times \C$ 
by 
$[g\cdot s_0,\dots, g\cdot s_N] = [s_0,\dots,s_N] M$. 
By the last condition in (a), $M_p := M|_{\{p\}\times \C}$ is 
a constant matrix independent of $z$. 
The frame $[s_0,\dots,s_N]M_p$ 
yields a trivialization of $\cF|_{U_p\times \C}$ 
satisfying the properties (i)--(ii) above 
since $M_p$ is constant. 
On the other hand, since $\nabla$ and $(\cdot,\cdot)$ are 
$\Aut(p)$-invariant, the connection matrices and 
Gram matrix of the pairing $(\cdot,\cdot)$ do not 
change under the gauge transformation by $M$, 
and hence the trivialization given by the frame 
$[s_0,\dots, s_N]M$ 
also satisfies the properties (i)--(ii) above. 
The uniqueness argument shows that the two trivializing 
frames are the same, 
i.e.~$M=M_p$ is a constant matrix. 
\end{proof} 

We now apply the above general method to 
the B-model \logDTEP structure $(\cFB, \nablaB, \pairingB)$. 
Recall from \S \ref{sec:2d_log_TEP} that $(\cFB, \nablaB,\pairingB)$ is the 
Deligne extension of the B-model TEP structure 
$(\cFB^\times, \nablaB, \pairingB)$ with logarithmic singularities along 
\[
D= \ov{\{y_1y_2 =0\}} \cup \{y_1=-1/27\}
\]
For each point $y$ in $\cMB$, 
we shall construct an opposite module $\bP$ 
for the B-model TEP structure over $U_y\setminus D$ 
for a sufficiently small neighbourhood $U_y$ of $y$, 
which is compatible with the Deligne extension. 
This yields a \logDtrTLEP structure with base $(U_y, U_y\cap D)$ 
which underlies the B-model \logDTEP structure. 

\subsubsection{Step 1, Case 1: 
$y \in \cMCY \setminus (\DCY\cup \{\infty\})$}

Let $y$ be of the form $y =(y_1,y_2)$ such that $y_2 = 0$ and 
$y_1 \ne 0$,~${-\frac{1}{27}}$,~$\infty$.  
In other words, we take $y \in \cMCY \setminus \big(\DCY \cup \{\infty\}\big)$. 
In this case there are many possible choices for $\bP$: 

\begin{proposition}
  \label{pro:existence_of_compatible_opposites}
Let $y \in \cMCY \setminus \big(\DCY \cup \{\infty\}\big)$. 
The following are equivalent:
  \begin{itemize}
  \item[(A)] a subspace $P$ of $\Hvec|_y$ such that $F^2_\vec \oplus P
    = \Hvec|_y$;
  \item[(B)] an opposite module $\bP$ defined on $U_y \setminus D$,
    where $U_y$ is a neighbourhood of $y$ in $\cMB$, such
    that $\bP$ is compatible with the Deligne extension.
  \end{itemize}
\end{proposition}
\begin{proof} 
By Proposition~\ref{pro:correspondence_of_opposites}, 
a subspace $P$ of $\Hvec|_y$ such that $F^2_\vec \oplus P = \Hvec|_y$ 
is equivalent to an extension of $\cFB|_{\{y\}\times \C}$ 
to a free $\cO_{\{y\}\times \Proj^1}$-module 
satisfying the condition (a) of Proposition~\ref{pro:extending_trivialization}.  
The conclusion follows from Proposition~\ref{pro:extending_trivialization}. 
\end{proof} 

\subsubsection{Step 1, Case 2: the Large-Radius, Conifold, and 
Orbifold Points. }
We next consider the large-radius, conifold, and orbifold points.  
In each case there is a unique choice for $\bP$. 
For the large-radius and the orbifold points, the uniqueness 
has been shown in \cite{CIT:wall-crossings}*{Theorem 3.5} 
for the case at hand, 
and in \cite{Iritani:Ruan}*{Theorem 3.13} for a more general target.

\begin{proposition} 
\label{pro:LR_conifold_orbifold} 
We have the following: 
  \begin{enumerate}
  \item Suppose that $y$ is the large-radius limit point $y =
    (y_1,y_2) = (0,0)$.  There is a unique opposite module 
  $\bP_\LR$, 
    defined near $y$, 
    which is compatible with the Deligne extension.  
    The corresponding flat trivialization of $\cFB$ along $\{y\}\times \C$ 
is given by the frame: 
 \[
   \text{$1$, $ D_2$, $D_2^2$, $D_2^3$, 
      $D_1$, $ D_1^2$}
\]   
\item Suppose that $y$ is the conifold point $y = (t,y_2) = (0,0)$.
There is a unique opposite module $\bP_\con$, defined near $y$,  
which is compatible with the Deligne extension. 
The corresponding flat trivialization of $\cFB$ along $\{y\}\times \C$ 
is given by the frame: 
\[
\text{$1$, $D_2$, $D_2^2$, $D_2^3$, 
      $D_1$, $(1+27y_1) D_1^2$}
\]
\item Suppose that $y$ is the orbifold point 
$y = (\fry_1,\fry_2) = (0,0)$.  
There is a unique opposite module $\bP_\orb$, defined near $y$, 
which is compatible with the Deligne extension.  
The corresponding flat trivialization of $\cFB$ along $\{y\}\times \C$ 
is given by the frame: 
\[
\text{$1$, $\frD_2$, $\frD_2^2$, $\frD_2^3$, 
 $\frd_1$, $\frd_1^2$}
\]
  \end{enumerate}
\end{proposition}

\begin{proof}
In all three cases, in view of Proposition~\ref{pro:extending_trivialization}, 
it suffices to check that there exists a unique 
extension of $\cFB|_{\{y\}\times \C}$ 
to a free $\cO_{\{y\}\times \Proj^1}$-module satisfying 
the condition (a) of Proposition~\ref{pro:extending_trivialization} 
and that it is defined by the frame given in the proposition. 
This has been proved in Proposition~\ref{pro:opposite_cusps}. 
\end{proof}

\subsubsection{Step 1, Case 3: $y\not \in
  \cMCY$} \label{sec:step_1_case_3} 
We now turn to the remaining case, where $y \notin \cMCY$.  
This means either that $y =(y_1,y_2)$ with $y_2 \ne 0$, 
or that $y = (\fry_1,\fry_2)$ with $\fry_1=0$ and $\fry_2 \ne 0$. 
We will use the fact that any connection $\nabla$ as in
Definition~\ref{def:compatible_with_Deligne} defined on $U
\times \C$ extends canonically to a connection on $V \times \C$, where
$V$ is the orbit of $U$ under the flow of the Euler field:
see e.g.~Kim--Sabbah~\cite{Kim--Sabbah}*{Example~1.3}.  
In the case at hand, the Euler field is $2 y_2 \partial_{y_2} = 2
\fry_1 \partial_{\fry_2}$.  The opposite submodules constructed in
Step~1, Cases~1 and~2, are defined on neighbourhoods $\{U_y\}$ that
together cover the locus $\cMCY \subset \cMB$
where $y_2 = 0$ or $\fry_2 = 0$, and so the orbits of these
neighbourhoods under the Euler flow cover all of $\cMB$.
Thus we construct, for any $y \in \cMB$ with $y \not \in
\cMCY$, a neighbourhood $U_y$ of $y$ in $\cMB$ and an
opposite module $\bP$ over $U_y\setminus D$ 
which is compatible with the Deligne extension. 

More precisely, we have the following statement: 

\begin{proposition}
  \label{pro:Euler_flow}
  Let $p \colon \cMB \to \cMCY$ be the map that sends
  $(y_1,y_2) \in \cMB$ to the point $(y_1,0) \in
  \cMCY \subset \cMB$, and which sends
  $(\fry_1,\fry_2) \in \cMB$ such that $\fry_1 = 0$ to the
  orbifold point $(\fry_1,\fry_2) = (0,0) \in \cMCY \subset
  \cMB$.  Let $y \in \cMCY$.  For a sufficiently
  small open neighbourhood $U_y$ of $y$, an opposite module defined
  over $U_y \setminus D$ which is compatible with the Deligne
  extension extends to an opposite module over
  $p^{-1}(p(U_y))\setminus D$.
\end{proposition}

\begin{proof}
Suppose for simplicity that $y \in \cMCY$   
is neither the large radius limit point, nor the conifold point,  nor the orbifold point. 
(The argument in these three cases is essentially identical.)
Then $p(x)$ is
the limit as $t \to -\infty$ of the image of $x$ under the time-$t$ flow 
  of the Euler field.  
With respect to the flat 
  trivialization defined by $\bP$, we have:
  \begin{equation*}
    \begin{aligned}
      & \nabla_{z \partial_z} = z \partial_z - 2 z^{-1} B(y_1,y_2) + C(y_1,y_2)
      \\
      & \nabla_{\partial_{y_1}} = \partial_{y_1} + z^{-1} A(y_1,y_2) \\
      & \nabla_{y_2 \partial_{y_2}} = y_2 \partial_{y_2} + z^{-1} B(y_1,y_2) 
    \end{aligned}
  \end{equation*}
  for $(y_1,y_2)$ in $U_y$, for some regular
  endomorphism-valued functions $A, B, C$ on $U_y$. 
  Flatness of $\nabla$ gives that $C$ is independent of $y_1$ and $y_2$ 
  (see Remark \ref{rem:C1}) 
  and yields the following differential equations: 
  \begin{align*}
  2 y_2 \partial_{y_2} B  & =B - [C,B]   \\
    2 y_2 \partial_{y_2} A & = 2 \partial_{y_1} B 
    = A - [C,A]
  \end{align*}
  These differential equations imply: 
  \begin{equation*}
    \begin{aligned}
      B(y_1,y_2t^2) & = t \cdot t^{-C} B(y_1,y_2) t^C \\
      A(y_1,y_2t^2) & = t \cdot t^{-C} A(y_1,y_2) t^C
    \end{aligned}
    \qquad \qquad t \in \C^\times
  \end{equation*} 
  The right-hand side defines an analytic continuation 
  of $B(y_1,y_2)$, $A(y_1,y_2)$ -- which are originally defined 
  only near $y_2=0$ -- to all of $V_y = p^{-1}(p(U_y))$. 
  By the discussion after Definition \ref{def:compatible_with_Deligne} 
  this yields an opposite module over $V_y \setminus D$ which 
  is compatible with the Deligne extension. 
\end{proof}

\begin{remark}
  This completes Step~1: we have constructed, for each $y \in
  \cMB$, a neighbourhood $U_y$ of $y$ in
  $\cMB$ and an opposite module $\bP$ over $U_y \cap
  \cMB^\times$ which is compatible with the Deligne extension.  In
  particular, $\bP$ determines a \logDtrTLEP structure with base
  $U_y$.
\end{remark}

\subsection{Step 2: Unfolding the \logDtrTLEP Structures Locally} 
\label{sec:unfolding_locally}
We now delete the conifold locus, $y_1 ={-\frac{1}{27}}$, from
$\cMB$, setting:
\[
\cMB^\circ := \big\{(y_1,y_2) \in \cMB : \textstyle
y_1 \ne {-\frac{1}{27}}\big\}
\]
Consider $y \in \cMB^\circ$, a neighbourhood $U_y$ of $y$ in
$\cMB^\circ$, and an opposite module $\bP$ over 
$U_y \setminus D$ such that $\bP$ is compatible with the Deligne extension, as
constructed in Step~1.  The choice of $U_y$ and $\bP$ defines a
\logDtrTLEP structure with base $U_y$ such that the underlying TEP
structure coincides with the B-model \logDTEP structure.  
The section $\xi$ of $\cFB$ corresponding 
to the element $1 \in \cFGKZ$ satisfies the conditions (IC), (GC),
(EC), and flatness in \cite{Reichelt}*{Theorem~1.12}.  We therefore
consider Reichelt's universal unfolding of our \logDtrTLEP structure. 
This is a \logDtrTLEP structure with base 
$(U_y \times V_y, (D\cap U_y)\times V_y)$, 
where $V_y$ is a neighbourhood of the origin in $\C^4$, such that the
restriction to $U_y \times \{0\}$ coincides with the \logDtrTLEP
structure with base $(U_y,D\cap U_y)$ defined by $\bP$. 
The underlying \logDTEP structure is miniversal in the sense of 
Definition \ref{def:miniversal}. 

\begin{remark}
  We delete the conifold locus $y_1 ={-\frac{1}{27}}$ because
  Reichelt's generation condition (GC) fails there.  
\end{remark}

\subsection{Step~3: A Global Miniversal TEP Structure}

Now that we have completed Steps~1 and~2, we are in the following
situation.  Given a sufficiently small open subset $U$ of
$\cMB^\circ$, there exists an opposite module $\bP$ over $U \setminus D$ 
that is compatible with the Deligne extension.  
Thus there exists a \logDtrTLEP structure with base 
$(U \times V, (U\cap D) \times V)$, 
where $V$ is an open neighbourhood of the origin in $\C^4$; 
this \logDtrTLEP structure is constructed as a universal unfolding 
of the \logDtrTLEP structure with base $(U,U\cap D)$ defined by $\bP$.  
The \logDtrTLEP structure with base
$(U \times V, (U\cap D)\times V)$ 
determines a \logDTEP structure with the same base, 
and we now show that all these \logDTEP structures glue together, 
after shrinking the base $U \times V$ if necessary, 
to give a global, miniversal \logDTEP structure, 
defined on a six-dimensional complex manifold $\cMBbig$ 
that contains $\cMB^\circ$ as a closed submanifold. 
This global \logDTEP structure is the \emph{big B-model \logDTEP structure}.

\subsubsection{The Gluing Map} 
\label{sec:gluing_map} 
To simplify the notation, when there is no risk of confusion, 
we denote a \logDTEP (or \logDtrTLEPns) structure 
simply by the corresponding locally free sheaf $\cF$, omitting the flat connection 
$\nabla$ and the pairing $(\cdot,\cdot)_\cF$. 

\begin{lemma} 
\label{lem:gluing_logTEP}
Let $U$ be an open set of $\cMB^\circ$. 
Suppose that we have opposite modules $\bP$, $\bP'$ for 
the B-model TEP structure $\cFB^\times$ over 
$U\setminus D$ that are compatible with the Deligne extension.  
These opposite modules define the \logDtrTLEP structures underlying 
the B-model \logDTEP structure $\cFB$. 
Suppose that $U$ is sufficiently small so that 
the \logDtrTLEP structures admit the following 
universal unfolding as in Step 2: 
\begin{align*} 
& \text{miniversal \logDtrTLEP structure 
$\cE_{\bP}$ with base $(U\times V, (U\cap D) \times V)$} \\  
& \text{miniversal \logDtrTLEP structure 
$\cE_{\bP'}$ with base $(U\times V', (U\cap D) \times V')$} 
\end{align*}  
where $V$, $V'$ are open neighbourhoods of the origin in $\C^4$. 
We write $\cF_{\bP} = \cE_{\bP}|_{(U\times V) \times \C}$ 
and $\cF_{\bP'} = \cE_{\bP'}|_{(U\times V')\times \C}$ for 
the underlying \logDTEP structures. 
Let $\theta_{\bP\bP'}$ denote the canonical isomorphism 
of \logDTEP structures 
\[
\theta_{\bP\bP'} \colon 
\cF_{\bP}|_{U\times \{0\}\times \C} \cong 
\cFB|_{U\times \C} \cong 
\cF_{\bP'}|_{(U\times \{0\}) \times \C} 
\]
given by the construction.  
There exist open sets $O_{\bP \bP'} \subset U \times V$, 
$O_{\bP' \bP} \subset U \times V'$ and 
a biholomorphic map $\varphi_{\bP \bP'} \colon 
O_{\bP \bP'} \to O_{\bP' \bP}$ such that: 
\begin{itemize}
\item $U \times \{0\} \subset O_{\bP \bP'}$ and $U \times \{0\} \subset
  O_{\bP' \bP}$;
\item $\varphi_{\bP \bP'}|_{U \times \{0\}}$ is the identity map; 
\item $\varphi_{\bP\bP'}$ maps the divisor $((U\cap D)\times V) 
\cap O_{\bP\bP'}$
onto $((U\cap D) \times V')\cap O_{\bP'\bP}$; 
\item there is an isomorphism $\Theta_{\bP\bP'} \colon 
\cF_{\bP}|_{O_{\bP\bP'}\times\C} 
\to (\varphi_{\bP\bP'}\times \id)^\star(\cF_{\bP'}|_{O_{\bP'\bP}\times \C})$ 
of \logDTEP structures which restricts to $\theta_{\bP\bP'}$ 
over $(U\times \{0\})\times \C$. 
\end{itemize}
Moreover, the map $\varphi_{\bP\bP'}$ and the isomorphism 
$\Theta_{\bP\bP'}$ are unique as germs. 
\end{lemma} 
\begin{proof} 
By construction, the \logDTEP structures $\cF_\bP$ and $\cF_{\bP'}$ are 
equipped with natural opposite modules $\bP$ and $\bP'$ 
that are compatible 
with Deligne extensions and give rise to the \logDtrTLEP structures 
$\cE_{\bP}$ and $\cE_{\bP'}$.  
Recall from Proposition~\ref{pro:extending_trivialization} 
that a Deligne-extension-compatible 
opposite module for $\cF_{\bP'}$ near $p\in U\times \{0\}$ 
corresponds bijectively 
to an extension of $\cF_{\bP'}|_{\{p\} \times \C}$ to a free 
$\cO_{\{p\}\times \Proj^1}$-module satisfying certain conditions. 
By the isomorphism  $\theta_{\bP\bP'} \colon 
\cF_{\bP}|_{U \times \{0\} \times \C} 
\cong \cF_{\bP'}|_{U\times \{0\} \times \C}$, one can shift 
the opposite module $\bP$ for $\cF_{\bP}|_{U\times \{0\}\times \Proj^1}$ 
to an opposite module $\bP''$ for $\cF_{\bP'}|_{U\times \{0\} \times \C}$.  
For every point $p\in U\times \{0\}$, $\bP''$ gives rise to an 
extension of $\cF_{\bP'}|_{\{p\}\times \C}$ to a free 
$\cO_{\{p\}\times \Proj^1}$-module  
satisfying the conditions of Proposition~\ref{pro:extending_trivialization}, (a). 
Therefore, the opposite module $\bP''$ for 
$\cF_{\bP'}|_{U\times \{0\} \times \C}$ extends 
to a Deligne-extension-compatible 
opposite module (which we denote by $\bP''$ again) 
for $\cF_{\bP'}$ over an open neighbourhood $O$ of 
$U\times \{0\}$ in $U\times V'$. 
This gives rise to a \logDtrTLEP structure 
$\cE_{\bP''}$ over $O\times \Proj^1$ underlain by the \logDTEP structure 
$\cF_{\bP'}$ and one has an isomorphism 
of \logDtrTLEP structures: 
\[
\theta_{\bP\bP''} \colon 
\cE_{\bP}|_{(U\times \{0\})\times \Proj^1} \cong 
\cE_{\bP''}|_{(U\times \{0\})\times \Proj^1} 
\]
(The isomorphism $\theta_{\bP\bP''}$ is induced from 
$\theta_{\bP\bP'}$.) 
The universal property of Reichelt's unfolding implies 
that there exist a biholomorphic map $\varphi_{\bP\bP'} 
\colon O_{\bP\bP'} \to O_{\bP'\bP}$ between 
an open neighbourhood $O_{\bP\bP'}$ of $U\times \{0\}$ 
in $U\times V$ and an open neighbourhood $O_{\bP'\bP}$ 
of $U\times \{0\}$ in $O \subset U\times V'$ 
such that $\varphi_{\bP\bP'}$ satisfies 
the properties listed in the statement 
and that $\theta_{\bP\bP''}$ extends 
to an isomorphism of \logDtrTLEP structures: 
\[
\Theta_{\bP\bP''} \colon \cE_{\bP}|_{O_{\bP\bP'} \times \Proj^1} 
\cong (\varphi_{\bP\bP'}\times \id)^\star
(\cE_{\bP''}|_{O_{\bP'\bP} \times \Proj^1})
\]
The map $\varphi_{\bP\bP'}$ and isomorphism $\Theta_{\bP\bP''}$ 
are unique as germs. 
Restricting $\Theta_{\bP\bP''}$ to $O_{\bP\bP'}\times \C$, we 
obtain the desired isomorphism $\Theta_{\bP\bP'}$ between 
$\cF_{\bP}$ and $(\varphi_{\bP\bP'}\times \id)^\star(\cF_{\bP'})$. 

We show the uniqueness of $\varphi_{\bP\bP'}$ and $\Theta_{\bP\bP'}$. 
Suppose we have $\varphi_{\bP\bP'}$ and $\Theta_{\bP\bP'}$ satisfying 
the conditions in the statement. 
Then the isomorphism $\Theta_{\bP\bP'}\colon 
\cF_\bP \cong (\varphi_{\bP\bP'}\times \id)^\star \cF_{\bP'}$ 
of \logDTEP structures induces a \logDtrTLEP structure $\cE_{\bP''}$ 
underlain by the \logDTEP structure $\cF_{\bP'}$ which 
is isomorphic to $\cE_{\bP}$ as a \logDtrTLEP structure. 
By the uniqueness of Reichelt's universal unfolding, 
$\varphi_{\bP\bP'}$ and $\Theta_{\bP\bP'}$ should be 
the same (as germs) as what we constructed above. 
\end{proof} 

\subsubsection{The Big B-model \logDTEP Structure} 
The above Lemma \ref{lem:gluing_logTEP} says that 
the underlying \logDTEP structures of the 
miniversal \logDtrTLEP structures constructed locally in Step 2 
do not depend on the choice of opposite modules. 
Therefore, they are glued together to give a global miniversal 
\logDTEP structure over a 6-dimensional base $\cMBbig$. 
At first sight, the gluing construction looks obvious: however 
it is not so straightforward to show that the glued space is Hausdorff. 
We leave this technical (but elementary) problem 
to a separate paper \cite{Coates--Iritani:gluing} 
and adapt the result there to our setting. 

We take an open covering $\{U_i\}_{i\in I}$ of $\cMB^\circ$ such that 
for each $i\in I$ there exists an opposite module $\bP_i$ for $\cFB^\times$ 
over $U_i \setminus D$ which is compatible with the Deligne extension 
and that the \logDtrTLEP structure associated to $\bP_i$ admits Reichelt's 
universal unfolding $\cE_i$ with base $(U_i\times V_i, (U_i\cap D) \times V_i)$ 
for an open neighbourhood $V_i$ of the origin in $\C^4$. 
We write $\cF_i = \cE_i|_{(U_i\times V_i) \times \C}$ 
for the \logDTEP structure underlying $\cE_i$. 
We glue the local charts $U_i\times V_i$ first and then glue
the local \logDTEP structures $\cF_i$.  

First we construct an ambient space $\cMBbig$ containing $\cMB^\circ$. 
Write $\iota \colon U_i \cong U_i\times \{0\} \to U_i\times V_i$ 
for the inclusion map and define the sheaf  
of algebras over $U_i$ by $\cA_i := \iota^{-1}\cO_{U_i\times V_i}$. 
For $i,j\in I$, 
the sheaves $\cA_i$ and $\cA_j$ are canonically 
isomorphic along $U_i\cap U_j$ by the map $\varphi_{\bP_i\bP_j}$ in 
Lemma \ref{lem:gluing_logTEP}. 
The gluing maps $\varphi_{\bP_i\bP_j}$ satisfy the cocycle condition 
by their uniqueness. 
Therefore $\cA_i$ for all $i\in I$ are glued together to give a global sheaf 
$\cA$ of algebras over $\cMB^\circ$. 
The sheaf $\cA$ is naturally equipped with a surjection 
$\cA \to \cO_{\cMB^\circ}$. 
By \cite{Coates--Iritani:gluing}*{Theorem 1}, there exists a global 
6-dimensional complex manifold $\cMBbig$ together with a 
closed embedding $\iota \colon \cMB^\circ \to \cMBbig$ 
such that we have an isomorphism 
$\cA \cong \iota^{-1} \cO_{\cMBbig}$ which commutes 
with the natural surjections to $\cO_{\cMB^\circ}$.  
The space $\cMBbig$ is unique in the sense explained 
in \emph{loc.~cit}. 

Next we construct a \logDTEP structure on $\cMBbig$. 
Consider the inclusion $\iota\times \id \colon\cMB^\circ \times \C 
\to \cMBbig \times \C$ and 
set $\tcA = (\iota\times \id)^{-1} \cO_{\cMBbig\times \C}$. 
This is the sheaf of algebras over $\cMB^\circ \times \C$. 
Consider the pull-back $(\iota\times \id)^{-1} \cF_i$. 
This is a locally free $\tcA|_{U_i\times \C}$-module of rank 6.  
Recall that the gluing maps $\varphi_{\bP_i\bP_j}$ are 
determined so that the \logDTEP structures $\cF_i$ and 
$(\varphi_{\bP_i\bP_j}\times\id )^\star\cF_j$ are isomorphic. 
In view of the construction of $\cA$, this means that 
$(\iota\times \id)^{-1}\cF_i|_{(U_i\cap U_j) \times\C}$ 
and 
$(\iota\times \id)^{-1}\cF_j|_{(U_i\cap U_j)\times \C}$ 
are canonically isomorphic as 
$\tcA|_{(U_i\cap U_j)\times \C}$-modules for each $i,j\in I$. 
Therefore the sheaves $(\iota\times \id)^{-1}\cF_i$ are 
glued together to a locally free $\tcA$-module $\cB$ of rank 6.  
By \cite{Coates--Iritani:gluing}*{Theorem 2, Remark 4}, 
there exists a locally free sheaf $\cFBbig$ of rank 6 over an open 
neighbourhood of $\cMB^\circ\times \C$ in $\cMBbig \times \C$ 
such that $(\iota\times \id)^{-1}\cFBbig \cong \cB$. 
Similarly, we can glue the divisors $(U_i \cap D)\times V_i$ on local 
charts to construct a global divisor $\Dbig$ in $\cMBbig$
by regarding them as a coherent $\tcA$-module and 
applying \cite{Coates--Iritani:gluing}*{Theorem 2}. 
The flat connection and the pairing on the local charts 
are glued to give \emph{germs} of connections and pairings: 
\begin{align*} 
& \nablaB \colon (\iota\times \id)^{-1}\cFBbig 
\to (\iota\times \id)^{-1} \left(
\Omega^1_{\cMBbig \times \C} (\log \hZ) \otimes 
\cFBbig(\cM\times \{0\}) \right) \\ 
& \pairingB \colon (\iota \times \id)^{-1} 
\left((-)^\star \cFBbig \otimes \cFBbig\right)
\to (\iota\times \id)^{-1} \cO_{\cMBbig \times \C} 
\end{align*} 
where $\hZ = \cMBbig \times \{0\} \cup \Dbig \times \C$. 
These germs extend to an actual open neighbourhood 
of $\cMB^\circ \times \C$ and satisfy the properties  
of a miniversal \logDTEP structure. 
Because of the flat connection $\nabla$, the structure 
$(\cFBbig, \nablaB, \pairingB)$ 
extends automatically to an open set of the form 
$O \times \C$, where $O$ is an open neighbourhood 
of $\cMB^\circ$ in $\cMBbig$. 
The proof of Theorem \ref{thm:unfolding} is now complete. 

\subsection{A Mirror Theorem for Big Quantum Cohomology}

The opposite module $\bP$ in
Proposition~\ref{pro:LR_conifold_orbifold}(1) coincides under mirror
symmetry (Theorem~\ref{thm:mirrorsymmetryforYbar}) with the canonical
opposite module for Gromov--Witten theory defined in
Example~\ref{ex:opposite_A}: this is
\cite{CIT:wall-crossings}*{Theorem~3.5}.  Thus in a neighbourhood of
the large-radius limit point $(y_1,y_2) = (0,0)$, the A-model
\logDTEP structure (Example~\ref{ex:log_TEP_A}) is isomorphic to the
big B-model \logDTEP structure.  Since the universal unfolding of a
\logDTEP structure is unique as an analytic germ, this proves:

\begin{theorem}
  \label{thm:big_mirror_symmetry_Ybar}
  Let $(\cMA{\Ybar},\DA{\Ybar})$ 
denote the base of the A-model \logDTEP
  structure for $\Ybar$, as described in Example~\ref{ex:log_TEP_A}. 
  Let $\Dbig$ be the divisor in $\cMBbig$ as above. 
Consider:
  \begin{itemize}
  \item  the A-model \logDTEP structure 
$\big(\cFA{\Ybar},\nablaA{\Ybar}, \pairingA{\Ybar}\big)$ 
for $\Ybar$; this is a \logDTEP structure with base 
$(\cMA{\Ybar}, \DA{\Ybar})$.
\item the big B-model \logDTEP
  structure $\big(\cFBbig,\nablaB,\pairingB\big)$; 
this is a \logDTEP structure with base $(\cMBbig,\Dbig)$.
\end{itemize}
There exist:
\begin{itemize}
\item  an open neighbourhood $U^{\rm big}$ 
of the large-radius limit point in $\cMBbig$;
\item an open embedding of pairs $\Mir \colon 
(U^{\rm big},\Dbig \cap U^{\rm big}) 
\to (\cMA{\Ybar}, \DA{\Ybar})$; and
\item  an isomorphism of \logDTEP structures
  \begin{equation}
    \label{eq:big_log_TEP_isomorphism}
    \big(\cFBbig,\nablaB,\pairingB\big)\Big|_{U^{\rm big}\times \C} 
\cong 
    \Mir^\star
    \big(\cFA{\Ybar},\nablaA{\Ybar}, \pairingA{\Ybar} \big).
  \end{equation}
\end{itemize}
The map $\Mir$ is called the mirror map; it sends the large-radius
limit point in $\cMBbig$ to the origin in $\cMA{\Ybar}$, and
coincides with the map $\mir_{\Ybar}$ 
in Theorem \ref{thm:mirrorsymmetryforYbar} 
when restricted to the small parameter space $U^\times \subset 
 U^{\rm big}$ there.  
The isomorphism \eqref{eq:big_log_TEP_isomorphism}
intertwines the opposite module $\bP_\LR$ 
for $\cFBbig$ defined in 
Proposition~\ref{pro:LR_conifold_orbifold}(1) with 
the canonical opposite module $\bP_{\rm A}$ 
  from Example~\ref{ex:opposite_A}.
\end{theorem}

\begin{remark}
  An analogous mirror symmetry statement for \logDtrTLEP structures
  was proved by Reichelt--Sevenheck \cite{Reichelt--Sevenheck}.
\end{remark}

\begin{remark}
\label{rem:big_mirror_symmetry_cXbar} 
  A similar statement holds for the big quantum cohomology of the
  orbifold $\cXbar$: cf.~\cite{CIT:wall-crossings}*{proof of
    Theorem~3.12}.
\end{remark}

\section{Quantization Formalism and Fock Sheaf}
\label{sec:Fock_sheaf}

As discussed in the Introduction, Givental's quantization formalism
has been an essential ingredient in much recent work in
Gromov--Witten theory. Givental's formulation of his quantization
rules depends on a choice of flat co-ordinate system and so, in the
context of mirror symmetry, is applicable only over certain patches of
the moduli space $\cMBbig$. In previous work, we constructed a global and
co-ordinate-free version of Givental's quantization, associating to a
miniversal cTEP structure a \emph{Fock sheaf} on (an open subset of)
the total space of that cTEP structure~\cite{Coates--Iritani:Fock}.
Furthermore we showed that whenever the cTEP structure is semisimple,
such as is the case for the A-model cTEP structure associated to a
target space $X$ with semisimple quantum cohomology, there is a
canonically defined global section of this Fock sheaf. (In the A-model
case, this global section coincides with the total descendant
potential $\cZ_X$.) In this section, we review the construction of the
Fock sheaf.

\subsection{cTEP Structures and $\log$-cTEP Structures}

We will need the notions of cTEP structure and $\log$-cTEP structure.
One can think of these as being obtained from the
notions of TEP structure (Definition~\ref{def:TEP}) and \logDTEP
structure (Definition~\ref{def:logDTEP}) by taking the formal
completion along the divisor $z=0$.

\begin{definition}
  Let $\hA^1 = \Spf \C[\![z]\!]$ denote the formal neighbourhood of
  zero in $\C$.  Recall that a sheaf of modules over $\cM \times
  \hA^1$ is the same thing as a sheaf of $\cO_{\cM}[\![z]\!]$-modules.
  Let $(-)\colon \cM \times \hA^1 \to \cM \times \hA^1$ denote the map
  sending $(t, z)$ to $(t,-z)$; this is consistent with our previous
  definition of $(-)$, in Definition~\ref{def:TEP}.  For an
  $\cO_{\cM}[\![z]\!]$-module $\sfF$, we give the pull-back $(-)^\star
  \sfF$ the structure of an $\cO_\cM[\![z]\!]$-module by setting:
  \begin{align*}
    f(z) (-)^\star \alpha = (-)^\star f(-z) \alpha && \text{for all $f(z) \in
      \cO_{\cM}[\![z]\!]$ and $\alpha \in \sfF$.}
  \end{align*}
  Write $\sfF[z^{-1}]$ for the locally free
  $\cO_{\cM}(\!(z)\!)$-module $\sfF \otimes_{\cO_{\cM}[\![z]\!]}
  \cO_{\cM}(\!(z)\!)$, and $\sfF_0$ for the quotient $\sfF/z\sfF$.
\end{definition}

\begin{notation}
We will use sans serif font ($\sfF$, $\mathsf{G}$, etc.) to denote 
sheaves and similar structures over $\hA^1$ or $\cM \times \hA^1$.
\end{notation}

\begin{definition}[cf.~Definition~\ref{def:TEP}]
\label{def:cTEP}
Let $\cM$ be a complex manifold.  A \emph{cTEP structure} 
$(\sfF, \nabla,(\cdot,\cdot)_{\sfF})$ with base $\cM$ 
consists of a locally free
$\cO_\cM[\![z]\!]$-module $\sfF$ of rank $N+1$, 
a meromorphic flat connection:
\[
\nabla \colon \sfF \to 
(\Omega_{\cM}^1 \oplus \cO_{\cM} z^{-1}dz) 
\otimes_{\cO_{\cM}} z^{-1} \sfF
\]
and a non-degenerate pairing:
\[
(\cdot,\cdot)_{\sfF} \colon 
(-)^\star \sfF \otimes_{\cO_{\cM}[\![z]\!]} \sfF \to \cO_{\cM}[\![z]\!] 
\]
which satisfies:
\begin{align*} 
\begin{split} 
((-)^\star s_1,s_2)_{\sfF} 
& = (-)^\star ((-)^\star s_2, s_1)_{\sfF}   \\ 
d ((-)^\star s_1, s_2)_{\sfF}  
& = ((-)^\star \nabla s_1, s_2)_{\sfF} 
+ ((-)^\star s_1,\nabla s_2)_{\sfF}  
\end{split} 
\end{align*}
for $s_1,s_2 \in \sfF$.  Here we regard $z^{-1} \sfF$ as a subsheaf of
$\sfF[z^{-1}]$; non-degeneracy of the pairing $(\cdot,\cdot)_{\sfF}$
means that the induced pairing on $\sfF_0 = \sfF/z\sfF$
\[
(\cdot,\cdot)_{\sfF_0} \colon \sfF_0\otimes_{\cO_{\cM}} 
\sfF_0 \to \cO_{\cM} 
\]
is non-degenerate. 
\end{definition} 

\begin{definition}[cf.~Definition~\ref{def:logDTEP}]
\label{def:logDcTEP}
Let $D$ be a divisor with normal crossings in a complex manifold
$\cM$.  A \emph{$\log$-cTEP structure} 
$(\sfF, \nabla, (\cdot,\cdot)_{\sfF})$
with base $(\cM,D)$ consists of a locally free $\cO_\cM[\![z]\!]$-module
$\sfF$ of rank $N+1$, a meromorphic flat connection:
\[
\nabla \colon \sfF \to 
\big(\Omega_{\cM}^1(\log D) \oplus \cO_{\cM} z^{-1}dz \big) 
\otimes_{\cO_{\cM}} z^{-1} \sfF
\]
and a pairing:
\[
(\cdot,\cdot)_{\sfF} \colon 
(-)^\star \sfF \otimes_{\cO_{\cM}[\![z]\!]} \sfF \to \cO_{\cM}[\![z]\!] 
\]
which satisfies the properties listed in Definition~\ref{def:cTEP} and
is non-degenerate in the same sense.
\end{definition} 

\begin{example}
  \label{ex:log_cTEP_A}
  Our first key example is the \emph{A-model $\log$-cTEP structure},
  which is the formalization at $z=0$ of the A-model \logDTEP
  structure (Example~\ref{ex:log_TEP_A}). 
Write $\cFA{X}$ for the sheaf underlying the A-model \logDTEP structure and 
$(\cMA{X},\DA{X})$ for its base.  
The A-model $\log$-cTEP structure 
$\big(\FA{X},\nablaA{X}, \pairingA{X})\big)$ has
base $(\cMA{X},\DA{X})$ and:
  \[
  \FA{X} := \cFA{X} 
\otimes_{\cO_{\cMA{X} \times \C}} 
\cO_{\cMA{X}}[\![z]\!]
  \]
  with meromorphic flat connection $\nabla$ and pairing
  $(\cdot,\cdot)$ induced by the meromorphic flat connection and
  pairing on the A-model \logDTEP structure.  We refer to the
  restriction to $\cMA{X} \setminus \DA{X}$ 
of the A-model $\log$-cTEP structure
  as the \emph{A-model cTEP structure}.
\end{example}

\begin{example}
  \label{ex:log_cTEP_B}
  Our second key example is the \emph{big B-model $\log$-cTEP
    structure}, which is the formalization at $z=0$ of the big B-model
  \logDTEP structure from Theorem \ref{thm:unfolding}.  
This is a $\log$-cTEP structure with base
  $(\cMBbig,\Dbig)$:
  \[
  \FBbig
  := \cFBbig \otimes_{\cO_{\cMBbig \times \C}} \cO_{\cMBbig}[\![z]\!]
  \]
  with meromorphic flat connection $\nabla$ and pairing
  $(\cdot,\cdot)$ induced by the meromorphic flat connection and
  pairing on the big B-model \logDTEP structure.
\end{example}

\begin{remark}
  A cTEP structure with base $\cM$ is the same thing as a $\log$-cTEP
  structure with base $(\cM,D)$ where $D = \varnothing$.  Thus the
  definitions of symplectic pairing, miniversality, etc.~for
  $\log$-cTEP structures given below also define the corresponding notions for
  cTEP structures.
\end{remark}

\begin{definition}
  Let $(\sfF,\nabla, (\cdot,\cdot)_\sfF)$ be a $\log$-cTEP structure.
  The pairing $(\cdot,\cdot)_{\sfF}$ induces a symplectic pairing:
  \[
  \Omega \colon \sfF[z^{-1}] \otimes_{\cO_{\cM}} 
  \sfF[z^{-1}] \to \cO_\cM, 
  \]
  defined by:
  \begin{equation}
    \label{eq:symplecticform} 
    \Omega(s_1, s_2) =  
    \Res_{z=0} ((-)^\star s_1,s_2)_{\sfF} \, dz
  \end{equation}
\end{definition} 

\begin{definition}
  Let $n \in \Z$ and let $(\sfF,\nabla, (\cdot,\cdot)_\sfF)$ be a
  $\log$-cTEP structure.  We set:
  \begin{align}
    \label{eq:dualmodules} 
    \begin{split}  
      (z^n \sfF)^\vee &:= 
      \injlim_{l} \sHom_{\cO_{\cM}}
(z^n \sfF/z^l \sfF,\cO_{\cM}), 
      \\
      \sfF[z^{-1}]^\vee &:= \projlim_n 
      \injlim_l \sHom_{\cO_{\cM}}
(z^{-n}\sfF/z^l \sfF,\cO_{\cM}).  
    \end{split}  
  \end{align}  
\end{definition}

There are natural surjections:
\[
\sfF[z^{-1}]^\vee \surj \cdots \surj (z^{-2} \sfF)^\vee 
\surj (z^{-1} \sfF)^\vee \surj \sfF^\vee \surj (z\sfF)^\vee 
\surj \cdots. 
\] 
The dual $(z^n\sfF)^\vee$ has the structure of an
$\cO_{\cM}[\![z]\!]$-module such that the action of $z$ is nilpotent;
it is locally isomorphic to $(\cO_{\cM}(\!(z)\!)/\cO_{\cM}[\![z]\!])^{N+1}$
as an $\cO_{\cM}(\!(z)\!)$-module.  Also
$\sfF[z^{-1}]^\vee$ is locally free as an
$\cO_{\cM}(\!(z)\!)$-module.  The symplectic pairing gives an isomorphism 
\[
\sfF[z^{-1}] \cong \sfF[z^{-1}]^\vee, \quad 
s \mapsto \iota_s \Omega = \Omega(s, \cdot)
\]
and thus a dual symplectic pairing $\Omega^\vee$ on
$\sfF[z^{-1}]^\vee$:
\begin{equation*}
\Omega^\vee \colon \sfF[z^{-1}]^\vee \otimes_{\cO_{\cM}} 
\sfF[z^{-1}]^\vee \to \cO_{\cM}
\end{equation*} 
The dual flat connection $\nabla^\vee$ is defined by:
\[
\nabla^\vee \colon 
(z^{-1}\sfF)^\vee \to \Omega^1_{\cM}(\log D) \otimes_{\cO_{\cM}} \sfF^\vee, 
\quad 
\pair{\nabla^\vee \varphi}{s} := d\pair{\varphi}{s} - 
\pair{\varphi}{\nabla s}
\]

\subsection{The Total Space of a $\log$-cTEP Structure}

We now consider the total space $\LL$ of a $\log$-cTEP structure. This
is an algebraic analogue of Givental's Lagrangian
submanifold~\cite{Givental:symplectic}.

\begin{definition} 
 Let $(\sfF,\nabla, (\cdot,\cdot)_{\sfF})$ be a $\log$-cTEP
  structure.  The \emph{total space} $\LL$ of $(\sfF,\nabla,
  (\cdot,\cdot)_{\sfF})$ is the total space of the infinite
  dimensional vector bundle associated to $z\sfF$.
\end{definition}

As a set, $\LL = \{ (t,\bx) : t\in \cM, \bx\in z \sfF_t\}$. Let
$\pr \colon \LL \to \cM$ denote the natural projection. We regard
$\LL$ as a ``fiberwise algebraic variety'' over $\cM$, endowing it
with the structure of a ringed space exactly as
in~\cite[Definition~4.7]{Coates--Iritani:Fock}. Let $\bcO$ denote the
structure sheaf of $\LL$. For a connected open set $U\subset \cM$ such
that $\sfF|_U$ is a free $\cO_U[\![z]\!]$-module, the ring of regular
functions on $\pr^{-1}(U)$ is the polynomial ring over $\cO(U)$:
\begin{equation}
\label{eq:strsheaf_prinvU}
\bcO(\pr^{-1}(U)) := \Sym_{\cO(U)}^\bullet 
\Gamma(U,(z\sfF)^\vee). 
\end{equation} 
To make this concrete, take a trivialization $\sfF|_U \cong \C^{N+1}
\otimes \cO_U[\![z]\!]$.  Consider the induced trivialization
$\sfF[z^{-1}]|_U \cong \C^{N+1} \otimes \cO_U(\!(z)\!)$, and the dual
frame $x_n^i \in \sfF[z^{-1}]^\vee$, $n\in \Z$, $0\le i\le N$, defined
by:
\begin{align} 
\label{eq:frame_x}
\begin{split}  
  x_n^i \colon \sfF[z^{-1}]\Bigr|_U \cong \C^{N+1} \otimes \cO_U(\!(z)\!)
  & \longrightarrow \cO_U \\
  \sum_{m\in \Z} \sum_{j=0}^N a_m^j e_j z^m
 & \longmapsto a_n^i
\end{split}
\end{align}  
where $e_i$, $0\le i\le N$, denotes the standard basis of $\C^{N+1}$.
Restricting $x_n^i$ to $z \sfF$, we obtain co-ordinates $x_n^i$, $n\ge
1$, $0\le i\le N$, on the fibers of $\LL|_U$.  

\begin{definition} \label{def:algebraic_local_co-ordinates}
  Let $(\sfF,\nabla, (\cdot,\cdot)_{\sfF})$ be a $\log$-cTEP structure
  of rank $N+1$ with base $(\cM,D)$.  
Let $t^0,q_1,\dots,q_r,t^{r+1},\dots, t^R$ be local
  co-ordinates on an open set $U \subset \cM$ such that $D \cap U$ is
  given by $(q_1 q_2 \cdots q_r = 0)$.  We call the co-ordinate system
  \[
\{(t^j, q_k, x_n^i) : \text{$j\in \{0,r+1,\dots,R\}$, 
$1\le k\le r$, $0 \leq i \leq N$, $n\ge
    1$}\}
\] 
an \emph{algebraic local co-ordinate system} on $\LL$. 
We also set $q_k = e^{t^k}$ for $k=1,\dots,r$ so that 
$(t^0,t^1,\dots,t^r, t^{r+1},\dots,t^R)$ gives a 
multi-valued co-ordinate system on $U\setminus D$. 
We write $t$ for a point on $\cM$; this is a slight abuse of notation. 
\end{definition}

We have:
\[
\bcO(\pr^{-1}(U)) = 
\cO(U)\left[x_n^i \, |\, n\ge 1, 0\le i\le N\right]. 
\]
We equip $\bcO(\pr^{-1}(U))$ 
with a grading and filtration as follows.  The grading
on $\bcO(\pr^{-1}(U))$ is given by the degree as polynomials in the
variables $x_n^i$.  The $l$th part of the filtration, $l \geq 0$, is
given by:
\begin{equation} 
\label{eq:filtration} 
\bcO_l(\pr^{-1}(U)) = \Biggl \{\sum_{n\ge 0}  
\sum_{\substack{l_1,\dots,l_n\ge 0 \\ 
l_1 + \cdots + l_n \le l}} 
\sum_{i_1,\dots,i_n \ge 0} 
f^{l_1,\dots,l_n}_{i_1,\dots,i_n}(t) 
x_{l_1+1}^{i_1}\cdots x_{l_n+1}^{i_n}\;\Big | 
\;  
f^{l_1,\dots,l_n}_{i_1,\dots,i_n}(t)
\in \cO(U)\Biggr \}. 
\end{equation} 
This is an increasing filtration $\bcO_l(\pr^{-1}(U)) \subset 
\bcO_{l+1}(\pr^{-1}(U))$. 

\subsection{Miniversality of $\log$-cTEP structures}

Suppose now that $(\sfF,\nabla, (\cdot,\cdot)_{\sfF})$ is a
$\log$-cTEP structure with base $(\cM,D)$.  Writing the connection
$\nabla$ in terms of our trivialization $\sfF|_U \cong \C^{N+1}
\otimes \cO_U[\![z]\!]$ gives:
\begin{equation}
  \label{eq:definition_of_C}
  \nabla s = d s - \frac{1}{z} \cC(t,z) s 
\end{equation}
where $s \in \C^{N+1} \otimes \cO_U[\![z]\!]  \cong \sfF|_U$ and
$\cC(t,z) \in \End(\C^{N+1})\otimes \Omega_U^1(\log D)[\![z]\!]$.  The
residual part $\cC(t,0) = (-z\nabla)|_{z=0}$ determines a section of
$\End(\sfF_0) \otimes \Omega_U^1(\log D)$ which is independent of
choice of trivialization.

\begin{example}
  In the case of the A-model $\log$-cTEP structure 
(Example~\ref{ex:log_cTEP_A}), 
we have $\cC(t,0) = (\phi_0*) 
  dt^0 + \sum_{i=1}^r (\phi_i*) \frac{dq_i}{q_i} + \sum_{i=r+1}^N (\phi_i*)
  dt^i$. 
\end{example}

\begin{definition}
  \label{def:miniversal_cTEP} 
  Let $(\sfF,\nabla, (\cdot,\cdot)_{\sfF})$ be a $\log$-cTEP
  structure.  Let:
  \begin{align*}
    \sfF_{0,t}^\circ & := 
    \{ x_1 \in \sfF_{0,t} \, |\, 
    \Theta_\cM(\log D)_t \to \sfF_{0,t}, \ v\mapsto \iota_v \cC(t,0) x_1 
    \text{ is an isomorphism}\}  \\  
    \LLo & := \{ (t,\bx) \in \LL \,|\, t\in \cM, \ 
    \bx \in z \sfF_t, \ 
    (\bx/z)|_{z=0} \in \sfF_{0,t}^\circ \} \\
    \sfF_0^\circ & := \bigcup_{t\in \cM} \sfF_{0,t}^\circ
  \end{align*}
  These are open subsets of, respectively, $\sfF_{0,t}$, $\LL$, and
  $\sfF_0$.  If for every point $t\in \cM$, $\sfF_{0,t}^\circ$ is a
  non-empty Zariski open subset of $\sfF_{0,t}$, then we say that
  $(\sfF,\nabla, (\cdot,\cdot)_{\sfF})$ is \emph{miniversal}.
\end{definition}

\begin{remark}
  A miniversal $\log$-cTEP structure $(\sfF,\nabla,
  (\cdot,\cdot)_{\sfF})$ with base $(\cM,D)$ satisfies $\dim \cM =
  \rank \sfF$.
\end{remark}

We henceforth assume that our $\log$-cTEP structure is miniversal.  Let
$\{t^j,q_k, x_n^i\}$ be an algebraic local co-ordinate system 
on $\LL$, and write $\cC(t,z) = \sum_{i=0}^N \cC_i(t,z) dt^i$. 
Here recall that $dt^i = \frac{dq_i}{q_i}$ for $1\le i\le r$. 
Consider:
\begin{equation} 
\label{eq:discriminant}
P(t,x_1) := (-1)^{N+1}\det(\cC_0(t,0) x_1, \cC_1(t,0) x_1, \dots, \cC_N(t,0) x_1) 
\end{equation} 
This is a polynomial of degree $N+1$, $P(t,x_1) \in
\cO(U)[x_1^0,\dots,x_1^N]$, called the \emph{discriminant}. The set
$\LLo$ is the complement\footnote{More invariantly, we can think of $P(t,x_1) dt^0\wedge \dots \wedge
  dt^N$ as a section of the line bundle
$\pr^\star(\det(\sfF_0) \otimes \Omega^{N+1}_\cM(\log D))$ over $\LL$, and
of $\LLo$ as the complement to the zero locus of that
section.} of the zero-locus of $P(t,x_1)$.  
  The ring of regular functions over
$\pr^{-1}(U)^\circ := \pr^{-1}(U) \cap \LLo$ is:
\[
\bcO(\pr^{-1}(U)^\circ)  
= \cO(U)[\{x_n^i\}_{n\ge 1,0\le i\le N}, 
P(t,x_1)^{-1}]
\]
Since $P(t,x_1)$ is homogeneous in $x_1$ and lies in the zeroth
filter, the grading and filtration on $\bcO(\pr^{-1}(U))$ 
extend canonically to $\bcO(\pr^{-1}(U)^\circ)$.

\subsection{One-Forms and Vector Fields on $\LL$}

The sheaf $\bOmega^1(\log D)$ of logarithmic one-forms on $\LL$ is
defined in algebraic local co-ordinates $\{t^j,q_k,x_n^i\}$ as:
\[
\bOmega^1(\log D) = 
 \bigoplus_{j=0}^N  \bcO dt^j \oplus
\bigoplus_{n =1}^\infty 
\bigoplus_{i=0}^N  \bcO d x_n^i   
\]
where recall again that $dt^i = \frac{dq_i}{q_i}$ for $1\le i\le r$. 
The grading on $\bOmega^1(\log D)$ is determined by\footnote
{More precisely, the grading and the filtration are defined on the 
module $\bOmega^1(\log D)(\pr^{-1}(U))$ or on 
$\bOmega^1(\log D)(\pr^{-1}(U)^\circ)$ for an 
open set $U\subset \cM$. 
We will omit the domains $\pr^{-1}(U)$, $\pr^{-1}(U)^\circ$, to ease the notation.}:
\begin{align*}
\deg(dt^j) = 0, &&
\deg(d x_n^i) = 1
\end{align*}
The filtration on $\bOmega^1(\log D)$ is determined by putting $dt^i$ in the
$(-1)$st filter but not the zeroth filter, and putting $d x_n^i$ in
the $(n-1)$st filter but not the $n$th filter.  
We write $\bcO^d_e$ and $\bOmega^1(\log D)^d_e$ 
for the $e$th filter of the $d$th graded piece of 
$\bcO$ and $\bOmega^1(\log D)$ respectively, so that
\[
\bOmega^1(\log D)^d_e = \bigoplus_{j=0}^N 
\bcO^d_{e+1} dt^j \oplus 
\bigoplus_{e_1+e_2 \le e} \bigoplus_{i=0}^N 
\bcO^{d-1}_{e_1} dx_{e_2+1}^i 
\]
We also set: 
\[
\big((\bOmega^1(\log D))^{\otimes k}\big)^d_e 
= \sum_{e_1+\cdots+e_k \le e} 
\sum_{d_1+ \cdots + d_k = d} 
\bOmega^1(\log D)^{d_1}_{e_1} \otimes 
\cdots \otimes \bOmega^1(\log D)^{d_k}_{e_k}   
\] 
The sheaf $\bTheta(\log D)$ of logarithmic vector fields on $\LL$ is
defined by 
\[
\bTheta(\log D) = \Hom_{\bcO}(\bOmega^1(\log D), \bcO).
\]
In algebraic local co-ordinates $\{t^j,q_k=e^{t^k}, x_n^i\}$, 
with $\partial_j :=
\frac{\partial}{\partial t^j}$ and $\partial_{n,i} :=
\frac{\partial}{\partial x_n^i}$, we have:
\[
\bTheta(\log D) = 
\prod_{j=0}^N \bcO \partial_j \times 
\prod_{n=1}^\infty
\prod_{i=0}^N \bcO \partial_{n,i}
\]
Note that $\bOmega^1(\log D)$ is the direct sum
whereas $\bTheta(\log D)$ is the direct product.

\subsection{The Yukawa Coupling and the Kodaira--Spencer Map}

As above, let $\{t^j,q_k,x_n^i\}$ be an algebraic local co-ordinate system
on $\LL$, and write $\cC(t,z) = \sum_{i=0}^N \cC_i(t,z) dt^i$.
From flatness of $\nabla$ and flatness of the pairing we have that:
\begin{align*}
  \big[ \cC_i(t,0), \cC_j(t,0) \big] & = 0 \\
  \big( \cC_i(t,0)s_1,s_2 \big)_{\sfF_0} & = \big( s_1, \cC_i(t,0)s_2 \big)_{\sfF_0}
\end{align*}
for all $i$,~$j$ and all $s_1$,~$s_2 \in \sfF_0$.  Thus the
endomorphisms $\cC_i(t,0)$ equip the fibers of $\sfF_0$ with a
structure similar to that of a Frobenius algebra (we need to choose an identity
element here; cf.~\cite[\S4.4]{Coates--Iritani:Fock}).

\begin{definition}
\label{def:Yukawa} 
  The \emph{Yukawa coupling} is a cubic tensor:
  \[
  \bY = \sum_{i,j,k} C^{(0)}_{ijk} dt^i \otimes dt^j \otimes dt^k \in
  \big((\bOmega^1(\log D))^{\otimes 3}\big)^2_{-3}
  \]
  defined in algebraic local co-ordinates $\{t^j,q_k,x_n^i\}$ by:
  \begin{align*}
    C_{ijk}^{(0)}(t,\bx) & = \big( \cC_i(t,0) x_1,
    \cC_j(t,0) \cC_k(t,0) x_1 \big)_{\sfF_0}
    & \text{where $x_1 = (\bx/z)|_{z=0}$}
  \end{align*}
  Recall again that $q_k = e^{t^k}$ for $1\le k\le r$. 
\end{definition}

\begin{remark}
  The Yukawa coupling is a symmetric cubic tensor on $\LL$ that is
  pulled back from $\sfF_0$.
\end{remark}

Let $\pr\colon \LL\to \cM$ denote the natural projection.  We define:
\begin{align*}
  \begin{split}  
    & \pr^\star (z^n\sfF)  := 
    \projlim_{l} \pr^\star (z^n\sfF/z^l\sfF) \cong 
    (\pr^{-1} z^n\sfF)  
    \otimes_{\pr^{-1}\cO_{\cM}[\![z]\!]}\bcO[\![z]\!] \\
    & \pr^\star \sfF[z^{-1}] 
    := \projlim_{l} \pr^\star(\sfF[z^{-1}]/z^l \sfF) \cong 
    (\pr^{-1} \sfF[z^{-1}])\otimes_{\pr^{-1}\cO_\cM(\!(z)\!)} \bcO(\!(z)\!) \\ 
    & \pr^\star (z^n \sfF)^\vee := (\pr^{-1} (z^n\sfF)^\vee)  
    \otimes_{\pr^{-1} \cO_{\cM}} \bcO \\
    & \pr^\star \sfF[z^{-1}]^\vee 
    := \projlim_l \pr^\star (z^{-l}\sfF)^\vee 
    \cong (\pr^{-1} \sfF[z^{-1}]^\vee)
    \otimes_{\pr^{-1}\cO_{\cM}(\!(z)\!)} \bcO(\!(z)\!).  
  \end{split}
\end{align*} 
These are locally free modules over $\bcO[\![z]\!]$, $\bcO(\!(z)\!)$,
$\bcO$, $\bcO(\!(z)\!)$ respectively.  Note
that, with the exception of $\pr^\star (z^n \sfF)^\vee$, these differ from
the standard notion of pullback.  For example, $\pr^\star(z^n\sfF)$ is the
completion of the standard pull-back $\pr^{-1}(z^n \sfF)
\otimes_{\pr^{-1}\cO_{\cM}} \bcO$ of $z^n\sfF$ with respect to the
$z$-adic topology.

The pull-back $\pr^\star\sfF$ admits a flat connection 
$\tnabla:=\pr^\star \nabla$:   
\begin{equation} 
\label{eq:tnabla}
\tnabla \colon \pr^\star \sfF \to \bOmega^1(\log D)
\hotimes \pr^\star (z^{-1}\sfF).  
\end{equation} 
where $\hotimes$ is the completed tensor product:
\[
\bOmega^1(\log D) \hotimes \pr^\star (z^{-1}\sfF) 
:=\projlim_n (\bOmega^1(\log D) \otimes 
\pr^\star (z^{-1}\sfF/z^n \sfF))
\]
Let $\{t^j,q_k,x_n^i\}$ be an algebraic local co-ordinate system on $\LL$,
where $(t^0,q_1,\dots,q_r,t^{r+1},\ldots,t^N)$ 
are local co-ordinates on an open subset $U$
of $\cM$, and consider a local trivialization $\sfF|_U \cong \C^{N+1}
\otimes \cO_U[\![z]\!]$.  The trivialization of $\sfF|_U$ allows us to
write:
\[
\nabla s = ds - \frac{1}{z} \cC(t,z) s
\]
where $\cC(t,z) = \sum_{i=0}^N \cC_i(t,z) dt^i$ 
(recall that $dt^i = \frac{dq_i}{q_i}$ for $1\le i\le r$). 
The trivialization
of $\sfF|_U$ also induces a trivialization $\pr^\star \sfF|_{\pr^{-1}(U)}
\cong \C^{N+1}\otimes \bcO[\![z]\!]$, and with respect to this
trivialization we have:
\begin{align*}
  & \tnabla_{\partial_j} =  
  \partial_i - \frac{1}{z} \cC_i(t,z) && 0\le i\le N \\
  & \tnabla_{\partial_{n,i}} = \partial_{n,i} 
  && \text{$0 \leq i \leq N$, $1 \leq n < \infty$}
\end{align*}

\begin{definition}
  The \emph{tautological section} $\bx$ of $\pr^\star (z\sfF)$ is defined
  by
  \[
  \bx(t,\bx) = \bx   
  \] 
  where $(t,\bx)$ denotes the point $\bx \in z \sfF_t$ 
  on $\LL$. 
\end{definition}

\begin{definition} 
\label{def:KS} 
The \emph{Kodaira--Spencer map} 
$\KS \colon \bTheta(\log D) \to \pr^\star\sfF$ 
is defined by:
\begin{align*} 
\KS(v) = \tnabla_v \bx
\end{align*} 
The \emph{dual Kodaira--Spencer map} 
$\KS^\star \colon \pr^\star\sfF^\vee 
\to \bOmega^1(\log D)$ is defined by:
\[
\KS^\star(\varphi) = \varphi(\tnabla \bx), 
\quad \varphi \in \pr^*\sfF^\vee. 
\]
\end{definition}

\begin{remark}
  The maps $\KS$ and $\KS^\star$ are isomorphisms over $\LLo\subset \LL$.
\end{remark}

\begin{definition}
  Let $\bThetao(\log D)$ denote the restriction of $\bTheta(\log D)$
  to $\LLo \subset \LL$, and let $\bOmegao^1(\log D)$ denote the
  restriction of $\bOmega^1(\log D)$ to $\LLo \subset \LL$.
\end{definition}

\begin{remark} 
Using the connection $\tnabla$ on $\pr^\star \sfF$ and 
the tautological section $\bx \in \pr^\star \sfF$, we can 
write the Yukawa coupling as follows:  
\[
\bY(X,Y,Z) = \Omega(\tnabla_X \tnabla_Y \bx, \tnabla_Z \bx). 
\]
\end{remark} 

\subsection{The Euler Vector Field and Grading Operators}

\begin{definition}
  An \emph{Euler vector field} for a $\log$-cTEP structure
  $(\sfF,\nabla, (\cdot,\cdot)_{\sfF})$ with base $(\cM,D)$ is a
  logarithmic vector field $E$ on $\cM$ such that
  $\nabla_{z \partial_z + E}$ is regular at $z=0$.
\end{definition}

\begin{remark}
  A miniversal $\log$-cTEP structure always admits an Euler vector
  field, and this Euler vector field is unique.
\end{remark}

\begin{definition}
  Suppose that $(\sfF,\nabla, (\cdot,\cdot)_{\sfF})$ is a $\log$-cTEP
  structure with Euler vector field $E$.  Define the grading operator
$\gr \in \End_\C(\sfF[z^{-1}])$ by 
  \begin{align*}
    \gr:=\nabla_{z \partial_z + E}
  \end{align*}
  The grading operator $\gr$ preserves $\sfF\subset \sfF[z^{-1}]$. 
  For $\varphi \in \sfF[z^{-1}]^\vee$, define $\gr^\vee(\varphi)$ by
  $\gr^\vee(\varphi)(x) = E(\varphi(x)) - \varphi(\gr(x))$.
\end{definition}

\begin{lemma}
  $\gr^\vee$ is a well-defined element of $\End_{\C}(\sfF[z^{-1}]^\vee)$.
\end{lemma}

\begin{proof}
  Let $(\cM,D)$ be the base of the $\log$-cTEP structure and suppose
  that $\varphi \in \sfF[z^{-1}]^\vee$. We need to show that
  $\gr^\vee(\varphi) \in \sfF[z^{-1}]^\vee$, i.e.~that
  $\gr^\vee(\varphi)$ is $\cO_\cM$-linear.  Let $f \in \cO_\cM$ and $x
  \in \sfF[z^{-1}]$.  Then:
  \begin{align*}
    \gr^\vee(\varphi)(f x) & = E(f \varphi(x)) - \varphi(\gr(f x)) \\
    & = E(f) \varphi(x) + f E(\varphi(x)) - E(f) \varphi(\gr(x)) - f
    \varphi(\gr(x)) \\
    & = f \gr^\vee(\varphi)(x)
  \end{align*}
  as required.
\end{proof}

\begin{example}
\label{ex:Gr_gr} 
  Consider the big B-model $\log$-cTEP structure $(\FBbig,\nabla,
  (\cdot,\cdot)_{\sfF})$ (Example \ref{ex:log_cTEP_B}).  Then the
  grading operator $\gr \in \End(\FBbig)$, when restricted to the
  small parameter space $\cMB^\times$, coincides with $\Gr - \frac{3}{2}$
  where $\Gr$ is the grading operator on the GKZ system (Definition~\ref{def:GKZ_grading}).  The shift by $\frac{3}{2}$ here reflects the shift by
  $\frac{3}{2}$ in Definition~\ref{def:BmodelconformalTEP}, which was
  made to ensure that the B-model TEP structure had weight zero.
\end{example}

\begin{definition}
  Let $\sharp \colon \sfF[z^{-1}] \to \sfF[z^{-1}]^\vee$ be the map
  $\alpha \mapsto \Omega(\alpha,{-})$, and let $\flat \colon
  \sfF[z^{-1}]^\vee \to \sfF[z^{-1}]$ be the inverse map.  Write
  $\alpha^\sharp$ for $\sharp(\alpha)$, and $\varphi^\flat$ for
  $\flat(\varphi)$.
\end{definition}
\begin{lemma} 
  \label{lem:grading_things}
  We have:
  \begin{itemize}
  \item[(a)] $\gr^\vee\big(\alpha^\sharp\big) = \big((\gr + 1)\alpha\big)^\sharp$;
  \item[(b)] $\big(\gr^\vee(\varphi)\big)^\flat = (\gr + 1)\big(\varphi^\flat\big)$;
  \item[(c)] $\big(\gr^\vee \otimes 1 + 1 \otimes \gr^\vee\big) \Omega
    = \Omega$;
  \item[(d)] $\big(\gr \otimes 1 + 1 \otimes \gr\big) \Omega^\vee
    = {-\Omega^\vee}$.
  \end{itemize}
\end{lemma}
\begin{proof}
  For $\alpha$,~$\beta \in \sfF[z^{-1}]$, we have:
  \[
  \big(z \partial_z + E \big) ((-)^\star \alpha,\beta) =
  \big((-)^\star \gr(\alpha),\beta\big) + 
  \big((-)^\star \alpha,\gr(\beta)\big) 
  \]
  and hence:
  \begin{equation}
    \label{eq:gr_and_Omega}
    (E-1) \Omega(\alpha,\beta) = \Omega\big(\gr(\alpha),\beta\big) + 
    \Omega\big(\alpha,\gr(\beta)\big)
  \end{equation}
  Rearranging gives:
  \[
  E \Omega(\alpha,\beta) - \Omega\big(\alpha,\gr(\beta)\big) = 
  \Omega\big(\gr(\alpha),\beta\big) + \Omega(\alpha,\beta)
  \]
  which is (a).  Part (b) follows immediately.  Rearranging
  \eqref{eq:gr_and_Omega} again gives:
  \[
  E \Omega(\alpha,\beta) - \Omega\big(\alpha,\gr(\beta)\big) -
  \Omega\big(\gr(\alpha),\beta\big) = \Omega(\alpha,\beta)
  \]
  which is (c).  For (d), we have:
  \begin{align*}
    \big(\gr \otimes 1 + 1 \otimes \gr\big) \Omega^\vee &= 
    \big(\gr \otimes 1 + 1 \otimes \gr\big) (\flat \otimes \flat) \Omega
    \\
    &= 
    (\flat \otimes \flat) \big((\gr^\vee-1) \otimes 1 + 1 \otimes
    (\gr^\vee - 1) \big) \Omega \\
    & = (\flat \otimes \flat) ({-\Omega}) && \text{by (c)} \\
    & = {-\Omega^\vee}
  \end{align*}
\end{proof}

\subsection{Opposite Modules and Propagators}

\begin{definition}[cf.~Definition~\ref{def:opposite}]
\label{def:opposite_log-cTEP} 
Let $(\sfF,\nabla, (\cdot,\cdot)_{\sfF})$ be a $\log$-cTEP structure
with base $(\cM,D)$.  Let $\sfP$ be a locally free
$\cO_\cM[z^{-1}]$-submodule $\sfP$ of $\sfF[z^{-1}]$.  We say that:
\begin{enumerate}
\item $\sfP$ is opposite to $\sfF$ if $\sfF[z^{-1}] = \sfF \oplus \sfP$;
\item $\sfP$ is isotropic if $\Omega(s_1,s_2) = 0$ for all $s_1$,~$s_2
  \in \sfP$;
\item $\sfP$ is parallel if $\nabla_X \sfP \subset \sfP$ for all $X
  \in \Theta_\cM(\log D)$;
\item $\sfP$ is homogeneous if $\nabla_{z\partial_z} \sfP \subset \sfP$. 
\end{enumerate}
An \emph{opposite module} for $(\sfF,\nabla, (\cdot,\cdot)_{\sfF})$ is
a locally free $\cO_\cM[z^{-1}]$-submodule $\sfP$ of $\sfF[z^{-1}]$
such that $\sfP$ is opposite to $\sfF$, isotropic, parallel, and
homogeneous.  Let $U$ be an open subset of $\cM$.  We say that $\sfP$
is an \emph{opposite module over $U$} if $\sfP$ is an opposite module
for the restriction $(\sfF,\nabla, (\cdot,\cdot)_{\sfF})\big|_U$.
\end{definition} 

\begin{remark}
 Conditions (3) and (4) here imply that an opposite module $\sfP$ is preserved by the grading
  operator $\gr$.
\end{remark}

\begin{example}[opposites compatible with Deligne give opposites for
  $\log$-cTEP structures]
  \label{ex:opposites_compatible}
  Let $(\cF, \nabla, (\cdot,\cdot)_{\cFbar})$ be a \logDTEP
  structure with base $(\cM,D)$ which is the Deligne extension of a
  TEP structure $(\cF^\times,\nabla,(\cdot,\cdot)_{\cF})$ with base $\cM
  \setminus D$.  Let $(\sfF,\nabla, (\cdot,\cdot)_{\sfF})$ be the
  $\log$-cTEP structure with base $(\cM,D)$ obtained from
  $(\cF,\nabla, (\cdot,\cdot)_{\cFbar})$ by taking the formal
  completion along the divisor $z=0$ in $\cM \times \C$.  Suppose that
  $\bP$ is an opposite module for $(\cF^\times,\nabla,(\cdot,\cdot)_{\cF})$
  which is compatible with the Deligne extension
  (Definition~\ref{def:compatible_with_Deligne}).  Then $\bP$
  determines a trivialization of $\cF$ and hence a trivialization
  of $\sfF$.  Thus $\bP$ determines an opposite module $\sfP$ for
  $(\sfF,\nabla, (\cdot,\cdot)_{\sfF})$.
\end{example}

\begin{example} \label{ex:opposite_B_log-cTEP}
  In particular, Proposition~\ref{pro:LR_conifold_orbifold} determines opposite modules for the big B-model log-cTEP structure: 
  \begin{align*}
    &\text{$\sfP_\LR$, defined near the large-radius limit point} \\    
    &\text{$\sfP_\con$, defined near the conifold point} \\    
    &\text{$\sfP_\orb$, defined near the orbifold point.} 
  \end{align*}
\end{example}

\begin{example} \label{ex:opposite_A_log-cTEP}
  The canonical opposite module $\bP_{\rm A}$ for the
  A-model TEP structure defined in Example~\ref{ex:opposite_A} is
  compatible with the Deligne extension.  It thus determines a
  canonical opposite submodule $\sfP_{\rm A}$
 for the A-model log-cTEP structure. 
\end{example}

An opposite module $\sfP$ determines flat connections on the
logarithmic tangent sheaf and logarithmic cotangent sheaf of $\LLo$,
as follows.  The connection $\tnabla$ on $\pr^\star \sfF$
(equation~\eqref{eq:tnabla}) extends $z^{-1}$-linearly to a flat
connection $\tnabla \colon \pr^\star \sfF[z^{-1}] \to \bOmega^1(\log
D) \hotimes \pr^\star (\sfF[z^{-1}])$, where:
\[
\bOmega^1(\log D) \hotimes \pr^\star(\sfF[z^{-1}]) :=\projlim_n
(\bOmega^1(\log D) \otimes \pr^\star(\sfF[z^{-1}]/z^n \sfF))
\]
The dual flat connection $\tnabla^\vee \colon \pr^\star
\sfF[z^{-1}]^\vee \to \bOmega^1(\log D) \hotimes \pr^\star
\sfF[z^{-1}]^\vee$ is defined by:
\begin{align*}
  \big \langle \tnabla^\vee \varphi,s \big \rangle
  := d \langle \varphi,s \rangle -
  \big \langle \varphi,\tnabla s \big \rangle
  && 
  \text{$s \in \pr^\star \sfF[z^{-1}]$, $\varphi \in  \pr^\star
    \sfF[z^{-1}]^\vee$}
\end{align*}
where $\bOmega^1(\log D) \hotimes \pr^\star(\sfF[z^{-1}]^\vee) :=\projlim_n
(\bOmega^1(\log D) \otimes \pr^\star(z^{-n} \sfF)^\vee)$.  This
induces flat connections $\tnabla^\vee \colon \pr^\star (z^n
\sfF)^\vee \to \bOmega^1 \otimes \pr^\star(z^{n+1} \sfF)^\vee$ for
each $n \in \Z$.

\begin{definition}
  \label{def:Nabla}
  Let $\sfP$ be an opposite module for the $\log$-cTEP structure
  $(\sfF,\nabla, (\cdot,\cdot)_{\sfF})$, and let $\Pi \colon
  \sfF[z^{-1}] \to \sfF$ be the projection along $\sfP$.  The
  composition of the maps:
  \begin{align*}
    & \xymatrix{
      \pr^\star \sfF \ar[rr]^-{\tnabla} && 
      \bOmega^1(\log D) \hotimes \pr^\star(z^{-1} \sfF) 
      \ar[rr]^-{\id \otimes \Pi} &&
      \bOmega^1(\log D) \hotimes \pr^\star \sfF
    } \\
    & \xymatrix{
      \pr^\star \sfF^\vee \ar[rr]^-{\Pi^\vee} && 
      \pr^\star (z^{-1} \sfF)^\vee \ar[rr]^-{\tnabla^\vee} &&
      \bOmega^1(\log D) \otimes \pr^\star \sfF^\vee 
    }
  \end{align*}
  (restricted to $\LLo$) with the Kodaira--Spencer isomorphisms $\KS
  \colon \bThetao(\log D) \to \pr^\star\sfF$, 
  $\KS^\star \colon \pr^\star\sfF^\vee
  \to \bOmegao^1(\log D)$ induces connections:
  \begin{equation}
    \label{eq:Nabla}
    \begin{aligned}
      &\Nabla \colon \bThetao(\log D) \to 
      \bOmegao^1(\log D) \hotimes \bThetao(\log D) \\
      &\Nabla \colon \bOmegao(\log D) \to 
      \bOmegao(\log D) \otimes \bOmegao(\log D) 
    \end{aligned}
  \end{equation}
  where $\bOmegao^1(\log D) \hotimes \bThetao(\log D) :=
  \projlim_n\big(\bOmegao^1(\log D) \otimes \big( \bThetao(\log D) /
  \KS^{-1}(\pr^\star (z^n \sfF)) \big) \big)$.  
\end{definition}

The connections in \eqref{eq:Nabla} are dual to each other.
Proposition~4.108 in~\cite{Coates--Iritani:Fock} shows that they are flat.

\begin{definition}
  \label{def:propagator}
  Let $\sfP_1$,~$\sfP_2$ be opposite modules for the
  $\log$-cTEP structure $(\sfF,\nabla, (\cdot,\cdot)_{\sfF})$.  Let
  $\Pi_i \colon \sfF[z^{-1}] \to \sfF$, $i \in \{1,2\}$, be the
  projection along $\sfP_i$ defined by the decomposition $\sfF[z^{-1}]
  = \sfP_i \oplus \sfF$.  The \emph{propagator}
  $\Delta=\Delta(\sfP_1,\sfP_2) \in
  \sHom_{\bcO}(\bOmegao^1(\log D) \otimes
  \bOmegao^1(\log D),\bcO)$
  is defined by:
  \begin{align*}
    \Delta(\omega_1, \omega_2) = 
    \Omega^\vee\big(\Pi_1^\star (\KS^\star)^{-1}\omega_1, 
    \Pi_2^\star (\KS^\star)^{-1} \omega_2\big),
    && \omega_1,\omega_2 \in \bOmegao^1(\log D).
  \end{align*}
  The logarithmic bivector field $\Delta$ coincides, via the
  Kodaira--Spencer isomorphism $\KS^\star$, with the push-forward along
  $\Pi_1 \times \Pi_2$ of the Poisson bivector field on $\sfF[z^{-1}]$
  defined by $\Omega^\vee$.
\end{definition}

The propagator $\Delta := \Delta(\sfP_1,\sfP_2)$ is symmetric,
i.e.~$\Delta(\omega_1, \omega_2) = \Delta(\omega_2, \omega_1)$ for all
$\omega_1,\omega_2 \in \bOmegao^1$~\cite[Proposition~4.110]{Coates--Iritani:Fock}.
Furthermore, if $\sfP_1$, $\sfP_2$, $\sfP_3$ are opposite modules for
the $\log$-cTEP structure $(\sfF,\nabla, (\cdot,\cdot)_{\sfF})$ and
$\Delta_{ij} := \Delta(\sfP_i,\sfP_j)$ then~\cite[Proposition~4.111]{Coates--Iritani:Fock}:
\[
\Delta_{13} = \Delta_{12} + \Delta_{23}
\]
In particular, $\Delta(\sfP_1,\sfP_2) =
{-\Delta(\sfP_2,\sfP_1)}$.

\begin{lemma}
  \label{lem:Pi_commutes_with_gr}
  Let $\sfP$ be an opposite module for the $\log$-cTEP structure
  $(\sfF,\nabla, (\cdot,\cdot)_{\sfF})$, and let $\Pi \colon
  \sfF[z^{-1}] \to \sfF$ be the projection along $\sfP$.  Then $\gr
  \circ \Pi = \Pi \circ \gr$.
\end{lemma}

\begin{proof}
  Let $\alpha \in \sfF[z^{-1}]$, and write $\alpha = \alpha_{\sfF} +
  \alpha_{\sfP}$ with $\alpha_{\sfF} \in \sfF$ and $\alpha_{\sfP} \in
  \sfP$.  Then $\gr(\alpha) = \gr(\alpha_{\sfF}) +
  \gr(\alpha_{\sfP})$.  The operator $\gr$ preserves both $\sfF$ and
  $\sfP$, so $\gr(\alpha_{\sfF}) \in \sfF$ and $\gr(\alpha_{\sfP}) \in
  \sfP$.  Thus $\Pi \circ \gr(\alpha) = \gr(\alpha_F) = \gr \circ
  \Pi(\alpha)$, as required.
\end{proof}

\begin{lemma}
  \label{lem:propagator_grading}
  Let $\sfP_1$,~$\sfP_2$ be opposite modules for the $\log$-cTEP
  structure $(\sfF,\nabla, (\cdot,\cdot)_{\sfF})$, and let $\Pi_i \colon
  \sfF[z^{-1}] \to \sfF$ be the projection along $\sfP_i$, $i \in
  \{1,2\}$.  Let $V = (\Pi_1 \otimes \Pi_2) \Omega^\vee$.  Then $(\gr
  \otimes 1 + 1 \otimes \gr) V = {-V}$.
\end{lemma}

\begin{proof}
  Combine Lemma~\ref{lem:grading_things} and
  Lemma~\ref{lem:Pi_commutes_with_gr}:
  \begin{align*}
    (\gr \otimes 1 + 1 \otimes \gr) (\Pi_1 \otimes \Pi_2) \Omega^\vee
    &= (\Pi_1 \otimes \Pi_2)(\gr \otimes 1 + 1 \otimes \gr) \Omega^\vee
    \\
    &= {-(\Pi_1 \otimes \Pi_2)\Omega^\vee} = {-V}
  \end{align*}
\end{proof}

\subsection{The Fock Sheaf}
\label{sec:Fock_log}
Consider a miniversal $\log$-cTEP structure $(\sfF,\nabla,
(\cdot,\cdot)_{\sfF})$ with base $(\cM,D)$.  
As before, let $\{t^j,q_k,x_n^i\}$ be an
algebraic local co-ordinate system on $\LL$ 
(see Definition~\ref{def:algebraic_local_co-ordinates}) 
where $\{t^j,q_k\}$ are co-ordinates on an open set $U \subset \cM$.  
Write the co-ordinates
$\{t^0,\log q_1,\ldots,\log q_r,t^{r+1},\ldots,t^N,x^i_n\}$ as
$\{\sx^\mu\}$, so that:
\[
\begin{cases}
  \text{
    $d\sx^\mu = \frac{dq_j}{q_j}$ 
    and
    $\frac{\partial}{\partial \sx^\mu}
    = q_j \frac{\partial}{\partial q_j}$
  }
  & \text{if $\sx^\mu = \log q_j$ and $1 \leq j \leq r$}
  \\
  \text{
    $d\sx^\mu = dt^j$ 
    and
    $\frac{\partial}{\partial \sx^\mu}
    = \frac{\partial}{\partial t^j}$
  } & \text{if $\sx^\mu = t^j$ and $j=0$ or $r<j \leq N$} \\
  \text{
    $d\sx^\mu = dx^i_n$ 
    and
    $\frac{\partial}{\partial \sx^\mu}
    = \frac{\partial}{\partial x^i_n}$
  }  & \text{if $\sx^\mu = x^i_n$} 
\end{cases}
\]
We use Einstein's summation convention for repeated indices,
expressing the Yukawa coupling and propagator $\Delta =
\Delta(\sfP_1,\sfP_2)$ as:
\begin{align*}
  \bY = C^{(0)}_{\mu \nu \rho}
  d\sx^\mu \otimes d\sx^\nu \otimes d\sx^\rho &&
  \Delta = \Delta^{\mu \nu} \partial_\mu \otimes \partial_\nu
\end{align*}
where $\partial_\mu := \frac{\partial}{\partial \sx^\mu}$. Let $\sfP$
be an opposite module for $(\sfF,\nabla, (\cdot,\cdot)_{\sfF})$ and
consider the flat connection $\Nabla$ on $\bOmegao^1(\log D)$
determined by $\sfP$ (Definition~\ref{def:Nabla}).  The Christoffel
symbols of $\Nabla$ are defined by:
\[
\Nabla_{\partial_\nu} d\sx^\mu = 
{-\Gamma^\mu_{\nu \rho}} d\sx^\rho
\]
The flat connection $\Nabla$ acts on $n$-tensors $C_{\mu_1 \cdots
  \mu_n} d\sx^{\mu_1} \otimes \cdots \otimes d\sx^{\mu_n} \in
\big(\bOmegao^1(\log D)\big)^{\otimes n}$ by:
\[
\Nabla(C_{\mu_1 \cdots  \mu_n} 
d\sx^{\mu_1} \otimes \cdots \otimes d\sx^{\mu_n})
=
(\Nabla_\nu C_{\mu_1 \cdots  \mu_n})
d\sx^\nu \otimes d\sx^{\mu_1} \otimes \cdots \otimes d\sx^{\mu_n}
\]
where:
\begin{equation} 
\label{eq:covariant_derivative_tensor}
(\Nabla_\nu C_{\mu_1 \cdots  \mu_n})
:=
\partial_\nu C_{\mu_1 \cdots  \mu_n}
-
\sum_{i=1}^n
C_{\mu_1 \cdots \underset{i}{\rho} \cdots \mu_n}
\Gamma^\rho_{\mu_i \nu}
\end{equation} 
\begin{definition}[local Fock space] 
\label{def:local_Fock} 
The \emph{local Fock space} $\Fock(U;\sfP)$  consists of collections:
\[
\big\{ \Nabla^n C^{(g)} \in
\big(\bOmega^1(\log D)\big)^{\otimes n}
\big(\pr^{-1}(U)^\circ \big): 
g \geq 0, n \geq 0, 2g-2+n>0 
\big \}
\]
of completely symmetric logarithmic $n$-tensors on $\pr^{-1}(U)^\circ$
such that the following conditions hold:
\begin{itemize}
\item (Yukawa) $\Nabla^3 C^{(0)}$ is the Yukawa coupling $\bY$;
\item (Jetness) $\Nabla(\Nabla^n C^{(g)}) = \Nabla^{n+1} C^{(g)}$;
\item (Grading and Filtration) $\Nabla^n C^{(g)} \in
  \big(\big(\bOmega^1(\log D)\big)^{\otimes n} \big(\pr^{-1}(U)^\circ
  \big)\big)^{2-2g}_{3g-3}$;
\item (Pole) $P \Nabla C^{(1)}$ extends to a regular $1$-form on
  $\pr^{-1}(U)$, where $P$ is the discriminant~\eqref{eq:discriminant}.  Furthermore for $g \geq 2$ we have:
  \[
  C^{(g)} \in 
  P^{5-5g} \cO(U)[x_1,x_2,Px_3,\ldots,P^{3g-4} x_{3g-2}]
  \]
\end{itemize}
Writing:
\[
\Nabla^n C^{(g)} = 
C^{(g)}_{\mu_1\cdots\mu_n} d\sx^{\mu_1} \otimes \cdots \otimes
d\sx^{\mu_n}
\]
we refer to $\Nabla^n C^{(g)}$ or $C^{(g)}_{\mu_1\cdots\mu_n}$ as
\emph{$n$-point correlation functions}.
\end{definition} 

We encode elements of the local Fock space $\Fock(U;\sfP)$ as formal
functions on the total space of the logarithmic tangent bundle
$\bTheta(\log D)|_{\pr^{-1}(U)^\circ}$, called jet potentials.  Let
$\{\sy^\mu\}$ denote the fiber co-ordinates of the logarithmic tangent
bundle $\bTheta(\log D)$ dual to $\{\frac{\partial}{\partial \sx^\mu}\}$, so that $(\sx,\sy)$
denotes a point in the total space of $\bTheta(\log
D)|_{\pr^{-1}(U)^\circ}$.

\begin{definition}[jet potential] 
\label{def:jetpotential}
Given an element $\wave = \{\Nabla^n C^{(g)}\}_{g,n}$ of
$\Fock(U;\sfP)$, set:
\begin{equation}
\label{eq:jetpotential} 
\begin{aligned}
  & \cW^g(\sx,\sy) 
  = \sum_{n=\max(0,3-2g)}^\infty 
  \frac{1}{n!} 
  C^{(g)}_{\mu_1,\dots,\mu_n}(\sx) 
  \sy^{\mu_1} \cdots \sy^{\mu_n} \\
  & \cW(\sx,\sy) 
  = \sum_{g=0}^\infty \hbar^{g-1} \cW^g(\sx,\sy) 
\end{aligned}
\end{equation} 
We call $\cW^g$ the \emph{genus-$g$ jet potential} and $\exp(\cW)$ the
\emph{total jet potential} associated to $\wave$.
\end{definition} 

\begin{remark}
  $\exp(\cW)$ is well-defined as a power series in $\hbar$ and
  $\hbar^{-1}$: cf.~\cite[Remark~4.63(2)]{Coates--Iritani:Fock}.
\end{remark}

The Fock sheaf is constructed by gluing local Fock spaces
$\Fock(U;\sfP_1)$, $\Fock(U;\sfP_2)$ according to the following
\emph{transformation rule}.  Let $\Delta$ denote the propagator
$\Delta(\sfP_1,\sfP_2)$.  The transformation rule $T(\sfP_1, \sfP_2)\colon
\Fock(U;\sfP_1) \to \Fock(U;\sfP_2)$ is a map which assigns to a jet
potential $\exp(\cW)$ for an element of $\Fock(U;\sfP_1)$, the jet
potential $\exp(\hcW)$ for an element of $\Fock(U;\sfP_2)$ given by:
\begin{equation} 
\label{eq:transformationrule-jet} 
\exp\big(\hcW(\sx,\sy)\big) =  
\exp\left(\frac{\hbar}{2}\Delta^{\mu\nu} \partial_{\sy^\mu}
\partial_{\sy^\nu}\right) 
\exp\big(\cW(\sx,\sy)\big). 
\end{equation} 
This is equivalent to expressing the correlation functions
$\{\hC^{(g)}_{\mu_1,\dots,\mu_n}\}_{g,n}$ for $\hcW$ in terms of sums
over Feynman graphs, the vertex terms of which are the correlation
functions $\{C^{(g)}_{\mu_1,\dots,\mu_n} \}_{g,n}$ for $\cW$.  We use
the notation for graphs established in Appendix~\ref{sec:graph_notation}.
The transformation rule \eqref{eq:transformationrule-jet} is
equivalent to the \emph{Feynman rule}:
\[
\hC^{(g)}_{\mu_1,\dots,\mu_n} 
= \sum_\Gamma \frac{1}{|\Aut(\Gamma)|} 
\Cont_{\Gamma}(\Delta, \{C^{(h)}\}_{h\le g}
)_{\mu_1,\dots,\mu_n}
\]
Here the summation is over all connected decorated graphs 
$\Gamma$ such that 
\begin{itemize}
\item 
To each vertex $v\in V(\Gamma)$ is 
assigned a non-negative integer $g_v\ge 0$, called genus; 
  
\item 
$\Gamma$ has labelled $n$-legs: an isomorphism 
$L(\Gamma) \cong \{1,2,\dots,n\}$ is given; 

\item $\Gamma$ is stable, i.e.\ 
$2 g_v - 2 + n_v>0$ for every vertex $v$.  
Here $n_v = |\pi_V^{-1}(v)|$ denotes the number of edges or 
legs incident to $v$; 

\item $g = \sum_{v} g_v + 1- \chi(\Gamma)$. 
\end{itemize}
We put the index $\mu_i$ on the $i$th leg, the correlation
function $\Nabla^{n_v} C^{(g_v)}$ on the vertex $v$, and the propagator
$\Delta$ on every edge.  Then
$\Cont_\Gamma(\Delta,\{C^{(h)}\}_{h\le g})_{\mu_1,\dots,\mu_n}$ 
is defined to be the
contraction of all these tensors with the indices $\mu_1,\dots,\mu_n$
on the legs fixed.  Here $\Aut(\Gamma)$ denotes the automorphism group
of the decorated graph $\Gamma$.

\begin{remark}
  We showed in~\cite[Proposition~4.115]{Coates--Iritani:Fock} that the transformation rule
  \eqref{eq:transformationrule-jet} is well-defined, i.e.~that it
  preserves the conditions (Yukawa), (Jetness), (Grading and
  Filtration), and (Pole) in the definition of the local Fock space
  $\Fock(U;\sfP_i)$.
\end{remark}

\begin{remark}
  The transformation rule \eqref{eq:transformationrule-jet} satisfies
  the cocycle condition~\cite[Proposition~4.111]{Coates--Iritani:Fock} : if $\sfP_1$,~$\sfP_2$,~$\sfP_3$ are opposite
  modules for $\sfF$ over $U$ and $T_{ij} = T(\sfP_i,\sfP_j)$ is the
  transformation rule from $\Fock(U;\sfP_i)$ to $\Fock(U;\sfP_j)$ then
  $T_{13} = T_{23} \circ T_{12}$.
\end{remark}

\begin{assumption}[Covering Assumption]
  \label{assumption:covering}
  There is an open covering $\{U_a : a \in A\}$ of $\cM$ such that for
  each $a \in A$ there exists an opposite module $\sfP_a$ for $\sfF$
  over $U_a$.
\end{assumption}

\begin{definition}[Fock sheaf] 
If Assumption~\ref{assumption:covering} holds, then we define the
\emph{Fock sheaf} to be the sheaf of sets on $\cM$ obtained by gluing
the local Fock spaces $\Fock(U_a;\sfP_a)$, $a \in A$, using the
transformation rule
\begin{align*}
  T(\sfP_a,\sfP_b) : \Fock(U_a \cap U_b;\sfP_a) \to \Fock(U_a \cap
  U_b;\sfP_b)
  && a, b \in A
\end{align*}
over $U_a \cap U_b$.
\end{definition} 
\begin{remark}
  Note that the Fock sheaf is a sheaf over all of $\cM$, not just over
  $\cM \setminus D$.
\end{remark}
\begin{remark} 
We can define the Fock sheaf without the covering assumption: 
see \cite[\S 4.13]{Coates--Iritani:Fock}. The definition there 
requires an analysis of anomaly equations for curved (i.e.~non-parallel) opposite 
modules. 
\end{remark} 

\begin{definition}[Gromov--Witten wave function] 
\label{def:GW-wave}
Let $X$ denote either $\cXbar$ or $\Ybar$, and 
consider the A-model log-cTEP structure for $X$ defined 
in Example~\ref{ex:log_cTEP_A}. 
The base of this log-cTEP structure is $(\cMA{X},\DA{X})$, 
and we denote the corresponding Fock sheaf on 
$\cMA{X}$ by $\Fock_{{\rm A},X}$. 
The Gromov--Witten ancestor potentials of $X$ define 
a global section $\wave_X$ of $\Fock_{{\rm A},X}$, 
the \emph{Gromov--Witten wave function}, as we now explain.

Let $\{\phi_i \}_{i=0}^N$ be a homogeneous
basis of $H_X$ as in \S\ref{sec:bases}, and write a general point $t \in \cMA{X}$ as $t = \sum_{i=0}^N t^i \phi_i$.  Recalling that $\phi_1,\ldots,\phi_r$ form a basis for $H^2(X)$, set $q_i = e^{t^i}$, $1 \leq i \leq r$, and write $\{t^j,q_k,x_n^i\}$ for the corresponding
algebraic local co-ordinate system on the total space $\LL$ of the A-model log-cTEP structure.  Let $\sfP_{\rm A}$ denote the canonical opposite module defined in Example~\ref{ex:opposite_A_log-cTEP}. The Gromov--Witten wave-function $\wave_X$ is defined by the element 
$\{ \Nabla^n C^{(g)}_X \}_{g,n} \in 
\Fock_{\rm A}(\cMA{X}; \sfP_{\rm A})$ where:
\begin{equation}
  \label{eq:GW_wave}
  \begin{aligned} 
    \Nabla^3 C_X^{(0)} &= \bY = 
    \sum_{i=0}^N \sum_{j=0}^N \sum_{k=0}^N
    dt^i \otimes dt^j  \otimes dt^k 
    \bigg( \phi_i * \phi_j * x_1, \phi_k * x_1 \bigg) 
    \\
    \Nabla C_X^{(1)} & = d(F^1_X(t) + \bar{\cF}_X^1)
    \Big |_{a_0 =0, Q_1 = \cdots = Q_r =1} \\ 
    C_X^{(g)} &= \bar{\cF}^g_X \Big |_{a_0 = 0, Q_1 = \cdots = Q_r =1} 
    \qquad \text{for $g\ge 2$} 
  \end{aligned}
\end{equation} 
and $\Nabla$ denotes the covariant derivative $\Nabla^{\sfP_{\rm A}}$ from Definition~\ref{def:Nabla}.  
Here $*$ is the quantum product \eqref{eq:quantum_product}, 
$\bar\cF_X^g$ is the genus-$g$ ancestor potential \eqref{eq:ancestor}, 
$F^1_X(t)$ is the non-descendant genus-one 
Gromov--Witten potential: 
\[
F^1_X(t) = \sum_{n=0}^\infty \sum_{\substack{
d\in \NE(X) \\ (n,d)\neq (0,0)}}
\frac{Q^d}{n!} 
\corr{t,\dots,t}_{1,n,d} 
\]
and we used the Dilaton shift: 
\begin{equation} 
\label{eq:Dilaton_shift} 
a_n^i = x_n^i + \delta_n^1 \delta^i_0 \qquad n\ge 1 
\end{equation} 
to identify the variables $\{t^i,q_k, a_n^i\}$ on the right-hand side
with the co-ordinates $\{t^i, q_k, x_n^i\}$ on $\LL$. Our convergence
results in~\cite{Coates--Iritani:convergence} imply that the
Gromov--Witten wave-function is well-defined -- that is, that 
the specialization $Q_1=\cdots=Q_r=1$ in \eqref{eq:GW_wave} 
makes sense and yields an analytic function, 
and the resulting correlation functions satisfy the conditions (Yukawa), (Jetness), (Grading \& Filtration) and 
(Pole).  See~\cite[Section~6]{Coates--Iritani:Fock} for details.
\end{definition}

\begin{remark}[{\cite[Theorem~6.8]{Coates--Iritani:Fock}}]  
  The Ancestor--Descendant relation 
\cite[Theorem 2.1]{Kontsevich--Manin:relations}, 
\cite[\S 5]{Givental:quantization} 
implies that for 
$t \in \cMA{X}$ sufficiently close to the large-radius limit point for $X$ and $\bx$ sufficiently close to $-1 z$, there are flat co-ordinates $\bq = (q_n^i)$ on a neighbourhood of $(t,\bx)$ in $\LL$ such that:
  \begin{align*} 
    \Nabla^3 C^{(0)}_X & = \bY = \sum_{l,m,n=0}^\infty \sum_{i,j,k=0}^N 
                         \parfrac{^3\cF^0_{X}}{q_l^i \partial q_m^j \partial q_n^k}(\bq) \,
                         dq_l^i \otimes dq_m^j \otimes dq_n^k \\ 
    \Nabla C^{(1)}_X &  = d\cF^1_{X}(\bq) \\ 
    C^{(g)}_X & = \cF^g_{X}(\bq) \qquad \text{for $g\ge 2$.}
  \end{align*}
  Here we regard the genus-$g$ descendant potential $\cF^g_X$, which was defined in \S\ref{subsec:descendant_and_ancestor} as a function of variables $t_n^i$, as a function of $q_n^i$ via the Dilaton Shift $t_n^i = q_n^i + \delta_n^1 \delta^i_0$.  The flat co-ordinates $\bq$ and 
  the algebraic co-ordinates $(t,\bx)$ are related by 
  \[
  \bq(z) = \left [L(t,-z)^{-1} \bx(z)\right]_+
  \]
  where $L(t,-z)$ is the fundamental solution \eqref{eq:GW_fundsol},  
  $[\cdots]_+$ denotes the non-negative part as a $z$-series, 
  $\bq(z) = \sum_{n=0}^\infty q_n z^n$, 
  $\bx(z) = \sum_{n=1}^\infty x_n z^n$, 
  $q_n = \sum_{i=0}^N q_n^i \phi_i$, and $x_n = \sum_{i=0}^N x_n^i \phi_i$. 
Thus one can think of the Gromov--Witten wave function $\wave_X$ as encoding the total descendant potential $\cZ_X$ of $X$.
\end{remark}

\subsection{A Global Section of the Fock Sheaf for the Big B-Model
  $\log$-cTEP Structure}

We now construct a global section of the Fock sheaf for the big
B-model $\log$-cTEP structure.  This global section coincides under
mirror symmetry with the Gromov--Witten wave functions $\wave_{\Ybar}$ and $\wave_{\cXbar}$.

\begin{proposition} \label{pro:covering_assumption}
  The Covering Assumption (Assumption~\ref{assumption:covering}) holds
  for the big B-model $\log$-cTEP structure.
\end{proposition}

\begin{proof}
  Let $y \in \cMBbig$ be a point of $\cMB^\circ \subset \cMBbig$.  In
  \S\ref{sec:enlarge_base} we constructed, for a sufficiently small
  neighbourhood $U_y^{\rm sm}$ of $y$ in $\cMB$, an
  opposite module $\bP_y^{\rm sm}$ for the B-model TEP structure on
  $U_y^{\rm sm} \setminus D$ which is compatible with the Deligne
  extension.  After shrinking $U_y^{\rm sm}$ if necessary, we may
  assume that $U_y^{\rm sm} \subset \cMB^\circ$.  By the
  construction of $\cMBbig$ in \S\ref{sec:enlarge_base}, the opposite
  module $\bP_y^{\rm sm}$ extends to an opposite module $\bP_y$ for
  the big B-model TEP structure on $U_y \setminus \Dbig$ which is
  compatible with the Deligne extension, for some neighbourhood $U_y$
  of $y$ in $\cMBbig$.  
Recall that $\cMBbig$ was constructed as the germ of
  a thickening of $\cMB^\circ$; after shrinking $\cMBbig$ if necessary
  we may assume that the open sets $\{U_y : y \in \cMB^\circ\}$ just
  constructed form an open covering of $\cMBbig$.  
By Example~\ref{ex:opposites_compatible}, the opposite module $\bP_y$
  over $U_y$ determines an opposite module $\sfP_y$ for the big
  B-model $\log$-cTEP structure over $U_y$.  Thus
  Assumption~\ref{assumption:covering} holds for the big B-model
  $\log$-cTEP structure.
\end{proof}

\begin{definition}
\label{def:Fock_B}
In view of Proposition~\ref{pro:covering_assumption}, 
there is a Fock sheaf on $\cMBbig$ determined by the big
B-model $\log$-cTEP structure.  We denote this by $\Fock_{\rm B}$.
\end{definition}

\begin{definition}
  Recall the definition of $\cC(t,z)$ from equation
  \eqref{eq:definition_of_C}.  We say that a $\log$-cTEP structure
  $(\sfF,\nabla, (\cdot,\cdot)_\sfF)$ with Euler field $E$ and base
  $(\cM,D)$ is \emph{tame semisimple} at $t \in \cM$ if the
  endomorphism $\iota_E \cC(t,0) \in \End_{\C}(\sfF_{0,t})$ 
  is semisimple with pairwise distinct eigenvalues. 
  This endomorphism is ``multiplication by the
  Euler field'' and coincides with the action 
of $\nabla_{z^2 \partial_z}$ on $\sfF_{0,t}$. 
Tame semisimplicity implies that all 
operators $\iota_v \cC(t,0)\in \End(\sfF_{0,t})$ 
with $v \in T\cM(\log D)$ are semisimple. 
\end{definition}

Consider now the big B-model $\log$-cTEP structure 
$(\FBbig,\nablaB,\pairingB)$ with 
base $(\cMBbig,\Dbig)$. 
Let $U_{\tss}$ denote the set of points at which this 
$\log$-cTEP structure is tame semisimple. 
The complement of $U_{\tss}$ in $\cMBbig$ is a union of 
divisors $\{B_i : i \in I\}$ and, 
after shrinking the thickening $\cMBbig$ of $\cMB^\circ$ 
if necessary, we may insist that each irreducible component 
$B_i$ meets $\cMB^\circ \subset \cMBbig$. 
The critical values of the superpotential $W_y$ are distinct for 
$y \in \cMB^\times = \cMB^\circ \setminus D$ 
(see equation~\eqref{eq:critical_values}), 
and so the tame semisimple locus $U_{\tss}$ contains $\cMB^\times$.  
This implies that, for each $i \in I$, the intersection of the divisor $B_i$ 
with $\cMB^\circ$ either contains the component $(y_1=0)$ 
or the component $(y_2=0)$ of $D \cap \cMB^\circ$.  
In particular, each divisor $B_i$ contains the large-radius limit point 
$y_1=y_2=0$. 
Moreover, we can also see that $U_{\tss}$ does not intersect 
with $\Dbig$. In fact, by our local construction of $\cFBbig$ 
in \S\ref{sec:unfolding_locally}, 
the residues of $z\nabla$ along $\Dbig$ define 
nilpotent operators 
(see Proposition \ref{pro:Birkhoff_logarithmic_connection}) 
in $\End(\sfF_{{\rm B},0}^{\rm big})$, 
which are non-zero by miniversality. 
Therefore a point on $\Dbig$ 
cannot be tame semisimple\footnote
{We can also check that this holds for the big quantum cohomology of $\Ybar$: 
the quantum product $h_i\star$ coincides with the nilpotent 
operator $h_i \cup$ along $q_i=0$ because of the Divisor Equation.}. 

In previous work we have shown -- see \cite[Definition~7.9]{Coates--Iritani:Fock} -- that
Givental's formula~\cite{Givental:semisimple} for higher-genus potentials 
defines a section $\wave_{\tss}$ of the B-model Fock sheaf $\Fock_{\rm B}$ 
over the tame semisimple locus $U_{\tss} \subset \cMBbig$.
The mirror isomorphism of \logDTEP structures from 
Theorem~\ref{thm:big_mirror_symmetry_Ybar}: 
\[
\big(\cFBbig,\nablaB,\pairingB\big)\Big|_{U^{\rm big}\times \C} 
\cong 
    \Mir^\star
    \big(\cFA{\Ybar},\nablaA{\Ybar}, \pairingA{\Ybar} \big)
\]
induces an isomorphism of Fock sheaves 
\[
\Fock_{\rm B}\big|_{U^{\rm big}} \cong \Mir^\star 
\Fock_{{\rm A}, \Ybar} 
\]
and Teleman's theorem~\cite{Teleman} implies that, 
under this isomorphism, 
$\wave_{\tss}$ corresponds to the Gromov--Witten wave function 
$\Mir^\star \wave_{\Ybar}$ (see Definition \ref{def:GW-wave})  
over $U_{\tss} \cap U^{\rm big}$.  (This is explained in detail in~\cite[Theorem~7.15]{Coates--Iritani:Fock}.) 
The same is true when we replace $\Ybar$ with $\cXbar$ 
and work near the orbifold point, which is the large-radius limit point for $\cXbar$. 
We obtain the following: 
\begin{theorem}
\label{thm:wave_B} 
After shrinking the thickening $\cMBbig$ of $\cMB^\circ$ if necessary, 
there exists a global section $\wave_{\rm B}$ of $\Fock_{\rm B}$ 
over $\cMBbig$ 
extending $\wave_{\tss}$ such that the following holds: 
\begin{itemize} 
\item[(a)] near the large radius limit point for $\Ybar$, 
$\wave_{\rm B}$ corresponds to the Gromov--Witten wave 
function of $\Ybar$ under the mirror isomorphism in 
Theorem~\ref{thm:big_mirror_symmetry_Ybar}; 
\item[(b)] near the large radius limit point for $\cXbar$, 
$\wave_{\rm B}$ corresponds to the Gromov--Witten wave 
function of $\cXbar$ under the mirror isomorphism in 
Remark~\ref{rem:big_mirror_symmetry_cXbar}. 
\end{itemize} 
\end{theorem}  

\begin{proof}
In view of our discussion, it suffices to show that the section $\wave_{\tss}$ 
extends holomorphically across the divisors $B_i$, $i\in I$.  
The divisors $B_i$,~$i \in I$, all meet the open set~$U^{\rm big}$.
  By Hartog's Principle, it suffices to check that the correlation
  functions for $\wave_{\tss}$ with respect to one opposite module
  extend to holomorphic functions on all of~$U^{\rm big}$. 
We check this using the opposite module~$\sfP_\LR$ from
  Example~\ref{ex:opposite_B_log-cTEP}; under mirror symmetry, this
  corresponds to the canonical opposite module ~$\sfP_{\rm A}$ from
  Example~\ref{ex:opposite_A_log-cTEP} (see
  Theorem~\ref{thm:big_mirror_symmetry_Ybar}) and $\wave_{\tss}$
  corresponds to $\wave_{\Ybar}$. But the correlation functions
  \eqref{eq:GW_wave} for $\wave_{\Ybar}$ are evidently holomorphic on
  all of $\Mir(U^{\rm big})$. 
\end{proof}

\begin{remark} \label{rem:CRC}
The existence of a global section $\wave_{\rm B}$ 
with these properties establishes a higher-genus version of 
the Crepant Resolution Conjecture for $\cXbar= \Proj(1,1,1,3)$.  See
Theorem~8.1 and Corollary~8.2 in~\cite{Coates--Iritani:Fock} 
for a more general result for weak-Fano toric orbifolds.  
\end{remark} 

\section{The Finite-Dimensional Fock Sheaf}
\label{sec:Fock_sheaf_fd}

In this section we construct a finite-dimensional version of the Fock sheaf, which one can think of as arising from the big B-model Fock sheaf by taking the conformal limit.  Recall from \S\ref{sec:conformal_limit} that we have a three-dimensional vector bundle $\widebar{H} \to \cMCY$ equipped with a logarithmic flat connection $\nabla$, a two-dimensional flat subbundle $\Hvec$ of $\widebar{H}$, a two-dimensional flat affine subbundle $\Haff$ of $\widebar{H}$, and a distinguished section $\zeta$ of $\Haff$.  
There is a canonical identification, for each $y \in \cMCY$, between
any tangent space to $\Haff|_y$ and the fiber $\Hvec|_y$, so $\Haff$
is parallel to $\Hvec$; $\Haff$ is a symplectic affine bundle, 
$\Hvec$ is a symplectic vector bundle, and 
this identification between $\Haff$ and $\Hvec$ intertwines 
the symplectic structures.  
The base $\cMCY$ of $\widebar{H}$, $\Haff$, and $\Hvec$ 
is isomorphic to $\Proj(3,1)$, and the flat connection $\nabla$ 
has logarithmic poles at the divisor $\DCY = \{0,{-\frac{1}{27}}\}$. 
The finite-dimensional Fock sheaf that we will construct has base $\cMCY$.  

\begin{figure}[htb]
\centering
\includegraphics[scale=0.7, bb=32 37 443 273]{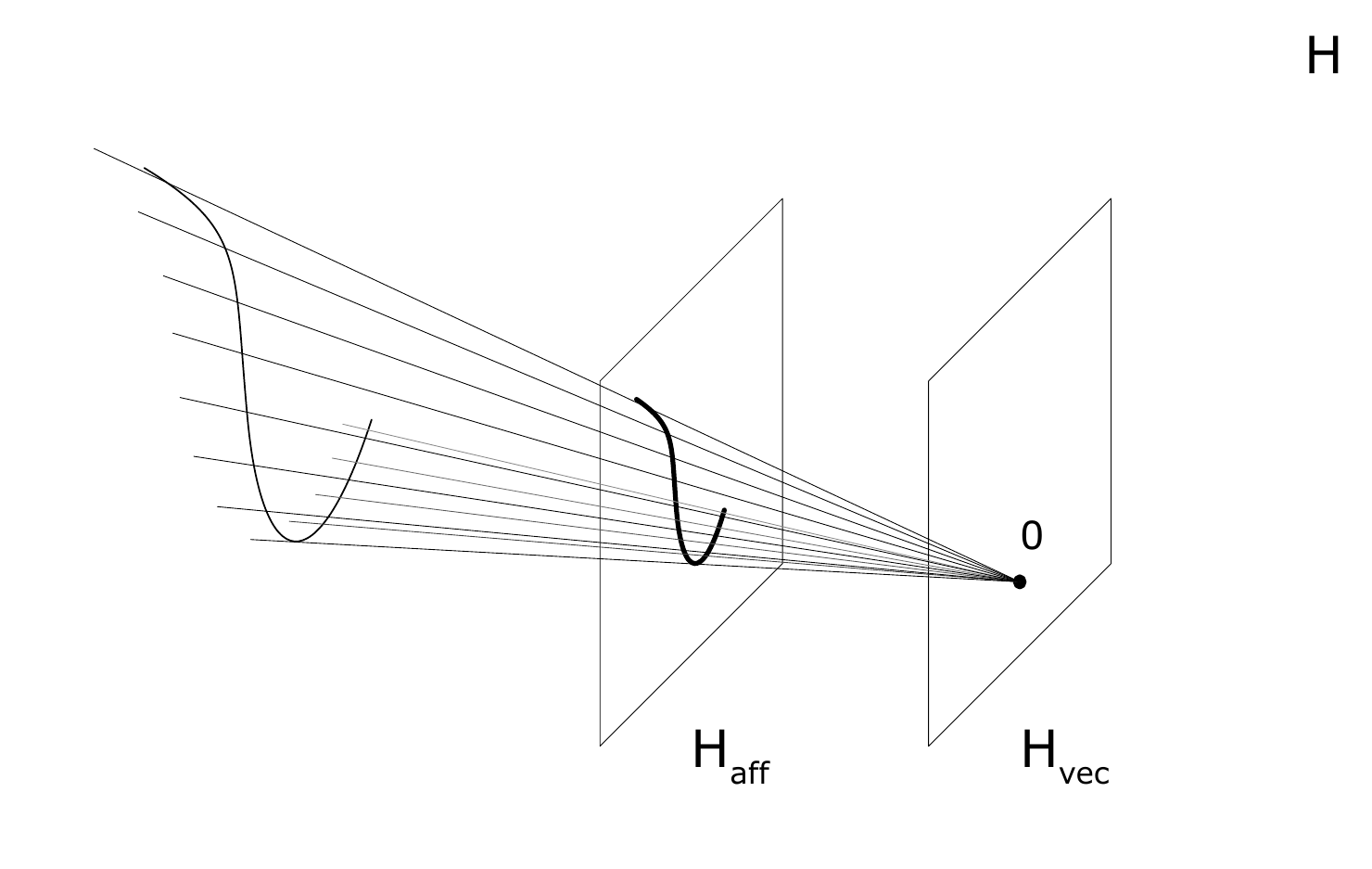} 
\caption{The finite-dimensional cone $\hcL$: the primitive 
section $\zeta$ sweeps out a Lagrangian curve $\cL 
\subset \Haff|_{t_0}$ via parallel translation to the fibre at $t_0$.}
\label{fig:finite-dimensional_cone}
\end{figure}

\subsection{The Yukawa Coupling and the Kodaira--Spencer Map}
\label{sec:Yukawa_fd}

\begin{notation} 
We denote by $\Theta(\log \DCY)$ the sheaf of tangent vector 
fields on $\cMCY$ 
logarithmic along $\DCY$ and by $\Omega^1(\log \DCY)$ 
the sheaf of 1-forms on $\cMCY$ logarithmic along $\DCY$. 
Similarly, we denote by $\Theta(\log \{0\})$ 
(respectively $\Omega^1(\log \{0\})$)  
the sheaf of tangent vector fields (respectively 1-forms) logarithmic 
only at $0\in \DCY$. 
\end{notation}

\begin{definition}[cf.~Definition \ref{def:Yukawa}]
  \label{def:Yukawa_fd} 
  The \emph{Yukawa coupling} $Y_{\CY} \in \big(\Omega^1(\log
  \DCY)\big)^{\otimes 3}$ is defined by:
  \begin{align*}
    Y_{\CY}(X_1,X_2,X_3) = \Omega\big(\nabla_{X_1} \nabla_{X_2}
    \zeta,\nabla_{X_3} \zeta\big)
    &&
    \text{$X_1$,~$X_2$,~$X_3 \in \Theta(\log \DCY)$}
  \end{align*}
Here we regard $\nabla_{X_3} \zeta$ and 
$\nabla_{X_1} \nabla_{X_2} \zeta$ as sections of $\Hvec$, 
via the identification of tangent spaces to $\Haff$ 
with fibers of $\Hvec$ discussed above. 
\end{definition}

\begin{definition}[Definition~\ref{def:KS}] 
  \label{def:Kodaira_Spencer_fd}
  The \emph{Kodaira--Spencer map} is:
  \begin{align*}
    \KS \colon \Theta(\log \{0\}) & \to \cO(\Hvec) \\
    X & \mapsto \nabla_X \zeta
  \end{align*}
\end{definition}

\begin{remark} 
  The Kodaira--Spencer map gives an isomorphism between the
  logarithmic tangent bundle $\Theta(\log\{0\})$ 
of $(\cMCY,\{0\})$ and the subbundle $F^2_\vec
  \subset \Hvec$ defined in \S\ref{sec:opposite_filtrations}.
\end{remark}

\begin{remark} 
\label{rem:pole_Yukawa} 
From the previous remark, it follows that the Yukawa 
coupling has a pole of order 3 at the large-radius limit point 
$y_1 = 0$ and a pole of order 1 at the conifold point 
$y_1 = -\frac{1}{27}$.  We give an explicit formula 
for $Y_{\CY}$ in Example~\ref{ex:Yukawa_conformal_limit}.
\end{remark} 

Let $t_0\in \cMCY$ be a point away from $\DCY$. 
Locally near $t_0$, we can encode the information 
of the filtered flat bundle $(\widebar{H},  
\widebar{F}^1 \subset 
\widebar{F}^2 \subset \widebar{H},\nabla)$ 
discussed in \S\ref{sec:opposite_filtrations}
as a finite-dimensional cone $\hcL$ in $\widebar{H}|_{t_0}$. 
Parallel translation defines an isomorphism 
$\widebar{H}|_t \cong \widebar{H}|_{t_0}$ for $t$ 
in a small neighbourhood of $t_0$. 
Via this isomorphism, the flag 
$(0\subset \widebar{F}^1_t\subset \widebar{F}^2_t
\subset \widebar{H}|_t)$ can be identified 
with a flag\footnote{The map $t\mapsto 
(0\subset \widebar{F}^1_t \subset \widebar{F}^2_t \subset 
\widebar{H}|_t) \in \operatorname{Fl}_{1,2,3}(\widebar{H}|_{t_0})$ 
can be viewed as a period map. } in $\widebar{H}|_{t_0}$. 
With this identification in mind, we define the finite dimensional cone 
$\hcL \subset \widebar{H}|_{t_0}$ to be: 
\[
\hcL = \bigcup_{t} \widebar{F}^1_t 
\]
where $t$ varies in a neighbourhood of $t_0$. 
See Figure \ref{fig:finite-dimensional_cone}. 
Recall from \S \ref{sec:opposite_filtrations} 
that $\widebar{F}^1_t$ is a line generated by the primitive 
section $\zeta = -z$, that $\widebar{F}^2_t$ is generated by 
$\zeta$ and $\theta \zeta$, 
and that $\theta \zeta = \nabla_{y_1\parfrac{}{y_1}}\zeta$. 
The tangent space of $\hcL$ along the 
line $\widebar{F}^1_t$ is therefore $\widebar{F}^2_t$. 
Recall also that $\zeta$ lies in the affine subbundle 
$\Haff$. Under the above identification, $t\mapsto \zeta(t)$ sweeps out 
a one-dimensional Lagrangian submanifold $\cL$ 
in $\Haff|_{t_0}$: 
\[
\cL = \hcL \cap \Haff|_{t_0} 
= \{ \zeta(t) : \text{$t$ in a neighbourhood of $t_0$}\}. 
\]
In other words, $\cMCY\setminus \DCY$ can be locally embedded 
into a fiber $\Haff|_{t_0}$ as a Lagrangian submanifold. 
The tangent space of $\cL$ at 
$\zeta(t)$ is identified with $F^2_\vec|_t 
\subset \Hvec|_t \cong T_{\zeta(t)} \Haff|_t$. 
By the same construction, we can realize the universal cover 
of $\cMCY\setminus \DCY$ as an immersed Lagrangian 
submanifold in $\Haff|_{t_0}$. 

\subsection{Opposite Line Bundles and Propagators}

\begin{definition}[cf.~Definitions \ref{def:opposite}, \ref
{def:opposite_log-cTEP}] 
Let $U$ be an open subset of $\cMCY$.  
An \emph{opposite line bundle} over $U$ 
is a one-dimensional subbundle $P$ of $\Hvec|_U$ such that
\begin{itemize} 
\item[(1)] $P$ is flat, i.e.~$\nabla \cO(P) 
\subset \Omega^1(\log \DCY) \otimes \cO(P)$;
\item[(2)]  
for each $t \in U$, 
we have $P_t \oplus F^2_\vec|_t = \Hvec|_t$.
\end{itemize}
\end{definition}

\begin{definition}[Flat connection $\Nabla^P$, 
cf.~Definition \ref{def:Nabla}] 
\label{def:Nabla_fd}
Let $P$ be an opposite line bundle over $U\subset \cMCY$. 
Let $\Pi \colon \Hvec \to \Hvec/P \cong F^2_\vec$ denote 
the projection along $P$. The flat connection $\nabla$ on $\Hvec$ 
induces a flat connection on $F^2_\vec$ over $U$: 
\begin{equation} 
\label{eq:conn_P_fd} 
\cO(F^2_\vec) \xrightarrow{\nabla} \Omega^1(\log \DCY) \otimes 
\cO(\Hvec) \xrightarrow{\id\otimes \Pi} 
\Omega^1(\log \DCY) \otimes \cO(F^2_\vec) 
\end{equation} 
Composing this with the Kodaira-Spencer isomorphism 
$\Theta(\log \{0\}) \cong \cO(F^2_\vec)$, we obtain 
a logarithmic flat connection on $U\subset \cMCY$: 
\[
\Nabla^P \colon \Theta(\log \{0\}) \to 
\Omega^1(\log \DCY) \otimes \Theta(\log \{0\}) 
\]
We denote the dual connection 
\[
\Nabla^P \colon \Omega^1(\log \{0\}) \to 
\Omega^1(\log \DCY) \otimes \Omega^1(\log \{0\}) 
\]
by the same symbol. 
\end{definition}
\begin{remark} 
\label{rem:pole_NablaP} 
The description of $\Hvec$ in \S\ref{sec:two} shows that
the residue endomorphism $N$ of $\nabla$ at a point $t_0\in \DCY$ 
is nilpotent. 
Monodromy invariance then forces that an opposite 
line bundle $P$ around $t_0$ is unique and has 
$\Image N$ as the fiber at $t_0$; such a line bundle 
will be denoted by $P_\LR$ for $t_0=0$ and 
by $P_\con$ for $t_0=-\frac{1}{27}$ -- see Notation \ref{nota:opposite_fd} below. 
Therefore the connection \eqref{eq:conn_P_fd} has no logarithmic 
singularities along $\DCY$ for such $P$, i.e.~gives a map
\[
\cO(F^2_\vec) \to \Omega^1 \otimes \cO(F^2_\vec) 
\] 
Consequently, the connection $\Nabla^P$ gives a map:  
\[
\Nabla^P \colon \Omega^1(\log \{0\}) \to 
\Omega^1 \otimes \Omega^1(\log\{0\}) 
\subset 
\Omega^1(\log\{0\})^{\otimes 2}. 
\]
In particular, a flat co-ordinate associated with $\Nabla^{P_\con}$ 
is \emph{holomorphic} at the conifold point, whereas 
a flat co-ordinate for $\Nabla^{P_\LR}$ is \emph{logarithmic} 
at the large-radius limit point. 
Note however that the connection \eqref{eq:conn_P_fd} 
can have poles along $\DCY$ if we do not require the opposite line bundle $P$ to be flat 
(see \S\ref{subsec:curved} for curved opposite line bundles). 
\end{remark} 

Recall from the previous section 
that a neighbourhood of $t_0 \in \cMCY \setminus \DCY$ 
can be embedded into $\Haff|_{t_0}$ as a Lagrangian submanifold 
$\cL$. 
Choose affine Darboux co-ordinates $(p,x)$ on $\Haff|_{t_0}$ 
such that $\partial/\partial p$ is parallel to $P_{t_0}$ and 
that $\Omega = \frac{1}{3} dp \wedge dx$. 
The fact that $P$ is an opposite line bundle  implies that 
$P_{t_0}=\langle \partial/\partial p\rangle $ 
is transversal to the tangent space 
$T_{\zeta(t)} \cL = F^2_\vec|_t$ 
(note that $P_t$ is independent of $t$ when transported to 
the fiber $\Hvec|_{t_0}$). 
Therefore $\cL$ can be written as the graph of 
a function 
(see Figure \ref{fig:flat_coordinate_on_fd_Lagrangian}), $p = p(x)$.  
We may regard a function $\cF_{\rm B}^0(x)$ satisfying 
\[
p(x) = 3\parfrac{\cF_{\rm B}^0}{x} 
\]
as a ``genus-zero potential'' for the B-model; this depends on the choice of $P$. 
The co-ordinate $x$ restricted to $\cL
\subset \Haff|_{t_0}$ 
defines an affine flat co-ordinate with respect to $\Nabla^P$.
More invariantly, the affine flat structure 
is given by the projection along the linear foliation $P_{t_0}$: 
\[
\cL \subset \Haff|_{t_0} \longrightarrow 
\Haff|_{t_0}/P_{t_0}  
\]
\begin{figure}[htb] 
\includegraphics[bb=200 610 400 710]{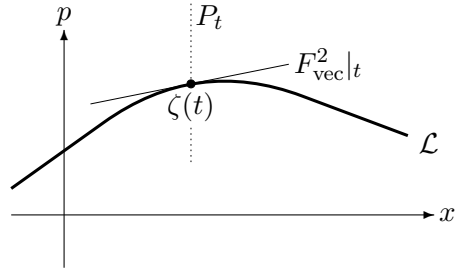}
\caption{Writing $\cL \subset \Haff|_{t_0}$ 
as a graph of $p=p(x)$: Darboux co-ordinates $(p,x)$ are chosen 
so that $P_{t_0} = \langle \partial/\partial p \rangle$;
the co-ordinate $x$ on $\cL$ 
then defines the affine flat structure associated to $P$.}
\label{fig:flat_coordinate_on_fd_Lagrangian} 
\end{figure} 

\begin{example}[The Yukawa coupling in flat co-ordinates]
\label{ex:Yukawa}
Let $P$ be an opposite line bundle in a neighbourhood of 
$t_0 \in \cMCY \setminus \DCY$. 
As above, 
choose affine Darboux co-ordinates $(p,x)$ of $\Haff|_{t_0}$ 
such that $P_{t_0} = \langle \partial/\partial p\rangle$ 
and that $\Omega = \frac{1}{3} dp \wedge dx$, 
and write $\cL$ as the graph 
of a function $p=p(x)$.  
We set:  
\begin{equation} 
\label{eq:tau} 
\tau = \parfrac{p}{x} 
\end{equation} 
Then the Yukawa coupling is given by:
  \begin{align*}
Y_{\CY} & =    Y_{\CY}(\partial_x,\partial_x,\partial_x) dx^{\otimes 3}  = 
    \Omega(\nabla_{\partial_x} \nabla_{\partial_x} \zeta,
    \nabla_{\partial_x} \zeta) dx^{\otimes 3} \\
    &= \Omega\left(
      \begin{pmatrix}
        \frac{\partial \tau}{\partial x} \\ 0
      \end{pmatrix},
      \begin{pmatrix}
        \tau \\ 1 
      \end{pmatrix}
    \right) dx^{\otimes 3} 
    = \frac{1}{3} \frac{\partial \tau}{\partial x} dx^{\otimes 3} 
  \end{align*}
\end{example}

\begin{remark}
This calculation shows that the Yukawa coupling is 
given by the 3rd derivative of a generating function 
$\cF_{\rm B}^0$ for $\cL$. 
Note that the Yukawa coupling is independent of the 
choice of $P$ whereas $\cF_{\rm B}^0$ depends on $P$. 
For Givental's  infinite-dimensional Lagrangian cone, 
this is explained in 
\citelist{\cite{Coates--Ruan}*{\S6.1}\cite{Givental:homological}*{\S6}}. 
\end{remark}

\begin{proposition} 
\label{pro:existence_opposite_fd}
For every $t\in \cMCY$, there exists an opposite line bundle $P$ 
in a neighbourhood of $t$. 
\end{proposition} 

\begin{proof} 
When $t \in \cMCY$ is not equal to the large-radius, conifold, or
orbifold points, one can construct an opposite line bundle $P$ on a
neighbourhood $U$ of $t$ by choosing a one-dimensional subspace
$P_t$ of $\Hvec|_t$ that is opposite to $F^2_\vec$, and then
extending $P_t$ to a line bundle over $U$ by parallel translation. 

When $t$ is equal to the large radius, conifold, or orbifold points, 
we have a canonical choice for an opposite line bundle near $t$. 
As shown in Proposition \ref{pro:LR_conifold_orbifold}, 
there are unique opposite modules $\bP_\LR$, $\bP_\con$, $\bP_\orb$ 
for $\cFB^\times$ near the large radius, 
conifold and orbifold points, which are 
compatible with the Deligne extension $\cFB$. 
These opposite modules induce, via the correspondence in 
Proposition \ref{pro:correspondence_of_opposites}, 
opposite line bundles over neighbourhoods of 
the large radius, conifold and orbifold points respectively. 
\end{proof} 

\begin{notation} 
\label{nota:opposite_fd} 
We write $P_\LR$, $P_\con$, $P_\orb$ for the opposite line bundles 
in a neighbourhood of the large-radius, conifold and orbifold points 
(respectively) discussed in the proof above. 
The fibres of 
$P_\LR$, $P_\con$, $P_\orb$ at the respective limit points are 
(see Proposition \ref{pro:opposite_cusps}): 
\begin{align*} 
P_\LR & \big |_{y_1=0}  = \langle \theta^2 \zeta\rangle, 
\\
P_\con & \big|_{y_1=-\frac{1}{27}} 
= \langle (1+27 y_1) \theta^2 \zeta \rangle, 
\\ 
P_\orb &\big|_{y_1=\infty}  = \langle z^{-1} \frd_1^2 \rangle 
= \langle \fry_1^{-2} \theta(\theta+\tfrac{1}{3}) \zeta \rangle 
\end{align*}  
These are unique opposite line bundles, respectively, around 
the large-radius, conifold and orbifold points: see Remark~\ref{rem:pole_NablaP} 
and Proposition~\ref{prop:opposite_cusps_flat}. 
\end{notation} 

\begin{definition}[cf.~Definition \ref{def:propagator}] 
\label{def:propagator_fd}
  Let $P_1$,~$P_2$ be opposite line bundles over $U$, and let $\Pi_i \colon
  \Hvec \to F^2_\vec$ be the projection along $P_i$, $i \in \{1,2\}$.
  The \emph{propagator} $\Delta = \Delta(P_1,P_2)$ is the logarithmic
  bivector field $\Delta \in \big(\Theta(\log \{0\})\big)^{\otimes 2}$
  defined by:
  \[
  \Delta := (\KS \otimes \KS)^{-1} (\Pi_1 \otimes \Pi_2) \Omega^\vee
  \]
  where $\Omega^\vee \in \Hvec\otimes \Hvec$ 
is the dual symplectic form on $\Hvec^\vee$.
\end{definition}

\begin{example}[The propagator in flat co-ordinates, 
cf.~{\cite[Lemma 5.22]{Coates--Iritani:Fock}}] 
\label{ex:propagator}
Let $P_1$, $P_2$ be opposite line bundles over a 
neighbourhood $U$ of $t_0 \in \cMCY \setminus \DCY$. 
We embed $U$ into $\Haff|_{t_0}$ as a 
Lagrangian curve $\cL$ as above. 
Let $(p,x)$ and $(p',x')$ denote affine Darboux co-ordinates
on $\Haff|_{t_0}$ associated to $P_1$ and $P_2$ respectively 
as in Example \ref{ex:Yukawa}, 
so that 
\[
P_1|_{t_0} = \left\langle \parfrac{}{p} \right\rangle, \quad 
P_2|_{t_0} = \left\langle \parfrac{}{p'} \right\rangle, \quad 
\Omega = \frac{1}{3} dp \wedge d x = \frac{1}{3} 
dp' \wedge dx'. 
\] 
Then $x$ and $x'$ restricted to $\cL$ 
give flat co-ordinates for $P_1$ and $P_2$ respectively.  
If 
\begin{align*} 
\begin{aligned} 
p' &= a p + b x + e \\ 
x' &= c p + d x + f 
\end{aligned} 
&& 
\text{with 
$\begin{pmatrix} a & b \\ c & d \end{pmatrix} 
\in \SL(2,\C)$} 
\end{align*} 
is the affine symplectic co-ordinate change between $(p,x)$ and $(p',x')$, 
the slope parameters \eqref{eq:tau} of $\cL$ are related by 
\[
\tau' = \frac{a\tau+b}{c\tau+d} 
\]
and the flat vector fields $\partial_x$, $\partial_{x'}$ 
on $\cL\cong U$ are related by 
\begin{equation} 
\label{eq:relation_flat_vector_fields}
\partial_x = \parfrac{x'}{x} \partial_{x'} = (c \tau + d) \partial_{x'}. 
\end{equation} 
Let $\Pi_i \colon \Hvec \to F^2_\vec$ denote the projection along $P_i$ 
for $i=1,2$. Then we have 
\begin{align*}
& \KS^{-1} \Pi_1(\partial_p) = 0 & &  
\KS^{-1} \Pi_2(\partial_p) = \KS^{-1} \Pi_2(a \partial_{p'} 
+ c \partial_{x'} ) = c \partial_{x'} = 
\frac{c}{c\tau+d} \partial_x \\ 
& \KS^{-1} \Pi_1(\partial_x) = \partial_x && 
\KS^{-1} \Pi_2(\partial_x) = 
\KS^{-1} \Pi_2(b\partial_{p'} + d \partial{x'}) = 
d \partial_{x'} = 
\frac{d}{c\tau+d}\partial_x 
\end{align*} 
and the propagator is: 
\begin{align*}
\Delta(P_1,P_2) & = (\KS \otimes \KS)^{-1} (\Pi_1 \otimes \Pi_2)
(3 \partial_p \otimes \partial_x - 3 \partial_x \otimes \partial_p) \\
&= -\frac{3c}{c\tau+d} 
\partial_x\otimes \partial_x. 
\end{align*}
\end{example} 

\begin{lemma}[Propagator calculus, cf.~{\cite[Proposition 4.45]{Coates--Iritani:Fock}}]
  \label{lem:formulas_for_propagator}
Let $P_1$,~$P_2$ be opposite line bundles over $U$, let $t_0 \in U$,
and let $x$ be a flat co-ordinate on $U$ corresponding to $P_1$.  
Write the propagator $\Delta(P_1,P_2)$ as $\Delta(x) \partial_x
  \otimes \partial_x$, and write the Yukawa coupling as $Y_{\CY}(x) \, dx
  \otimes dx \otimes dx$.  Then:
  \begin{align*}
\Nabla^{P_2} - \Nabla^{P_1} & = \Delta(P_1,P_2)\cdot Y_{\CY} 
= \Delta(x) Y_{\CY}(x) dx \\  
\frac{\partial \Delta}{\partial x} & = \Delta(x)^2 Y_{\CY}(x)
  \end{align*} 
where we regard $\Nabla^{P_1}$, $\Nabla^{P_2}$ as connections
on $\Omega^1(\log \{0\})$. 
\end{lemma}
\begin{proof}
Choose co-ordinates as in Examples~\ref{ex:Yukawa}, 
\ref{ex:propagator} so that, with notation as there, 
\[
Y_{\CY}(x) = \frac{1}{3} \parfrac{\tau}{x}, 
\qquad 
\Delta(x) = -\frac{3c}{c\tau+d} .
\]
Let us denote derivatives with respect to $x$ by subscripts. 
Recalling that $x$ and $x'$ are flat co-ordinates for $P_1$ and $P_2$ 
respectively, and using \eqref{eq:relation_flat_vector_fields}, 
we have: 
\begin{align*}
\Nabla^{P_1} (dx) & = 0 \\  
\Nabla^{P_2}(dx) & = \Nabla^{P_2} \left( \frac{dx'}{c\tau+d} \right) 
= - \frac{c \tau_x}{(c\tau+d)^2} dx \otimes dx' 
= - \frac{c\tau_x}{c\tau +d} dx\otimes dx. 
\end{align*} 
Hence 
\[
\Nabla^{P_2}(dx) - \Nabla^{P_1} (dx) = \Delta(x) Y_{\CY}(x) dx \otimes dx 
\]
and:
  \[
  \frac{\partial \Delta}{\partial x} = 
  \frac{3 c^2 \tau_x}{(c\tau+d)^2}
  =
  \Delta(x)^2 Y_{\CY}(x)
  \]
as claimed.
\end{proof}

\subsection{The Fock Sheaf}
\label{sec:Fock_fd}
We describe the Fock sheaf in the finite-dimensional setting. 
The construction is almost parallel to \S\ref{sec:Fock_log}. 

\begin{definition}
\label{def:local_Fock_space_fd}
Let $U$ be an open subset of $\cMCY$ and let $P$ be an opposite line
bundle over $U$. 
First suppose that $U$ does not contain the conifold point
$y_1= -\frac{1}{27}$. 
The \emph{local Fock space} $\Fock_{\CY}(U;P)$ consists
of collections:
\[
  \big \{ 
  \Nabla^n C^{(g)} \in 
\big(\Omega_U^1(\log \{0\}) \big)^{\otimes n}: 
  \text{$g \geq 0$,~$n \geq 0$,~$2g-2+n>0$} 
  \big\}
\]
of completely symmetric logarithmic $n$-tensors on $U$ such that:
  \begin{itemize}
  \item (Yukawa) $\Nabla^3 C^{(0)}$ is the Yukawa coupling $Y_{\CY}$;
  \item (Jetness) $\Nabla(\Nabla^n C^{(g)}) = \Nabla^{n+1} C^{(g)}$.
  \end{itemize}
Here $\Nabla = \Nabla^P$ is the flat connection on 
$\Omega^1_U(\log\{0\})$ defined by $P$, 
extended to logarithmic $n$-tensors in the 
obvious way; cf.~ the discussion in \S\ref{sec:Fock_log}. 
When $U$ contains the conifold point, we define $\Fock_{\CY}(U;P)$ similarly 
except that we allow $\Nabla^n C^{(g)}$ 
to have poles of order at most $2g-2+n$ at the conifold point, and 
impose the same conditions (Yukawa) and (Jetness). 
\end{definition}
\begin{remark}
The Yukawa coupling $Y_{\CY} = \Nabla^3 C^{(0)}$ has a pole of order 1 
at the conifold point -- see Remark \ref{rem:pole_Yukawa} -- and 
thus satisfies the last condition in the definition.  
\end{remark} 

Let $t$ be a co-ordinate on $U$.  If the point $t=0$ is the large-radius 
limit point then write $u = \log t$, $du = \frac{dt}{t}$, and $\partial_u =
t \frac{\partial}{\partial t}$; otherwise write $u = t$, $du = dt$,
$\partial_u = \frac{\partial}{\partial t}$.  Then, as in the
infinite-dimensional case, if:
\[
\Nabla^n C^{(g)} = C^{(g)}_n du^{\otimes n}
\]
then we refer to the tensors $\Nabla^n C^{(g)}$ or the functions
$C^{(g)}_n$ as \emph{$n$-point correlation functions}.  We again
encode elements of the local Fock space $\Fock_{\CY}(U;P)$ as formal
functions on the logarithmic tangent bundle.  Let $v$ be the fiber
co-ordinate on the logarithmic tangent bundle $\Theta_U(\log\{0\})$ that
is dual to $\partial_u$, so that $(u,v)$ denotes a point in the total
space of $\Theta_U(\log \{0\})$.  

\begin{definition}[jet potential]
\label{def:jetpotential_fd}
Given an element $\wave = \{\Nabla^n C^{(g)}\}_{g,n}$ of 
$\Fock_{\CY}(U;P)$,
set:
\begin{align*}
  \cW^g(u,v) =
  \sum_{n=\max(0,3-2g)}^\infty 
  \frac{C^{(g)}_n(u)}{n!} v^n && \text{and} &&
  \cW(u,v) = \sum_{g=0}^\infty \hbar^{g-1} \cW^g(u,v)
\end{align*}
We call $\cW^g$ the \emph{genus-$g$ jet potential} and $\exp(\cW)$ the
\emph{total jet potential} associated to $\wave$.
\end{definition} 

\begin{remark}
  We regard $\cW^{(g)}(u,v)$ as a formal function on the total space
  of the logarithmic tangent bundle.  As in the infinite-dimensional
  case, $\exp(\cW)$ is well-defined as a power series in $\hbar$ and
  $\hbar^{-1}$.
\end{remark}

\begin{definition}[transformation rule]
  \label{def:transformation_rule_fd}
  Let $P_1$ and $P_2$ be opposite line bundles over $U$, and consider
  the propagator $\Delta(P_1,P_2) = \Delta(u) \partial_u
  \otimes \partial_u$.  The \emph{transformation rule} 
\[
T(P_1,
  P_2)\colon \Fock_{\CY}(U;P_1) \to \Fock_{\CY}(U;P_2)
  \]
assigns to a jet potential
  $\cW$ for an element $\wave \in \Fock_{\CY}(U,P_1)$, the jet potential
  $\hcW$ for an element $\widehat{\wave} \in 
\Fock_{\CY}(U,P_2)$ given by: 
  \[
  \exp\big(\hcW(u,v)\big) =  
  \exp\left(\frac{\hbar}{2}\Delta(u) \parfrac{^2}{v^2}
    \right) 
  \exp\big(\cW(u,v)\big)
  \]
\end{definition}

Suppose that $\wave = \{C^{(g)}_n  du^{\otimes n} 
: \text{$g \geq 0$,~$n \geq 0$,~$2g-2+n>0$}\}$ 
are the correlation functions for $\cW$ and that
$\widehat{\wave} = \{\hC^{(g)}_n du^{\otimes n} : 
\text{$g \geq 0$,~$n \geq  0$,~$2g-2+n>0$}\}$ 
are the correlation functions for $\hcW$.  
The transformation rule in Definition~\ref{def:transformation_rule_fd} is
equivalent to the Feynman rule:
\begin{equation}
  \label{eq:Feynman_rule_fd}
  \hC^{(g)}_n du^{\otimes n} 
  = \sum_\Gamma \frac{1}{|\Aut(\Gamma)|} 
  \Cont_{\Gamma}(\Delta, \{C^{(h)}\}_{h\le g})
\end{equation}
where the summation is over all connected decorated graphs $\Gamma$
such that:
\begin{itemize}
\item 
To each vertex $v\in V(\Gamma)$ is 
assigned a non-negative integer $g_v\ge 0$, called genus; 
  
\item 
$\Gamma$ has $n$~labelled legs: an isomorphism 
$L(\Gamma) \cong \{1,2,\dots,n\}$ is given; 

\item $\Gamma$ is stable, i.e.\ 
$2 g_v - 2 + n_v>0$ for every vertex $v$.  
Here $n_v = |\pi_V^{-1}(v)|$ denotes the number of edges or 
legs incident to $v$; 

\item $g = \sum_{v} g_v + 1- \chi(\Gamma)$. 
\end{itemize}
(See Appendix~\ref{sec:graph_notation} for our notation for
graphs.)\phantom{.} 
We put the correlation function $C^{(g_v)}_{n_v} du^{\otimes n_v}$ 
on the vertex $v$ 
and put the propagator $\Delta(P_1,P_2)$ on every edge.  
Then $\Cont_\Gamma(\Delta,\{C^{(h)}\}_{h\le g})$ is defined to be the 
contraction of all these tensors; the result is an $n$-tensor with the
$n$~tensor indices corresponding to the $n$~labelled legs\footnote
{Since the base $\cMCY$ is one-dimensional, there is only one 
kind of tensor indices; the labelling of legs still plays a role 
in reducing automorphisms of $\Gamma$. }. 
As before,
$\Aut(\Gamma)$ denotes the automorphism group of the decorated graph
$\Gamma$.

\begin{proposition}
\label{prop:transf_rule_well_defined}
  The transformation rule is well-defined.  In other words, if:
  \begin{align*}
    & \wave = \{ C^{(g)}_n du^{\otimes n}: \text{$g \geq 0$,~$n \geq
      0$,~$2g-2+n>0$}\} 
    \intertext{is an element of $\Fock_{\CY}(U,P_1)$ and:}
    & \widehat{\wave} = \{\hC^{(g)}_n du^{\otimes n}: 
\text{$g \geq 0$,~$n \geq
      0$,~$2g-2+n>0$}\}
  \end{align*}
  is defined by the Feynman rule \eqref{eq:Feynman_rule_fd} then
  $\widehat{\wave} \in \Fock_{\CY}(U,P_2)$.
\end{proposition}

\begin{proof}
First note that if $U$ contains either large-radius or conifold points, 
there is a unique opposite line bundle over $U$ 
-- see Remark \ref{rem:pole_NablaP}.
Therefore the transformation rule is trivial and there is 
nothing to prove. In particular, we do not need to discuss 
the `pole order $2g-2+n$' condition at the conifold point. 
(When we consider curved opposite line bundles, however, 
this condition matters: see \S \ref{subsec:curved}.) 

We need to show that $\widehat{\wave}$ satisfies the properties
  (Yukawa) and (Jetness) in Definition~\ref{def:local_Fock_space_fd}.
  (Yukawa) is obvious, as there is exactly one stable 3-valent graph
  with $g=0$ and so $\hC^{(0)}_3 = C^{(0)}_3$.  
To establish (Jetness), we shall differentiate the right-hand side of
\eqref{eq:Feynman_rule_fd} with respect to $\Nabla^{P_2}$ 
and check if it coincides with the Feynman rule for 
$\hC^{(g)}_{n+1} du^{\otimes (n+1)}$. 

As discussed, it suffices to check (Jetness) away from $\DCY$. 
Therefore we may choose the co-ordinate $u$ to be a flat co-ordinate $x$ 
associated with $P_1$ 
and use notation as in Lemma~\ref{lem:formulas_for_propagator}. 
Since $\Nabla^{P_2} = \Nabla^{P_1} + \Delta(x) Y_{\CY}(x) dx$ by 
Lemma \ref{lem:formulas_for_propagator}, 
$\Nabla^{P_2}$ applied to \eqref{eq:Feynman_rule_fd} 
yields a sum over stable Feynman graphs
  as above, but with either:
  \begin{itemize}
  \item[($a$)] a distinguished vertex that carries $\Nabla^{P_1} \big(
    C^{(g_v)}_n dx^{\otimes n} \big) 
= C^{(g_v)}_{n_v+1} dx^{\otimes (n_v+1)}$; or
  \item[($b$)] a distinguished edge which carries $d \Delta(x) 
= Y_{\CY}(x) \Delta(x)^2 dx$; or 
   \item[($c$)] a distinguished leg that carries $Y_{\CY}(x) \Delta(x) dx$ in place 
   of $dx$. 
  \end{itemize}
The first possibility here arises from differentiating a vertex term in
\eqref{eq:Feynman_rule_fd}; the second possibility arises from
differentiating an edge term; and the third possibility arises from 
the difference of $\Nabla^{P_1}$ and $\Nabla^{P_2}$ 
-- recall how $\Nabla$ acts on $n$-tensors from Equation~\eqref{eq:covariant_derivative_tensor}. 
Note that we have used (Jetness) for $\wave$ in (a) and 
Lemma~\ref{lem:formulas_for_propagator} in (b). 
Observe that these are precisely the contributions appearing 
in the Feynman sum for $\hC^{(g)}_{n+1} dx^{\otimes (n+1)}$; 
in fact ($a$)--($c$) correspond respectively to Feynman 
graphs such that 
\begin{itemize} 
\item[($a'$)] the leg labelled by $n+1$ is on a vertex $v$ such that 
$2g_v - 2 + n_v >1$; 
\item[($b'$)] the leg labelled by $n+1$ is on a genus-zero vertex 
with 1 leg and 2 adjacent edges;   
\item[($c'$)] the leg labelled by $n+1$ is on a genus-zero vertex 
with 2 legs and 1 adjacent edge. 
\end{itemize}  
The proposition follows. 
\end{proof}

We now show that the transformation rule satisfies the cocycle
condition.

\begin{proposition}
\label{prop:propagator_cocycle} 
  Let $P_1$,~$P_2$, and~$P_3$ be opposite line bundles over $U$, and
  let $\Delta_{ij} = \Delta(P_i,P_j)$ be the corresponding
  propagators.  We have:
  \[
  \Delta_{13} = \Delta_{12} + \Delta_{23}
  \]
  In particular, $\Delta_{12} = {-\Delta_{21}}$.
\end{proposition}

\begin{proof}
  Let $\Pi_i \colon \Hvec \to F^2_\vec$ be the projection along $P_i$.
  Then, for any sections $\omega$,~$\omega'$ of $(F^2_\vec)^\vee$, we
  have:
  \begin{align*}
    \Delta_{13}(\KS^\star \omega, \KS^\star \omega') 
    &= \Omega^\vee\big( (\Pi_1^\star - \Pi_3^\star) \omega,
    \Pi_3^\star \omega' \big) \\
    &= \Omega^\vee\big( (\Pi_1^\star - \Pi_2^\star) \omega,
    \Pi_3^\star \omega' \big) +
    \Omega^\vee\big( (\Pi_2^\star - \Pi_3^\star) \omega,
    \Pi_3^\star \omega' \big) \\
    &= \Omega^\vee\big( (\Pi_1^\star - \Pi_2^\star) \omega,
    \Pi_2^\star \omega' \big) +
    \Delta_{23}(\KS^\star \omega, \KS^\star \omega') \\
    &= \Delta_{13}(\KS^\star \omega, \KS^\star \omega')  +
    \Delta_{23}(\KS^\star \omega, \KS^\star \omega') 
  \end{align*}
  For the first equality here we used the fact that $\Image
  \Pi_3^\star = P_3^\perp$ is isotropic; for the third equality we
  used the fact that $\Image (\Pi_2^\star - \Pi_3^\star)$ and 
  $\Image(\Pi_1^\star - \Pi_2^\star)$ are contained in the 
  isotropic subspace $(F^2_\vec)^\perp$.
\end{proof}

\begin{corollary}
  The transformation rule
  (Definition~\ref{def:transformation_rule_fd}) satisfies the cocycle
  condition: if $P_1$,~$P_2$, and~$P_3$ are opposite line bundles over
  $U$ then:
  \[
  T(P_1,P_3) = T(P_2,P_3) \circ T(P_1,P_2)
  \]
\end{corollary}

Thus the following Definition makes sense.

\begin{definition}[Fock sheaf]
\label{def:Fock_CY}
From Proposition~\ref{pro:existence_opposite_fd} 
we know that there is an open covering
$\{U_a : a \in A\}$ of $\cMCY$ such that for each $a \in A$ there
exists an opposite line bundle $P_a$ over $U_a$.  
The \emph{Fock sheaf} $\Fock_{\CY}$ is defined 
to be the sheaf of sets over $\cMCY$ obtained by 
gluing the local Fock spaces $\Fock_{\CY}(U_a;P_a)$, $a \in A$, 
using the transformation rule
\begin{align*}
 T(P_a,P_b) : \Fock_{\CY}(U_a \cap U_b;P_a) \to 
\Fock_{\CY}(U_a \cap U_b;P_b)
    && a, b \in A
\end{align*}
over $U_a \cap U_b$.
\end{definition}

\begin{remark}
  The Feynman rule \eqref{eq:Feynman_rule_fd} coincides with that used
  by Aganagic--Bouchard--Klemm \cite{ABK}*{\S2}.  It arises there
  through stationary phase approximation of certain integral operators
  acting on wave functions, which suggests a possible non-perturbative
  extension of the quantization formalism.  Our approach here
  emphasizes rigorous mathematical constructions, but in doing so
  hides this possible link to a non-perturbative theory.
\end{remark}

\subsection{Curved Opposite Line Bundles} 
\label{subsec:curved} 
We discuss a generalization of the previous framework 
to possibly curved (i.e.~not necessarily parallel) opposite line bundles. 
Of particular interest to us are the complex conjugate line bundle 
and the algebraic opposite line bundles which will be introduced 
later in \S\ref{subsec:alg_cc}. 
A general theory for curved opposite modules in the infinite-dimensional setting 
was developed in \cite[\S 4.13 and \S 9]{Coates--Iritani:Fock}, 
and the discussion here is parallel to that. 

\begin{definition} 
Let $U\subset \cMCY$ be an open set. 
A \emph{possibly curved opposite line bundle} over $U$ 
is a topological (or $C^\infty$) line subbundle $P$ of $\Hvec|_U$ 
such that $\Hvec|_U = F^2_\vec|_U \oplus P$. 
\end{definition} 

For possibly curved opposite line bundles $P_1,P_2$, 
the propagator 
\[
\Delta(P_1,P_2) := (\KS^{-1}\otimes \KS^{-1})  
(\Pi_1 \times \Pi_2)_* \Omega^\vee 
\]
is still well-defined as a \emph{continuous} (or $C^\infty$) section of 
$(\Theta(\log \{0\}))^{\otimes 2}$ 
(see Definition \ref{def:propagator_fd}). 
Let $P_0$ be an opposite line bundle and suppose 
that an element of the local Fock space for $P_0$ 
\[
\wave=\{C^{(g)}_n du^{\otimes n} 
: 2g-2 + n >0\} \in \Fock_{\CY}(U;P_0) 
\] 
is given. For a possibly curved opposite line bundle $P$ over $U$, 
we define \emph{genus-$g$, $n$-point correlation functions} 
$\hC^{(g)}_n du^{\otimes n}$ \emph{with respect to $P$} 
by the same Feynman rule as before 
\begin{equation} 
\label{eq:Feynman_P0_P}
\hC^{(g)}_n du^{\otimes n} = \sum_{\Gamma} 
\frac{1}{|\Aut(\Gamma)|} 
\Cont_\Gamma( \Delta(P_0,P), \{C^{(h)}\}_{h\le g}) 
\end{equation} 
where $\Gamma$ ranges over all connected, decorated, 
genus-$g$ stable graphs (see the list of conditions below 
equation~\eqref{eq:Feynman_rule_fd}). 
Note that $\hC^{(0)}_3 du^{\otimes 3} = C^{(0)}_3 du^{\otimes 3}$ 
is the Yukawa coupling. 

\begin{lemma} 
\label{lem:regularity_curved} 
If $U$ does not contain the conifold point, 
then the correlation functions \eqref{eq:Feynman_P0_P} 
are continuous sections of $\Omega^1(\log\{0\})^{\otimes n}$. 
If $U$ contains the conifold point, $(1+27 y)^{2g-2+n}  \hC^{(g)}_n 
du^{\otimes n}$ extends continuously across the conifold point.  
\end{lemma} 
\begin{proof}
The former statement is obvious from the definition. 
The latter statement follows from the condition 
that $C^{(g)}_n du^{\otimes n}$ has a pole of order $2g-2+n$ 
at the conifold point (see Definition~\ref{def:local_Fock_space_fd}), 
and the fact that the ``Euler number'' $2g-2+n$ is additive 
under graph contractions. 
\end{proof} 

\begin{remark} 
The Feynman rule involving curved opposite modules 
still satisfies the cocycle condition. This is because 
Proposition~\ref{prop:propagator_cocycle} and its proof 
are valid also for curved opposite line bundles. 
In particular, we can invert the 
Feynman rule \eqref{eq:Feynman_P0_P} to get 
\[
C^{(g)}_n du^{\otimes n} = 
\sum_{\Gamma} \frac{1}{|\Aut(\Gamma)|} 
\Cont_\Gamma( \Delta(P,P_0), \{\hC^{(h)}\}_{h\le g}).  
\]
\end{remark} 

The main difference from the parallel case is 
that correlation functions with respect to a possibly 
curved opposite line bundle 
\emph{do not satisfy (Jetness)} in general. 
In place of (Jetness), they satisfy certain anomaly equations. 
We assume henceforth that a possibly curved opposite line bundle $P$ 
is a $C^\infty$ subbundle of $\Hvec$. 
Also, for simplicity, we work over the locus 
$\cMCY \setminus \DCY$ where the connection $\nabla$ 
has no singularities.  
We can define a (not necessarily flat) connection $\Nabla^P$ associated 
with a curved $P$ by the same formula as in Definition~\ref{def:Nabla_fd}. 
Namely, the projection $\Pi \colon \Hvec \to F^2_\vec$ along $P$ 
induces a (not necessarily flat) connection 
\begin{align*} 
C^\infty(F^2_\vec) \xrightarrow{\nabla} 
C^\infty(T_\C^\vee \otimes \Hvec) 
\xrightarrow{\id \otimes \Pi} 
C^\infty(T_\C^\vee \otimes F^2_\vec) 
\end{align*} 
which in turn defines the connection: 
\begin{align*} 
\Nabla^P & \colon 
C^\infty(T^{1,0}) \to C^\infty(T_\C^\vee
\otimes T^{1,0}) \\ 
\intertext{and its dual:} 
\Nabla^P & \colon 
C^\infty((T^{1,0})^\vee) \to C^\infty(T_\C^\vee \otimes 
(T^{1,0})^\vee) 
\end{align*} 
via the Kodaira--Spencer isomorphism 
$F^2_\vec \cong T^{1,0}$.  Here $T_\C$ denotes the complexified tangent bundle 
of $\cMCY \setminus \DCY$, $T_\C^\vee$ its dual,
and $T_\C = T^{1,0}\oplus T^{0,1}$ the type decomposition. 
The connection $\Nabla^P$ can be naturally 
extended to $n$-tensors: 
\[
\Nabla^P \colon C^\infty((T^{1,0})^\vee\otimes 
\cdots \otimes (T^{1,0})^\vee) 
\to C^\infty(T_\C^\vee \otimes (T^{1,0})^\vee\otimes 
\cdots \otimes (T^{1,0})^\vee) 
\]
Note that the $(0,1)$-part of $\Nabla^P$ is the standard 
Dolbeault operator. 

\begin{definition}[torsion] 
\label{def:torsion} 
Let $P$ be a possibly curved opposite line bundle. The \emph{torsion} 
of $P$ is the $C^\infty$ tensor $\Lambda\colon 
(T^{1,0})^\vee \otimes (T^{1,0})^\vee \to T_\C^\vee$  
defined by 
\[
\Lambda(\omega_1,\omega_2) = \Omega^\vee(
\nabla^\vee \Pi^*(\KS^*)^{-1} \omega_1, 
\Pi^*(\KS^*)^{-1} \omega_2)  \qquad 
\omega_1,\omega_2 \in C^\infty((T^{1,0})^\vee)
\]
where $\nabla^\vee$ is the connection on $\Hvec^\vee$ 
dual to $\nabla$, $\Pi^* \colon (F^2_\vec)^\vee \to \Hvec^\vee$ 
is the dual of the projection $\Pi\colon \Hvec\to F^2_\vec$, and
$\KS^* \colon (F^2_\vec)^\vee 
\cong (T^{1,0})^\vee$ is the dual of the 
Kodaira--Spencer isomorphism.  
\end{definition} 

Note that $\Lambda$ vanishes if and only if $P$ is parallel for 
$\nabla$. We can generalize the propagator calculus 
in Lemma \ref{lem:formulas_for_propagator} as follows: 
\begin{lemma} 
\label{lem:propagator_calc_curved}
Let $P_0$ be a (parallel) opposite line bundle over $U$ and 
let $P$ be a possibly curved opposite line bundle 
over $U$. 
Let $x$ be a flat co-ordinate associated with $P_0$. Write the propagator $\Delta(P_0,P)$ as 
$\Delta(x) \partial_x \otimes \partial_x$,
the Yukawa coupling as 
$Y_{\CY}(x) dx\otimes dx\otimes dx$, and
the torsion of $P$ as $\Lambda= \Lambda(dx,dx) = 
\Lambda_x dx + \Lambda_{\ov{x}} d\ov{x} \in C^\infty(T_\C^\vee)$. 
Then: 
\begin{align*} 
\Nabla^P - \Nabla^{P_0} &= \Delta(P_0,P) \cdot Y_{\CY}  = \Delta(x) Y_{\CY}(x) dx \\ 
d \Delta(x) &= \Delta(x)^2 Y_{\CY}(x) dx - \Lambda_x dx - \Lambda_{\ov{x}} 
d\ov{x}. 
\end{align*} 
Moreover, the curvature of $\Nabla^P$ on the cotangent 
bundle $(T^{1,0})^\vee$ is  
$Y_{\CY}(x) \Lambda_{\ov{x}} \, dx \wedge d\ov{x}$. 
\end{lemma} 
\begin{proof} 
We use notation as in Examples~\ref{ex:Yukawa} and~\ref{ex:propagator}. 
Take a point $t_0\in U$ and embed a neighbourhood 
of $t_0$ as a Lagrangian curve $\cL 
\subset \Haff|_{t_0}$.  Fix  
Darboux co-ordinates $(p,x)$ on $\Haff|_{t_0}$ 
such that $P_0 = \langle \partial/\partial p\rangle$, 
$\Omega = \frac{1}{3} dp \wedge dx$, and that 
$x$ coincides with the given flat co-ordinate when restricted to 
$\cL$. 
In terms of the co-ordinates $(p,x)$, we have 
\[
\KS(\partial_x) 
= \tau \parfrac{}{p} + \parfrac{}{x} 
= \begin{pmatrix} \tau \\ 1 \end{pmatrix} 
\]
where $\tau$ is the slope parameter of $\cL$ in \eqref{eq:tau}. 
Suppose that the fiber $P_x$ at $x$ is written 
in the form 
(via parallel translation to $\Hvec|_{t_0}$):  
\[
P_x = \C \begin{pmatrix} c \\ 1 \end{pmatrix} 
\]
for a smooth function $c=c(x)$. 
Then a computation similar to Example~\ref{ex:propagator} gives: 
\[
\Delta(P_0,P) = \Delta(x) \partial_x \otimes \partial_x 
= - \frac{3}{\tau-c} \partial_x \otimes \partial_x.  
\]
In terms of the dual frame $\{dp,dx\}$ of $\Hvec^\vee|_{t_0}$, we have:
\[
\Pi^*(\KS^*)^{-1}(dx) = \frac{1}{\tau-c} (dp - c dx) = 
\frac{1}{\tau-c} (1,-c) 
\]
Hence:
\begin{equation} 
\label{eq:torsion} 
\Lambda = \Lambda(dx,dx) = 
\Omega^\vee\left( d \left[ \frac{1}{\tau-c} 
(1, -c) \right], 
\frac{1}{\tau-c} 
(1, -c) \right) 
= \frac{3}{ (\tau-c)^2} dc 
\end{equation} 
and: 
\[
\Nabla^P dx = \KS^*\left( \nabla^\vee  
\Pi^*(\KS^*)^{-1}(dx) |_{F^2_\vec} \right) 
= \KS^*\left[ \left. d 
\left( \frac{1}{\tau-c} (1, -c) \right)\right|_{F^2_\vec}\right] 
= - \frac{d\tau\otimes dx}{\tau-c}. 
\]
Therefore, using $Y_{\CY}(x) = \frac{1}{3} \tau_x$, we find: 
\[
\Nabla^P dx - \Nabla^{P_0} dx = -\frac{3}{\tau-c} Y_{\CY}(x) dx\otimes dx 
= \Delta(x) Y_{\CY}(x) dx\otimes dx  
\]
and: 
\[
d \Delta =\frac{3 d\tau}{(\tau -c)^2} 
- \frac{3 dc}{(\tau-c)^2}
= \Delta(x)^2 Y_{\CY}(x) dx - \Lambda. 
\]
Finally, the curvature of $\Nabla^P$ is given by: 
\[
\left(\Nabla^P\right)^2 (dx) 
= \Nabla^P\left(-\frac{\tau_x}{\tau-c} dx \otimes dx 
\right) 
= -\frac{\tau_x c_{\ov{x}}}{(\tau -c)^2} (d\ov{x}\wedge dx) \otimes dx
= Y_{\CY}(x)\Lambda_{\ov{x}}  (dx\wedge d\ov{x}) \otimes dx
\]
as claimed.  
\end{proof} 

Using Lemma \ref{lem:propagator_calc_curved}, we deduce the 
following anomaly equation. We omit the proof since the 
argument is very similar to that proving (Jetness)
in Proposition~\ref{prop:transf_rule_well_defined}. 

\begin{proposition}[anomaly equation, 
cf.~{\cite[Theorem 4.86]{Coates--Iritani:Fock}}]  
\label{prop:anomaly_equation} 
Let $P_0$ be a (parallel) opposite line bundle and 
$P$ be a possibly curved opposite line bundle. 
Let $x$ denote a flat co-ordinate associated with $P_0$ 
and let $\wave = 
\{ C^{(g)}_n dx^{\otimes n}\}_{2g-2+n>0}$ 
be an element of the local Fock space $\Fock_{\rm CY}(U;P_0)$. 
Let $\hC^{(g)}_n dx^{\otimes n}$ denote the genus-$g$, 
$n$-point correlation functions with respect to $P$  
produced from $\wave$ by the 
Feynman rule \eqref{eq:Feynman_P0_P}. 
Let $\Lambda= 
\Lambda(dx,dx) = \Lambda_x dx+ \Lambda_{\ov{x}} d \ov{x}$ 
denote the torsion of $P$. Then we have: 
\[
\hC^{(g)}_{n+1} dx^{\otimes (n+1)} 
=\Nabla^P (\hC^{(g)}_n dx^{\otimes n}) 
+ \frac{1}{2} 
\sum_{\substack{h+k = g \\ i+j=n}} 
\binom{n}{i} \Lambda \otimes 
\hC^{(h)}_{i+1} \hC^{(k)}_{j+1} dx^{\otimes n} 
+ \frac{1}{2} \Lambda \otimes \hC^{(g-1)}_{n+2} dx^{\otimes n}
\]
Equivalently:
\begin{align*} 
\hC^{(g)}_{n+1} 
& = \parfrac{\hC^{(g)}_n}{x} + n \Delta(x) Y_{\CY}(x) 
\hC^{(g)}_n 
+ \frac{1}{2} 
\sum_{\substack{h+k = g \\ i+j=n}} 
\binom{n}{i} \Lambda_x  
\hC^{(h)}_{i+1} \hC^{(k)}_{j+1} 
+ \frac{1}{2} \Lambda_x \hC^{(g-1)}_{n+2},  \\ 
0 & = \parfrac{\hC^{(g)}_n}{\ov{x}} + \frac{1}{2} 
\sum_{\substack{h+k = g \\ i+j=n}} 
\binom{n}{i} \Lambda_{\ov{x}} 
\hC^{(h)}_{i+1} \hC^{(k)}_{j+1}  
+ \frac{1}{2} \Lambda_{\ov{x}} \hC^{(g-1)}_{n+2} 
\end{align*} 
where we use notation from Lemma 
$\ref{lem:propagator_calc_curved}$. 
\end{proposition} 

\begin{remark} 
\label{rem:anomaly_equation} 
When we apply this in \S\ref{sec:calc}, 
we shall restrict to the following special cases: 
\begin{itemize} 
\item[(a)] $P$ is an anti-holomorphic line subbundle, i.e.~preserved by 
the $(1,0)$-part of $\nabla$ 
(e.g.~the complex conjugate opposite line bundle 
in Definition \ref{def:Pcc}); 

\item[(b)] $P$ is a holomorphic line subbundle which is not flat 
(e.g.~the algebraic opposite line bundle 
in Definition \ref{def:Palg}). 
\end{itemize} 
In the case (a), we have $\Lambda_x =0$. 
Therefore the correlation functions satisfy `partial' jetness 
$(\Nabla^P)^{1,0} (\hC^{(g)}_n dx^{\otimes n}) = 
\hC^{(g)}_{n+1} dx^{\otimes (n+1)}$ and 
the holomorphic anomaly equations: 
\begin{align*} 
0 & = \parfrac{\hC^{(1)}_1}{\ov{x}} + 
\frac{1}{2} \Lambda_{\ov{x}} Y_{\CY}(x)  && \text{for $g=1$} \\ 
0 & = \parfrac{\hC^{(g)}_0}{\ov{x}} + \frac{1}{2} 
\sum_{h=1}^{g-1} \Lambda_{\ov{x}} \hC^{(h)}_1 \hC^{(g-h)}_1 
+ \frac{1}{2} \Lambda_{\ov{x}} \hC^{(g-1)}_2  
&& \text{for $g \ge 2$} 
\end{align*} 
where $\hC^{(g)}_1 = \partial_x \hC^{(g)}_0$ for $g\ge 2$ 
and $\hC^{(g)}_2 = \partial_x \hC^{(g)}_1 + \Delta(x) Y_{\CY}(x) 
\hC^{(g)}_1$ for $g\ge 1$. 
Note that the holomorphic anomaly equation in genus 1 says that 
$\hC^{(1)}_1 dx$ behaves as a `connection 1-form' on 
the square root of the canonical bundle: 
\[
d(\hC^{(1)}_1 dx) = \frac{1}{2} \Lambda_{\ov{x}} Y_{\CY}(x) 
dx \wedge d\ov{x} = \frac{1}{2} (\Nabla^P)^2. 
\]
In the case (b), we have $\Lambda_{\ov{x}} = 0$. 
Thus the correlation functions $\hC^{(g)}_n$ are holomorphic, 
but do not satisfy (Jetness). 
\end{remark}

\section{The Conformal Limit of the Fock Sheaf}
\label{sec:conformal_limit_Fock}

Let $\cMCY^\circ$ denote the complement of the conifold locus:
\[
\cMCY^\circ := \cMCY \setminus \{{\textstyle -\frac{1}{27}}\}
\]
and let $\Fock_{\CY}^\circ$ denote the restriction to $\cMCY^\circ$ 
of the finite-dimensional Fock sheaf $\Fock_{\CY}$ 
from \S\ref{sec:Fock_sheaf_fd}. 
Recall from \S\ref{sec:Fock_sheaf} 
that the B-model Fock sheaf is defined on $\cMBbig$ and 
let $i\colon \cMCY^\circ \to \cMBbig$ denote the inclusion. 
In this section, we prove: 
\begin{theorem} \label{thm:Fock_restriction}
There exists a restriction map of Fock sheaves, $i^{-1} \Fock_{\rm B} \to 
\Fock_{\CY}^\circ$. 
\end{theorem} 

\noindent There are several things to understand:
\begin{itemize}
\item how correlation functions for $\Fock_{\rm B}$ give rise to correlation functions for $\Fock_{\CY}^\circ$ 
(\S\ref{sec:correlation_functions_conformal_limit});
\item how 
opposite modules for the big B-model log-cTEP structure that are compatible with the Deligne extension
give rise to opposite line bundles for $\Hvec$ 
(\S\ref{sec:opposites_conformal_limit}); 
\item how the transformation rule used to assemble $\Fock_{\rm B}$ 
out of local Fock spaces gives rise to the transformation rule used to 
assemble $\Fock_{\CY}^\circ$ out of local Fock spaces 
(\S\S\ref{sec:connections_conformal_limit}--\ref
{sec:propagators_conformal_limit}).
\end{itemize}
With this material in place, we define the restriction map $i^{-1} \Fock_{\rm B} \to \Fock_{\CY}^\circ$ in 
\S\ref{sec:restriction_conformal_limit}.  In \S\ref{sec:global_section_conformal_limit} we show that there is a global section $\wave_{\rm CY}$ of the finite-dimensional Fock sheaf $\Fock_{\CY}^\circ$, which arises via restriction from the global section $\wave_{\rm B}$ of $\Fock_{\rm B}$.  Near the large-radius limit point, correlation functions of $\wave_{\rm CY}$ encode Gromov--Witten invariants of the non-compact Calabi--Yau $3$-fold $Y = K_{\Proj^2}$ and near the orbifold point,  the correlation functions of $\wave_{\rm CY}$ encode Gromov--Witten invariants of the non-compact orbifold $\cX = \big[\C^3/\mu_3\big]$.  We will see in \S\ref{sec:conifold_estimate} that the genus-$g$, $n$-point correlation functions of $\wave_{\CY}$, which \emph{a priori} are holomorphic functions on $\cMCY^\circ$, are in fact meromorphic functions on $\cMCY$ with poles at the conifold point ${-\frac{1}{27}} \in \cMCY$ of order at most $2g-2+n$.  
Thus we can think of $\wave_{\CY}$ as a 
global section of the finite-dimensional Fock sheaf $\Fock_{\CY}$. 

\subsection{Correlation Functions in the Conformal Limit}
\label{sec:correlation_functions_conformal_limit}

Recall from Theorem \ref{thm:unfolding} 
that the big B-model \logDTEP structure $\cFBbig$ 
(and hence the big B-model $\log$-cTEP structure $\FBbig$) 
has logarithmic singularities along $\Dbig\subset \cMBbig$.  
Let $\LL$ denote the total space of the big B-model $\log$-cTEP 
structure and $\LLo$ denote its open subset as defined in 
Definition \ref{def:miniversal_cTEP}. 
Correlation functions for the B-model Fock sheaf $\Fock_{\rm B}$ 
are local sections $\Nabla^n C^{(g)}$ of 
$\bOmegao^1(\log \Dbig)^{\otimes n}$, 
where $\bOmegao^1(\log \Dbig)$ is the sheaf of one-forms 
on $\LLo$ logarithmic along $\pr^{-1} \Dbig$, 
satisfying the conditions (Yukawa), (Jetness), (Grading and Filtration), 
and (Pole). 
Correlation functions for the finite-dimensional Fock sheaf $\Fock_{\CY}$ 
are local sections $\Nabla^n C^{(g)}$ of 
$\Omega^1(\log \{0\})^{\otimes n}$, where $\Omega^1(\log \{0\})$ 
is the sheaf of one-forms on $\cMCY$ logarithmic at $y_1=0$, 
satisfying the conditions (Yukawa) and (Jetness).  
Roughly speaking, to relate $\Fock_{\rm B}$ to $\Fock_{\CY}$, 
we want to pull back correlation functions for $\Fock_{\rm B}$ 
along\footnote{The primitive section $\zeta$ lands in 
$\LLo \subset \LL$ because Reichelt's 
conditions (IC), (GC) hold along $\cMCY^\circ$: 
see \S\ref{sec:unfolding_locally}.  
}
the primitive section $\zeta \colon \cMCY^\circ \to \LLo$.  This requires care, because in general there is no canonical way to restrict logarithmic forms to the logarithmic locus.

\begin{example}
\label{ex:no_canonical_map}
Let $i \colon D \to \cM$ be the inclusion of a normal crossing divisor 
into a complex manifold.  
Then there is no canonical map $i^\star \Omega^1_{\cM} ( \log D ) 
\to \Omega^1_D$; 
indeed the canonical map goes in the other direction, 
and fits into an exact sequence
\[
\xymatrix{
  0 \ar[r] & \Omega^1_D \ar[r] & i^\star \Omega^1_{\cM} ( \log D ) \ar[r]^-{\text{res}} & \cO_D \ar[r] & 0 
}
\]
where the map $\text{res}$ takes the residue along $D$.
\end{example}

To pull back correlation functions for $\Fock_{\rm B}$, we first restrict to the image of $\zeta \colon \cMB^\circ \to \LLo$.  
Here there is a well-defined pullback, as $\zeta\big|_{\cMB^\circ}$ 
is transverse to the logarithmic locus in $\LLo$; 
over $\cMCY^\circ$, it defines a map: 
\[
\zeta^\star \bOmegao^1(\log \Dbig) \to \Omega^1_{\cMB^\circ}(\log D)\big|_{\cMCY^\circ}, 
\]
where here and hereafter $\zeta^\star$  
means the pull-back by $\zeta \colon \cMCY^\circ \to \LL^\circ$ 
and $D \subset \cMB$ is the divisor \eqref{eq:log_divisor}. 
Then we choose a splitting of
\begin{equation}
  \label{eq:splitting}
  \xymatrix{
    0 \ar[r] & \Omega^1_{\cMCY^\circ}(\log \{0\}) \ar[r] & \Omega^1_{\cMB^\circ}(\log D)\big|_{\cMCY^\circ} \ar[r] & \cO_{\cMCY^\circ} \ar[r] \ar@/_/@{-->}(75,3)& 0. 
  }
\end{equation}
Here the dashed arrow is multiplication by 
\[
du_2 = \frac{1}{3} \frac{dy_1}{y_1} + \frac{dy_2}{y_2} 
= \frac{d\fry_2}{\fry_2}. 
\] 
As we will see in \S\ref{sec:connections_conformal_limit}, 
in our situation this choice of splitting \emph{is} canonical.  
Combining, we get a restriction map 
\begin{equation}
\label{eq:restriction}
\zeta^\star \bOmegao^1(\log \Dbig) \longrightarrow 
\Omega^1_{\cMCY^\circ}(\log \{0\}) 
\end{equation}
as the composition
\[
\xymatrix{
  \zeta^\star \bOmegao^1(\log \Dbig) \ar[r] 
& \Omega^1_{\cMB^\circ}(\log D)\big|_{\cMCY^\circ} \ar[r] 
& \Omega^1_{\cMCY^\circ}(\log \{0\}) \ar[r]  
\oplus \cO_{\cMCY^\circ} \ar[r] 
& \Omega^1_{\cMCY^\circ}(\log \{0\}) 
}
\]
where the middle arrow is the splitting and the right-hand arrow is projection to the first factor.  

\begin{example} \label{ex:Yukawa_conformal_limit}
We can compute the B-model Yukawa coupling $\bY$ 
(Definition~\ref{def:Yukawa}) using Proposition~\ref{pro:GKZ}.  
Restricting the result to $\zeta(\cMB^\circ)$, which is possible because 
$\zeta\big|_{\cMB^\circ}$ is transverse to the logarithmic locus, yields
  \[
 \bY\big|_{\zeta(\cMB^\circ)} = { - \frac{1}{3(1+27y_1)}} \Big( \frac{dy_1}{y_1} \Big)^{\otimes 3} + 9 (du_2)^{\otimes 3}.
  \]
But the Yukawa coupling in the finite-dimensional setting 
(Definition~\ref{def:Yukawa_fd}) is 
\[
 Y_{\CY} = \Omega(\theta^2 \zeta, \theta \zeta) \, \Big( \frac{dy_1}{y_1} \Big)^{\otimes 3} = { - \frac{1}{3(1+27y_1)}} \Big( \frac{dy_1}{y_1} \Big)^{\otimes 3} 
  \]
-- see \eqref{eq:symplectic_pairing}.   Thus the restriction map \eqref{eq:restriction} takes the Yukawa coupling $\bY$ to $Y_{\CY}$.
\end{example}

\subsection{Opposite Modules in the Conformal Limit}
\label{sec:opposites_conformal_limit}
We now discuss how opposite modules 
for the big B-model log-cTEP structure that are compatible with the Deligne extension give rise 
to opposite line bundles in the conformal limit.  
This is largely a summary of material from 
\S\S\ref{sec:conformal_limit}--\ref{sec:enlarge_base}.  
In \S\ref{sec:conformal_limit}, we considered the restriction 
$\cFCY$ of the B-model log-TEP structure to $\cMCY \times \C$.  
This is a log-TEP structure with base $(\cMCY,\DCY)$, 
which carries an endomorphism $N\colon \cFCY \to z^{-1} \cFCY$ 
given by the residue of the B-model connection 
along the divisor $\cMCY \times \C \subset \cMB \times \C$.  
Taking the formalization\footnote
{Since $\cFCY$ is graded -- see \eqref{eq:restricted_GKZz} -- 
no information is lost by the formalization 
$(\cFCY,\nabla) \rightsquigarrow (\sFCY,\nabla)$.} 
 of $\cFCY$ at $z=0$ defines 
a log-cTEP structure $\sFCY$ with base $(\cMCY,\DCY)$, 
equipped with a residue endomorphism $N \colon \sFCY \to z^{-1} \sFCY$ 
induced by that on $\cFCY$.  
In \S\ref{sec:conformal_limit} we considered a six-dimensional vector bundle $H$, a three-dimensional vector bundle $\widebar{H}$, and a two-dimensional vector bundle $\Hvec$; these are related to the log-cTEP structure $\sFCY$ as follows:
\begin{equation}
  \label{eq:diagram_ambients}
  \begin{aligned}
    \xymatrix{
      \Hvec\  \ \ar@{^{(}->}[dr] && H \ar@{->>}[dl]\ar@{^{(}->}[dr] \\
      & \widebar{H} && \sFCY[z^{-1}]
    }
  \end{aligned}
\end{equation}
The vector bundle $H$ is included in $\sFCY[z^{-1}]$ as the degree-$1$ part; it is preserved by the action of the residue endomorphism $N$, so we can regard $N$ as an endomorphism of $H$.  There is a canonical surjection from $H$ onto $\widebar{H} = H/\Image N$, and $\Hvec = (\Ker N)/(\Image N \cap \Ker N)$ sits canonically as a subbundle of $\widebar{H}$.  The diagram \eqref{eq:diagram_ambients} induces the following diagram of Hodge subbundles:
\begin{equation}
  \label{eq:diagram_Hodge}
  \begin{aligned}
    \xymatrix{
      F^2_{\vec}\  \ \ar@{^{(}->}[dr] && F^2 \ar@{->>}[dl]\ar@{^{(}->}[dr] \\
      & \widebar{F}^2 && \sFCY  
    }
  \end{aligned}
\end{equation}
Here $F^2 \subset H$ is three-dimensional, $\widebar{F}^2 \subset \widebar{H}$ is two-dimensional, and $F^2_{\vec} \subset \Hvec$ is one-dimensional.  These were introduced in \S\ref{sec:six} and \S\ref{sec:opposite_filtrations}; we give explicit bases for them below.  
Let $U$ be an open neighbourhood of $y\in \cMCY^\circ$ in 
$\cMBbig$. 
Let $\bP$ be an opposite module for 
the big B-model \logDTEP structure $(\cFBbig,\nablaB,\pairingB)$ 
over $U\setminus \Dbig$ 
such that $\bP$ is compatible with the Deligne extension 
$(\cFBbig,\nablaB,\pairingB)|_U$ 
in the sense of Definition \ref{def:compatible_with_Deligne}. 
The opposite module $\bP$ naturally yields an opposite module $\sfP$ 
for the big B-model $\log$-cTEP structure $(\FBbig,\nablaB,\pairingB)$ 
over $U$. 
By a slight abuse of language, we call such a $\sfP$ 
a \emph{Deligne-extension-compatible opposite module} for 
the big B-model $\log$-cTEP structure $\FBbig$. 
Let $\sfP_{\CY}$ denote the restriction $\sfP$ to $\cMCY^\circ \cap U$. 
Combining Propositions~\ref{pro:extending_trivialization},~\ref{pro:correspondence_of_opposites}
and~\ref{pro:opposite_cusps}, we find that 
$\bP$ (or $\sfP_{\CY}$) induces an opposite line bundle 
$P$ over $U \cap \cMCY^{\circ}$ 
and that \eqref{eq:diagram_ambients} induces the following 
diagram of opposite modules, filters, and line bundles:
\begin{equation}
  \label{eq:diagram_opposites}
  \begin{aligned}
    \xymatrix{
      P \ar@{=}[dr] && U_1 \ar@{->>}[dl]\ar@{^{(}->}[dr] \\
      & \widebar{U}_1 && \sfP_{\CY}
    }
  \end{aligned}
\end{equation}
Here $U_1$ is the degree-one part of $\sfP_{\CY}$, which is three-dimensional; $\widebar{U}_1$ is the image of $U_1$ under the projection to $\widebar{H}$, which is one-dimensional; and the opposite line bundle $P$ is equal to $\widebar{U}_1$.  
We have that $U_1 = \big \langle z^{-1} D_2^2 \big \rangle 
+ \{s \in \Ker N : [s] \in P\}$. 

Let us now give explicit bases for the bundles in \eqref{eq:diagram_Hodge} and \eqref{eq:diagram_opposites}, summarizing the discussion in \S\ref{sec:conformal_limit}.  
We have, in the manifold chart $\cMCY\setminus \{y_1=\infty\}$:
\begin{align*}
  \Ker N & = \big \langle D_1 - \textstyle \frac{1}{3}D_2, z^{-2} D_2^3, z^{-1} (1 + 27y_1) (D_1 - \textstyle \frac{1}{3}D_2)^2 \big \rangle \\
  \Image N & = \big \langle D_2, z^{-1} D_2^2, z^{-2} D_2^3 \big \rangle \\
  \intertext{Furthermore:}
  F^2 & = \big\langle z, D_1 - \textstyle \frac{1}{3}D_2, D_2 \big\rangle \\
  \widebar{F}^2 & = \big\langle [z], [D_1 - \textstyle \frac{1}{3}D_2] \big\rangle = \big\langle [z], [D_1] \big \rangle = \big \langle \zeta, \theta \zeta \big \rangle \\
  F^2_{\vec} & = \big\langle [D_1 - \textstyle \frac{1}{3}D_2] \big\rangle = \big\langle [D_1] \big \rangle = \big \langle \theta \zeta \big \rangle \\
  \intertext{and:}
  U_1 & =  \big \langle z^{-1} D_2^2, z^{-2} D_2^3, z^{-1} (1 + 27y_1) (D_1 - \textstyle \frac{1}{3}D_2)^2 + a (D_1 - \textstyle \frac{1}{3}D_2) \big \rangle \\
  P & =  \big \langle (1 + 27y_1) \theta^2 \zeta + a \theta \zeta \big \rangle
\end{align*}
where $a$ is a scalar-valued function of $y_1$ that parameterizes the opposite filter $U_1$ or line bundle $P$.

A key observation is that the surjection $\Ker N \surj \Hvec$ 
induces an isomorphism 
$F^2 \cap \Ker N \xrightarrow{\sim}  F^2_{\vec}$.   
That is, the residue endomorphism $N$ singles out a canonical lift of 
$F^2_\vec$ to $H$. This is also true near the orbifold point. 
As we now explain, it is this that makes our choice of splitting in 
\eqref{eq:splitting}  canonical.   
Note that the Kodaira--Spencer map (see Definition~\ref{def:KS}) 
gives an isomorphism $\zeta^\star \bThetao(\log \Dbig) \xrightarrow{\sim}
\zeta^\star \pr^\star \FBbig = \sFCY|_{\cMCY^\circ}$, 
and consider the diagram
\begin{equation} 
\label{eq:canonical_splitting} 
\begin{aligned}
\xymatrix{
  \zeta^\star \bThetao(\log \Dbig) \ar[rr]_-{\sim}^-{\KS} && \sFCY\big|_{\cMCY^\circ} \\
  \Theta_{\cMB^\circ}(\log D)\big|_{\cMCY^\circ} \ar@{^{(}->}^-{\KS}[rr] \ar@{^{(}->}[u] && F^2\big|_{\cMCY^\circ} \ar@{^{(}->}[u] \ar@{=}[r] &  \big \langle z, D_1 - {\textstyle \frac{1}{3} D_2}, D_2 \big \rangle\\
  && \big(F^2 \cap \Ker N\big)\big|_{\cMCY^\circ} \ar@{^{(}->}[u] \ar@{=}[r] &  \big \langle D_1 - {\textstyle \frac{1}{3} D_2} \big \rangle\\
  \Theta_{\cMCY^\circ}(\log \{0\}) \ar[rr]_-{\sim}^-{\KS} \ar@{-->}[uu] && F^2_{\vec}\big|_{\cMCY^\circ} \ar^-{\rotatebox{270}{$\sim$}}[u] \ar@{=}[r] &  \big \langle [ D_1 - {\textstyle \frac{1}{3} D_2} ] \big \rangle \\
}
\end{aligned} 
\end{equation} 
where the lower-right vertical isomorphism is the canonical lift of $F^2_\vec$ to $H$ and $\KS$ denotes the Kodaira--Spencer map.  There is a unique choice for the dashed arrow that makes the diagram commute: the bottom horizontal map takes $y_1 \frac{\partial}{\partial y_1}$ to $[D_1] = [D_1 - {\textstyle \frac{1}{3} D_2} ]$, and so the dashed map must take $y_1 \frac{\partial}{\partial y_1}$ to $y_1 \frac{\partial}{\partial y_1} - {\textstyle \frac{1}{3}} y_2 \frac{\partial}{\partial y_2}$.  Thus our choice of splitting in \eqref{eq:splitting} is the unique choice such that this diagram commutes.  Dualizing gives:

\begin{lemma}
  \label{lem:splitting_is_canonical}
The restriction map \eqref{eq:restriction} is the unique map that makes the following diagram commute:
  \[
  \xymatrix{
    \zeta^\star \bOmegao^1(\log \Dbig) \ar@{-->}[dd] && 
\sFCY^\vee \big|_{\cMCY^\circ}
\ar_-{\KS^\star}^-{\sim}[ll] \ar[d] \\
    && \big(F^2 \cap \Ker N \big)^\vee \big|_{\cMCY^\circ} \ar^{\rotatebox{90}{$\sim$}}[d]\\
    \Omega^1_{\cMCY^\circ}(\log \{0\})  
&& \big(F^2_{\vec}\big)^\vee\big|_{\cMCY^\circ} \ar_-{\KS^\star}^-{\sim}[ll]
  }
  \]
\end{lemma}

\subsection{Connections in the Conformal Limit}
\label{sec:connections_conformal_limit}

In this section we will show that the restriction map \eqref{eq:restriction} 
sends the connection (Definition \ref{def:Nabla})
\begin{align}
  \Nabla^{\sfP} \colon \bOmegao^1(\log \Dbig) & \longrightarrow \bOmegao^1(\log \Dbig) \otimes \bOmegao^1(\log \Dbig) \label{eq:first_connection} \\
  \intertext{to the connection (Definition \ref{def:Nabla_fd})}
  \Nabla^{P} \colon \Omega^1_{\cMCY^\circ}(\log \{0\}) 
& \longrightarrow \Omega^1_{\cMCY^\circ}(\log \{0\}) \otimes 
\Omega^1_{\cMCY^\circ}(\log \{0\}) \label{eq:second_connection}
\end{align}
where the opposite line bundle $P$ is induced by the 
Deligne-extension-compatible opposite module $\sfP$ 
as in \S\ref{sec:opposites_conformal_limit}. 
More precisely, these connections are defined on open sets 
where $\sfP$ or $P$ are defined, but we shall omit 
the restriction signs to ease the notation. 
Note that it suffices to check the correspondence between the connections 
\eqref{eq:first_connection}, \eqref{eq:second_connection} 
on the manifold chart $\{y_1 \neq \infty\}$; we will work only with 
this chart.

\begin{remark}
  Since $\Nabla^{\sfP}$ is not $\cO_{\LLo}$-linear, it \emph{does not} induce a map from $\zeta^\star \bOmegao^1(\log \Dbig)$ to  $\zeta^\star \bOmegao^1(\log \Dbig)^{\otimes 2}$
\end{remark}

Since $\zeta|_{\cMB^\circ}$ is transverse to the logarithmic locus, 
we can pull back the connection \eqref{eq:first_connection} 
to get a connection
\[
\bOmegao^1(\log \Dbig)\big|_{\zeta(\cMB^\circ)}  
\longrightarrow \Omega^1_{\cMB^\circ}(\log D) \otimes 
\left( \bOmegao^1(\log \Dbig)\big|_{\zeta(\cMB^\circ)}  \right) 
\]
Restricting to $\cMCY^\circ$ gives
\[
\zeta^\star \bOmegao^1(\log \Dbig) \longrightarrow 
\Big(\Omega^1_{\cMB^\circ}(\log D) \Big|_{\cMCY^\circ}\Big)
\otimes 
\zeta^\star \bOmegao^1(\log \Dbig) 
\]
and using the splitting \eqref{eq:splitting} gives a connection 
\[
\nabla' \colon \zeta^\star \bOmegao^1(\log \Dbig) \longrightarrow  
\Omega^1_{\cMCY^\circ}(\log \{0\}) \otimes 
\zeta^\star \bOmegao^1(\log \Dbig). 
\]
Explicitly:
\[
\nabla' \alpha = \frac{dy_1}{y_1} \otimes 
\Big( \Nabla^{\sfP}_{\big(y_1 \frac{\partial}{\partial y_1} 
- \frac{1}{3} y_2 \frac{\partial}{\partial y_2}\big)} \alpha \Big). 
\]

Let us identify $\zeta^\star \bOmegao^1(\log \Dbig)$ with $\sFCY^\vee$ 
using the Kodaira--Spencer map, so that
\[
\nabla' \colon \sFCY^\vee \longrightarrow 
\Omega^1_{\cMCY^\circ}(\log \{0\}) \otimes \sFCY^\vee 
\]
We need to show that $\nabla'$ induces a connection on $(F^2_\vec)^\vee$ via the map $\sFCY^\vee \to (F^2 \cap \Ker N)^\vee \cong (F^2_\vec)^\vee$ -- see Lemma~\ref{lem:splitting_is_canonical} -- and that this induced connection coincides, via the Kodaira--Spencer map, with $\Nabla^P$.  To see this, consider the dual connection
\[
\nabla' \colon \sFCY \longrightarrow 
\Omega^1_{\cMCY^\circ}(\log \{0\}) \otimes \sFCY
\]
and compute:
\begin{align*}
  \nabla' (D_1 - {\textstyle \frac{1}{3}} D_2) 
&= \frac{dy_1}{y_1}\otimes 
\Nabla^{\sfP}_{\big(y_1 \frac{\partial}{\partial y_1} - \frac{1}{3} y_2 \frac{\partial}{\partial y_2}\big)} (D_1 - {\textstyle \frac{1}{3}} D_2) 
 \\
  &= \frac{dy_1}{y_1}\otimes 
\Pi_{\sfP} \Big(z^{-1} (D_1 - {\textstyle \frac{1}{3}} D_2)^2 \Big)   \\
  &= - \frac{dy_1}{y_1}\otimes 
{\frac{a}{1+27y_1}} (D_1 - {\textstyle \frac{1}{3}} D_2) 
\end{align*}
where $\Pi_\sfP \colon \sFCY[z^{-1}] \to \sFCY$ is the projection 
along $\sfP$. 
Here we used the fact that $\Pi_{\sfP}$ on $H$ is the same as projection  $H \to F^2$ along $U_1$, together with the explicit bases from \S\ref{sec:opposites_conformal_limit}.  Thus $\nabla'$ preserves 
$F^2 \cap \Ker N$, and so induces a connection on $F^2_\vec$.  
It remains to show that this induced connection is $\Nabla^P$.  
But this is obvious:
\begin{align*}
  \Nabla^P \big( \theta \zeta \big) &= 
\frac{dy_1}{y_1}\otimes \Pi_P (\theta^2 \zeta) \\
  &= -\frac{dy_1}{y_1}\otimes  \frac{a}{1+27y_1} \big( \theta \zeta \big)  
\end{align*}
where we again used the explicit bases in \S\ref{sec:opposites_conformal_limit}.  So under the identification $F^2_\vec \xrightarrow{\sim}  F^2 \cap \Ker N$, which sends $\theta \zeta$ to $D_1 - {\textstyle \frac{1}{3}} D_2$, $\nabla'$ coincides with $\Nabla^P$.  Thus we have shown that the restriction map \eqref{eq:restriction} sends the connection \eqref{eq:first_connection} to the connection \eqref{eq:second_connection}.

\subsection{The Propagators Agree in the Conformal Limit}
\label{sec:propagators_conformal_limit}
In this section, we prove: 
\begin{proposition} 
\label{pro:propagator_reduction} 
Let $\sfP_1$,~$\sfP_2$ be Deligne-extension-compatible 
opposite modules for the big B-model $\log$-cTEP structure 
$(\FBbig,\nablaB,\pairingB)$. 
Let $P_1$,~$P_2$ be the corresponding opposite line bundles.  
The pull-back by $\zeta \colon \cMCY^\circ \to \LLo$ 
of the propagator in the infinite-dimensional 
setting $($Definition $\ref{def:propagator})$ 
\[
\zeta^\star \Delta(\sfP_1,\sfP_2) 
\in \sHom(\zeta^\star \bOmegao(\log \Dbig)^{\otimes 2}, 
\cO_{\cMCY^\circ}) 
\]
is induced from the propagator in the finite-dimensional setting 
$(\text{Definition~}\ref{def:propagator_fd})$ 
\[
\Delta(P_1,P_2) \in \Theta_{\cMCY^\circ} (\log \{0\})^{\otimes 2} 
= 
\sHom((\Omega^1_{\cMCY^\circ}(\log \{0\}))^{\otimes 2}, 
\cO_{\cMCY^\circ})
\]
via the restriction map \eqref{eq:restriction}.
\end{proposition}  

The $\log$-cTEP structure $(\sFCY, \nabla, (\cdot,\cdot))$ 
with base $(\cMCY,\DCY)$ 
carries an $\cO_{\cMCY}$-linear grading operator:
\[
\Gr(P) = \big[ \textstyle z\frac{\partial}{\partial z}, P \big]
\]
This is the grading inherited from the GKZ system 
(Definition~\ref{def:GKZ_grading}).
It is a shift of the grading operator $\gr$ 
inherited from the big B-model $\log$-cTEP structure 
so that $\Gr = \gr + \frac{3}{2}$: see Example \ref{ex:Gr_gr}.  
The $\cO_{\cMCY}[\![z]\!]$-module $\sFCY$ decomposes as:
\[
\sFCY = \prod_{i=0}^\infty \sFCY^{(i)}
\]
where $\sFCY^{(i)}$ is the sub-bundle of degree~$i$ with respect to $\Gr$.  We have:
\[
\sFCY^{(i)} = 
\begin{cases}
  \big\langle 1 \big\rangle & i=0 \\
  \big\langle z, D_2, D_2-3D_1 \big\rangle & i=1 \\
  \big\langle z^2, z D_2, z(D_2-3D_1), D_2^2, D_1 (D_2-3D_1) \big\rangle & i=2 \\
  \big\langle z^3, z^2 D_2, z^2(D_2-3D_1), z D_2^2, z D_1 (D_2-3D_1), D_2^3 \big\rangle & i=3 \\
  z^{i-3} \sFCY^{(3)} & i \geq 4
\end{cases}
\]
Note that $\sFCY^{(1)} = F^2 \subset H$. 
Recall from Definition \ref{def:propagator} that the propagator 
$\Delta(\sfP_1,\sfP_2)$ 
is induced from the tensor 
$V\in \FBbig \hotimes \FBbig = 
\sHom(\FB^{{\rm big}\vee} \otimes \FB^{{\rm big} \vee}, 
\cO_{\cMBbig})$: 
\[
V(\varphi_1,\varphi_2) 
:= \Omega^\vee(\Pi_1^\star\varphi_1, 
\Pi_2^\star\varphi_2)
\]
as $\Delta(\sfP_1,\sfP_2) = (\KS \otimes \KS) \pr^\star(V)$, 
where $\Pi_i \colon \FBbig[z^{-1}] \to \FBbig$ is the projection 
along $\sfP_i$. 
\begin{proposition}
\label{pro:propagator_restricts}
Let $V_{\CY}$ denote the restriction of 
$V$ to $\cMCY^\circ$. 
Then we have: 
\[
V_{\CY} \in 
\left\langle (D_2 - 3D_1)^{\otimes 2} \right\rangle = 
 (F^2\cap \Ker N)^{\otimes 2} \subset \sFCY^{(1)} 
\otimes \sFCY^{(1)}. 
\]  
\end{proposition}

\begin{proof}
  Lemma~\ref{lem:propagator_grading} implies that 
$(\Gr \otimes 1 + 1 \otimes \Gr) V_{\CY} = 2 V_{\CY}$,
  and therefore that:
  \[
  V_{\CY} \in \left(\sFCY^{(0)} \otimes \sFCY^{(2)}\right) \oplus
  \left( \sFCY^{(1)} \otimes \sFCY^{(1)} \right) \oplus
  \left( \sFCY^{(2)} \otimes \sFCY^{(0)} \right)
  \]
  Let us write:
  \[
  V_{\CY} = 1 \otimes a_2 + 
  \sum \gamma_{ij} \phi_i \otimes \phi_j
  + a_2 \otimes 1
  \]
where $a_2 \in \sFCY^{(2)}$, $\gamma_{ij}$ is symmetric in $i$ and $j$,
and $(\phi_1,\phi_2,\phi_3) = (z,D_2,D_2-3D_1)$ 
is a basis for $\sFCY^{(1)}$. 
We claim that the following equation holds: 
\begin{equation}
    \label{eq:D2_and_propagator}
    \big( D_2 \otimes 1 - 1 \otimes D_2 \big) V_{\CY} = 0
\end{equation}
where $D_2 = -z N$ is the endomorphism of $\sfF_{\CY}$ 
(see \S \ref{sec:six}). 
To see this, note that, since $D_2=-zN$ preserves both $\sfF_{\CY}$ 
and $\sfP_i|_{\cMCY}$, we have that $D_2 \Pi_i = \Pi_i D_2$.  
Thus:
\begin{align*} 
\big( D_2 \otimes 1 - 1 \otimes D_2 \big) V_{\CY} 
& = \big( D_2 \otimes 1 - 1 \otimes D_2 \big) 
\big(\Pi_1 \otimes \Pi_2\big) \Omega^\vee \\
& = \big(\Pi_1 \otimes \Pi_2\big) 
\big( D_2 \otimes 1 - 1 \otimes D_2 \big) \Omega^\vee
\end{align*} 
which is zero as $D_2$ is self-adjoint with respect to $\Omega$.  
Writing out the graded pieces of \eqref{eq:D2_and_propagator} yields:
  \begin{align*}
    0 &= 1 \otimes D_2 a_2 && \text{$(0,3)$ component}\\
    0 &= D_2 \otimes a_2 - \sum 
\gamma_{ij} \phi_i \otimes D_2 \phi_j && \text{$(1,2)$ component} 
\end{align*}
The first equation shows that $a_2 \in \Ker D_2$.  
The second equation gives $\phi_i \ne D_2 \implies D_2 \phi_j = 0$, 
i.e.~$\phi_j=D_2-3D_1$,  
and thus $\gamma_{11}=\gamma_{12}=\gamma_{31}=
\gamma_{32} = 0$.  
Symmetry of $\gamma$ gives 
\[
\gamma = \begin{pmatrix} 0 & 0 & 0 \\ 0 & \gamma_{22} & 0 \\ 
0 & 0 & \gamma_{33} \end{pmatrix} 
\]
The second equation now becomes
$D_2\otimes a_2 = \gamma_{22} D_2\otimes D_2^2$,  
and thus $a_2 = \gamma_{22} D_2^2$. 
Since $D_2 a_2 = 0$, we conclude that $\gamma_{22}=0$ 
and $a_2=0$. 
Thus $V_{\CY} = \gamma_{33}(D_2-3D_1)
\otimes (D_2-3D_1)$ and the Proposition follows.
\end{proof}

\begin{proof}[Proof of Proposition~\ref{pro:propagator_reduction}] 
In view of the diagram \eqref{eq:canonical_splitting} and 
Proposition~\ref{pro:propagator_restricts}, 
it suffices to show that the element $V_{\rm fd} \in 
F^2_\vec\otimes F^2_\vec$ defined by 
\[
V_{\rm fd}(\varphi_1,\varphi_2) := 
\Omega^\vee(\pi_1^\star \varphi_1,\pi_2^\star \varphi_2) 
\]
coincides with $V_{\CY}$ under the identification 
$F^2_\vec \cong F^2 \cap \Ker N$,  
where $\pi_i \colon \Hvec \to F^2_\vec$ is the projection 
along $P_i$ and $\varphi_i \in (F^2_\vec)^\vee$. 
Take $\varphi \in \sFCY^\vee$ and 
choose $v_i \in \sfP_i|_{\cMCY^\circ}$ 
such that $\varphi = \Omega(v_i,\cdot)$ 
for $i=1,2$; 
then we have 
\[
\iota_\varphi V_{\CY} = v_1 - v_2 \in \sFCY. 
\]
We know that $v_1-v_2$ lies in $F^2\cap \Ker N$ 
by Proposition \ref{pro:propagator_restricts}. 
On the other hand, let $\ov{\varphi}\in 
(F_2\cap \Ker N)^\vee \cong (F^2_\vec)^\vee$ 
be the image of $\varphi$ 
and let $w_i\in H$ be the degree-1 part of $v_i$. 
Then we have $\ov{\varphi}= \Omega(w_i,\cdot)$ on 
$F^2_\vec\cap \Ker N$. 
By the correspondence between $\sfP_i$ and $P_i$ 
in \S\ref{sec:opposites_conformal_limit}, 
the image $[w_i]$ of $w_i$ in $\widebar{H}=\Cok N$ lies in $P_i$ 
and thus: 
\[
\iota_{\ov{\varphi}} V_{\rm fd} = [w_1] - [w_2] 
\in F^2_\vec \subset \widebar{H}. 
\]
Since $v_1-v_2$ is of degree 1, 
we have $v_1-v_2 = w_1- w_2$. Thus $v_1-v_2$ corresponds 
to $[w_1]-[w_2]$ under the isomorphism 
$F^2 \cap \Ker N \cong F^2_\vec$. 
The conclusion follows.  
\end{proof} 

\subsection{The Restriction Map on Fock Sheaves}
\label{sec:restriction_conformal_limit}

As discussed, correlation functions for the B-model Fock sheaf $\Fock_{\rm B}$ are local sections $\Nabla^n C^{(g)}$ of $\bOmegao^1(\log \Dbig)^{\otimes n}$ satisfying the conditions (Yukawa), (Jetness), (Grading and Filtration), and (Pole). Applying the restriction map \eqref{eq:restriction} to such correlation functions $\{\Nabla^n C^{(g)}\}_{g,n}$ yields local sections of $\Omega^1_{\cMCY^\circ}(\log \{0\})^{\otimes n}$ which satisfy (Yukawa), by Example~\ref{ex:Yukawa_conformal_limit}, and (Jetness), by \S\ref{sec:connections_conformal_limit}.  To show that we get a restriction map on Fock sheaves
\begin{equation} 
\label{eq:restriction_Focksheaf}
i^{-1} \Fock_{\rm B} \to \Fock_{\CY}^\circ
\end{equation} 
it remains only to check that the restriction map takes the propagator for the big B-model Fock sheaf $\Fock_{\rm B}$ to the propagator for the finite-dimensional Fock sheaf $\Fock_{\CY}^\circ$.  This is the content of \S\ref{sec:propagators_conformal_limit}.  Theorem~\ref{thm:Fock_restriction} is proved.

\subsection{A Global Section of the Finite-Dimensional Fock Sheaf}
\label{sec:global_section_conformal_limit}

Applying the restriction map \eqref{eq:restriction_Focksheaf} to the global section $\wave_{\rm B}$ of $\Fock_{\rm B}$ (see Theorem \ref{thm:wave_B}) gives a global section 
$\wave_{\CY}$ of $\Fock_{\CY}^\circ$.  
We can compute the correlation functions of $\wave_{\CY}$ with respect to the opposite line bundle $P_\LR$ by applying the restriction map to the Gromov--Witten wave function $\wave_{\Ybar}$, which is a global section of 
$\Fock_{{\rm A}, \Ybar}$.  
With notation as in Definition~\ref{def:GW-wave}, the mirror map (see Theorem~\ref{thm:mirrorsymmetryforYbar}) gives:
\begin{align*}
  \begin{aligned}
      t^1= \log q_1 &= \log y_1 - 3 g(y_1) \\
      t^2 = \log q_2 &= \log y_2 + g(y_1)
  \end{aligned}
  && \text{where} && g(y_1) = \sum_{d=1}^\infty \frac{(3d-1)!}{(d!)^3} (-1)^{d+1} y_1^d.
\end{align*}
Thus $\frac{1}{3} d t^1 + dt^2 = \frac{1}{3} d \log y_1 + d\log y_2$, and the splitting \eqref{eq:splitting} in these co-ordinates 
is given by $du_2 = \frac{1}{3} dt^1 + dt^2$.  
Restricting to the image of $\zeta\big|_{\cMCY^\circ}$ sets 
\begin{align*}
t^0 = t^3 = t^4 = t^5 = 0 && q_2=0 && x^i_n =
 \begin{cases}
  {-1} & \text{if $n=1$ and $i=0$} \\
   0 & \text{otherwise.}
 \end{cases}
\end{align*} 
Let 
$\wave_{\Ybar} = \{\Nabla^n C^{(g)}_{\Ybar}\}_{g,n}$ 
denote the Gromov--Witten wave function of $\Ybar$. 
Then the restriction 
of $\Nabla^n C^{(g)}_{\Ybar}$ to $\zeta(\cMCY^{\circ})$ 
under the map \eqref{eq:restriction}
is given by: 
\[
\left. \left( \partial_{t^1} - {\textstyle \frac{1}{3}}\partial_{t^2}\right)^n 
F^{(g)}_{\Ybar}(t)
\right|_{Q_1=Q_2 =1, t^0=t^3=t^4=t^5=0, q_2 =0} (dt^1)^{\otimes n}  
\]
where $F^{(g)}_{\Ybar}$ is the Gromov--Witten potential 
\eqref{eq:GW_potential} of $\Ybar$. 
Writing 
\begin{align*}
  n_{g,d} = \corr{\vphantom{\big\vert}}^{\Ybar}_{g,n,(d,0)} = \corr{\vphantom{\big\vert}}^{Y}_{g,n,d}
  && d>0, 
\end{align*}
we have 
\begin{align}
\label{eq:GW-wave_Y}
\begin{aligned}  
\text{the restriction of $\Nabla^3 C^{(0)}_{\Ybar}$} 
& = 
\left( {- \frac{1}{3}} + \sum_{d=1}^\infty d^3 n_{0,d} 
\, q_1^d \right)\Big( \frac{dq_1}{q_1} \Big)^{\otimes 3} && \\ 
\text{the restriction of $\Nabla C^{(1)}_{\Ybar}$} 
& = \left( {- \frac{1}{12}} + \sum_{d=1}^\infty d n_{1,d} 
\, q_1^d \right) \frac{dq_1}{q_1} && \\ 
\text{the restriction of $C^{(g)}_{\Ybar}$} & = 
\sum_{d=0}^\infty n_{g,d} \, q_1^d && \text{for $g\ge 2$}
\end{aligned} 
\end{align} 
where $q_1 = e^{t^1}$ and we used the fact that 
$\corr{(h_1-\frac{1}{3} h_2)^3}_{1,3,0}^{\Ybar} = 
\int_{\Ybar} (h_1-\frac{1}{3} h_2)^3 = -\frac{1}{3}$ 
and 
$\corr{h_1-\frac{1}{3} h_2}^{\Ybar}_{1,1,0} 
= {-\frac{1}{24}} \int_{\Ybar} (h_1-\frac{1}{3} h_2) 
\cup c_2(\Ybar) = -\frac{1}{12}$.  
These are the correlation functions of $\wave_{\CY}$ with 
respect to the opposite line bundle $P_\LR$ and coincide with 
(the derivatives of) the Gromov--Witten potentials of $Y$. 

In a similar way, we can compute the correlation functions of $\wave_{\CY}$ with respect to the opposite line bundle $P_\orb$ by applying the restriction map \eqref{eq:restriction_Focksheaf} to the Gromov--Witten wave function $\wave_{\cXbar}$.
Let $\{\frt=t^4, \log q=t^1\}$ denote the co-ordinates on 
$H^2_\orb(\cXbar)$ dual to $\{\unit_{\frac{1}{3}}, h\}$ 
defined in \S\ref{sec:bases}. 
Recall that the mirror map in Theorem~\ref{thm:mirrorsymmetryforXbar} gives
\begin{align*}
\frt & = \sum_{n=0}^\infty (-1)^n 
\frac{\prod_{j=0}^{n-1}(\frac{1}{3}+j)^3}{(3n+1)!}
\fry_1^{3n+1}, && 
\log q  = 3 \log \fry_2.
\end{align*} 
Thus the splitting \eqref{eq:splitting} is given 
in these co-ordinates by $du_2 = \frac{1}{3}d\log q$. 
Writing 
\begin{align*} 
n^\orb_{g,k} = \corr{\unit_{\frac{1}{3}},\dots,
\unit_{\frac{1}{3}}}
^{\cXbar}_{g,k,0} = 
\corr{\unit_{\frac{1}{3}},\dots,\unit_{\frac{1}{3}}}^{\cX}_{g,k,0} 
&& \text{when $2g-2+k>0$} 
\end{align*} 
we have: 
\begin{align}
\label{eq:GW-wave_X}  
\begin{aligned} 
\text{the restriction of $\Nabla^3C^{(0)}_{\cXbar}$} &= 
\left(\sum_{k=0}^\infty n^\orb_{0,k+3} \frac{\frt^k}{k!}
\right)  (d\frt)^{\otimes 3} && \\
\text{the restriction of $\Nabla C^{(1)}_{\cXbar}$} & =  
\left(\sum_{k=0}^\infty n^\orb_{1,k+1} \frac{\frt^k}{k!} \right) 
d\frt  && \\ 
\text{the restriction of $C^{(g)}_{\cXbar}$} & = 
\sum_{k=0}^\infty n^\orb_{g,k} \frac{\frt^k}{k!} 
&& \text{for $g\ge 2$.}  
\end{aligned} 
\end{align} 
These are the correlation functions for $\wave_{\CY}$ 
with respect to the opposite line bundle $P_\orb$ 
and coincide with (the derivatives of) the Gromov--Witten potentials of $\cX$.  This proves:

\begin{theorem}[cf.~Theorem~\ref{thm:wave_B}] 
\label{thm:wave_CY} 
Let $\wave_{\CY}$ be the section of the Fock sheaf $\Fock_{\CY}^\circ$ 
over $\cMCY^{\circ}$ given as the restriction of  the global section
$\wave_{\rm B}\in \Fock_{\rm B}$ 
under the map \eqref{eq:restriction_Focksheaf}. 
Then: 
\begin{itemize} 
\item[(a)] around $y_1=0$ and 
with respect to the opposite line bundle $P_\LR$, the 
correlation functions of $\wave_{\CY}$ are given by 
the Gromov--Witten potential of $Y$ as in \eqref{eq:GW-wave_Y}; 

\item[(b)] around $y_1=\infty$ and 
with respect to the opposite line bundle $P_\orb$, the 
correlation functions of $\wave_{\CY}$ are given by 
the Gromov--Witten potential of $\cX$ as in \eqref{eq:GW-wave_X}. 
\end{itemize} 
\end{theorem} 

\begin{remark}[cf.~Remark~\ref{rem:CRC}] 
The existence of a global section $\wave_{\CY}$ 
with these properties establishes a 
higher-genus version of the Crepant Resolution Conjecture~\cite{Ruan:crepant,Bryan--Graber,CIT:wall-crossings,Coates--Ruan,Iritani:Ruan}
for the crepant resolution $Y \to \cX$. 
\end{remark}

\section{Estimates at the Conifold Point}
\label{sec:conifold_estimate}

Given an open set $U \subset \cMCY^\circ$ and an opposite line bundle $P$ over $U$,  correlation functions with respect to $P$ for the global section $\wave_{\CY}$ of $\Fock_{\CY}^\circ$ are holomorphic functions on $U$.  Recall that there is a unique opposite line bundle $P_\con$ near the conifold point ${-\frac{1}{27}} \in \cMCY$: see Notation~\ref{nota:opposite_fd}. 
  In this section we show that genus-$g$, $m$-point correlation functions for $\wave_{\CY}$ with respect to $P_\con$ extend meromorphically across the conifold point and have a pole of order at most $2g-2+m$ there. 
This shows that $\wave_{\CY}$ satisfies the conifold pole condition in 
Definition~\ref{def:local_Fock_space_fd}, and thus that $\wave_{\CY}$ extends to a global section of 
$\Fock_{\CY}$ over $\cMCY$.  This follows immediately from the corresponding statement about  $\wave_{\rm B}$:

\begin{theorem}
  \label{thm:conifold_estimate}
Let $\bP_\con$ denote the unique opposite module for $\cFBbig$ 
at the conifold point that is compatible with the Deligne extension 
$($see Proposition~$\ref{pro:LR_conifold_orbifold})$  
and let $\sfP_\con$ denote the corresponding opposite module for 
$\FBbig$ $($see Example~$\ref{ex:opposite_B_log-cTEP})$.   
Consider the pull-back of the genus-$g$, $m$-point correlation 
function for $\wave_{\rm B}$ with respect to $\sfP_\con$ along the primitive section $\zeta \colon 
\cMB^\circ \to \LLo$; 
this gives a local section of $\Omega^1_{\cMB^\circ}(\log D)^{\otimes m}$
which is defined in a neighbourhood of 
the conifold divisor $(y_1=-\frac{1}{27})$
but is not defined on the divisor itself. 
This extends meromorphically across the conifold divisor, and has 
a pole of order at most $2g-2+m$ there. 
\end{theorem}

In outline: this will follow from the Givental's higher-genus formula -- which we used to define the B-model global section $\wave_{\rm B}$, and which expresses each genus-$g$ correlation function as a finite sum over Feynman graphs -- together with an analysis of the stationary phase asymptotics of various oscillating integrals.  The stationary phase analysis will allow us to estimate the pole order of each ingredient of Givental's formula.

\subsection{Critical Points}
\label{sec:critical_points}

Consider the Landau--Ginzburg mirror $(\pi, W, \omega)$ from \S\ref{sec:LG_Ybar}, and identify the fiber of $\pi$ with $(\C^\times)^3$ by setting
\begin{align*}
  w_1 = x_1 x_3, && w_2 = x_2 x_3, && w_3 = \frac{y_1 x_3}{x_1 x_2}, && w_4 = x_3, && w_5 = \frac{y_2}{x_3}
\end{align*}
where $(x_1,x_2,x_3) \in (\C^\times)^3$.  Then the superpotential becomes
\[
W(x_1,x_2,x_3) 
= x_1 x_3 + x_2 x_3 + \frac{y_1 x_3}{x_1 x_2} + x_3 + \frac{y_2}{x_3}
\]
and there are six critical points:
\[
(x_1^c, x_2^c, x_3^c) = \left( \sqrt[3]{y_1}, \sqrt[3]{y_1}, \sqrt{\frac{y_2}{1 + 3 \sqrt[3]{y_1}}} \right)
\]
Writing $T = y_1 + \frac{1}{27}$ for the co-ordinate near the conifold point, we see that four of the critical points extend holomorphically across $T=0$ and the other two escape to infinity there.  The divergent critical points are those for which the critical value
\begin{equation} 
\label{eq:critical_values} 
W(x_1^c, x_2^c, x_3^c) = 2 \sqrt{y_2 (1+3 \sqrt[3]{y_1})}
\end{equation} 
approaches zero as $T \to 0$. 
We also note that $x_3^c = O(T^{-1/2})$ for a divergent $c$. 

Introduce logarithmic co-ordinates near a critical point $c$, 
setting 
\begin{align*}
x_1 = x_1^c \exp((x_3^c)^{-1/2}\theta_1), && 
x_2 = x_2^c \exp((x_3^c)^{-1/2} \theta_2), && 
x_3 = x_3^c \exp((x_3^c)^{1/2} \theta_3),
\end{align*} 
and writing
\[
W_{i j \cdots k}(c) = \left( 
\frac{\partial}{\partial \theta_i} \frac{\partial}{\partial \theta_j} \cdots \frac{\partial}{\partial \theta_k} W \right) (x_1^c, x_2^c, x_3^c)
\]
for the multiple logarithmic derivative of W at $c$.  
Then the logarithmic Hessian at $c$ satisfies:
\begin{align*}
  H_c & =
\begin{pmatrix}
     2 \sqrt[3]{y_1} & \sqrt[3]{y_1} & 0 \\
     \sqrt[3]{y_1} & 2\sqrt[3]{y_1} & 0 \\
     0 & 0 & 2 y_2
\end{pmatrix} 
&&& H_c^{-1} & = 
\begin{pmatrix}
 \frac{2}{3\sqrt[3]{y_1}} & -\frac{1}{3\sqrt[3]{y_1}} & 0 \\
 -\frac{1}{3\sqrt[3]{y_1}} & \frac{2}{3\sqrt[3]{y_1}} & 0 \\
0 & 0 & \frac{1}{2 y_2}
\end{pmatrix} 
\end{align*} 
and $\det(H_c) = 6 y_1^{2/3} y_2$. 
These quantities are holomorphic at $T=0$. 
For a divergent critical point $c$, and 
for $m$,~$n\ge 1$ and $l\ge 0$, we have: 
\begin{align*}
   & W_{\underbrace{\scriptstyle 1\cdots 1}_{m}\underbrace{\scriptstyle 3\cdots 3}_{l}}(c)
      = 
      W_{\underbrace{\scriptstyle 2 \cdots 2}_{m} 
      \underbrace{\scriptstyle 3\cdots 3}_{l}}(c) = 
  \begin{cases}
    0 & \text{$m$ odd} \\
    O(T^{\frac{m-l-2}{4}}) & \text{$m$ even}\\
  \end{cases}
  \\ 
    &  W_{\underbrace{\scriptstyle 3\cdots 3}_{m}}(c)
       =
        \begin{cases}
          0 & \text{$m$ odd} \\
          O(T^{\frac{-m+2}{4}}) & \text{$m$ even}\\
        \end{cases}
     \qquad \quad 
         W_{\underbrace{\scriptstyle 1\cdots 1}_{n}
\underbrace{\scriptstyle 2\cdots 2}_{m}
\underbrace{\scriptstyle 3\cdots 3}_{l}}(c)
     =
       O(T^{\frac{n+m-l-2}{4}})
\end{align*}
as $T\to 0$. 
At non-divergent critical points, the multiple logarithmic 
derivatives of $W$ are holomorphic at $T=0$. 
Altogether, we get 
\begin{equation} 
\label{eq:order_derivative_W} 
W_{i_1\dots i_k}(c) =
\begin{cases}  
O(T^{-\frac{k}{4}+\frac{1}{2}})  & \text{if $c$ is divergent;} \\ 
O(1) & \text{if $c$ is non-divergent.} 
\end{cases} 
\end{equation}
\subsection{Givental's Higher-Genus Formula}
\label{subsec:Givental_formula} 

Choose a point $t\in \cMB \setminus D\subset \cMBbig$.  The B-model 
$\log$-cTEP structure is tame semisimple at $t$ 
because $W$ has pairwise distinct eigenvalues. 
Correlation functions for the B-model wave function $\wave_{\rm B}$ 
with respect to $\sfP_\con$ are obtained 
by applying a certain quantized operator $\widehat{R}_t$ 
to the product of Kontsevich--Witten tau-functions
\begin{align*}
  \cT = \prod_{c} \tau(\bq^c) && \text{where $\bq^c = q^c_0 + 
q^c_1 z + q^c_2 z^2 + \cdots \in \C[\![z]\!]$.}
\end{align*}
Here $c$ ranges over critical points of $W$ and 
$R_t$ is an invertible $\C[\![z]\!]$-linear operator: 
\[
R_t \colon \prod_{c} \C [\![z]\!] \longrightarrow \FBbig|_t
\]
This is Givental's formula for the ancestor potentials of a semisimple Frobenius manifold. It is discussed, in a notation and framework convenient for our setting, in~\cite[\S\S3--4]{Coates--Iritani:convergence}; the original reference is~\cite{Givental:semisimple}.  
The operator $R_t$ here is a certain ``asymptotic fundamental solution'' 
for the connection $\nablaB$, whose existence near $t$ 
is guaranteed in general by~\cite[Proposition~1.1]{Givental:elliptic} 
and which in our setting we can obtain from a genuine fundamental 
solution matrix by taking stationary phase asymptotics. 

Recall from Proposition~\ref{pro:LR_conifold_orbifold} that 
the flat trivialization of $\cFB$ corresponding to $\bP_\con$ 
is given by the frame $1$, $D_2$, $D_2^2$, $D_2^3$,  $D_1$, 
$(1+27y_1) D_1^2$ at the conifold point $(y_1,y_2) = (-1/27, 0)$.  
Let $\cD_1,\ldots,\cD_6$ denote differential operators whose classes in the 
GKZ system give the flat trivialization associated with $\bP_\con$ 
and coincide with the above frame at the conifold point. 
Let~$c$ be a critical point, and let $\Gamma_+(c)$ 
denote the Lefschetz thimble given by upward 
gradient flow from $c$ of the function 
$x \mapsto \Re\big( {W(x) \over z}\big)$.  
Let $\{s_c\}$ denote the flat sections of $\cFB$ dual to 
the cycles $\{\Gamma_+(c)\}$ so that 
\begin{align*}
  \frac{1}{(2\pi z)^{3/2}}\int_{\Gamma_+(c)} s_{c'} = \delta_{c,c'} 
  && \text{(cf.~\S\ref{sec:B_model_TEP})}.
\end{align*}
Define a matrix $S_t = (s_{jc})$ with rows indexed by $j \in \{1,2,\ldots,6\}$ and columns indexed by critical points $c$ of $W$, 
by expressing the sections $s_c$ with respect to the frame $\cD_1,\ldots,\cD_6$:
\[
s_c = \sum_{j=1}^6 s_{jc} \cD_j
\]
The matrix $S_t$ is a fundamental solution matrix for $\nablaB$; 
its entries are multi-valued holomorphic functions on 
$\cMB^\circ\times \C^\times$.   
The duality between the sections $\{s_c\}$ and the cycles 
$\{\Gamma_c\}$ implies that the $(c,j)$ entry 
of the inverse matrix $S_t^{-1}$ is the oscillating integral
\begin{equation}
  \label{eq:matrix_entry}
[S_t^{-1}]_{(c,j)} = {\frac{1}{(2\pi z)^{3/2}} 
\cD_j \int_{\Gamma_+(c)} e^{-W/z} \omega}.
\end{equation}
In the basis $\{\cD_i\}$ of $\FBbig|_t$, the linear operator 
$R_t^{-1}$ is represented by a formal power series in $z$ with coefficients 
in 6 by 6 matrices. 
The $(c,j)$-entry of $R_t^{-1}$ is obtained from 
the $(c,j)$-entry \eqref{eq:matrix_entry} 
of $S_t^{-1}$ by stationary phase expansion: 
\begin{align*}
  e^{W(c)/z} [S_t^{-1}]_{(c,j)} \ 
  \sim \ 
  [R_t^{-1}]_{(c,j)} && \text{as $z\to +0$.}   
\end{align*}
The basis $\cD_i$ can be calculated explicitly up to 
order $O(T)$. 
\begin{lemma} 
\label{lem:hcD_along_conifold_divisor} 
Define $(\cD'_1,\dots,\cD'_6) := 
(1, D_2, D_2^2-y_2, D_2^3 + z y_2 -2 y_2D_2, D_1,(1+27y_1) D_1^2)$. 
Then $\cD_i = \cD'_i + O(T)$ in the GKZ system. 
\end{lemma} 
\begin{proof} 
We need to show that $\cD'_i$ gives a flat trivialization 
associated with $\bP_\con$ along 
the divisor $(y_1=-1/27)$.  
Since $\{\cD'_i\}$ coincides with the frame 
$1, D_2, D_2^2, D_2^3, D_1, (1+27y_1) D_1^2$ 
at $(y_1,y_2) = (-1/27,0)$, it suffices to check that 
the action of $D_2$ in the basis $\{\cD'_i\}$ is 
represented by a matrix independent of $z$. 
Indeed, the action of $D_2$ in the basis $\{\cD'_i\}$ is 
\[
\begin{pmatrix} 
0 & y_2 & 0 & -2y_2^2 & 0 & 0 \\ 
1 & 0 & y_2 & 0 & 0 & 0 \\ 
0 & 1 & 0 & y_2 & \frac{1}{3} & 0 \\
0 & 0& 1 & 0 & 0 & 0 \\ 
0 & 0 & 0 & 0 & 0 & 0 \\
0 & 0 & 0 & 9y_2 & 0 & 0 
\end{pmatrix} 
\]
along the divisor $y_1=-1/27$. The lemma follows. 
\end{proof}

\subsection{Stationary Phase Asymptotics}

We say that a function $f$ of $T$ \emph{has $T$-order $\alpha$} 
if $f(T) = O(T^\alpha)$ as $T\to 0$. 
We evaluate the $T$-order of $R_t^{-1}$ by 
examining the stationary phase asymptotics of \eqref{eq:matrix_entry}, 
where $T = y_1 + \frac{1}{27}$ is the co-ordinate of $t$ 
and we shall keep $y_2\neq 0$ fixed. 
Let $c$ be a divergent critical point, and start 
with the oscillatory integral associated with $c$. We have 
\[
\frac{1}{(2\pi z)^{3/2}}
\int_{\Gamma_+(c)} e^{-W/z} \omega \sim 
\frac{e^{-W(c)/z}}{ \sqrt{x_3^c} \sqrt{\det H_c}} 
\left[
  e^{\frac{z}{2} \Delta}
  \exp \left( -\frac{1}{z} \sum_{k=3}^\infty \frac{1}{k!} 
\sum_{i_1,\ldots,i_k} W_{i_1\cdots i_k}(c) \theta_{i_1} \cdots \theta_{i_k} \right) \right]_{\theta_1 = \theta_2 = \theta_3 = 0}
\]
where $\Delta = \sum_{i,j=1}^3 H_c^{ij} \frac{\partial}{\partial \theta_i} \frac{\partial}{\partial \theta_j}$ and 
$H_c^{ij}$ are the matrix entries of the inverse Hessian. 
We can obtain this asymptotic expansion by expanding the 
integrand in Taylor series in $\theta_1,\theta_2,\theta_3$ 
and performing termwise Gaussian integrals: see 
\citelist{\cite{CCIT:MS}*{\S 6.2} \cite{Lho-Pandharipande:CRC}*{Appendix A}}. 
The factor $(x_3^c)^{-1/2}$ comes from 
$\omega = (x_3^c)^{-1/2} d\theta_1 d\theta_2 d\theta_3$. 
By Wick's theorem, 
the term in square brackets is the sum over graphs\footnote
{As an illustration, we listed all graphs that contribute whose number of edges is 
less than or equal to 3. 
Such a graphical technique is standard in quantum 
field theory: see e.g.~\citelist{\cite{BIZ}*{\S 2} \cite{Zee}*{Chapter I.7}}.} 
\begin{center}
\includegraphics[bb=113 628 493 687]{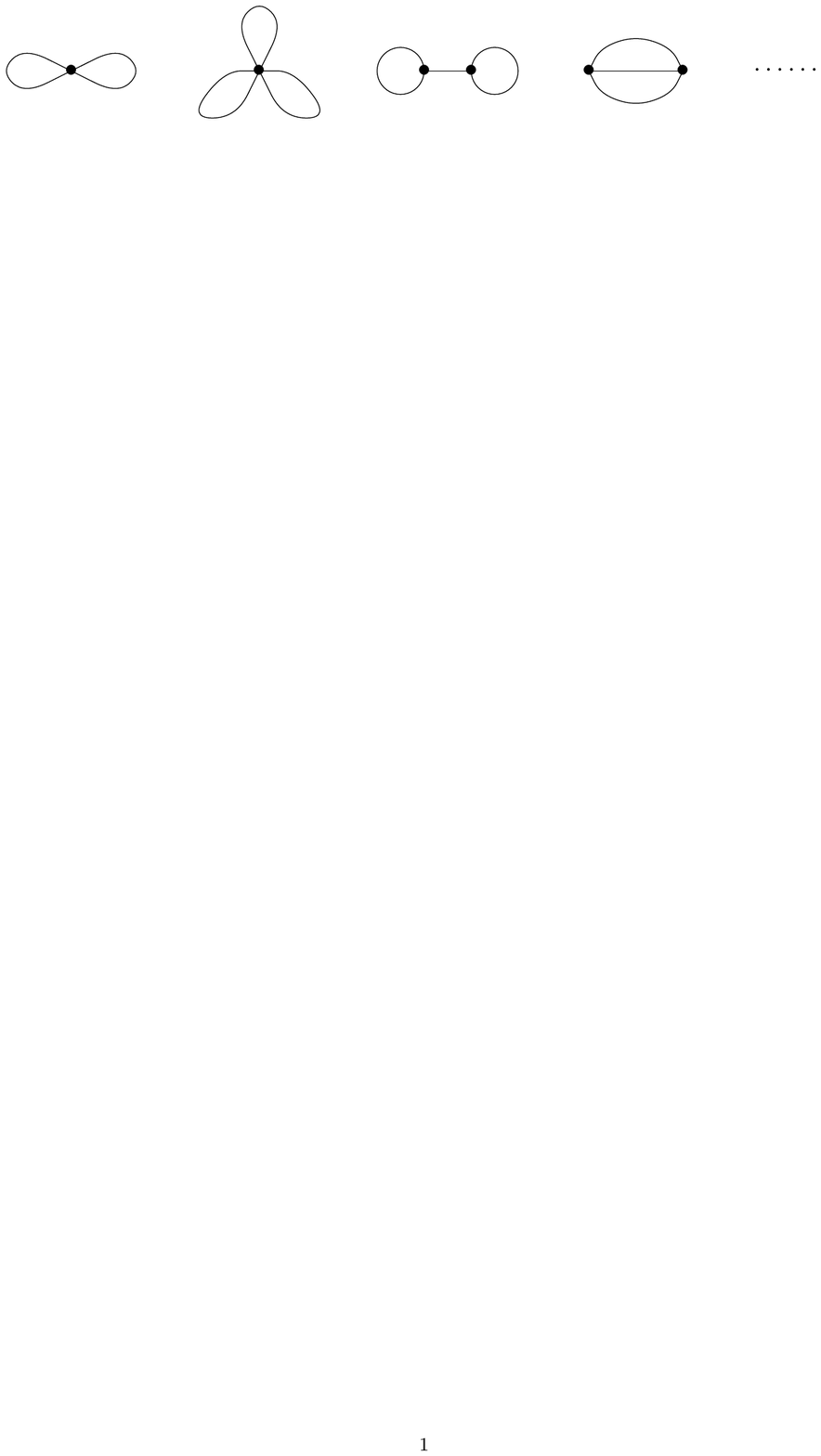}
\end{center}
where each vertex has valency at least $3$, 
we place the tensor $-\frac{1}{z} \sum_{i_1,\ldots,i_k} W_{i_1\cdots i_k}
d\theta_{i_1} \cdots d\theta_{i_k}$ at a vertex of valency $k$, 
and contract using the bivector field $z\Delta$ on each edge.  
A graph with $E$~edges and $V$~vertices contributes to the coefficient 
of $z^{E-V}$ in the asymptotic expansion, and 
if the graph has vertices of valencies $k_1,\dots,k_V$, 
then its contribution has $T$-order 
$\sum_{i=1}^V (-\frac{k_i}{4} + \frac{1}{2}) = 
-\frac{E-V}{2}$; here we used \eqref{eq:order_derivative_W}. 
Thus
\[
\frac{1}{(2\pi z)^{3/2}}
\int_{\Gamma_+(c)} e^{-W/z} \omega \sim 
e^{-W(c)/z}\sum_{n =0}^\infty a_n z^n  
\qquad 
\text{with $a_n = O(T^{\frac{1}{4} - \frac{n}{2}})$ as $T\to 0$.}
\]
The stationary phase asymptotics of 
$\cD \int_{\Gamma_+(c)} e^{-W/z} \omega$, 
with $\cD$ a differential operator in $y_1$ and $y_2$, 
can be computed similarly. 
For example, if $\cD = D_2= -z y_2 \parfrac{}{y_2}$, we have 
\begin{align*} 
& \frac{1}{(2\pi z)^{3/2}}D_2 \int_{\Gamma_+(c)} 
e^{-W/z} \omega 
= \frac{1}{(2\pi z)^{3/2}} \frac{y_2}{x_3^c} \int_{\Gamma_+(c)} 
e^{-W/z - \sqrt{x_3^c} \theta_3} 
\omega \\ 
& \sim 
\frac{e^{-W(c)/z}}{\sqrt{x_3^c}\sqrt{\det(H_c)}} 
\frac{y_2}{x_3^c} 
\left[e^{\frac{z}{2}\Delta} 
\exp\left(-\sqrt{x_3^c} \theta_3 - 
\frac{1}{z} \sum_{k=3}^\infty \frac{1}{k!} 
\sum_{i_1,\dots,i_k} W_{i_1 \dots i_k}(c) 
\theta_{i_1} \cdots \theta_{i_k}\right) 
\right]_{\theta_1=\theta_2=\theta_3=0} 
\end{align*} 
We apply Wick's theorem again to express this as a sum over graphs. 
In this case, we allow graphs to have additional vertices of valency 
$1$ where we place the tensor $(-\sqrt{x_3^c} d\theta_3)$. 
If a graph has $V$ vertices of valencies $k_1,\dots,k_V \ge 3$, 
$L$ vertices of valency $1$, and $E$ edges, its contribution 
has $T$-order $\sum_{i=1}^V (-\frac{k_i}{4} + \frac{1}{2}) 
- \frac{L}{4} = -\frac{E-V}{2}$. 
Hence
\[
 \frac{1}{(2\pi z)^{3/2}}D_2 \int_{\Gamma_+(c)} 
e^{-W/z} \omega 
\sim e^{-W(c)/z} \sum_{n=0}^\infty 
b_n z^n 
\qquad 
\text{with $b_n = O(T^{\frac{3}{4} - \frac{n}{2}})$ 
as $T\to 0$.} 
\]
Combining this method with Lemma \ref{lem:hcD_along_conifold_divisor}, 
we can compute the $T$-orders of the asymptotic expansion 
of $\cD_i \int_{\Gamma_+(c)} e^{-W/z} \omega$ for all $i$. 
The analysis for a non-divergent critical 
point is identical except for the fact that everything is 
holomorphic at $T=0$.  
Ordering the critical points so that the first two are divergent 
and the last four are non-divergent, we see that 
\begin{equation}
  \label{eq:estimate_R_inverse}
  \begin{aligned}
    R_t^{-1} = \sum_{n=0}^\infty 
    \begin{pmatrix}
      [ {\textstyle \frac{1}{4} - \frac{n}{2} } ] &
      [ {\textstyle \frac{3}{4} - \frac{n}{2} } ] &
      [ {\textstyle \frac{1}{4} - \frac{n}{2} } ] &
      [ {\textstyle \frac{3}{4} - \frac{n}{2} } ] &
      [ {\textstyle - \frac{1}{4} - \frac{n}{2} } ] &
      [ {\textstyle \frac{1}{4} - \frac{n}{2} } ] \\
      [ {\textstyle \frac{1}{4} - \frac{n}{2} } ] &
      [ {\textstyle \frac{3}{4} - \frac{n}{2} } ] &
      [ {\textstyle \frac{1}{4} - \frac{n}{2} } ] &
      [ {\textstyle \frac{3}{4} - \frac{n}{2} } ] &
      [ {\textstyle - \frac{1}{4} - \frac{n}{2} } ] &
      [ {\textstyle \frac{1}{4} - \frac{n}{2} } ] \\
      [0] & [0] & [0] & [0] & [0] & [1] \\ 
      [0] & [0] & [0] & [0] & [0] & [1] \\ 
      [0] & [0] & [0] & [0] & [0] & [1] \\ 
      [0] & [0] & [0] & [0] & [0] & [1] 
    \end{pmatrix}
    z^n 
  \end{aligned}
\end{equation}
where $[\alpha]$ denotes a term that has $T$-order $\alpha$. 

We will need similar estimates for the matrix entries of $R_t$.  
For this we use the unitarity condition
\[
R_t(-z)^{-\mathrm{T}} R_t(z)^{-1}= G 
\quad 
\text{or equivalently,} \quad 
\sum_c [R_t(-z)^{-1}]_{(c,i)} [R_t(z)^{-1}]_{(c,j)} = g_{ij} 
\]
from~\cite[\S1.3]{Givental:semisimple}; 
here $G=(g_{ij}) 
=( ((-)^\star \cD_i, \cD_j)_{\rm B})$ 
is the Gram matrix in Proposition~\ref{pro:GKZ}(b) 
evaluated at the conifold point $(y_1,y_2) = (-1/27,0)$. 
The unitarity follows directly from the description 
\eqref{eq:stationary_phase} of the B-model pairing.  
Thus
\[
R_t(z) = G^{-1} R_t (-z)^{-\mathrm{T}}
\]
and since 
\[
G^{-1} = 
\begin{pmatrix}
  0 & 0 & 0 & \frac{1}{9} & 0 & 0 \\
  0 & 0 & \frac{1}{9} & 0 & 0 & 1 \\
  0 & \frac{1}{9} & 0 & 0 & 0 & 0 \\
  \frac{1}{9} & 0 & 0 & 0 & 0 & 0 \\
  0 & 0 & 0 & 0 & 0 & -3 \\
  0 & 1 & 0 & 0 & -3 & 0 
\end{pmatrix}
\]
we conclude that
\begin{equation}
  \label{eq:estimate_R}
  \begin{aligned}
R_t = \sum_{k=0}^\infty 
    \begin{pmatrix}
      [ {\textstyle \frac{3}{4} - \frac{n}{2} } ] & [ {\textstyle \frac{3}{4} - \frac{n}{2} } ] &  [0] & [0] & [0] & [0] \\ 
      [ {\textstyle \frac{1}{4} - \frac{n}{2} } ] & [ {\textstyle \frac{1}{4} - \frac{n}{2} } ] &  [0] & [0] & [0] & [0] \\ 
      [ {\textstyle \frac{3}{4} - \frac{n}{2} } ] & [ {\textstyle \frac{3}{4} - \frac{n}{2} } ] &  [0] & [0] & [0] & [0] \\ 
      [ {\textstyle \frac{1}{4} - \frac{n}{2} } ] & [ {\textstyle \frac{1}{4} - \frac{n}{2} } ] &  [0] & [0] & [0] & [0] \\ 
      [ {\textstyle \frac{1}{4} - \frac{n}{2} } ] & [ {\textstyle \frac{1}{4} - \frac{n}{2} } ] &  [1] & [1] & [1] & [1] \\ 
      [ {\textstyle -\frac{1}{4} - \frac{n}{2} } ] & [ {\textstyle -\frac{1}{4} - \frac{n}{2} } ] &  [0] & [0] & [0] & [0] 
    \end{pmatrix}
    z^n.
  \end{aligned}
\end{equation}

\subsection{The Proof of Theorem~\ref{thm:conifold_estimate}}

It remains to translate these estimates for the pole order in $T$ of stationary phase asymptotics to estimates for the pole order in $T$ of correlation functions.  
Choose a point $t\in \cMB\setminus D\subset \cMBbig$. 
As before, we fix the co-ordinate $y_2\neq 0$ of the point $t$  
and study the asymptotics of correlation functions 
as $T = y_1+ \frac{1}{27}$ goes to zero. 
Introduce algebraic co-ordinates $(\tit,\bx)$ on the total space 
$\LL$ of the big B-model $\log$-cTEP structure,  
where 
\begin{itemize} 
\item $\tit\in \cMBbig$ represents a point in a neighbourhood of $t$;
and
\item $\bx = \sum_{n=1}^\infty \sum_{i=1}^6 x_n^i e_i z^n
\in z\C^6[\![z]\!]$ are co-ordinates along the fiber of $\LL\to \cMBbig$ 
associated with the frame $\cD_1,\dots,\cD_6$ of $\FBbig$.  
\end{itemize}  
See Definition \ref{def:algebraic_local_co-ordinates}. 
There is a distinguished flat co-ordinate system\footnote
{Note that the flat co-ordinate system $\hbq$ depends on 
the choice of $t$.}
$\hbq\in \C^6[\![z]\!]$ 
in the formal neighbourhood of the fiber $\LL_t^\circ$ in $\LL^\circ$,  
associated with the frame $\{\cD_i\}$:
see \cite[Definition 4.28]{Coates--Iritani:Fock}. 
This is given by 
\begin{equation} 
\label{eq:from_algebraic_to_flat} 
\hbq = [M(\tit, z) \bx]_+ 
\end{equation} 
where $[\cdots]_+$ means the non-negative part as a $z$-series.  The
inverse fundamental solution matrix $M(\tit,z)$ here is characterized by the conditions $M(t,z) = \Id$ 
and $d M(\tit,z) =\frac{1}{z} M(\tit,z) A(\tit)$, 
where $d$ is the differential in the $\tit$-direction and 
$\frac{1}{z}A(\tit)$ is the matrix-valued connection $1$-form 
for $\nablaB$ written in the frame $\{\cD_i\}$. 

Let $\{\Nabla^m C^{(g)}\}_{g,m}$ denote the correlation functions for 
$\wave_{\rm B}$ with respect to $\sfP_\con$. 
Givental's higher genus formula discussed in \S \ref{subsec:Givental_formula} 
gives correlation functions along $\LL^\circ_t$, expressed in terms of 
the flat co-ordinate system $\hbq = \sum_{n=0}^\infty 
\sum_{i=1}^6 \hq_n^i e_i z^n$. 
Writing 
\[
\Nabla^m C^{(g)}\big|_{\LL_t^\circ}= 
\sum C^{(g)}_{(n_1,i_1),\dots, (n_m,i_m)}(t,\bx) \,
dq_{n_1}^{i_1} \otimes \cdots \otimes d\hq_{n_m}^{i_m}
\]
and setting\footnote{$\cA^\con_{(t,\bx)} $ is a formal power series in the shifted flat co-ordinate 
$\hba = \sum_{n=0}^\infty \sum_{i=1}^6 \ha_n^i e_i z^n 
:= \hbq - \bx$ on a neighbourhood of $(t,\bx)$.}
\[
\cA^\con_{(t,\bx)} =  \exp\left( \sum_{g=0}^\infty 
\sum_{m:2g-2+m>0}  \frac{\hbar^{g-1}}{m!}
\sum_{n_1,\dots,n_m} \sum_{1\le i_1,\dots,i_m\le 6} 
C^{(g)}_{(n_1,i_1),\dots,(n_m,i_m)}(t,\bx) \ha_{n_1}^{i_1} 
\cdots \ha_{n_m}^{i_m} 
\right),  
\] 
we have 
\begin{equation} 
\label{eq:conifold_ancestor} 
\cA^\con_{(t,\bx)} 
= \left[ \exp\left( \frac{\hbar}{2} \sum V_t^{(n,c),(n',c')} \parfrac{}{q^c_n} 
\parfrac{}{q^{c'}_{n'}} \right) \cT'  
\right]_{\bq^c =  [R_t^{-1}(\bx + \hba)]^c} 
\end{equation} 
where $V_t^{(n,c),(n',c')}$ are coefficients of 
Givental's propagator defined below, 
and $\cT'$ is the product of the Kontsevich--Witten tau function 
modified at genus $1$: 
\[
\cT'(\bq) = \prod_c
\exp\left(\sum_{g=0}^\infty \hbar^{g-1} 
\left(\cF^g_{\pt}(\bq^c) - \delta_{g,1}\cF^1_\pt([R_t^{-1}\bx]^c)\right)
\right).  
\]
Recall that $\cF^g_\pt$ is the genus-$g$ descendant potential 
\eqref{eq:genus_g_descendant_potential} of a point; 
we regard it 
as a function of the dilaton-shifted co-ordinate 
$\bq^c = -z +\bt$. 
Formula \eqref{eq:conifold_ancestor} 
follows from the definition of $\wave_{\tss}$ 
\cite[Definition 7.9]{Coates--Iritani:Fock}, the fact that 
$\{\Nabla^m C^{(g)}\}$ can be obtained from $\wave_{\tss}$ 
by the transformation rule $T(\sfP_{\tss}, \sfP_\con)$,
and the following facts: 
\begin{itemize} 
\item the `conifold ancestor potential' $\cA^\con_{(t,\bx)}$ 
is the image under the formalization map of $\{\Nabla^m C^{(g)}\}$ 
at $(t,\bx)$, see \cite[Definition 5.11]{Coates--Iritani:Fock}, 
where the formalization map is the one associated with the frame 
$\{\cD_i\}$.

\item $\cT'(\bq)$ is the image under the formalization map 
of $\wave_{\tss}$ at $(t,R_t^{-1} \bx)$, where the formalization 
map is the one associated with the semisimple trivialization.  To see this 
combine Lemma~5.13 and Lemma~7.13 of \cite{Coates--Iritani:Fock}.

\item the transformation rule $T(\sfP_{\tss}, \sfP_\con)$ 
is expressed in terms of the action of Givental's quantized 
operator $\widehat{R}_t$ through the formalization map; 
see \cite[Theorem 5.14]{Coates--Iritani:Fock}. 
\end{itemize} 
\begin{definition} 
\emph{Givental's propagator} $\{V_t^{(n,c),(n',c')}\}$ associated with 
$R_t$ is defined by 
\[
\sum_{n=0}^\infty \sum_{n'=0}^\infty (-1)^{n+n'} V_t^{(n,c),(n',c')} w^n z^{n'} 
= \left[\frac{R_t(-w)^{-1} R_t(z)- \Id}{z+w}  \right]_{(c,c')}   
\]
where $c$,~$c'$ range over critical points of $W$. 
\end{definition} 
The formula \eqref{eq:conifold_ancestor} 
together with the discussion in \cite[\S3]{Coates--Iritani:convergence}  
implies that $C^{(g)}_{(n_1,i_1),\dots,(n_m,i_m)}(t,\bx)$ 
is given by the sum over decorated connected Feynman graphs 
\begin{center}
\includegraphics[bb=283 653 344 710]{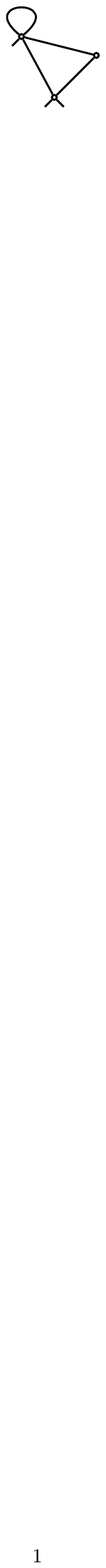} 
\end{center}
where 
\begin{itemize} 
\item each vertex $v$ is labelled by an integer $g_v\ge 0$;
\item the graph has $m$ external half-edges, called \emph{legs}, 
labelled by $\{1,\dots,m\}$; 
\item a label $(l,c) \in \Z_{\geq 0} \times \{\text{critical points of $W$}\}$ 
is assigned to each pair of a vertex and a half-edge incident to it; 
note that we assign a label $(l,c)$ to legs (external half-edges) too 
so that legs have two different kinds of labels; 
\item the Euler number $\chi$ of the graph satisfies 
$g = 1- \chi + \sum_{v:\text{vertex}} g_v $;
\end{itemize} 
and we require that, for each 
vertex $v$, if $(l_1,c_1),\dots,(l_k,c_k)$ are all the
labels attached to half-edges incident to $v$, then 
\begin{equation*}
  l_1 + \dots + l_k   \le 3 g_v -3 + k \qquad 
  \text{and} \qquad 
  2g_v-2 + k  > 0
\end{equation*} 
There are finitely many such decorated Feynman graphs~\cite{Givental:semisimple}.
The contribution of such a graph
$\Gamma$ to $C^{(g)}_{(n_1,i_1),\dots,(n_m,i_m)}(t,\bx)$ is:
\begin{equation}
  \label{eq:Gammacontribution}
  \frac{1}{|\Aut(\Gamma)|} 
  \prod_{e \in E(\Gamma)} (\text{edge term for $e$}) 
  \prod_{v \in V(\Gamma)} (\text{vertex term for $v$})
  \prod_{\ell \in L(\Gamma)} (\text{leg term for $\ell$})
\end{equation}
where the edge term for an edge with labels $(l,c)$, $(l',c')$ 
is the coefficient $V_t^{(l,c),(l',c')}$ of Givental's propagator;
the vertex term for a vertex $v$ incident to half-edges with labels 
$(l_1,c_1),\dots, (l_k,c_k)$ is\footnote{We used the Dilaton equation 
in the second line.} 
\begin{align}
\nonumber 
& \left. \parfrac{^k\cF^{g_v}_\pt}{q_{l_1}^c \partial q_{l_2}^c 
\cdots \partial q_{l_k}^c}\right|_{\bq^c =[R_t^{-1}\bx]^c} 
= \sum_{p=0}^\infty \frac{1}{p!} 
\corr{\psi_1^{l_1},\dots,\psi_k^{l_k},
\bt^c(\psi_{k+1}), \dots 
\bt^c(\psi_{k+p})}_{g_v,k+p,0}^\pt 
\\ 
\label{eq:vertex_term} 
& \qquad =  \sum_{p=0}^\infty (-q^c_1)^{-(2g_v-2+k+p)}
\frac{1}{p!} 
\corr{\psi_1^{l_1},\dots,\psi_k^{l_k}, 
\bq^c_{\ge 2}(\psi_{k+1}),\dots,\bq^c_{\ge 2}(\psi_{k+p})
}_{g_v,k+p,0}^\pt 
\end{align}
if $c_1 = \cdots = c_k = c$ and zero otherwise, where 
we set 
\begin{equation} 
\label{eq:Dilaton_shift_in_sstriv} 
\bq^c(z) = -z + \bt^c(z) = [R_t(z)^{-1} \bx(z)]^c, 
\end{equation} 
and $\bq^c_{\ge 2}(z)$ denotes the truncation of the $z$-series 
$\bq^c(z)$ at degree two; and finally the leg term of a leg $\ell$ 
with labels $s \in \{1,\dots,m\}$ and $(l,c)$ is 
$[R_t^{-1} e_{i_s} z^{n_s}]_l^c$, where 
$[\cdots]_l^c$ denotes the coefficient of $z^l$ in 
the $c$th component. 

Now we restrict $(t,\bx)$ to lie on the image of the primitive section 
$\zeta\colon \cMB^\circ \to \LLo$, and evaluate the $T$-order 
of $C^{(g)}_{(n_1,i_1),\dots,(n_m,i_m)}(t,\bx)$. 
The primitive section $\zeta$ is given by $-z = -z \cD'_1$ in the GKZ 
system. Since the frame $\{\cD_i\}$ is homogeneous and 
$\cD'_1=1$ has the lowest possible degree, it follows 
that $\cD_1 = f(T) \cD'_1$ for some holomorphic function $f(T)$ 
such that $f(0) = 1$ and $f$ is independent of $y_2$ or $z$. 
Therefore the section $\zeta$ is given in terms of $\bx$ by 
\begin{equation} 
\label{eq:zeta_by_x}
x_1^1 = -1 +O(T), \quad x_1^2 = \cdots = x_1^6 =0, 
\quad x_n^1 = \cdots = x_n^6 = 0 \quad \text{for all $n\ge 2$.} 
\end{equation} 
On this locus, using \eqref{eq:Dilaton_shift_in_sstriv} 
and \eqref{eq:estimate_R_inverse}, we have for $n\ge 1$, 
\[
q_n^c =
\begin{cases}  
O(T^{\frac{3}{4} - \frac{n}{2}}) & \text{if $c$ is divergent;} \\ 
O(1) & \text{if $c$ is non-divergent.} 
\end{cases} 
\] 
A more careful analysis of the first column of $R_t^{-1}$ 
shows that $q_1^c$ is exactly of $T$-order $\frac{1}{4}$ if $c$ is 
divergent, and is exactly of $T$-order $0$ if $c$ is non-divergent. 
Thus we find for every critical point $c$, 
\[
(q_1^c)^{-1} = O(T^{-1/4}) \qquad \text{and} \qquad 
q_n^c = O(T^{\frac{3}{4}-\frac{n}{2}}) \quad 
\text{for $n\ge 2$.}
\]
We can estimate the $T$-order of the 
vertex term \eqref{eq:vertex_term} from this.  
Using the fact that the coefficient of $q_{j_1}^c \cdots q_{j_p}^c$ 
(with $j_1,\dots,j_p \ge 2$) 
in \eqref{eq:vertex_term} is non-zero only when 
$l_1+ \cdots +l_k + j_1 + \cdots + j_p = 3g_v-3+k+p$, 
we find that the $T$-order of the vertex term \eqref{eq:vertex_term} 
is at least 
\begin{equation} 
\label{eq:vertex_contribution}  
\frac{-(2g_v-2+k+p)}{4} + \sum_{r=1}^p \left( \frac{3}{4} - 
\frac{j_r}{2}\right) = 
-(2g_v-2) - \sum_{r=1}^k \left(\frac{3}{4} - \frac{l_r}{2}\right). 
\end{equation} 
The $T$-order of the leg term 
$[R_t^{-1} e_{i_s} z^{n_s}]_l^c$ 
is $-\frac{1}{4} - \frac{l-n_s}{2}$ by 
\eqref{eq:estimate_R_inverse}. 
To compute the $T$-order of the edge term, writing 
\begin{align*} 
\sum_{l,l'\ge 0} (-1)^{l+l'} V_t^{(l,c),(l',c')} w^l z^{l'} 
& = \left[R_t(-w)^{-1}\frac{R_t(z)- R_t(-w)}{z+w}  \right]_{(c,c')}   \\ 
& = \sum_{l=0}^\infty \sum_{l'= 0}^\infty 
(-1)^l w^l z^{l'} \sum_{e+e'=l+l'+1, e\le n} \left[ (R_t^{-1})_e 
(R_t)_{e'}\right]_{(c,c')}
\end{align*} 
with $R_t(z)^{-1} = \sum_{e\ge 0} (R_t^{-1})_e z^e$ 
and $R_t(z) = \sum_{e\ge 0} (R_t)_e z^e$, 
and using \eqref{eq:estimate_R_inverse} and \eqref{eq:estimate_R}, 
we find that $V^{(l,c),(l',c')} = O\big(T^{-\frac{l+l'+1}{2}}\big)$ 
as $T \to 0$, for all pairs $(c,c')$ of critical points. 
Since
\[
-\frac{l+l'+1}{2} 
= \left(\frac{3}{4} - \frac{l}{2}\right) 
+ \left(\frac{3}{4} - \frac{l'}{2}\right) - 2
\]
let us split the contribution $O\big(T^{-\frac{l+l'+1}{2}}\big)$ 
from an edge $e$ with labels $(l,c)$ and $(l',c')$ 
up into contributions $O(T^{\frac{3}{4} - \frac{l}{2}})$ 
and $O(T^{\frac{3}{4} - \frac{l'}{2}})$ 
carried by the two half-edges given by $e$ 
and a contribution $O(T^{-2})$ carried by $e$ itself. 
We include this new contribution $O(T^{\frac{3}{4} - \frac{l}{2}})$ 
from a half-edge with label $(l,c)$ into the $T$-order of 
vertices or legs incident to it. 
Then, the new $T$-order of the vertex term of a vertex $v$ becomes 
$-(2g_v-2)$ -- see \eqref{eq:vertex_contribution} --
and the new $T$-order of the leg term of a leg labelled by 
$s\in\{1,\dots,m\}$ is $-1+\frac{n_s}{2}$. 
Therefore the total $T$-order of the contribution from a graph $\Gamma$ 
is at least 
\[
\overbrace{\overset{}{- 2 |E(\Gamma)|}
}^{\text{edge terms}}  
\overbrace{- \sum_{v\in V(\Gamma)} (2g_v-2)}
^{\text{vertex terms}} 
\overbrace{ - \sum_{1\le s\le m}
\left(1 - \frac{n_s}{2}\right)}^{\text{leg terms}}  
\ge  - (2g-2 +m)
\]
where we used $g=\sum_{v\in \Gamma(V)} g_v + 1 -\chi$.  
Summing over all graphs, we find that 
\[
C^{(g)}_{(n_1,i_1),\dots,(n_m,i_m)}(t,\bx) = O(T^{-(2g-2+m)})
\] 
as $T\to 0$ on the image of $\zeta$. 

We need to check that the change of co-ordinates  
\eqref{eq:from_algebraic_to_flat} does not affect the 
pole order in $T$. 
Recall that $C^{(g)}_{(n_1,i_1),\dots,(n_m,i_m)}(t,\bx)$ 
is an $m$-tensor written in the basis $\{d\hq_n^i\}$ of 1-forms. 
Write $\tit=(\tit^1,\dots,\tit^6)$ for a co-ordinate system 
centered at $t$, and write 
$\frac{1}{z}A = 
\frac{1}{z}\sum_{\alpha=1}^6 A_\alpha(\tit) d\tit^\alpha$ 
for the connection $1$-form of $\nablaB$. 
Equation \eqref{eq:from_algebraic_to_flat} gives 
\begin{align*} 
\hq_0^i & = \sum_{\alpha} \tit^\alpha [A_\alpha(t) x_1]^i 
+ O(|\tit|^2)  && \\ 
\hq_n^i & = x_n^i + \sum_{\alpha} \tit^\alpha [A_\alpha(t) x_{n+1}]^i 
+ O(|\tit|^2) && \text{for $n\ge 1$}. 
\end{align*} 
Since the section $\zeta=-z$ has co-ordinates $x_n^i =0$ 
for $n\ge 2$ -- see \eqref{eq:zeta_by_x} -- we have:  
\begin{align*} 
\zeta^\star(d\hq_0^i)\big|_t & = 
[\nablaB \zeta]^i_0
= \delta_{1,i} \frac{dy_1}{y_1} 
+ \delta_{2,i} \frac{dy_2}{y_2} +O(T) && \\ 
\zeta^\star(d \hq_n^i) \big|_t & = dx_n^i 
&& \text{for $n\ge 1$}
\end{align*} 
where $[\cdots]_0^i$ means the coefficient in front of $z^0 \cD_i$ 
when expanded in the basis $\{z^n \cD_i\}$. 
These $1$-forms are regular along $T=0$. 
This means that $\zeta^\star(\Nabla^m C^{(g)})$ 
has poles of order $2g-2+m$ along $y_1 = -1/27$, for 
any fixed $y_2 \neq 0$. We already know from 
Theorem \ref{thm:wave_B} that $\Nabla^m C^{(g)}$ 
extends regularly across $y_2 =0$ as a logarithmic tensor; 
Hartog's Principle applied to a section of 
$\Omega^1_{\cMB}(\log D)^{\otimes m}$ thus proves 
Theorem~\ref{thm:conifold_estimate}. 

\begin{remark} 
In this section, we studied correlation functions 
on the image of $\zeta$, but the pole order along $T=y_1+\frac{1}{27}=0$ 
depends on the choice of slice. 
A similar analysis shows that $C^{(g)}(t,\bx)$ (the $0$-point correlation 
function, with $g\ge 2$) 
has pole of order $g-1$ along $T=0$ for a fixed \emph{generic} $\bx$. 
The restriction to the image of $\zeta$ is special because 
$\zeta$ touches the discriminant divisor $P(t,x_1) = 0$ 
(see equation~\eqref{eq:discriminant}) at the conifold point; 
this follows from $q_1^c =[R_t^{-1}x_1]^c= O(T^{1/4})$ 
on the image of $\zeta$ for a divergent $c$.  
We have the $5g-5$ pole order condition 
along the discriminant (Definition~\ref{def:local_Fock}), and correlation functions on the image of~$\zeta$ acquire part of their 
poles from this. 
\end{remark}

\section{Modularity}
\label{sec:calc} 

We now apply the theory developed in the preceding sections to 
show that the Gromov--Witten potential of 
local $\Proj^2$ is a quasi-modular function with respect to 
the congruence subgroup $\Gamma_1(3)$ of $SL(2,\Z)$: 
\[
\Gamma_1(3) = \left\{\begin{pmatrix} a & b \\ c & d \end{pmatrix} 
\in SL(2,\Z) 
: a \equiv d \equiv 1, \ c \equiv 0 \bmod 3\right\}.   
\]

\subsection{The Mirror Family for Local $\Proj^2$} 

As discussed in the Introduction, the mirror to the non-compact Calabi--Yau manifold $Y$ is a certain family of elliptic curves $\{E_y : y \in \cMCY\}$.  This family has been studied by many authors: see for example~\citelist{\cite{CKYZ:local_mirror} \cite{Hosono:central_charges} \cite{Stienstra:resonant} \cite{Takahashi:log_mirror} \cite{ABK} \cite{Konishi--Minabe:local_B} \cite{Doran--Kerr}}.  We summarize the aspects of this work that we need.

\subsubsection{A Family of Elliptic Curves with $\Gamma_1(3)$-Level Structure}
\label{subsubsec:family_level}
Recall that $\cMCY = \Proj(3,1)$ and $\DCY = \{-\frac{1}{27},0\}$. 
We will see the mirror family of $Y$ emerging in the conformal limit $y_2 \to 0$ of the 
Landau--Ginzburg potential mirror to $\Ybar = \Proj(\cO_{\Proj^2}\oplus 
\cO_{\Proj^2}(-2))$ from \S \ref{sec:LG_Ybar}: 
\[
\text{$W_y = w_1 + w_2 + w_3 + w_4 + w_5$ with $w_1w_2w_3=y_1 w_4^3$, $w_4w_5 = y_2$.}  
\]
Setting the last co-ordinate $w_5=y_2/w_4$ to zero 
and considering the zero locus of $W_y$ in the projective space 
with co-ordinates $[w_1,w_2,w_3,w_4]$, we obtain a family of elliptic curves
\begin{align*} 
E_y & = \left\{[w_1,w_2,w_3,w_4] \in \Proj^3: 
      \text{$w_1w_2w_3 = y w_4^3$, $w_1 + w_2 + w_3 + w_4 =0$} \right \} \\
    & = \text{compactification of }
      \left \{(w_1,w_2) \in (\C^\times)^2 : 
      w_1 + w_2 + \frac{y}{w_1 w_2} + 1 =0 \right\} 
\end{align*} 
parametrized by $y=y_1\in \C\subset \cMCY$. 
The second line is a presentation in the affine chart $w_4=1$. 
The curve $E_y$ has singularities when $y\in \DCY$. 
By introducing a co-ordinate $v = (w_1w_2w_3)^{1/3}$, 
we can extend the family across the orbifold point $y=\infty$ as  
\[
E_y = \left\{[w_1,w_2,w_3,v] \in \Proj^3 : 
  \text{$w_1w_2w_3 = v^3$, $w_1+w_2+w_3+ \fry v =0$} \right\} 
\]
with $\fry = \fry_1 = y^{-1/3}$. The isotropy group $\mu_3$ 
at $y=\infty \in \Proj(3,1)$ acts on the family 
as $v \mapsto \xi^{-1} v$, $\fry \mapsto \xi \fry$. 
A holomorphic volume form on $E_y$ is given by the one-form 
\[
\lambda_y = \frac{1}{3} \frac{d\log w_1 \wedge d \log w_2}
{d (w_1+w_2+ \frac{y}{w_1w_2}+1)} 
= \frac{dw_1}{3w_1(w_2-\frac{y}{w_1w_2}))}   
\]
where $(w_1,w_2) \in (\C^\times)^2$ are co-ordinates on the 
affine chart.  

\begin{remark} \label{rem:3-fold_covering}
  Aganagic--Bouchard--Klemm \cite{ABK} worked with 
  a 3-fold covering $\pi\colon \tE_y \to E_y$ given by  
  \begin{align*} 
    \tE_y = \left\{[X,Y,Z]\in \Proj^2 : 
    X^3 + Y^3+ Z^3 + \fry XYZ = 0 \right\}
  \end{align*} 
  where $\pi$ maps $[X,Y,Z]$ to 
  $[w_1,w_2,w_3,v] = [X^3, Y^3, Z^3, XYZ]$.    
\end{remark}

A \emph{$\Gamma_1(3)$-level structure} on an elliptic 
curve $E$ (equipped with a group structure) is by definition 
choice of a 3-torsion point $\tor$ on $E$. 
This is equivalent to the choice of an order-3 automorphism 
$\sigma$ of $E$ without fixed points, or to 
a non-zero element $\ell$ in $H_1(E,\Z/3\Z)$. 
We introduce a group structure on $E_y$ such that 
$[w_1,w_2,w_3,w_4] = [1,-1,0,0] \in E_y$ is the identity element,
and define a $\Gamma_1(3)$-structure on $E_y$ by 
the order 3 automorphism $\sigma$: 
\[
\sigma \colon [w_1,w_2,w_3,w_4] \mapsto [w_3,w_1,w_2, w_4]
\]
The corresponding 3-torsion point is $\tor=\sigma(0)= [0,1,-1,0]\in E_y$. 
For a path $\gamma$ connecting $0$ and $\tor$, 
$3 \gamma$ defines a non-zero element $\ell \in H_1(E_y,\Z/3\Z)$, 
which is independent of the choice of the path $\gamma$. 
The set of ordered bases $\{\alpha,\beta\}$ for $H_1(E_y,\Z)$ 
satisfying $\alpha\cdot \beta =1$ and $[\alpha] = \ell$ 
is a torsor over $\Gamma_1(3)$, via change of basis. 

A \emph{marked elliptic curve} is a pair $(E,\{\alpha,\beta\})$ 
of an elliptic curve $E$ (with group structure) and a symplectic 
basis, also called a marking, $\{\alpha,\beta\} \subset H_1(E,\Z)$ 
with $\alpha \cdot \beta=1$. 
The moduli space of marked elliptic curves can be identified 
with the upper-half plane $\HH=\{\tau\in \C : 
\Im(\tau) >0\}$ via the period map $(E,\{\alpha,\beta\}) 
\mapsto \tau = \int_\beta \lambda/\int_\alpha \lambda 
\in \HH$, where $\lambda$ is a non-zero holomorphic one-form on $E$.  
We call $\tau$ a \emph{modular parameter}. 
We let $\SL(2,\Z)$, and hence $\PSL(2,\Z)$, act on the upper-half plane 
by fractional linear transformations 
\[
\begin{pmatrix} 
a & b \\ 
c & d 
\end{pmatrix} \cdot \tau 
= \frac{a\tau + b}{c\tau+d} 
\]
which corresponds to the change of markings 
\begin{equation} 
\label{eq:action_on_homology_basis} 
(\alpha,\beta) \mapsto (\alpha',\beta') 
= (\alpha,\beta) \begin{pmatrix} d & b \\ c & a \end{pmatrix}. 
\end{equation}
The moduli stack of elliptic curves with $\Gamma_1(3)$-level structure 
is identified with the quotient: 
\begin{equation} 
\label{eq:moduli_elliptic_level} 
[\HH/\Gamma_1(3)]  
\end{equation} 
The $\Gamma_1(3)$-orbit of a marked elliptic curve 
$(E,\{\alpha,\beta\})$ corresponds to the elliptic curve $E$ 
with the $\Gamma_1(3)$-level structure $\ell = [\alpha] \in H_1(E,\Z/3\Z)$. 

\begin{remark} 
The $\Gamma_1(3)$-structure on $E_y$ lifts to 
a level-$3$ structure on $\tE_y$, i.e.~to a basis of 3-torsion points. 
The corresponding order-3 automorphisms are given by $[X,Y,Z] 
\mapsto [Z,X,Y]$ and $[X,Y,Z]\mapsto 
[X,\xi Y, \xi^2 Z]$ with $\xi\in \mu_3$.  
\end{remark}

\begin{proposition} 
\label{prop:moduli_elliptic_level} 
The base space $\cMCY \setminus \DCY$ of the mirror family 
can be identified 
with the moduli stack \eqref{eq:moduli_elliptic_level} 
of elliptic curves with $\Gamma_1(3)$-level structure. 
\end{proposition} 
\begin{proof} 
As we saw, $\cMCY\setminus \DCY$ is equipped with 
a family of elliptic curves with $\Gamma_1(3)$-level structure. 
Hence we have a canonical map 
$\cMCY\setminus \DCY \to [\HH/\Gamma_1(3)]$. 
The $j$-invariant of $E_y$ is given by 
\[
j(E_y) = -\frac{(1+24y)^3}{y^3(1+27y)}
\]
and this gives the composition $\cMCY \setminus \DCY 
\to [\HH/\Gamma_1(3)] \to \HH/\PSL(2,\Z) \cong \C$. 
We can easily see that this has the same degree ($=4$) and 
ramification data (at $j =0$, $1728$) 
as the covering $[\HH/\Gamma_1(3)] 
\to \HH/\PSL(2,\Z)$. Thus the coarse moduli spaces of 
$\cMCY \setminus \DCY$ and $[\HH/\Gamma_1(3)]$ 
are the same.  The $\mu_3$-orbifold structures at 
$y=\infty$, $\tau=e^{2\pi\iu/3}$ also match.  
\end{proof} 

\begin{figure}[htbp]
\includegraphics[bb=200 583 400 710]{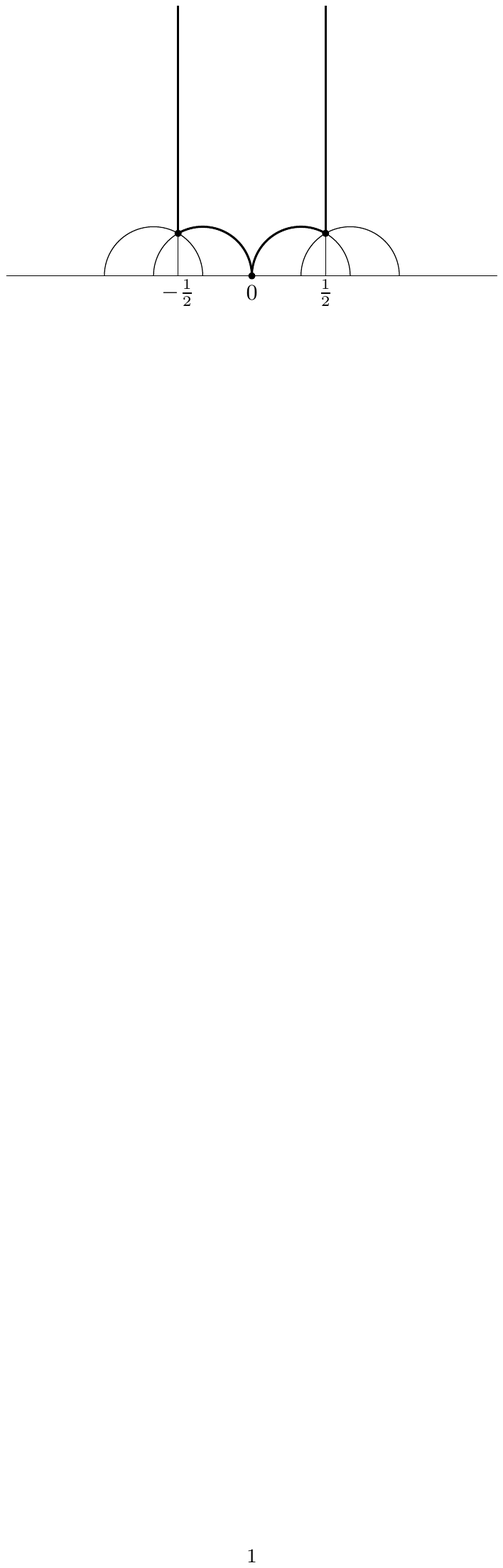}
\caption{A fundamental domain for $\HH/\Gamma_1(3)$. 
Note that $\Gamma_1(3)$ is generated by 
$\tau \mapsto \tau+1$ and $\tau \mapsto \tau/(3\tau +1)$. 
The large-radius limit point is $\tau=+\infty\iu$, 
the conifold point is $\tau= 0$, and the orbifold point is 
$\tau = -\frac{1}{2} + \frac{\iu}{2\sqrt{3}} 
= \frac{\xi-1}{3}$, where the parameter $\tau$ is as in 
Corollary \ref{cor:modular_parameter}.}
\end{figure} 

\subsubsection{A Relative Cohomology Mirror and the Picard--Fuchs Equation}
Let $F_y=F_y(w_1,w_2)$ denote the defining equation of $E_y$ 
on the affine chart $(w_1,w_2)\in (\C^\times)^2$: 
\[
F_y = w_1+ w_2 + \frac{y}{w_1w_2} + 1.  
\]
The corresponding affine elliptic curve 
\[
\Eaff_y = \big\{(w_1,w_2) \in (\C^\times)^2 : F_y(w_1,w_2) = 0\big\}
\] 
is $E_y\setminus \{0,\tor,2\tor\}$, where $\tor$ is the 
3-torsion point as before. 
Near $y=\infty$, by introducing variables 
$v_1 = \fry w_1$, $v_2 = \fry w_2$,  
we define 
\begin{equation} 
\label{eq:Eaff_near_infinity}
\Eaff_y = \big\{(v_1,v_2) \in (\C^\times)^2 : 
v_1+v_2+ 1/(v_1v_2) + \fry =0\big\} = E_y \setminus 
\{0,\tor,2\tor\}, 
\end{equation} 
where $\fry = y^{-1/3}$. 
A mirror for $Y$ is given by 
the relative cohomology of the pair $\big((\C^\times)^2, \Eaff_y\big)$; 
such a mirror has been analysed by 
Stienstra~\cite{Stienstra:resonant}, N.~Takahashi~\cite{Takahashi:log_mirror}  
and Konishi--Minabe~\cite{Konishi--Minabe:local_B}.  
We shall see that the variation of Hodge structure on $H^1(E_y)$ 
corresponds to the rank~2 vector bundle $\Hvec$ 
from \S\ref{sec:two}, and 
that the variation of mixed Hodge structure 
on $H^2\big((\C^\times)^2, \Eaff_y\big)$ 
corresponds to the rank~3 vector bundle $\widebar{H}$ there. 
Let $\zeta_y\in H^2\big((\C^\times)^2,\Eaff_y\big)$ denote 
the relative cohomology class given by 
\[
\zeta_y = 
\frac{dw_1}{w_1} \wedge \frac{dw_2}{w_2} 
= \frac{dv_1}{v_1} \wedge \frac{dv_2}{v_2}. 
\]
\begin{proposition}[\cites{Batyrev:MHS_affine, Stienstra:resonant, 
Takahashi:log_mirror,Konishi--Minabe:local_B}] 
\label{prop:PF} 
The classes $\zeta_y \in H^2\big((\C^\times)^2,\Eaff_y\big)$, 
$\lambda_y \in H^1(E_y)$ satisfy 
\[
\theta \zeta_y = \delta \left( \lambda_y\big|_{\Eaff_y}\right)  
\]
where $\theta = \nabla_{y\parfrac{}{y}}$ is the Gauss--Manin 
connection and 
$\delta \colon 
H^1(\Eaff_y) \to H^2\big((\C^\times)^2, \Eaff_y\big)$ is the connecting 
homomorphism.  
They satisfy the Picard--Fuchs equations:
\begin{align}
\label{eq:PF}
\begin{split}  
\left( \theta^3 + 3y\theta(3\theta+1)(3\theta+2) \right) \zeta_y &= 0 \\
\left( \theta^2 + 3 y (3\theta+1)(3\theta +2) \right) \lambda_y & = 0 
\end{split} 
\end{align} 
\end{proposition} 
\begin{proof}
Let $C \in H_2\big((\C^\times)^2,\Eaff_y\big)$ be a relative cycle. 
Working in the chart near $y=\infty$, we find 
\[
3 y\parfrac{}{y} \int_C \zeta_y = - 
\fry \parfrac{}{\fry} \int_C \zeta_y =  
\int_{\partial C} \fry \frac{d \log v_1 \wedge d\log v_2}
{d(v_1+v_2+\frac{1}{v_1v_2})} 
= 3 \int_{\partial C} \lambda_y
\] 
(see \cite[Lemma 1.8]{Takahashi:log_mirror}, 
\cite[Lemma 4.3]{Konishi--Minabe:local_B}).
This gives the first equation. The Picard--Fuchs equations 
are well-known:
see \cite[Theorem 14.2]{Batyrev:MHS_affine} and 
\cite[\S 6]{Stienstra:resonant}.  
\end{proof} 

\begin{corollary} 
\label{cor:VHS_isom} 
We have the following isomorphisms. 
\begin{itemize} 
\item[(1)]  
The rank $3$ vector bundle $\bigcup_y H^2((\C^\times)^2,\Eaff_y)$ 
over $\cMCY\setminus \DCY$ 
equipped with the Gauss--Manin connection is isomorphic to 
the vector bundle $(\widebar{H},\nabla)$ from \S\ref{sec:two} . 
\item[(2)] The rank $2$ vector bundle $\bigcup_y H^1(E_y)$ 
over $\cMCY\setminus \DCY$ 
equipped with the Gauss--Manin connection is isomorphic to 
the vector bundle $(\Hvec, \nabla)$ from \S\ref{sec:two}.
\end{itemize}  
These isomorphisms map $\zeta_y\in H^2((\C^\times)^2,\Eaff_y)$ 
to $\zeta\in\widebar{H}$ and $\lambda_y\in H^1(E_y)$  
to $\theta \zeta\in \Hvec$.
\end{corollary}

\begin{proof}
The vector bundles $(\widebar{H},\nabla)$, $(\Hvec,\nabla)$ 
are described by the same Picard--Fuchs equations \eqref{eq:PF};
see \eqref{eq:QDE_for_Y}. 
\end{proof}

Consider now the diagram: 
\[
\xymatrix{
0 \ar[r] & 
H_2\big((\C^\times)^2\big) \ar[r] & 
H_2\big((\C^\times)^2, \Eaff_y\big) 
\ar[r]^(0.6){\partial} 
& H_1(\Eaff_y) \ar[d]^{i_*} \ar[r]& H_1\big((\C^\times)^2\big) \ar[r] & 
0 \\ 
& & & H_1(E_y) & & 
}
\]
where we use $\Z$ coefficients and the top row is exact. 
Since $\Re(F_y) \colon (\C^\times)^2 \to \R$ is a Morse 
function with 3 critical points of Morse index 2, it follows from Morse theory  that 
\[
H_1((\C^\times)^2,\Eaff_y) = 0, \qquad 
H_2((\C^\times)^2,\Eaff_y)\cong \Z^3; 
\] 
see e.g.~\cite[\S 3.3.1]{Iritani:integral}.  Generators of $H_2\big((\C^\times)^2,\Eaff_y\big)$ are given by 
3 Lefschetz thimbles emanating from critical points of $F_y$.  
We define the \emph{lattice of vanishing cycles} to be 
\[
\VC_y := 
\Image\left(i_* \circ \partial \colon 
H_2\big((\C^\times)^2,\Eaff_y;\Z\big) 
\to H_1(E_y;\Z) \right).  
\]
\begin{proposition} 
\label{prop:VC}
The sublattice $\VC_y \subset H_1(E_y;\Z)$ 
is of index $3$ and is given by 
\[
\VC_y= 3 H_1(E_y;\Z) + \big\{ 
\alpha \in H_1(E_y;\Z) : [\alpha ] = \ell \big\} 
= \pi_* H_1(\tE_y;\Z) 
\]
where $\ell\in H_1(E_y,\Z/3\Z)$ is the $\Gamma_1(3)$-level structure 
of $E_y$ and $\pi\colon \tE_y \to E_y$ is the $3$-fold 
covering described in Remark~\ref{rem:3-fold_covering}.  
\end{proposition} 
\begin{proof} 
We work in the chart near $y=\infty$ and use the presentation 
\eqref{eq:Eaff_near_infinity} of $\Eaff_y$. 
Consider the projection $\Eaff_y \to \C$, $(v_1,v_2) \mapsto v_1$ 
to the $v_1$-plane, which extends to a ramified covering 
$E_y \to \Proj^1$. 
This has 4 branch points given by $v_1 = 0$ and $v_1(v_1+\fry)^2=4$; note that $v_1=\infty$ is not a branch point.  
The branch points move as $\fry$ varies, and two of them 
coalesce when $\fry= -3 \xi^j$, $j\in\{0,1,2\}$, with $\xi=e^{2\pi\iu/3}$, 
where $E_y$ is singular. 
The three vanishing cycles on $E_{y=\infty}$ 
associated with three paths $[0,-3 \xi^i]$, $i\in\{0,1,2\}$,
on the $\fry$-plane are given by the trajectories of 
coalescing branch points: see Figure \ref{fig:vanishing_cycles}. 
It is then easy to see that these vanishing cycles generate 
a sublattice of index 3. Thus $\VC_y$ is of 
index 3. 

On the other hand, the sublattice $\VC_y$ 
is clearly invariant under monodromy. 
Since we have $\cMCY \setminus \DCY \cong [\HH/\Gamma_1(3)]$, 
the monodromy group is $\Gamma_1(3)$ and acts on 
symplectic bases of $H_1(E_y,\Z)$ by 
\eqref{eq:action_on_homology_basis}. 
It is easy to see that there is a unique sublattice of index~3 
which is invariant under $\Gamma_1(3)$. The conclusion follows. 
\end{proof} 

\begin{figure}[htbp] 
\centering
\includegraphics[bb=200 610 400 710]{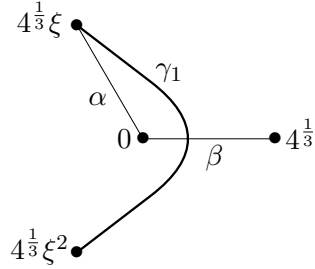} 
\caption{A vanishing cycle $\gamma_1$ on $E_\infty$, pictured on the $v_1$-plane. The black dots are branch points. 
Two other vanishing cycles 
$\gamma_2$, $\gamma_3$ are 
obtained from $\gamma_1$ by $2\pi/3$, $4\pi/3$ rotations 
respectively. 
The cycles $\{\alpha,\beta\}$ give a symplectic basis. 
With some choice of orientations, we find 
$\gamma_1 = 2 \alpha +\beta$, 
$\gamma_2= \alpha + 2\beta$, $\gamma_3 = -\alpha +\beta$.} 
\label{fig:vanishing_cycles} 
\end{figure} 

\begin{remark} 
\label{rem:other_mirror_models} 
As a mirror to $Y$, 
Chiang--Klemm--Yau--Zaslow~\cite{CKYZ:local_mirror} 
considered periods of a multi-valued one-form
\[
\int_{\gamma\subset E_y} \log (w_1) \frac{dw_2}{w_2} 
\] 
and periods of the 3-fold 
$\check{Y}= \big\{(w_1,w_2,u,v) \in (\C^\times)^2 \times \C^2 : 
F_y(w_1,w_2) + uv =0\big\}$: 
\[
\int_{S \subset \check{Y}} \frac{d w_1}{w_1}\wedge \frac{dw_2}{w_2}
\wedge \frac{du}{u}.  
\]
These are equivalent, up to a Tate twist, to the relative cohomology mirror~\cite{Konishi--Minabe:local_B}. 
\end{remark} 

\subsection{Periods and Compactly Supported $K$-theory}  
We next compute periods of the mirror family 
as explicit hypergeometric series.  To do this, we identify 
periods over integral cycles 
with elements of the compactly supported $K$-group of $Y$ (or $\cX$) 
via the $\hGamma$-integral structure 
\cite{Iritani:integral,Iritani:Ruan,Iritani:qcohperiod}. 
We then identify the modular parameter $\tau$ 
with the second derivative of the genus-zero Gromov--Witten 
potential of $Y$. 
Most of the computations in this section are already in the literature, 
in particular in work of Hosono~\cite{Hosono:central_charges}. 

\subsubsection{$I$-function, $\hGamma$-Integral Structure and Monodromy} 
\label{subsubsec:monodromy}
The $I$-functions \citelist{\cite{Givental:toric} \cite{CCIT:computing}}
of $Y$ and $\cX$ are the power series 
\begin{align*} 
I_Y(y,z) & = \sum_{d=0}^\infty 
y^{h/z+d} \frac{\prod_{m=0}^{3d-1}(-3h-mz)}
{\prod_{m=1}^d(h+mz)^3} \\ 
I_\cX(\fry,z) & = \sum_{d=0}^\infty 
\fry^d \frac{\prod_{
0\le m<d/3, \fracp{d/3}=\fracp{m}} (-mz)^3 }{d!z^d}
\unit_{\fracp{d/3}}
\end{align*} 
which take values, respectively, in $\HY=H^\bullet(Y)$ and 
$\HX =H^\bullet_{\orb}(\cX)$. In the second line, $\fracp{r}$ 
denotes the fractional part of a real number $r$. 
The components of $I_Y$ written in the basis $\{1,h,h^2\}$, or the components of $I_\cX$ written in the basis 
$\{\unit,\unit_{\frac{1}{3}},\unit_{\frac{2}{3}}\}$,
form a basis of solutions to the Picard--Fuchs equation 
\eqref{eq:PF} satisfied by $\zeta_y$. Therefore 
periods of $\zeta_y$ can be written as certain linear combinations  
of these hypergeometric series. 
In what follows, we set $z=1$ and write $I_Y(y) = I_Y(y,1)$ 
and $I_\cX(\fry) = I_\cX(\fry,1)$. 
The Mirror Theorem~\cite{Givental:elliptic}*{Theorem 4.2} implies that
the $I$-function of $Y$ can be 
expanded as 
\[
I_Y(y) = 1 + t h + \parfrac{F_Y^0}{t} (-3h^2) 
\]
where $t=t(y)$ is the mirror map for $Y$,
given by $t(y) = \log y + g(y)$ with $g(y)$ as 
in Theorem~\ref{thm:mirrorsymmetryforYbar}, 
and 
\begin{equation} 
\label{eq:genus-zero_pot_Y}
F_Y^0(t) = -\frac{1}{18} t^3 + \sum_{d=1}^\infty 
\corr{}_{0,0,d}^Y e^{t d}
\end{equation} 
is the genus-zero Gromov--Witten 
potential\footnote{We added a cubic term to $F_Y^0$ 
which is responsible for the cup product.} 
restricted to $H^2(Y)$.  
The Mirror Theorem~\cite{CCIT:computing}*{Theorem 4.6} implies that
the $I$-function of $\cX$ can be 
expanded as 
\[
I_\cX(\fry) = 1 + \frt \unit_{\frac{1}{3}} + 
\parfrac{F_\cX^0}{\frt} (3 \unit_{\frac{2}{3}})
\]
where $\frt=\frt(\fry)$ is the mirror map for $\cX$, 
which is the same map as appeared in 
Theorem \ref{thm:mirrorsymmetryforXbar},
and 
\begin{equation} 
\label{eq:genus-zero_pot_X}
F_\cX^0(\frt)  = 
\sum_{n=3}^\infty \corr{\unit_{\frac{1}{3}}, 
\dots, \unit_{\frac{1}{3}}}_{0,n,0}^\cX \frac{\frt^n}{n!}.    
\end{equation} 
is the genus-zero Gromov--Witten 
potential restricted to $H^2_{\orb}(\cX)$.

Consider now the $\hGamma$-integral structure \citelist{\cite{Iritani:integral}*{\S 2.4} \cite{Iritani:Ruan}*{\S 2}}. 
The classes $\hGamma_Y \in \HY$, $\hGamma_\cX \in \HX$ 
are defined by:
\begin{align*}
  \hGamma_Y := \Gamma(1+h)^3 \, \Gamma(1-3h) = 1+\pi^2 h^2, && 
  \hGamma_\cX :=  \bigoplus_{i=0}^2 \Gamma(1-\tfrac{i}{3})^3 \unit_{\frac{i}{3}}. 
\end{align*}
Let $X$ denote either $Y$ or $\cX$ and consider  
the $K$-group $K_c(X)$ of coherent sheaves on $X$ with compact support. 
The groups $K_c(Y)$, $K_c(\cX)$ are freely generated by 
3 coherent sheaves: 
\begin{align*}
  K_c(Y) = \left\langle 
  \cO_{\pt}, \cO_{\Proj^1}(-1), \cO_{\Proj^2}(-1) \right\rangle, 
  &&
  K_c(\cX) = \left\langle 
  \cO_0, \cO_0\otimes \varrho, \cO_0\otimes \varrho^2 
  \right \rangle 
\end{align*}
where $\Proj^1\subset \Proj^2$ denotes a line and $\varrho$ 
is the standard one-dimensional representation of~$\mu_3$. 
For $V\in K_c(X)$, we define a vector $\Psi(V)$ lying in 
the compactly supported (orbifold) cohomology $H_{X,c}$ of $X$ by
\[
\Psi(V) = \hGamma_X \cup (2\pi\iu)^{\frac{\deg}{2}} 
\inv^*\tch(V).  
\]
This is an analogue of the Mukai vector. 
For a precise definition of the right-hand side, we refer the reader to~\cite{Iritani:integral}*{\S 2.4} and~\cite{Iritani:Ruan}*{\S 2.5}.
In the case at hand, we have:
\begin{align*} 
\Psi(\cO_{\pt}) & = (2\pi\iu)^3 [\pt] \\
\Psi(\cO_{\Proj^1}(-1)) & 
= (2\pi\iu)^2 [\Proj^1]\\
\Psi(\cO_{\Proj^2}(-1)) & = (2\pi\iu) \left(1+\pi \iu h -\pi^2 h^2\right)\cap [\Proj^2] 
\end{align*} 
for $Y$ and 
\begin{align*} 
\Psi(\cO_0 \otimes \varrho^i) 
= (2\pi\iu)^3 \left( 
\frac{1}{3} [\pt] + \frac{\xi^{-i}}
{\Gamma(\frac{1}{3})^3}\unit_{\frac{1}{3}}
 - \frac{\xi^{-2i}}{\Gamma(\frac{2}{3})^3}\unit_{\frac{2}{3}} 
\right) 
&&  i \in \{0,1,2\}
\end{align*} 
for $\cX$, 
where $[\pt]\in H_{c}^6(\cX) \subset H_{\orb,c}^6(\cX)$ 
is the class of a non-stacky point, so that $(1,[\pt]) = 1$,
and $\xi = e^{2\pi\iu/3}$.  Cf.~\cite{Iritani:Ruan}*{Example 2.16}.

\begin{definition}[\cite{Iritani:integral,Iritani:Ruan}] 
Let $X$ be $Y$ or $\cX$. 
We define the \emph{quantum cohomology central charge} 
of $V \in K_c(X)$ to be 
\[
\Pi_X(V) = \left((-1)^{\deg/2} I_X, \Psi(V) \right) 
\]
where $I_X$ is the $I$-function of $X$ and 
$(\cdot,\cdot)$ is the natural pairing between 
(orbifold) cohomology and compactly supported (orbifold) cohomology. 
\end{definition} 

\begin{remark} 
\label{rem:central_charge_in_A-model_coord} 
The quantum cohomology central charge in \cite{Iritani:integral, 
Iritani:Ruan} is a function of the A-model co-ordinates (K\"ahler 
parameters) and is related to the present one by a change 
of co-ordinate given by the mirror map, together with a multiplicative factor of $(2\pi\iu)^{-3}$ . Under the mirror map 
$t=t(y)$ for $Y$, we have 
\begin{align} 
\label{eq:central_charges_Y} 
\begin{split} 
\Pi_Y(\cO_{\pt}) & = (2\pi\iu)^3 \\ 
\Pi_Y(\cO_{\Proj^1}(-1)) & = - (2\pi\iu)^2 t \\ 
\Pi_Y(\cO_{\Proj^2}(-1)) & = -(2\pi\iu) 
\left(\pi^2+ \pi \iu t + 3 \parfrac{F^0_Y}{t}\right).  
\end{split} 
\end{align} 
Similarly, under the mirror map $\frt=\frt(\fry)$ for $\cX$, 
we have 
\begin{equation} 
\label{eq:central_charges_X} 
\Pi_\cX(\cO_0\otimes \varrho^i) = 
(2\pi\iu)^3 \left( 
\frac{1}{3} + \frac{\xi^{-2i}}{3\Gamma(\frac{2}{3})^3} \frt 
+ \frac{\xi^{-i}}{\Gamma(\frac{1}{3})^3} 
\parfrac{F_\cX^0}{\frt} \right), 
\qquad 
i\in \{0,1,2\}. 
\end{equation} 
\end{remark} 

We introduce period vectors $\vPi_Y$ and $\vPi_\cX$ as follows: 
\begin{align*} 
\vPi_Y & :=\big( \Pi_Y(\cO_{\pt}), \Pi_Y(\cO_{\Proj^1}(-1)), 
\Pi_Y(\cO_{\Proj^2}(-1)) \big), \\  
\vPi_\cX & := \left ( \Pi_\cX(\cO_0), \Pi_\cX(\cO_0\otimes \varrho), 
\Pi_\cX(\cO_0\otimes \varrho^2) \right).  
\end{align*} 
They are power series solutions defined near $y=0$, $\fry=y^{-1/3} = 0$ 
respectively; 
since they satisfy the Picard--Fuchs equation, they analytically 
continue to the universal cover of $\cMCY\setminus \DCY$. 
Take a base point $y_0 \in \cMCY\setminus \DCY$  
such that $0<y_0 \ll 1$. 
We choose a branch of $\vPi_Y$ around $y_0$ 
by requiring that $\log y_0 \in \R$.

\begin{proposition}[\cites{DFR:noncompact_CY,
Hosono:typeIIA, Hosono:central_charges}]  
\label{prop:monodromy} 
Under analytic continuation along the positive real line 
in the $y$-plane, we have 
\[
\vPi_Y = \vPi_\cX \begin{pmatrix} 1 & 0 & 0 \\ 
1 & -1 & 0 \\ 1 & 1 & 1
\end{pmatrix}. 
\]
Moreover, the analytic continuation of $\vPi_Y$ along the 
loops $\gamma_\LR$, $\gamma_{\con}$,   $\gamma_\orb$ 
in Figure $\ref{fig:paths}$  
are given by $\vPi_Y M_\LR$, $\vPi_Y M_{\con}$, 
$\vPi_Y M_\orb$ respectively, where 
\begin{align*}
  M_\LR = \begin{pmatrix} 1 & -1 & 0 \\ 
    0 & 1 & -1 \\ 
    0 & 0 & 1 
  \end{pmatrix},  
  &&
     M_{\con} =  
     \begin{pmatrix} 
       1 & 0 & 0 \\ 0 & 1 & 0 \\ 0 & 3 & 1 
     \end{pmatrix}, 
      &&
         M_\orb = \begin{pmatrix} 1 & 1 & 1 \\ 0 & 1 & 1 \\ 0 & -3 & -2 
         \end{pmatrix}. 
\end{align*}
\end{proposition} 
\begin{proof} 
The analytic continuation has been computed 
in \cites{DFR:noncompact_CY, Hosono:typeIIA,Hosono:central_charges} 
in a slightly different basis. The Barnes integral representation 
for the $I$-function yields the connection formula 
between $\vPi_Y$ and $\vPi_\cX$: see e.g.~\citelist{\cite{Hosono:typeIIA}*{Appendix A}
\cite{CIT:wall-crossings}*{Appendix}}. 
It is easy to see that 
the monodromy around the orbifold point $y=\infty$ 
corresponds to 
$(-)\otimes \varrho$ on $K_c(\cX)$ and that the monodromy around 
the large radius limit point $y=0$ corresponds to $(-)\otimes \cO(-1)$ 
on $K_c(Y)$. 
This together with the connection formula yields $M_\LR$ and $M_\orb$. 
The conifold monodromy $M_{\con}$ is then given by $M_\LR^{-1} 
M_\orb^{-1}$.  
\end{proof} 

\begin{figure}[htbp]
\centering
\includegraphics[bb=200 610 400 710]{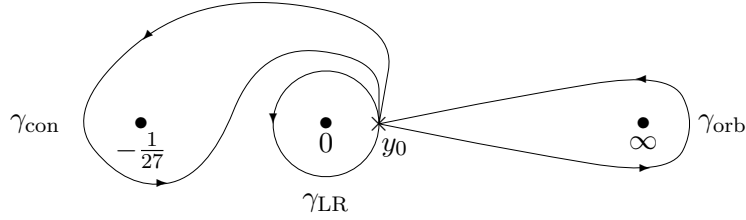}
\caption{Paths in $\cMCY \setminus \DCY$. 
The base point $y_0$ 
of the analytic continuation is chosen so that $0<y_0\ll 1$.}
\label{fig:paths} 
\end{figure} 

\begin{remark}[\cites{Horja, Hosono:typeIIA, 
Hosono:central_charges, Borisov-Horja:FM, Coates--Iritani--Jiang}]
The connection matrix relating $\vPi_Y$ and $\vPi_{\cX}$ coincides 
with the Fourier--Mukai transformation 
between compactly supported $K$-groups. 
Consider the diagram 
\[
\xymatrix{
& \left[\cO_{\Proj^2}(-1)/\mu_3\right]  
\ar[ld]_(0.6){f} \ar[rd]^(0.6){g} & \\ 
Y & & \cX 
} 
\]
and the Fourier--Mukai transformation $\Phi(-) = 
\bR g_*( f^*(-) \otimes \cO(-1) \otimes \varrho)$.  
Then we have: 
\[
\left(\Phi(\cO_{\pt}), 
\Phi(\cO_{\Proj^1}(-1)), \Phi(\cO_{\Proj^2}(-1)) \right) 
= \left( 
\cO_0, \cO_0\otimes \varrho, \cO_0\otimes \varrho^2 \right) 
\begin{pmatrix} 1 & 0 & 0 \\ 
1 & -1 & 0 \\ 1 & 1 & 1
\end{pmatrix}.  
\]
As we remarked in the proof, $M_\LR$ and $M_\orb$ correspond 
to the autoequivalences $(-)\otimes \cO(-1)$, $(-)\otimes \varrho$ 
respectively. 
The inverse conifold monodromy $M_{\con}^{-1}$ 
corresponds to the Seidel--Thomas spherical twist 
by the object $\cO_{\Proj^2}(-1)$. 
Observe also that the $2 \times 2$ 
right-lower submatrices of $M_\LR, M_{\con}, M_\orb$ 
generate $\Gamma_1(3)$. 
\end{remark} 

\begin{remark} 
The above matrices $M_\LR, M_{\con}, M_\orb$ represent 
the monodromy acting on \emph{homology} $H_2\big((\C^\times)^2, \Eaff_y\big)$; 
the monodromy acting on cohomology $H^2\big((\C^\times)^2, \Eaff_y\big)$, or equivalently the monodromy of $\big(\widebar{H},\nabla\big)$, is 
given by the adjoint-inverse of these matrices. 
\end{remark} 

\subsubsection{Identification of periods with hypergeometric series} 
We show that quantum cohomology central charges are  
periods of $\zeta_y$ over integral cycles, and vice versa. 

\begin{lemma}[\cite{Hosono:central_charges}*{Appendix A}] 
\label{lem:real_Lefschetz}  
For $0<y < \frac{1}{27}$, let 
$\Gamma_\R \in H_2\big((\C^\times)^2,\Eaff_{-y}\big)$ denote 
the class of a Lefschetz thimble associated to the critical value 
$1 -3 y^{1/3}$ of $F_{-y}$ and the straight path $[0,1-3 y^{1/3}]$, 
i.e.
\[
\Gamma_\R = \left\{(w_1,w_2)\in (\C^\times)^2: w_1<0, w_2<0, 
w_1+w_2+\frac{-y}{w_1w_2}+ 1 \ge 0 \right\}. 
\]
Then we have
\[
\Pi_Y(\cO_{\Proj^2}(-1))(e^{\pi \iu} y) = 
2\pi\iu \int_{\Gamma_\R} \zeta_{-y}.  
\]
\end{lemma} 
\begin{remark}
  The $(2\pi\iu)$ factor on the right-hand side here reflects the fact that we are working with a 2-dimensional relative cohomology mirror model, instead of a 3-dimensional mirror. 
\end{remark}
\begin{proof}[Proof of Lemma~\ref{lem:real_Lefschetz}]
Hosono \cite{Hosono:central_charges}*{equation (A.4)} 
evaluated the period integral over a vanishing sphere in the 
3-dimensional mirror model (see Remark \ref{rem:other_mirror_models}), 
and his computation implies the lemma. 
We give another proof using the Mellin transform, 
which was used by Katzarkov--Kontsevich--Pantev  
\cite{Katzarkov--Kontsevich--Pantev} to compute oscillatory 
integrals mirror to $\Proj^n$. 
Via the co-ordinate change $u_1 = -w_1$, $u_2 = -w_2$, 
$u_3 = -y/(w_1w_2)$, we write, for $0<y<\frac{1}{27}$, 
\[
\varphi(y) := \int_{\Gamma_\R} \zeta_{-y} 
= \int_{\substack{u_1>0, u_2>0, u_3>0, 
u_1+u_2+u_3\le 1 \\ u_1u_2u_3=y}}
\frac{d \log u_1 \wedge d \log u_2 \wedge d \log u_3}
{d\log y}. 
\]
We set $\varphi(y) = 0$ for $y\ge \frac{1}{27}$. 
The Mellin transform of $\varphi(y)$ can be computed as 
the Euler integral: 
\[
\int_0^\infty y^s \frac{dy}{y} \varphi(y) 
= \int_{\substack{u_1+u_2+u_3 \le 1\\ u_1>0,u_2>0,u_3>0}}
 (u_1u_2u_3)^s \frac{du_1}{u_1}
 \frac{du_2}{u_2} \frac{du_3}{u_3} 
 = \frac{\Gamma(s)^3}{\Gamma(1+3s)} 
\]
for $\Re(s)>0$. The Mellin inversion formula gives 
\[
\varphi(y) = \frac{1}{2\pi\iu} \int_{c-\iu \infty}^{c+\iu \infty} 
\frac{\Gamma(s)^3}{\Gamma(1+3s)} y^{-s} ds
\]
for $c>0$. Closing the contour to the left, we can write 
$\varphi(y)$ as the sum of residues at $s = -n$, $n=0,1,2,\dots$. 
Thus 
\begin{align*} 
\varphi(y) & = \sum_{n=0}^\infty \Res_{h=0} 
\frac{\Gamma(h-n)^3}{\Gamma(1+3h-3n)} y^{n-h} dh \\
& = \int_{\Proj^2} \sum_{n=0}^\infty 
h^3 \frac{\Gamma(h-n)^3}{\Gamma(1+3h-3n)} 
y^{n-h}  \\ 
& = \frac{1}{2\pi\iu} 
\left((-1)^{\deg/2} I_Y(e^{\pi\iu} y), \Psi(\cO_{\Proj^2}(-1))
\right) = \frac{1}{2\pi\iu} \Pi_Y(\cO_{\Proj^2}(-1))(e^{\pi\iu} y). 
\end{align*} 
In the second line here, $h$ is regarded as a cohomology class on $\Proj^2$. 
The lemma follows. 
\end{proof} 

\begin{proposition} 
\label{prop:central_charge_is_period}
Let $X$ denote either $Y$ or $\cX$. We have an isomorphism 
$\Mir \colon K_c(X) \cong H_2\big((\C^\times)^2, \Eaff_y;\Z\big)$ 
of integral lattices such that for $V \in K_c(X)$,
\begin{equation} 
\label{eq:central_charge_is_period}
\Pi_X(V) = 2\pi\iu \int_{\Mir(V)} \zeta_y.  
\end{equation} 
\end{proposition} 
\begin{proof} 
It suffices to prove this for $X=Y$. 
We saw in Lemma~\ref{lem:real_Lefschetz} 
that the identity \eqref{eq:central_charge_is_period} 
holds for $V = \cO_{\Proj^2}(-1)$ and 
$\Mir(V) = \Gamma_\R$.  
Recall from Proposition \ref{prop:monodromy} 
that monodromy $M_\LR$ around the large radius limit 
$y=0$ corresponds to $(-)\otimes \cO(-1)$ on $K_c(Y)$. 
Since $K_c(Y)$ is generated by $\cO_{\Proj^2}(-1)$ 
under $(-)\otimes \cO(-1)$, and Lefschetz thimbles are generated by $\Gamma_\R$ 
under monodromy around $y=0$, the conclusion follows. 
\end{proof} 

Next we describe cycles $\partial \Mir(V)$ on the elliptic curve $E_y$ 
in terms of the level structure. 

\begin{proposition} 
Let $\Mir \colon K_c(Y) \cong H_2\big((\C^\times)^2,\Eaff_y;\Z\big)$ 
as in Proposition~\ref{prop:central_charge_is_period} 
and set $\Gamma_1= \Mir(\cO_{\Proj^1}(-1))$, $\Gamma_2 
= \Mir(\cO_{\Proj^2}(-1))$. There exist a symplectic 
basis $\{\alpha, \beta\}$ of $H_1(E_y;\Z)$ 
and a sign $\varepsilon \in \{\pm 1\}$ 
such that 
$ [\alpha]$ is the level structure $\ell\in H_1(E_y;\Z/3\Z)$ 
and that 
\[
\partial \Gamma_1 = \varepsilon3 \beta, \qquad 
\partial \Gamma_2 = \varepsilon \alpha.  
\]
\end{proposition} 
\begin{proof} 
By differentiating \eqref{eq:central_charge_is_period} and 
using Proposition \ref{prop:PF}, we obtain  
\begin{equation} 
\label{eq:derivative_central_charge} 
y \parfrac{}{y} \Pi_X(V) = 2\pi\iu \int_{\partial \Mir(V)} \lambda_y, 
\end{equation} 
that is, the derivatives of the quantum cohomology 
central charges are precisely periods over 
cycles from $\VC_y$. 
Since $\{\cO_\pt, \cO_{\Proj^1}(-1), \cO_{\Proj^2}(-1)\}$ 
is a basis of $K_c(Y)$ and $y\parfrac{}{y} \Pi_Y(\cO_\pt) = 0$, 
$y\parfrac{}{y} \Pi_Y(\cO_{\Proj^1}(-1))$ and 
$y \parfrac{}{y} \Pi_Y(\cO_{\Proj^2}(-1))$ form a basis of periods 
over vanishing cycles, i.e.~
\[
\VC_y = \left\langle 
\partial \Gamma_1, \partial \Gamma_2  
\right\rangle.  
\] 
The monodromy of $y\parfrac{}{y} \Pi_Y(\cO_{\Proj^1}(-1))$, 
$y\parfrac{}{y} \Pi_Y(\cO_{\Proj^2}(-1))$ is given by the 
$2\times 2$ right-lower submatrices of $M_\LR, M_{\con}, M_\orb$ 
in Proposition \ref{prop:monodromy}. 
By reducing the monodromy modulo~3, we find 
that the class of $\partial \Gamma_1$ 
in $\VC_y/3 \VC_y$ generates a monodromy-invariant line 
over $\F_3 =\Z/3\Z$. 

Let us choose a symplectic basis $\{\alpha,\beta\}$ of $H_1(E_y;\Z)$ 
such that $[\alpha]$ is the given $\Gamma_1(3)$-level structure. 
Then $\{-3\beta,\alpha\}$ forms a basis of $\VC_y$ by 
Proposition \ref{prop:VC}. The monodromy in this basis is given by 
(see \eqref{eq:action_on_homology_basis}): 
\[
(-3\beta, \alpha) \mapsto (-3\beta',\alpha') = (-3\beta,\alpha) 
\begin{pmatrix} 
a & -c/3 \\ -3b & d
\end{pmatrix}.  
\]
Thus the basis $\{-3\beta,\alpha\}$ also transforms 
under $\Gamma_1(3)$, and we see that 
$-3 \beta$ generates a monodromy-invariant 
line of $\VC_y/3 \VC_y$.  
The discussion in the previous paragraph implies 
\begin{equation} 
\label{eq:O_P1(-1)_mod_3} 
\partial \Gamma_1 \equiv  \pm 3\beta \mod 3 \VC_y.  
\end{equation} 
Since $\partial \Gamma_1, \partial \Gamma_2$ are a basis 
of $\VC_y$, this implies that 
$[\partial \Gamma_2 ]  
= n [\alpha] + m [3 \beta]$ in $\VC_y/3\VC_y$ for some $n\in \F_3^\times$
and $m\in \F_3$; in particular 
\[
\partial \Gamma_2 \equiv \pm \alpha \mod 3 H_1(E_y;\Z).  
\]
Thus the class of $\partial \Gamma_2$ in $H_1(E_y;\Z/3\Z)$ 
equals $\varepsilon \ell$ for some $\varepsilon \in \{\pm 1\}$. 
Equation \eqref{eq:O_P1(-1)_mod_3} implies that $\partial \Gamma_1$ 
is divisible by $3$ in $H_1(E_y;\Z)$ and thus 
$\{\partial \Gamma_2, \partial \Gamma_1/3\}$ gives a basis of 
$H_1(E_y;\Z)$. It now suffices to show that this is a \emph{symplectic} basis:  
$\partial \Gamma_2 \cdot (\partial \Gamma_1/3) =1$. 
We will discuss this in the proof of the following 
Corollary \ref{cor:modular_parameter}. 
\end{proof} 

\begin{corollary}[cf.~Proposition \ref{prop:moduli_elliptic_level}] 
\label{cor:modular_parameter}
The multi-valued function 
\[
\tau = - \frac{y\parfrac{}{y} \Pi_Y(\cO_{\Proj^2}(-1))}{
y\parfrac{}{y} \Pi_Y(\cO_{\Proj^1}(-1))} 
=  -\frac{1}{2} - \frac{3}{2\pi\iu} \parfrac{^2 F^0_Y}{t^2} 
\]
takes values in the upper-half plane $\HH$ 
and induces an isomorphism 
$\cMCY\setminus \DCY \cong [\HH/\Gamma_1(3)]$, where 
$t=t(y)$ is the mirror map for $Y$.  
\end{corollary} 
\begin{proof} 
We have shown that there exist a symplectic basis 
$\{\alpha,\beta\}$ of $H_1(E_y;\Z)$ and $\varepsilon\in \{\pm 1\}$ 
such that 
$[\alpha] = \ell$ and 
$\partial \Gamma_1 = \pm \varepsilon 3\beta$ 
and $\partial \Gamma_2 = \varepsilon \alpha$. 
(The sign $\pm$ was not determined in the above discussion.) 
Recall from \S\ref{subsubsec:family_level} that the modular parameter 
for $E_y$, with respect to this marking,
is given by $\tau' = \int_\beta \lambda_y/\int_\alpha \lambda_y$. 
On the other hand, by \eqref{eq:derivative_central_charge}, we have 
\[
\tau = 
- \frac{y\parfrac{}{y} \Pi_Y(\cO_{\Proj^2}(-1))}{
y\parfrac{}{y} \Pi_Y(\cO_{\Proj^1}(-1))} 
= \frac{-\int_{\partial \Gamma_2} \lambda_y}
{\int_{\partial \Gamma_1} \lambda_y}   
= \pm \frac{1}{-3\tau'}. 
\]
This quantity satisfies
\begin{align*}
  \tau \sim -\frac{1}{2} + \frac{\log y}{2\pi\iu} + O(y) 
  && \text{as $y \to 0$}
\end{align*}
which lies in the upper-half plane $\HH$ when $|y|$ is sufficiently small.  
The Riemann bilinear inequality then implies that 
$(\partial \Gamma_1, -\partial \Gamma_2)$ is positively 
oriented, i.e.~that $\partial \Gamma_1 = \varepsilon 3 \beta$ 
and $\tau=1/(-3\tau')$.  
The isomorphism $\cMCY\setminus \DCY \cong [\HH/\Gamma_1(3)]$ 
in Proposition \ref{prop:moduli_elliptic_level} was given by 
the parameter $\tau'$; it now suffices to observe that 
the map $\tau' \mapsto \tau = 1/(-3\tau')$ 
induces an isomorphism $[\HH/\Gamma_1(3)] \cong [\HH/\Gamma_1(3)]$ 
via the involution on $\Gamma_1(3)$: 
\[
\begin{pmatrix} 
a & b \\ 
c & d 
\end{pmatrix} 
\longmapsto 
\begin{pmatrix} 
d & -c/3 \\ 
-3b & a
\end{pmatrix} 
= \begin{pmatrix} 
0 & 1 \\ 
-3 & 0
\end{pmatrix} 
\begin{pmatrix} 
a & b \\ c & d \end{pmatrix} 
\begin{pmatrix} 
0 & 1 \\ 
-3 & 0
\end{pmatrix}^{-1}.  
\]
The second expression for $\tau$ follows from 
\eqref{eq:central_charges_Y}. 
\end{proof} 

\begin{remark} 
\label{rem:Fricke} 
The parameter $\tau$ in Corollary \ref{cor:modular_parameter}  
is a modular parameter for $\tE_y$ rather than for $E_y$.
The map $\tau' \mapsto 1/(-3\tau')$ exchanging the modular parameters 
of $E_y$ and $\tE_y$ is known as the Fricke involution. 
The Fricke involution exchanges the large-radius ($\tau=+\infty \iu$) 
and conifold ($\tau=0$) points, and preserves the orbifold point 
($\tau = \frac{1-\xi}{3}$, $\frac{\xi^2-1}{3}$). 
The role of Fricke involution in this context has been studied extensively 
by Alim--Scheidegger--Yau--Zhou \cite{ASYZ}. 
\end{remark} 

Let $\chi(V_1, V_2) = \sum_{i=0}^3 (-1)^i \dim \Ext^i(V_1,V_2)$ 
denote the Euler pairing of  
coherent sheaves $V_1,V_2$ with compact support. 
Since we have $\partial \Gamma_1 \cdot \partial \Gamma_2 = -3$ 
and $\chi(\cO_{\Proj^1}(-1),\cO_{\Proj^2}(-1)) = 3$, we conclude:  
\begin{corollary} 
Let $X$ denote either $Y$ or $\cX$. 
For $V_1,V_2 \in K_c(X)$, we have 
\[
\chi(V_1,V_2) = 
- (\partial \Mir(V_1)) \cdot (\partial \Mir(V_2)). 
\]
\end{corollary} 

\subsection{Opposite Line Bundles at Cusps and 
the Crepant Resolution Conjecture}
\label{subsec:opposite_cusps_CRC} 
Recall the opposite line bundles $P_\LR$, $P_\con$, $P_\orb$ 
associated with large-radius, conifold and orbifold points that were defined in Notation \ref{nota:opposite_fd}. 
We next describe these opposite line bundles in terms of flat 
co-ordinates given by central charge functions, 
and obtain an explicit Feynman rule relating the 
Gromov--Witten potentials of $\cX$ and $Y$. 

As discussed in \S\ref{subsubsec:Haff}, any (local) function 
$\psi$ satisfying the Picard--Fuchs equation:
\begin{equation} 
\label{eq:PF_psi} 
\left[ \theta^3 + 3y \theta(3\theta+1)(3\theta+2) \right] \psi = 0 
\qquad \text{with $\theta = y\parfrac{}{y}$} 
\end{equation} 
defines (locally) a D-module homomorphism 
$\psi^\sharp \colon \cO(\widebar{H}) \to \cO$ sending $\zeta$ to $\psi$. 
In particular, the central charge functions 
$\Pi_Y(V)$, $\Pi_\cX(V)$ define ``flat co-ordinates'' on 
$\widebar{H}$ -- that is, flat sections 
of the dual bundle $\widebar{H}^\vee$. 
Recall from \S\S\ref{subsubsec:Haff}--\ref{subsubsec:Hvec}
that the subbundles $\Haff, \Hvec \subset \widebar{H}$ are 
cut out, respectively, by the equations
\begin{align*} 
\Pi_Y(\cO_\pt)^\sharp & = (2\pi \iu)^3, & & 
\Pi_Y(\cO_\pt)^\sharp = 0;
\end{align*} 
see also \eqref{eq:central_charges_Y}.
Introduce the following flat co-ordinates on $\widebar{H}$:   
\begin{align}
\label{eq:Darboux_central_charge} 
\begin{split} 
x & = \iu (2\pi\iu)^{-3/2} \Pi_Y(\cO_{\Proj^1}(-1))^\sharp \\
p & = -\iu (2\pi\iu)^{-3/2} \Pi_Y(\cO_{\Proj^2}(-1))^\sharp
\end{split} 
\end{align} 
where we set $\iu^{1/2} = e^{\pi\iu/4} = (1+\iu)/\sqrt{2}$.   
These co-ordinates are multi-valued:
they are originally defined near a point $y_0$ 
with $0<y_0\ll 1$, 
and then analytically continued over the 
universal cover of $\cMCY\setminus \DCY$. 
For example, if we analytically continue 
them to the orbifold point along the positive real line in the $y$-plane, 
we have, from the connection formula in Proposition~\ref{prop:monodromy}: 
\begin{align} 
\label{eq:Darboux_central_charge_X} 
\begin{split} 
x & =  \iu (2\pi\iu)^{-3/2} 
\left (-\Pi_\cX(\cO_0\otimes \varrho)^\sharp 
+ \Pi_\cX(\cO_0 \otimes \varrho^2)^\sharp \right) \\
p & = -\iu (2\pi\iu)^{-3/2} 
\Pi_\cX(\cO_0\otimes \varrho^2)^\sharp
\end{split} 
\end{align} 
$x$~and~$p$ give Darboux co-ordinates corresponding 
to an \emph{integral} basis of $K_c(Y)$ or $K_c(\cX)$. 

\begin{lemma} 
\label{lem:symplectic_pairing_in_flat} 
When we restrict $(p,x)$ to $\Haff$, we have 
$\Omega = \frac{1}{3} dp \wedge dx$.  
\end{lemma} 
\begin{proof} 
Set $\Pi_1 = \Pi_Y(\cO_{\Proj^1}(-1))$ 
and $\Pi_2=\Pi_Y(\cO_{\Proj^2}(-1))$. 
As $y \to 0$, we have (see equation~\eqref{eq:central_charges_Y}) 
\begin{align*} 
\Pi_1  & = - (2\pi\iu)^2 \log y 
+O(y)\\ 
\Pi_2  & = 
-(2\pi\iu) \left(\pi^2 + \pi\iu \log y -\frac{1}{2} (\log y)^2\right) 
+O(y\log y)  
\end{align*} 
The sections $\theta \zeta$, $\theta^2 \zeta \in \cO(\Hvec)$ 
form a fiberwise tangent frame of $\Haff$ near $y=0$. 
Since $x$ and $p$ are flat, it suffices to check that 
the asymptotics of $\Omega(\theta\zeta,\theta^2\zeta)$ 
and $\frac{1}{3}(dp\wedge dx)(\theta\zeta, \theta^2\zeta)$ 
agree. 
We have (see equation~\eqref{eq:symplectic_pairing}): 
\begin{align*} 
\Omega(\theta\zeta, \theta^2\zeta) & = \frac{1}{3(1+27y)} 
\sim \frac{1}{3}\\ 
(dp\wedge dx) (\theta\zeta,\theta^2 \zeta) 
 &= (2\pi\iu)^{-3} \left( \theta \Pi_2 \cdot 
 \theta^2 \Pi_1 
 - \theta^2 \Pi_2 \cdot \theta \Pi_1 \right) 
 \sim 1 
\end{align*} 
as $y\to 0$. The conclusion follows. 
\end{proof}

\begin{proposition} 
\label{prop:opposite_cusps_flat} 
Let $(p,x)$ be the coordinates of $\widebar{H}$ 
given by \eqref{eq:Darboux_central_charge}.  
If we analytically continue the co-ordinates $(p,x)$ 
along the paths shown in Figure $\ref{fig:paths}$, 
we have that:
\begin{itemize} 
\item[(a)] the opposite line bundle $P_\LR\subset \Hvec$ 
is cut out by $x = 0$; 
\item[(b)] the opposite line bundle $P_\con\subset \Hvec$ is 
cut out by $p=0$; 
\item[(c)] the opposite line bundle $P_\orb\subset \Hvec$ 
is cut out by $ x+(1-\xi) p=0$, where $\xi = e^{2\pi\iu/3}$. 
\end{itemize} 
\end{proposition} 
\begin{proof} 
The opposite line bundles $P_\LR$, $P_\con$, $P_\orb$ 
are flat subbundles of $\Hvec$ around the large-radius, conifold, and 
orbifold points respectively, and as such, 
they are necessarily invariant 
under the corresponding local monodromy. 
From the computation in Proposition \ref{prop:monodromy}, 
we find that $\{x=0\}$ is a unique invariant line 
in $\Hvec = \{\Pi_Y(\cO_\pt)^\sharp =0\}$ 
around the large-radius limit point; similarly 
$\{p=0\}$ is a unique invariant line 
around the conifold point. Parts (a) and (b) follow. 
The monodromy around the orbifold point is semisimple 
with eigenvalues $\{\xi,\xi^2\}$ 
and we have precisely two invariant lines 
given by 
$x+(1-\xi^i) p =0$, $i\in \{1,2\}$. 
On the other hand, the generator 
$v:=\nabla_{\partial/\partial \fry} \zeta$ 
of $F^2_\vec$ near the orbifold point  
has co-ordinates 
\begin{align*} 
x(v) & 
=  \iu (2\pi\iu)^{-3/2} \partial_{\fry}
\left(- \Pi_\cX(\cO_0\otimes \varrho) + 
\Pi_\cX(\cO_0\otimes \varrho^2) \right) = 
\iu \frac{(2\pi\iu)^{3/2}}{3\Gamma(\frac{2}{3})^3} (-\xi + \xi^2) 
+O(\fry) 
\\
p(v) &= -\iu (2\pi\iu)^{-3/2} \partial_{\fry} 
\Pi_\cX(\cO_0\otimes \varrho^2) 
= - \iu \frac{(2\pi\iu)^{3/2}}{3\Gamma(\frac{2}{3})^3} 
\xi^2 + O(\fry)  
\end{align*} 
where we used \eqref{eq:Darboux_central_charge_X} 
and the formula \eqref{eq:central_charges_X} 
for $\vPi_{\cX}$. 
Therefore $F^2_\vec|_{\fry=0}$ lies in the 
subspace 
$x + (1-\xi^2) p=0$. 
Part (c) follows since 
$P_\orb$ is transversal to $F^2_\vec$ near $\fry=0$. 
\end{proof} 

\begin{recapitulation} 
\label{rec:Caff}
Recall from \S\ref{sec:Yukawa_fd} that we can immerse 
the universal cover of $\cMCY \setminus \DCY$ into 
the fiber $\Haff|_{y_0}$ as an (immersed) Lagrangian submanifold 
$\cL$, by parallel translation of the primitive section $\zeta$. 
In terms of the ``integral'' 
co-ordinates $(p,x)$ 
on $\Haff|_{y_0}$ -- see equation~\eqref{eq:Darboux_central_charge} -- $\cL$ is given by:  
\[
p= p(\zeta) = \iu (2\pi\iu)^{-1/2} \left(\pi^2 + \pi \iu t + 3 
\parfrac{F_Y^0}{t} \right),  
\qquad 
x= x(\zeta) = -\iu(2\pi\iu)^{1/2} t 
\]
where $t=t(y)$ is the mirror map for $Y$ 
(see equation~\eqref{eq:central_charges_Y}), 
and the tangent space is: 
\begin{equation} 
\label{eq:slope_in_integral_coordinates}
T_{(x(\zeta),p(\zeta))} \cL = 
\C \begin{pmatrix} \tau \\ 1 \end{pmatrix} \quad 
\text{with $\tau = \parfrac{p(\zeta)}{x(\zeta)}
= -\frac{1}{2} - \frac{3}{2\pi\iu} \parfrac{^2F^0_Y}{t^2}$
}.   
\end{equation} 
We saw in Corollary~\ref{cor:modular_parameter} 
that the slope $\tau$ 
lies in $\HH$ and identifies 
the universal cover of $\cMCY \setminus \DCY$ with $\HH$. 
\end{recapitulation}

\begin{notation} 
\label{nota:GW_X_Y}
In this section, we denote by $F^g_X$ 
the genus-$g$ Gromov--Witten potential \eqref{eq:GW_potential} 
of $X = \cX$ or $Y$, restricted to the second cohomology  
and with Novikov parameters specialized to $Q=1$. 
As before, we write $t \mapsto t h \in H^2(Y)$, 
$\frt \mapsto \frt \unit_{\frac{1}{3}}\in H^2_\orb(\cX)$ 
for parameters on the second cohomology. 
Explicitly, we have 
(see equations~\eqref{eq:GW-wave_Y}--\eqref{eq:GW-wave_X} 
and \eqref{eq:genus-zero_pot_Y}--\eqref{eq:genus-zero_pot_X}):  
\begin{align*} 
F^0_Y(t) &= 
-\frac{1}{18} t^3 + \sum_{d=1}^\infty \corr{}_{0,0,d}^Y e^{dt}  \\ 
F^1_Y(t) & = -\frac{1}{12} t + 
\sum_{d=1}^\infty \corr{}_{1,0,d}^Y e^{dt} \\ 
F^g_Y(t) & = \sum_{d=0}^\infty \corr{}_{g,0,d}^Y e^{dt} && 
\text{for $g\ge 2$}\\ 
\intertext{and} 
F^g_\cX(\frt) &= \sum_{n:2g-2+n>0} 
\corr{\unit_{\frac{1}{3}},\dots,\unit_{\frac{1}{3}}}_{g,n,0}^\cX 
\frac{\frt^n}{n!} && \text{for $g\ge 0$}. 
\end{align*} 
\end{notation}

\begin{corollary}
\label{cor:analytic_HH} 
The following objects can be analytically continued to the 
universal cover $(\cMCY\setminus \DCY)\sptilde \cong \HH$ 
of $\cMCY\setminus \DCY$: 
\begin{itemize} 
\item[(a)] 
the opposite line bundles $P_\LR$, $P_\con$, $P_\orb$;  
\item[(b)] 
the Gromov--Witten potential $F^g_Y(t)$ of $Y$, 
when regarded 
as a function near the large-radius limit point $y=0$ 
via the mirror map $t=t(y)$; 
\item[(c)] 
the Gromov--Witten potential $F^g_\cX(\frt)$ of $\cX$, 
when regarded as a function near the orbifold point $\fry=0$ 
via the mirror map $\frt = \frt(y)$.  
\end{itemize} 
\end{corollary}
\begin{proof}  
We use notation as in Recapitulation \ref{rec:Caff}. 
Since $\tau$ is a non-zero holomorphic function on 
$(\cMCY\setminus \DCY)\sptilde$, 
it follows from Proposition \ref{prop:opposite_cusps_flat} 
that $\cL$ is transversal to both $P_\LR$ 
and $P_\con$ \emph{everywhere} on 
$(\cMCY\setminus \DCY)\sptilde$, 
i.e.~$P_\LR$ and $P_\con$ 
extend to opposite line bundles over 
$(\cMCY\setminus \DCY)\sptilde$. 
Similarly, since $1+(1-\xi) \tau$ never vanishes 
for $\tau \in \HH$, $P_\orb$ also extends to 
the universal cover. 
Part (a) follows. Parts (b) and (c) follow from Part (a) and 
the fact that the Gromov--Witten potentials of $Y$ and $\cX$ extend 
to a global section $\wave_{\CY}$ of $\Fock_{\CY}^\circ$:
see Theorem~\ref{thm:wave_CY}. 
\end{proof} 

\begin{notation}[Darboux co-ordinates at cusps]  
\label{nota:Darboux_cusps}  
Let $(p,x)$ be the ``integral'' Darboux co-ordinates on $\Haff$ 
given in \eqref{eq:Darboux_central_charge} 
and let $\tau$ be the slope \eqref{eq:slope_in_integral_coordinates} 
of $\cL$ in these co-ordinates.  
In view of Proposition \ref{prop:opposite_cusps_flat}, 
we introduce the following Darboux co-ordinates 
on $\Haff|_{y_0}$ 
associated to the large-radius, conifold and orbifold points 
(cf.~Examples \ref{ex:Yukawa}, \ref{ex:propagator}). 
Here all Darboux co-ordinates $(p',x')$ are normalized so that 
$\Omega = \frac{1}{3} dp' \wedge dx'$. 

\begin{itemize} 
\item[(1)] To the large-radius limit point, we associate the 
Darboux co-ordinates  
\begin{align*} 
\begin{pmatrix} p_\LR \\ x_\LR \end{pmatrix} 
= 
\iu (2\pi\iu)^{-1/2} 
\begin{pmatrix} 
- 2\pi\iu & - \pi \iu \\ 
0 & 1 
\end{pmatrix} 
\begin{pmatrix} p \\ x \end{pmatrix} 
+ 
\begin{pmatrix} -\pi^2 \\ 0 \end{pmatrix} 
\end{align*} 
such that $P_\LR = \langle \partial/\partial p_\LR \rangle$. 
In these co-ordinates, $\cL$ and its slope 
are given by: 
\[
\begin{cases} 
p_\LR = 3 \parfrac{F_Y^0}{t}\\
x_\LR = t
\end{cases} 
\ \text{and} \quad 
\tau_\LR = \parfrac{p_\LR}{x_\LR} 
= - 2\pi\iu\left(\tau+\frac{1}{2}\right) 
= 3 \parfrac{^2 F_Y^0}{t^2} 
\]
where $t=t(y)$ is the mirror map for $Y$ 
(see \S\ref{subsubsec:monodromy}). 

\item[(2)] To the conifold point, we associate the Darboux 
co-ordinates: 
\[
\begin{pmatrix} 
p_\con \\ x_\con 
\end{pmatrix} 
= \frac{1}{\sqrt{3}} (2\pi\iu)^{-1/2}
\begin{pmatrix} 
0 & 2\pi\iu
\\ 
- 3 & 0
\end{pmatrix} 
\begin{pmatrix} p \\ x \end{pmatrix}
\]
such that $P_\con = \langle \partial/\partial p_\con \rangle$. 
In these co-ordinates, the slope of $\cL$ is given by: 
\[
\tau_\con= \parfrac{p_\con}{x_\con}
= - \frac{2\pi\iu}{3\tau} 
\]
\item[(3)] 
To the orbifold point, we associate the Darboux co-ordinates 
\[
\begin{pmatrix} 
p_\orb \\ x_\orb 
\end{pmatrix} 
= \frac{1}{\iu (2\pi\iu)^{3/2}} 
\begin{pmatrix} 
3\Gamma(\frac{1}{3})^3 & (1-\xi) 
\Gamma(\frac{1}{3})^3 \\ 
3 \Gamma(\frac{2}{3})^3 & 
(1-\xi^2) \Gamma(\frac{2}{3})^3 
\end{pmatrix} 
\begin{pmatrix} 
p \\ x 
\end{pmatrix} 
+ \begin{pmatrix} 
\Gamma(\frac{1}{3})^3 \\ 
\Gamma(\frac{2}{3})^3 
\end{pmatrix} 
\]
such that $P_\orb = \langle \partial/\partial p_\orb \rangle$. 
In these co-ordinates, $\cL$ and its slope are 
given by: 
\begin{align*} 
\begin{cases} 
p_\orb = 3 \parfrac{F^0_\cX}{\frt}  \\ 
x_\orb = \frt 
\end{cases} 
\ \text{and} \quad 
\tau_\orb =  
\parfrac{p_\orb}{x_\orb} 
= \frac{\Gamma(\frac{1}{3})^3}{\Gamma(\frac{2}{3})^3} \cdot 
\frac{3 \tau + 1-\xi}{3 \tau + 1-\xi^2} 
= 3 \parfrac{^2 F^0_\cX}{\frt^2}
\end{align*} 
where $\frt = \frt(\fry)$ is the mirror map for $\cX$ 
(see \S\ref{subsubsec:monodromy}). 
\end{itemize} 
\end{notation} 
\begin{remark} 
The slope parameters $\tau_\LR, \tau_\con, \tau_\orb$ 
take values, respectively, in the right-half plane 
$\{\tau_\LR \in \C : \Re(\tau_\LR) >0\}$, 
the left-half plane $\{\tau_\con\in \C : \Re(\tau_\con)<0\}$ and 
the disc $\{\tau_\orb \in \C : |\tau_\orb|< \Gamma(\frac{1}{3})^3/
\Gamma(\frac{2}{3})^3\}$. 
In particular, the inequalities 
\[
\Re\left( \parfrac{^2F_Y^0}{t^2} \right) >0, \qquad 
\left| \parfrac{^2F_\cX^0}{\frt^2} \right| < 
\frac{\Gamma(\frac{1}{3})^3}{3 \Gamma(\frac{2}{3})^3} 
\]
hold. 
\end{remark} 
\begin{remark} 
\label{rem:x_con_asymptotics}
The conifold co-ordinate $x_\con = \sqrt{3} (2\pi\iu)^{-1/2} p$ 
restricted to $\cL$ can be written as the integral: 
\[
x_\con = \frac{\sqrt{3}}{2\pi} \int_{\Gamma_\R} \zeta 
\]
for $-1/27<y<0$ by Lemma \ref{lem:real_Lefschetz}. 
In particular, $\theta x_\con$ is a period over a vanishing cycle 
at the conifold point (see \eqref{eq:derivative_central_charge}). 
We also obtain the asymptotics $x_\con \sim 1+27 y$ near the 
conifold point by approximating the above integral by the area of an ellipse. 
Since $x_\con$ is invariant under the conifold monodromy, 
it is holomorphic near $y=-1/27$. 
\end{remark} 

The slope parameters in Notation \ref{nota:Darboux_cusps} 
are related to each other by: 
\begin{align*} 
\tau_\LR & = \frac{-3\pi\iu \tau_\con - 4\pi^2}{3\tau_\con}, 
&& 
\tau_\con = \frac{-4\pi^2}{3(\tau_\LR + \pi \iu)} 
\\
\tau_\LR & = -\frac{\pi}{\sqrt{3}} 
\frac{\Gamma(\frac{2}{3})^3 \tau_\orb + \Gamma(\frac{1}{3})^3}
{\Gamma(\frac{2}{3})^3 \tau_\orb - \Gamma(\frac{1}{3})^3} 
&& 
\tau_\orb= \frac{\Gamma(\frac{1}{3})^3}{\Gamma(\frac{2}{3})^3} 
\frac{\sqrt{3} \tau_\LR-\pi}{\sqrt{3} \tau_\LR +\pi} 
\\
\tau_\con & = 2\pi\iu 
\frac{\Gamma(\frac{2}{3})^3 \tau_\orb - \Gamma(\frac{1}{3})^3}
{(1-\xi^2)\Gamma(\frac{2}{3})^3 \tau_\orb - 
(1-\xi) \Gamma(\frac{1}{3})^3}
&& 
\tau_\orb = \frac{\Gamma(\frac{1}{3})^3}{\Gamma(\frac{2}{3})^3} 
\frac{(1-\xi)\tau_\con - 2\pi\iu}{(1-\xi^2)\tau_\con - 2\pi\iu} 
\end{align*} 
Therefore, by Example \ref{ex:propagator}, the propagators 
among the opposite line bundles $P_\LR$, $P_\con$, $P_\orb$ 
are given as follows: 
\begin{align}
\label{eq:propagator_explicit} 
\begin{aligned} 
\Delta(P_\LR,P_\con) & = \frac{-3}{\tau_\LR + \pi \iu} 
(\partial_{x_\LR})^{\otimes 2}  & 
\Delta(P_\orb,P_\con) & = \frac{-3\Gamma(\frac{2}{3})^3}
{\Gamma(\frac{2}{3})^3 \tau_\orb + \xi \Gamma(\frac{1}{3})^3}
(\partial_{x_\orb})^{\otimes 2}
\\ 
\Delta(P_\con,P_\LR)  & = -\frac{3}{\tau_\con} 
(\partial_{x_\con})^{\otimes 2} & 
\Delta(P_\con, P_\orb) & = \frac{-3(1-\xi^2)}{(1-\xi^2)\tau_\con - 2\pi\iu} 
(\partial_{x_\con})^{\otimes 2} \\ 
\Delta(P_\LR,P_\orb) & = \frac{-3\sqrt{3}}{\sqrt{3}\tau_\LR + \pi} 
(\partial_{x_\LR})^{\otimes 2} &
\Delta(P_\orb,P_\LR) & = \frac{-3\Gamma(\frac{2}{3})^3}
{\Gamma(\frac{2}{3})^3 \tau_\orb - \Gamma(\frac{1}{3})^3}
(\partial_{x_\orb})^{\otimes 2} 
\end{aligned} 
\end{align}
where $x_\LR= t(y)$, $x_\orb = \frt(\fry)$ 
are the mirror maps for $Y$ and $\cX$ respectively. 
The correlation functions of $\wave_{\CY}$ 
with respect to $P_\LR$, 
$P_\con$, $P_\orb$ are related by Feynman rules
given by these propagators. In particular, we get: 

\begin{theorem}[Crepant Resolution Conjecture for $\cX$: explicit form]  
\label{thm:CRC_explicit} 
As in Corollary $\ref{cor:analytic_HH}$, we regard the Gromov--Witten 
potentials of $Y$ and $\cX$ as holomorphic functions on the universal 
cover of $\cMCY\setminus \DCY$, 
and we use Notation $\ref{nota:GW_X_Y}$. 
After analytic continuation along the positive real line in the $y$-plane, 
the Gromov--Witten potentials of $Y$ and $\cX$ are related by 
a Feynman rule as in \S $\ref{sec:Fock_fd}$:  
\begin{align*} 
(\partial_{\frt}^3 F_\cX^0) \cdot  (d\frt)^{\otimes 3} 
&= (\partial_{t}^3 F_Y^0)  \cdot (dt)^{\otimes 3} 
= -\frac{1}{3(1+27y)} \left(\frac{dy}{y}\right)^{\otimes 3}
\\ 
(\partial_{\frt} F_\cX^1) d\frt  
& = \left(\partial_{t} F_Y^1 
+ \frac{1}{2} (\partial_{t}^3 F_Y^0) \Delta \right) 
dt  
\\ 
F_\cX^g & = \sum_{\Gamma} 
\frac{1}{|\Aut(\Gamma)|} 
\Cont_\Gamma\left( \Delta, 
\{\partial_t^\bullet \cF_Y^h : h\le g\} \right) && 
\text{for $g\ge 2$} 
\end{align*} 
where $\Gamma$ in the third line ranges over all connected 
stable decorated genus-$g$ graphs without legs, 
$\Delta$ is the propagator from $P_\LR$ to $P_\orb$: 
\[
\Delta = 
-\frac{3\sqrt{3}}{\sqrt{3}\tau_\LR+\pi} \qquad 
\text{with $\tau_\LR = 3\parfrac{^2F^0_Y}{t^2}$,} 
\] 
$\Cont_\Gamma(\Delta,\{\partial_t^\bullet \cF_Y^h : h\le g\})$ 
denotes the contraction along the graph $\Gamma$ 
with edge terms $\Delta$ and vertex terms  
$\partial_t^\bullet \cF_Y^h$ -- see the explanation after \eqref{eq:Feynman_rule_fd} -- and
\[
d\frt = -\frac{\sqrt{3}}{4\pi^2} \Gamma\left(\frac{2}{3}\right)^3 
\left(\sqrt{3} \tau_\LR +\pi \right) dt. 
\]
\end{theorem} 
\begin{proof} 
This follows from Theorem \ref{thm:wave_CY} and the definition 
of the finite-dimensional Fock sheaf $\Fock_{\CY}^\circ$. 
See Example \ref{ex:Yukawa_conformal_limit} for the Yukawa 
coupling (genus-zero term) and 
\eqref{eq:propagator_explicit} 
for the propagator. 
\end{proof}

\begin{remark} 
By the general theory developed in \S\ref{sec:Fock_fd}, we can invert 
the Feynman rule in the above theorem, 
by exchanging $F^g_\cX$ and $F^g_Y$, 
$x_\orb = \frt$ and $x_\LR = t$, 
and replacing $\Delta$ with 
$-3 \Gamma(\frac{2}{3})^3/(\Gamma(\frac{2}{3})^3 \tau_\orb - 
\Gamma(\frac{1}{3})^3)$. 
\end{remark} 

\begin{example} 
In Theorem \ref{thm:CRC_explicit}, 
the Feynman rule at genus two takes the form: 
\begin{align*} 
F^2_\cX & = F^2_Y + \frac{1}{2} \Delta 
(\partial_t^2 F^1_Y) + 
\frac{1}{2} \Delta( \partial_t F^1_Y)^2 
+ \frac{1}{2} \Delta^2 
(\partial_t F^1_Y)(\partial_t^3 F^0_Y) 
+ \frac{1}{8} \Delta^2 (\partial_t^4 F^0_Y)
+ \frac{5}{24} \Delta^3 (\partial_t^3 F^0_Y)^2.   
\end{align*} 
\end{example} 

\subsection{Algebraic and Complex Conjugate Opposite Line Bundles}
\label{subsec:alg_cc}
Recall the notion of curved opposite line bundle from \S\ref{subsec:curved}.
In this section we introduce two curved opposite line bundles 
$P_\alg$ and $P_\cc$. 
The algebraic opposite line bundle $P_\alg$ is a holomorphic subbundle of 
$\Hvec$ which is opposite to $F^2_\vec$ but is not flat; 
the complex conjugate opposite line bundle $P_\cc$ 
is a $C^\infty$-subbundle 
of $\Hvec$ which is opposite to $F^2_\vec$ but is not flat 
in the antiholomorphic direction.  
The key property of these line bundles is that they 
are \emph{single-valued} over $\cMCY\setminus \DCY$, 
and therefore they yield single-valued correlation functions 
of the global section $\wave_{\CY}$ of $\Fock_{\CY}^\circ$. 
This property plays a crucial role in the next section. 

As explained in \S\ref{subsec:opposite_cusps_CRC}, 
the central charges $\Pi_Y(V)$ of $V\in K_c(Y)$ 
give flat co-ordinates $\Pi_Y(V)^\sharp$ 
on $\widebar{H}$, and thus on $\Hvec$.   
Since the $\Z$-lattice formed by these co-ordinates $\Pi_Y(V)^\sharp$ 
is preserved under monodromy, they determine a real flat subbundle 
of $\Hvec|_{\cMCY\setminus \DCY}$  
\[
H_{\vec,\R} := \left\{ v \in \Hvec\big|_{\cMCY\setminus \DCY} 
: \text{$\iu(2\pi\iu)^{-3/2}  \Pi_Y(V)^\sharp (v) \in \R$ for all  $V\in K_c(Y)$} \right\}
\]
with the property that $\Hvec = H_{\vec,\R} \oplus \iu H_{\vec,\R}$. 
Recall from Corollary \ref{cor:VHS_isom} that 
$\Hvec|_y$ is isomorphic to $H^1(E_y,\C)$ 
for $y\in \cMCY\setminus \DCY$, 
and that via this isomorphism, $\Pi_Y(V)^\sharp$ corresponds\footnote{Up to a constant -- see equations~\eqref{eq:central_charge_is_period} and~\eqref{eq:derivative_central_charge}.} to 
the integration over the integral cycle $\partial\Mir(V)$.
 By scaling the isomorphism $\Hvec \cong 
\bigcup_y H^1(E_y,\C)$ by a constant,  therefore, 
we have that $H_{\vec,\R} \cong H^1(E_y,\R)$.

\begin{definition}[complex conjugate opposite]  
\label{def:Pcc} 
The \emph{complex conjugate opposite line bundle} 
$P_\cc$ is defined to be the $C^\infty$ complex 
subbundle of $\Hvec|_{\cMCY\setminus \DCY}$ 
given as the complex conjugate 
of $F^2_\vec$ with respect to the real form $H_{\vec,\R}$. 
\end{definition} 

Since $P_\cc$ is the complex conjugate of a holomorphic subbundle, 
$P_\cc$ is flat in the holomorphic direction -- that is, 
$\nabla_v C^\infty(P_\cc) \subset C^\infty(P_\cc)$ for 
any $(1,0)$-vector field $v$.  It is not 
flat in the antiholomorphic direction. 

\begin{lemma} 
\label{lem:Pcc} 
The line bundle $P_\cc\subset \Hvec|_{\cMCY\setminus \DCY}$ 
extends to a topological line subbundle of $\Hvec$ over 
$\cMCY$ such that $F^2_\vec \oplus P_\cc = \Hvec$ 
holds globally. 
Moreover, we have 
\[
P_\cc|_{y=0} = P_\LR|_{y=0}, \qquad  
P_\cc|_{y=-\frac{1}{27}} = P_\con|_{y=-\frac{1}{27}}, \qquad
P_\cc|_{y=\infty} = P_\orb|_{y=\infty}.
\] 
\end{lemma}
\begin{proof} 
Under the isomorphism $\Hvec|_y \cong H^1(E_y,\C)$, 
$F^2_\vec|_y$ corresponds to $H^{1,0}(E_y,\C) \cong \C \lambda_y$. 
As discussed above, we have $H_{\vec,\R}|_y \cong H^1(E_y,\R)$. 
The Hodge decomposition implies that 
$P_\cc|_y$ corresponds to $H^{0,1}(E_y,\C)$ 
and is opposite to $F^2_\vec|_y$ over $\cMCY\setminus \DCY$. 

The extension of $P_\cc$ across $\DCY$ and the oppositeness there 
follow from a property of the nilpotent orbit  
associated to a degeneration of Hodge structure (see \cite{Schmid}).  
We will give an elementary account below. 
Choose one of the limit points from $\DCY = \{0,-\frac{1}{27}\}$ and 
let $t$ denote a local co-ordinate centred at that point. 
From the description of $\widebar{H}$ in \S\ref{sec:two}, 
we can find a local basis $\{s_0,s_1\}$ of $\Hvec$ near $t=0$ 
such that $F^2_\vec = \langle s_1 \rangle$ 
and that the connection $\nabla$ is of the form: 
\[
(\nabla s_0, \nabla s_1) 
= (s_0, s_1) A(t) \frac{dt}{t} \qquad 
\text{with $A(0) = 
\begin{pmatrix} 0 & 1 \\ 0 &0  \end{pmatrix}$} 
\]
Then we can find a basis $\{f_0,f_1\}$ of flat sections
of the form (see e.g.~Proposition 
\ref{pro:fundamental_solution_logarithmic_connection}): 
\[
(f_0,f_1) = (s_0,s_1) G(t) 
\begin{pmatrix} 1 & -\log t \\ 0 & 1 \end{pmatrix} 
\]
where $G(t)$ is a holomorphic matrix-valued function 
near $t=0$ with $G(0) = I_2$. 
The flat section $f_0$ spans a monodromy-invariant line; hence 
after scaling $\{s_0,s_1\}$ by a constant, 
we may assume that $f_0\in H_{\vec,\R}$. 
Let $f_2 =a f_0 + b f_1$, $b\neq 0$ be another flat section 
taking values in $H_{\vec,\R}$ and linearly independent of $f_0$. 
The monodromy acts on $f_2$ as $f_2 \mapsto f_2 - 2\pi\iu b f_0$; 
 reality of the monodromy then implies $b\in \iu \R$. 
The complex conjugate $\ov{f_1}$ of $f_1$ 
with respect to $H_{\vec,\R}$ is then computed as:  
\[
\ov{f_1} = \frac{\ov{a}-a}{b}  f_0 - f_1. 
\]
Thus the complex conjugate of $\{s_0,s_1\}$ with respect to 
$H_{\vec,\R}$ is: 
\begin{align*} 
(\ov{s_0},\ov{s_1}) 
&= (\ov{f_0}, \ov{f_1}) 
\begin{pmatrix} 
1 & \ov{\log t} \\ 0 & 1 
\end{pmatrix} 
\ov{G(t)}^{-1}
= (f_0,f_1) 
\begin{pmatrix} 1 & \frac{\ov{a}-a}{b} \\ 0 & -1 \end{pmatrix} 
\begin{pmatrix} 
1 & \ov{\log t} \\ 0 & 1 
\end{pmatrix} 
\ov{G(t)}^{-1} \\ 
& = (s_0,s_1) G(t) 
\begin{pmatrix}
1 & -\log t \\ 0 & 1 
\end{pmatrix} 
\begin{pmatrix} 1 & \frac{\ov{a}-a}{b} \\ 0 & -1 \end{pmatrix} 
\begin{pmatrix} 
1 & \ov{\log t} \\ 0 & 1 
\end{pmatrix} 
\ov{G(t)}^{-1} \\
& = (s_0,s_1) \left[ 
\begin{pmatrix} 
1 & \frac{\ov{a} -a}{b} + 2 \log |t| \\ 
0 & -1
\end{pmatrix} + O(|t| \log|t|)\right] 
\end{align*} 
Since $P_\cc = \langle \ov{s_1} \rangle$, 
this implies that $P_\cc$ extends across $t=0$ 
as a topological line bundle, and that the fiber at $t=0$ is 
spanned by the invariant section $f_0|_{t=0}=s_0|_{t=0}$. 
This also shows the oppositeness of $P_\cc$ along $\DCY$ 
and that $P_\cc|_{y=0} = P_\LR|_{y=0}$ 
and $P_\cc|_{y=-\frac{1}{27}} = P_\con|_{y=-\frac{1}{27}}$ 
(we showed in the proof of Proposition \ref{prop:opposite_cusps_flat} 
that $P_\LR$ and $P_\con$ are spanned by invariant sections  
near cusps). 
To show that $P_\orb|_{y=\infty} = P_\cc|_{y=\infty}$, 
it suffices to note that these subspaces are uniquely characterized 
by invariance under $\mu_3$-monodromy and oppositeness 
to $F^2_\vec|_{y=\infty}$.  
\end{proof}  

Recall from \S\ref{sec:two} that $\cO(\Hvec)$ is 
isomorphic to $\cO(1)\oplus \cO(-1)$ as a 
vector bundle over $\cMCY \cong \Proj(3,1)$, 
and that the subsheaf $\cO(F^2_\vec)\cong \Theta(\log\{0\})$ 
is isomorphic to $\cO(1)$. 
There is a (precisely) one-dimensional family of holomorphic line subbundles 
of $\Hvec$ which correspond to the factor $\cO(-1)$. 
\begin{definition}[algebraic opposite]  
\label{def:Palg} 
The \emph{algebraic opposite line bundle} 
$P_\alg = P_\alg(a)$, where $a\in \C$,
is the holomorphic line subbundle of $\Hvec$ 
with basis given by: 
\begin{align*}
\begin{aligned}  
s_0 &=(9y +a) \theta\zeta + (1+27y) \theta^2\zeta 
& & \text{over $\cMCY\setminus \{y=\infty\}$} \\ 
s_\infty & = (1-3a) \fry^2 \partial_{\fry} \zeta+ (27+\fry^3)
\partial_{\fry}^2 \zeta = 9 \fry s_0
& & \text{over $\cMCY\setminus \{y=0\}$} 
\end{aligned} 
\end{align*} 
where 
$\theta = y\parfrac{}{y}$ and 
$\partial_{\fry} = \parfrac{}{\fry}$ act via the connection $\nabla$. 
Any holomorphic line subbundle of $\Hvec$ which is globally 
complementary to $F^2_\vec$ is of the form $P_\alg(a)$ for some $a \in \C$. 
\end{definition} 

\begin{lemma} 
Let $P_\alg(a)$ be the algebraic opposite line in Definition $\ref{def:Palg}$. 
We have 
\[
P_\alg(0)|_{y=0}=P_\LR|_{y=0}, \quad 
P_\alg(\tfrac{1}{3})|_{y=-\frac{1}{27}} = P_\con|_{y=-\frac{1}{27}}, 
\quad 
P_\alg(a)|_{y=\infty} = P_\orb|_{y=\infty}, 
\]
for all $a \in \C$.
\end{lemma} 
\begin{proof} 
This follows from Notation \ref{nota:opposite_fd} 
and Proposition \ref{pro:opposite_cusps}. 
\end{proof} 

\begin{proposition} 
\label{prop:propagator_with_cc}
We use Notation $\ref{nota:Darboux_cusps}$. 
The propagators between $P_\LR$, $P_\con$, $P_\orb$ 
and $P_\cc$ are given as follows: 
\begin{align*} 
\Delta(P_\LR, P_\cc) & = 
\frac{-3}{\tau-\ov{\tau}} (\partial_x)^{\otimes 2}
= \frac{3}{2\pi\iu(\tau-\ov{\tau})} 
(\partial_{x_\LR})^{\otimes 2} 
= -\frac{3}{\tau_\LR + \ov{\tau_\LR}} 
(\partial_{x_\LR})^{\otimes 2} 
\\ 
\Delta(P_\con, P_\cc) & = -3 \frac{\ov\tau}{\tau} 
\frac{1}{\tau-\ov\tau} (\partial_x)^{\otimes 2}
= -\frac{9}{2\pi\iu} \frac{|\tau|^2}{\tau-\ov\tau} 
(\partial_{x_\con})^{\otimes 2} 
= - \frac{3}{\tau_\con + \ov{\tau_\con}} 
(\partial_{x_\con})^{\otimes 2}
\\
\Delta(P_\orb,P_\cc) & = - 3 \frac{1+\ov{\tau}(1-\xi)}
{1+\tau(1-\xi)} \frac{1}{\tau-\ov\tau} (\partial_x)^{\otimes 2} 
= \frac{3\Gamma(\frac{2}{3})^6 }{(2\pi\iu)^3} 
\frac{(3\tau +1-\xi^2)(3\ov{\tau}+1-\xi^2)}{\tau-\ov{\tau}}
(\partial_{x_\orb})^{\otimes 2} \\ 
& = \frac{3\ov{\tau_\orb}}{ \Gamma(\frac{1}{3})^6 
\Gamma(\frac{2}{3})^{-6} - |\tau_\orb|^2} 
(\partial_{x_\orb})^{\otimes 2}
\end{align*} 
where we regard $x$,~$x_\LR$,~$x_\con$,~$x_\orb$ as co-ordinates 
on the immersed submanifold $\cL\looparrowright \Haff$, 
or on the universal cover of $\cMCY\setminus \DCY$ 
(see Recapitulation $\ref{rec:Caff}$). 
\end{proposition} 
\begin{proof} 
Since $(p,x)$ are real co-ordinates with respect to $H_{\vec,\R}$, 
the complex conjugation in these co-ordinates is 
the ordinary one. Written in this frame, we have:
\[
F^2_\vec = \C \begin{pmatrix} \tau \\ 1 \end{pmatrix}, 
\qquad 
P_\cc = \C \begin{pmatrix} \ov{\tau} \\ 1\end{pmatrix}, 
\quad 
\text{and}
\quad  
\KS(\partial_x) = 
\begin{pmatrix} \tau \\ 1 \end{pmatrix}. 
\] 
Hence by writing $\Pi_\cc \colon \Hvec \to F^2_\vec$ for the projection 
along $P_\cc$, we have 
\[
\KS^{-1} \Pi_\cc (\partial_p) = \frac{1}{\tau-\ov\tau}\partial_x, 
\qquad 
\KS^{-1} \Pi_\cc(\partial_x) = -\frac{\ov\tau}{\tau-\ov\tau} \partial_x.  
\]
Let $\Pi_\LR$, $\Pi_\orb$, $\Pi_\con \colon 
\Hvec \to F^2_\vec$ denote the projections 
along $P_\LR$, $P_\orb$, $P_\con$ respectively. 
Since $\KS^{-1} \Pi_\LR(\partial_p) =0$, $\KS^{-1}\Pi_\LR(\partial_x) 
= \partial_x$, we have (see Definition \ref{def:propagator_fd}) 
\[ 
\Delta(P_\LR, P_\cc) = 
(\KS^{-1}\otimes \KS^{-1}) (\Pi_\LR \otimes \Pi_\cc) 
(3 \partial_p \otimes  \partial_x - 3 \partial_x \otimes \partial_p) 
= - \frac{3}{\tau-\ov\tau} (\partial_x)^{\otimes 2}.   
\]
The other formulae can be obtained similarly 
using Notation~\ref{nota:Darboux_cusps} and 
\begin{align*} 
\KS^{-1} \Pi_\con(\partial_p) &= \frac{1}{\tau}\partial_x, 
&& \KS^{-1}\Pi_\con(\partial_x) = 0, \\
\KS^{-1}\Pi_\orb(\partial_p) &= \frac{1-\xi}{1 + \tau (1-\xi)} \partial_x, 
&& 
\KS^{-1} \Pi_\orb(\partial_x) = \frac{1}{1+\tau(1-\xi)} \partial_x 
\end{align*} 
which we deduce easily from Proposition \ref{prop:opposite_cusps_flat}. 
\end{proof} 
\begin{remark} 
The propagators $\Delta(P,P_\cc)$ with $P=P_\LR$, $P_\con$, $P_\orb$
approach zero at the corresponding limit points, confirming again 
the conclusion of Lemma~\ref{lem:Pcc}.
\end{remark} 

\begin{lemma} 
\label{lem:Wronskian} 
For any flat affine Darboux co-ordinates $(\tp,\tx)$ on $\Haff$ 
with $\Omega = \frac{1}{3} d\tp \wedge d\tx$,  
we have $\theta \tp(\zeta) \cdot \theta^2 \tx(\zeta) - 
\theta^2 \tp(\zeta) \cdot \theta \tx(\zeta) = (1+27y)^{-1}$,
where $\theta = y\parfrac{}{y}$.  
\end{lemma}
\begin{proof} 
This follows from $3\Omega(\theta\zeta,\theta^2\zeta) = (1+27y)^{-1}$:
see \eqref{eq:symplectic_pairing}. 
\end{proof} 

Let $E_2(\tau)$ and $\hE_2(\tau)$ 
denote the second Eisenstein series 
and its modular counterpart: 
\begin{align*} 
E_2(\tau)  = 1 - 24 \sum_{n=1}^\infty \frac{n Q^n}{1-Q^n} &&
\hE_2(\tau) = E_2(\tau) + \frac{6}{\pi\iu} \frac{1}{\tau-\ov\tau} 
\end{align*} 
with $Q= e^{2\pi\iu \tau}$. 
Then we have \cite{Kaneko--Zagier}: 
\begin{align}
\label{eq:E2_modular} 
\begin{split}  
E_2\left(\frac{a \tau + b}{c\tau +d}\right) & = (c\tau+d)^2 E_2(\tau) + 
\frac{6 c(c\tau + d)}{\pi\iu} \\ 
\hE_2 \left(\frac{a\tau +b}{c\tau + d}\right) & = (c\tau+d)^2 
\hE_2(\tau) 
\end{split} 
\end{align} 
for every $\begin{pmatrix} a & b \\ c & d \end{pmatrix}\in 
\SL(2,\Z)$. 

\begin{proposition} 
\label{prop:propagator_cc_alg} 
Let $P_\alg=P_\alg(a)$ be the algebraic opposite line bundle in 
Definition~$\ref{def:Palg}$.  Use Notation~$\ref{nota:Darboux_cusps}$. 
The propagator between $P_\cc$ and $P_\alg$ is given by 
\begin{align}
\label{eq:propagator_cc_alg}  
\Delta(P_\cc,P_\alg) 
&=3 \left( \frac{1}{\tau-\ov\tau} - 
\theta x \cdot 
\left((9y +a)  \theta x
+ (1+27 y)  \theta^2 x \right)   
\right) \partial_x \otimes \partial_x \\
\label{eq:propagator_cc_alg_E2} 
& = \frac{\pi\iu}{2} \hE_2(\tau) \partial_x \otimes \partial_x 
+ 3 \left(\frac{1}{12} - a\right) \theta \otimes \theta. 
\end{align} 
where we regard $x=x(\zeta)$ as a co-ordinate 
on the immersed submanifold $\cL\looparrowright \Haff$, 
or on the universal cover $\HH$ of $\cMCY\setminus \DCY$ 
(see Recapitulation $\ref{rec:Caff}$). 
\end{proposition} 
\begin{proof} 
In terms of the integral Darboux co-ordinates $(p,x)$ 
in \eqref{eq:Darboux_central_charge}, $P_\alg$ 
is given by 
\[
P_\alg = \C \begin{pmatrix} 
(9 y +a) \theta p(\zeta) + (1+27y) \theta^2p(\zeta) \\ 
(9 y + a) \theta x(\zeta) + (1+27y) \theta^2 x(\zeta)
\end{pmatrix} 
\]
Set $A := \theta x(\zeta) \cdot 
((9 y + a) \theta x(\zeta) + (1+27y) \theta^2 x(\zeta))$. 
Using Lemma \ref{lem:Wronskian}, we find 
\[
P_\alg = \C 
\begin{pmatrix} 
\frac{\theta p(\zeta)}{\theta x(\zeta)} A -1 \\ A 
\end{pmatrix} 
= \C \begin{pmatrix} \tau A -1 \\ A \end{pmatrix}. 
\]
Arguing as in Proposition \ref{prop:propagator_with_cc}, 
we find: 
\[
\Delta(P_\cc, P_\alg) = 3 \left(\frac{1}{\tau-\ov\tau} -A\right) 
 \partial_x \otimes \partial_x. 
\] 
This shows \eqref{eq:propagator_cc_alg}. 
Next we show that the expressions \eqref{eq:propagator_cc_alg} 
and \eqref{eq:propagator_cc_alg_E2} coincide. 
Recall that $(p,x)$ and $\tau$ transform 
under monodromy as (see Proposition \ref{prop:monodromy}) 
\[
\begin{pmatrix} p \\ x \end{pmatrix} 
\mapsto 
\begin{pmatrix} a & b \\ c & d \end{pmatrix} 
\begin{pmatrix} p \\ x \end{pmatrix} + 
\begin{pmatrix} * \\ * \end{pmatrix}, 
\qquad 
\tau \mapsto \frac{a \tau+b}{c\tau+d} 
\]
with $\begin{pmatrix} a & b \\ c & d \end{pmatrix} 
\in \Gamma_1(3)$,   whence $\theta x = \theta x(\zeta)$ transforms as: 
\begin{equation} 
\label{eq:theta_x_weight_1}
\theta x \mapsto c (\theta p) + d  (\theta x) 
= (c\tau +d) \theta x.  
\end{equation} 
Therefore the modular transformation property 
\eqref{eq:E2_modular} implies that 
\[
\hE_2(\tau)  \partial_x \otimes \partial_x 
= \hE_2(\tau) (\theta x)^{-2} \theta \otimes \theta 
\]
is invariant under monodromy,
and that the expression \eqref{eq:propagator_cc_alg_E2} 
descends to a single-valued bivector field on $\cMCY \setminus \DCY$. 
Moreover, \eqref{eq:propagator_cc_alg_E2} 
extends to a continuous global section of 
$\Theta(\log \{0\})^{\otimes 2}$ since 
$\hE_2(\tau) \to 1$ and 
$\theta x(\zeta) \to -\iu (2\pi\iu)^{1/2}$ 
in the large-radius limit $\tau \to +\infty \iu$,
and 
\[
\hE_2(\tau) \partial_x \otimes \partial_x 
= \hE_2(\tau) \frac{-3\tau^2}{2\pi\iu} \partial_{x_\con}\otimes 
\partial_{x_\con} 
= - \frac{3}{2\pi\iu} \hE_2(-1/\tau) 
\partial_{x_\con}\otimes \partial_{x_\con} 
\]
tends to $-3/(2\pi\iu) (\partial_{x_\con})^{\otimes 2}$ 
in the conifold limit $\tau \to 0$. 
Note that $\Delta(P_\cc,P_\alg)$ is a global continuous 
section of $\Theta(\log \{0\})^{\otimes 2}$ as 
$P_\cc$, $P_\alg$ are globally defined; moreover  
the difference between \eqref{eq:propagator_cc_alg} and 
\eqref{eq:propagator_cc_alg_E2} is holomorphic. 
Thus the difference between \eqref{eq:propagator_cc_alg} 
and \eqref{eq:propagator_cc_alg_E2} is a global 
holomorphic section of $\Theta(\log\{0\})^{\otimes 2} \cong 
\cO(2)$.  Such a section is unique up to a constant, so it suffices now to check that \eqref{eq:propagator_cc_alg} 
and \eqref{eq:propagator_cc_alg_E2} have the same value 
$-3a\theta\otimes \theta$ 
at $y=0$. 
\end{proof} 

Comparing \eqref{eq:propagator_cc_alg} and 
\eqref{eq:propagator_cc_alg_E2}, we obtain: 
\begin{corollary} 
\label{cor:E2}
$- 2\pi \iu E_2(\tau) = 
(\theta x) \cdot \left( \left(1 + 108 y\right) 
\theta x  
+ 12 (1+27y) \theta^2 x \right)$, where $x= x(\zeta)$. 
\end{corollary} 

\begin{corollary} 
\label{cor:eta} 
Let $\eta(\tau) = e^{\pi\iu\tau/12} \prod_{n=1}^\infty 
(1-e^{2\pi\iu \tau n})$ denote the Dedekind eta function. 
We have 
$\eta(\tau) =e^{-\frac{\pi\iu}{24}} 
y^{\frac{1}{24}} (1+27y)^{\frac{1}{8}} 
\sqrt{\iu (2\pi\iu)^{-1/2} \theta x}$.  
\end{corollary} 
\begin{proof} 
Using the identity 
$\parfrac{}{\tau} \log \eta(\tau) = \frac{\pi\iu}{12} E_2(\tau)$ 
and the above corollary, we have 
\[
\theta \log \eta(\tau) = -\frac{1}{(1+27 y)(\theta x)^2} 
\frac{\pi \iu}{12} E_2(\tau) 
= \frac{1}{24} +\frac{1}{8} \frac{27y}{1+27y} +  \frac{1}{2} 
\frac{\theta^2 x}{\theta x}. 
\] 
We arrive at the formula by integrating this. 
\end{proof}

Combining Propositions \ref{prop:propagator_with_cc},
\ref{prop:propagator_cc_alg}, we obtain: 
\begin{corollary} 
\label{cor:propagator_with_alg} 
With notation as in Propositions $\ref{prop:propagator_with_cc}$ 
and $\ref{prop:propagator_cc_alg}$, 
we have 
\begin{align*} 
\Delta(P_\LR, P_\alg)  - \Delta_a& = 
\frac{\pi\iu}{2} E_2(\tau) (\partial_x)^{\otimes 2}  
= -\frac{1}{4} E_2(\tau) (\partial_{x_\LR})^{\otimes 2} \\ 
\Delta(P_\con, P_\alg) - \Delta_a & = 
\left( 
\frac{3}{\tau} + \frac{\pi\iu}{2} E_2(\tau) \right) 
(\partial_x)^{\otimes 2} 
= \frac{3}{4} E_2\left(\frac{3\tau_\con}{2\pi\iu} \right) 
(\partial_{x_\con})^{\otimes 2} \\ 
\Delta(P_\orb, P_\alg) - \Delta_a & = 
\left( \frac{3(1-\xi)}{1+\tau(1-\xi)} 
+ \frac{\pi\iu}{2} E_2(\tau) \right) 
(\partial_x)^{\otimes 2} \\
& = \partial_\fry x_\orb \cdot \left( \frac{1}{12} \fry^2 
\partial_\fry x_\orb + 3 \left(1+\frac{\fry^3}{27}\right) 
\partial_{\fry}^2 x_\orb \right)  
(\partial_{x_\orb})^{\otimes 2}
\end{align*} 
where $\Delta_a = 3 \left(\frac{1}{12} -a\right) \theta\otimes \theta$.  
\end{corollary} 
\begin{proof}
We use $\Delta(P_\LR,P_\alg) = \Delta(P_\LR,P_\cc) + 
\Delta(P_\cc, P_\alg)$ etc.~from 
Proposition \ref{prop:propagator_cocycle}. 
We also use Notation \ref{nota:Darboux_cusps}, equation 
\eqref{eq:E2_modular}, 
Corollary \ref{cor:E2}, and Lemma \ref{lem:Wronskian} 
for $(p_\orb,x_\orb)$. 
Another way to compute these quantities will be explained in 
Lemma \ref{lem:propagator_with_alg}.
\end{proof} 

\subsection{Quasi-modularity of Gromov--Witten Potentials} \label{sec:quasi-modularity}
In this section we prove that the Gromov--Witten potential $F^g_Y$ is a quasi-modular function.  Let us begin by reviewing the theory of quasi-modular forms introduced\footnote
{To be more precise, Kaneko--Zagier considered quasi-modular 
forms which satisfy a standard growth condition at cusps. 
We do not impose the growth condition, since we  
deal with (quasi-)modular forms with non-positive weight.} by 
Kaneko--Zagier \cite{Kaneko--Zagier}. 
We say that a holomorphic function $f\colon \HH \to \C$ 
is a \emph{quasi-modular form of weight $k$} for $\Gamma_1(3)$ 
if there exist finitely many holomorphic functions 
$f_i\colon \HH \to \C$, $i=1,\dots,n$ such that 
\[
\hat{f}(\tau) = f(\tau) + \frac{f_1(\tau) }{\tau-\ov\tau} 
+ \cdots + \frac{f_n(\tau) }{(\tau-\ov\tau)^n} 
\]
is modular of weight $k$, i.e.~
\begin{align*} 
\hat{f}\left( \frac{a\tau+b}{c\tau+d}\right)  = 
(c\tau +d)^k \hat{f}(\tau)
&& \text{for all $\begin{pmatrix} a & b \\ c & d \end{pmatrix} \in \Gamma_1(3)$ and all $\tau \in \HH$.}   
\end{align*}
When $n=0$, $f$ is a (holomorphic) modular 
form of weight $k$. 
It is known that $f_1,\dots,f_n$ (and hence $\hat{f}$) 
are uniquely determined by $f$ \cite[Proposition 1]{Kaneko--Zagier}; see \cite[Proposition 3.4]{Bloch--Okounkov} for a proof. 
The function $\hat{f}$ is called 
the \emph{almost holomorphic modular form} associated with $f$, 
and $f$ is called the \emph{holomorphic limit} of $\hat{f}$.  
Equation \eqref{eq:E2_modular} shows that $E_2$ is 
a quasi-modular form of weight $2$ and $\hE_2$ is the 
associated almost holomorphic modular form. 
Every almost holomorphic modular form of weight $k$ 
can be uniquely expanded in the form: 
\[
\hat{f}(\tau)= \sum_{j=0}^n g_j(\tau) \hE_2(\tau)^j  
\]
where $g_j$ is holomorphic modular 
of weight $k-2j$. 
Taking the holomorphic limit, we find that 
the corresponding quasi-modular form $f$ admits a unique expansion: 
\begin{equation} 
\label{eq:E2_expansion} 
f(\tau) = \sum_{j=0}^n g_j(\tau) E_2(\tau)^j 
\end{equation} 
with $g_j$ holomorphic modular of weight $k-2j$. 
The ring of quasi-modular forms is therefore generated by 
modular forms and $E_2$ (see \cite{Kaneko--Zagier, 
Bloch--Okounkov}). 

\begin{remark}[modular quantities] 
\label{rem:modular_quantities} 
Let $(p,x)$ be the Darboux co-ordinates 
from \eqref{eq:Darboux_central_charge}, 
regarded as functions on $\cL \cong 
(\cMCY\setminus \DCY)\sptilde \cong \HH$ 
as in Recapitulation \ref{rec:Caff}.  
The following quantities are holomorphic modular for $\Gamma_1(3)$: 
\begin{itemize} 
\item the rational co-ordinates $y$, $\fry$, which are of weight 0; 
\item $\theta x$, which is of weight 1 
(see equation~\eqref{eq:theta_x_weight_1});  
\item $\theta \tau= \theta(\theta p/\theta x) 
= - (1+27y)^{-1} (\theta x)^{-2}$, which is of 
weight $-2$;  
\item the Yukawa coupling $Y_{\CY}(x) = \frac{1}{3}\partial_x \tau 
= \frac{1}{3}\theta \tau/\theta x$, which is of weight $-3$.  
\end{itemize} 
We also note the following: 
\begin{itemize}  
\item $f(\tau) (d\tau)^{\otimes k}$ is $\Gamma_1(3)$-invariant 
$\Longleftrightarrow$ 
$f(\tau)$ is of weight $2k$; 
\item $f(\tau) (\partial_x)^{\otimes k}$ is $\Gamma_1(3)$-invariant 
$\Longleftrightarrow$ $f(\tau)$ is of weight $k$; 
\item $f(\tau) (dx)^{\otimes k}$ is $\Gamma_1(3)$-invariant 
$\Longleftrightarrow$ $f(\tau)$ is of weight $-k$. 
\end{itemize} 
These follow from $d \tau = (\theta \tau) \frac{dy}{y}$, 
$\partial_x = (\theta x)^{-1} \theta$ and 
the above computation. 
\end{remark} 

\begin{notation}[correlation functions for $\wave_{\CY}$]
\label{nota:corr_wave_CY} 
Let $x = x(\zeta)$ denote the co-ordinate on 
$(\cMCY\setminus \DCY)\sptilde$ induced by 
the integral Darboux co-ordinates \eqref{eq:Darboux_central_charge}. 
We represent the global section $\wave_{\CY}$ of the finite-dimensional 
Fock sheaf 
by correlation functions as follows (see \S\S 
\ref{sec:conformal_limit_Fock}--\ref{sec:conifold_estimate}): 
\begin{itemize} 
\item[(1)] the correlation functions $C^{(g)}_{Y,n} dx^{\otimes n}$ with respect to $P_\LR$;

\item[(2)] the correlation functions $C^{(g)}_{\cX,n} dx_\orb^{\otimes n}$ with respect to $P_\orb$;

\item[(3)] the correlation functions $C^{(g)}_{\con,n} dx_\con^{\otimes n}$ with respect to $P_\con$;

\item[(4)] the correlation functions $C^{(g)}_{\cc,n} dx^{\otimes n}$ with respect to $P_\cc$;
\item[(5)] the correlation functions $C^{(g)}_{\alg,n} dx^{\otimes n}$ with respect to $P_\alg(a)$.  
\end{itemize} 
The co-ordinates $\frt = x_\orb$ and $x_\con$ here were defined 
in Notation~\ref{nota:Darboux_cusps}.  
We will set $a=\frac{1}{12}$ unless otherwise specified, 
and will write $F^g_\con = C^{(g)}_{\con,0}$. 
Theorem~\ref{thm:wave_CY} gives
\begin{align} 
\label{eq:corr_Y}
  C^{(g)}_{Y,n} = \parfrac{^n F^{g}_Y}{x^n} 
  = \iu^n (2\pi\iu)^{-n/2}\parfrac{^n F^g_Y}{t^n} 
  && \text{and} && 
  C^{(g)}_{\cX,n} = \parfrac{^n F^{g}_\cX}{\frt^n} 
\end{align}
where $F^{g}_Y$ and $F^g_\cX$ were defined in Notation~\ref{nota:GW_X_Y}.
\end{notation}

\begin{theorem} 
\label{thm:quasi-modularity}
Let $g$ and $n$ be non-negative integers satisfying $2g-2+n>0$.  
We have the following (quasi-)modularity with respect to the 
group $\Gamma_1(3)$ and the modular parameter 
$\tau$ from Corollary $\ref{cor:modular_parameter}$. 
\begin{itemize} 
\item[(a)] $C^{(g)}_{Y,n}$ is a quasi-modular form 
of weight $-n$; 

\item[(b)] $C^{(g)}_{\cc,n}$ 
is the almost-holomorphic modular form of weight $-n$ 
associated with $C^{(g)}_{Y,n}$; 

\item[(c)] $C^{(g)}_{\alg,n}$ 
is the holomorphic modular form of weight $-n$ which appears as
the constant term of the $E_2$-expansion \eqref{eq:E2_expansion} of $C^{(g)}_{Y,n}$. 
\end{itemize} 
\end{theorem} 
\begin{proof} 
The correlation functions $\{C^{(g)}_{Y,n} dx^{\otimes n}\}$, 
$\{C^{(g)}_{\cc,n}dx^{\otimes n}\}$, 
$\{C^{(g)}_{\alg,n}dx^{\otimes n}\}$ are different realizations 
of the same section $\wave_{\CY}$ of the Fock sheaf, and 
therefore they are related by the Feynman rule. 
The relationship between these correlation functions is shown 
in Figure~\ref{fig:triangle}: the propagators recorded there were computed in Proposition~\ref{prop:propagator_with_cc}, Proposition~\ref{prop:propagator_cc_alg}, and Corollary~\ref{cor:propagator_with_alg}. 
Since $P_\cc$ and $P_\alg$ are single-valued 
subbundles on $\cMCY$, we know that 
$C^{(g)}_{\cc,n}dx^{\otimes n}$ 
and $C^{(g)}_{\alg,n} dx^{\otimes n}$ 
are $\Gamma_1(3)$-invariant. Remark~\ref{rem:modular_quantities} then
implies that $C^{(g)}_{\cc,n}$ and $C^{(g)}_{\alg,n}$ 
are modular of weight $-n$. The Feynman rules between 
$P_\LR$ and $P_\cc$/$P_\alg$ imply  
that $C^{(g)}_{\cc,n}$ and $C^{(g)}_{\alg,n}$ 
can be written in the form: 
\begin{align*} 
C^{(g)}_{\cc,n} & = C^{(g)}_{Y,n} + \sum_{i=1}^{3g-3+n} f_i(\tau) 
\left( \frac{-3}{\tau-\ov\tau}\right)^i  \\
C^{(g)}_{Y,n} & = 
C^{(g)}_{\alg,n} + \sum_{i=1}^{3g-3+n} \tilde{f}_i(\tau) 
\left( -\frac{\pi\iu}{2} E_2(\tau)\right)^i 
\end{align*} 
for some holomorphic functions $f_i$, $\tilde{f}_i$ on $\HH$. 
Moreover $\tilde{f}_i$ is modular because it consists of 
products of several $C^{(h)}_{\alg,m}$'s (with total 
weight $-n-2i$). 
This implies that $C^{(g)}_{Y,n}$ is a quasi-modular form, 
that $C^{(g)}_{\cc,n}$ is the corresponding almost holomorphic 
modular form, and that $C^{(g)}_{\alg,n}$ is the constant 
term of the $E_2$-expansion of $C^{(g)}_{Y,n}$, as claimed. 
\end{proof}

\begin{figure}[htbp]
\includegraphics[bb=200 590 400 710]{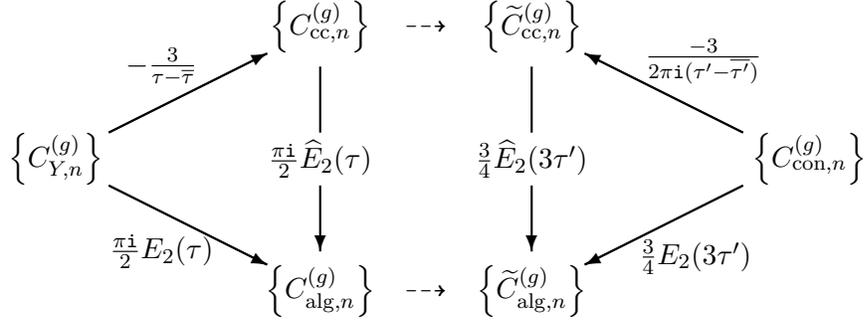}
\caption{Triples of correlation functions related by Feynman rules, with the arrows labelled by propagators.
The dashed arrows $\dashrightarrow$ mean a change of frames (multiplication by $(6\pi\iu)^{n/2}(\tau')^n$) 
and change of variables $\tau' = -1/(3\tau)$. } 
\label{fig:triangle}
\end{figure} 

Essentially the same argument shows the parallel results 
for conifold correlation functions:
\begin{proposition} 
Let $g$ and $n$ be non-negative integers satisfying $2g-2+n>0$.  
Let $\heartsuit$ denote one of $\alg$, $\con$ and $\cc$, 
and write $\tC^{(g)}_{\heartsuit,n} := (\parfrac{x}{x_\con})^n 
C^{(g)}_{\heartsuit,n}$ 
for the correlation functions in the frame $(dx_\con)^{\otimes n}$. 
We have the following (quasi-)modularity with respect to the 
group $\Gamma_1(3)$ and the modular parameter 
$\tau' := -1/(3\tau) 
= \tau_\con/(2\pi\iu)$ discussed in Remark~\ref{rem:Fricke}.
\begin{itemize} 
\item[(a)] $C^{(g)}_{\con,n}$ 
is a quasi-modular form of weight $-n$; 
\item[(b)] 
$\tC^{(g)}_{\cc,n}= (6\pi\iu)^{n/2}  (\tau')^n C^{(g)}_{\cc,n}$ 
is the associated almost-holomorphic modular form; 
\item[(c)] 
$\tC^{(g)}_{\alg,n}= 
(6\pi\iu)^{n/2} (\tau')^n C^{(g)}_{\alg,n}$ 
is the holomorphic modular form which appears as 
the constant term of the $E_2(3\tau')$-expansion 
of $C^{(g)}_{\con,n}$. 
\end{itemize} 
\end{proposition} 

\subsection{The Holomorphic Anomaly Equation and the Anomaly Equation} 

\begin{proposition} 
\label{prop:torsion_alg_cc} 
Let $(p,x)$ denote the integral Darboux co-ordinates \eqref{eq:Darboux_central_charge}, 
and regard $x=x(\zeta)$ as a co-ordinate on the universal 
cover of $\cMCY \setminus \DCY$. 
The torsion forms (see Definition~\ref{def:torsion}) 
of $P_\cc$ and $P_\alg(a)$ are given respectively as follows: 
\begin{align*} 
\Lambda_\cc(dx,dx) & = \frac{3}{(\tau-\ov\tau)^2} d\ov{\tau}  
= - \frac{3}{(\tau-\ov\tau)^2} \frac{d\ov{x}}
{(1+27 \ov{y}) (\ov{\theta x})^3} \\
\Lambda_\alg\left(\frac{dy}{y}, \frac{dy}{y}\right) 
& = \frac{9(1-6a) y - 3 a^2}{1+27y} \frac{dy}{y} 
\end{align*} 
\end{proposition} 
\begin{proof} 
Use \eqref{eq:torsion}. In the notation there, we have 
$c = \ov\tau$ for $P_\cc$ and $c = \tau - 1/A$ 
for $P_\alg$, where $A$ is as in the proof of Proposition \ref{prop:propagator_cc_alg}. 
We also use:
\begin{align*} 
& d \tau  = \tau_x dx = 3 Y_{\CY}(x) dx 
= - \frac{1}{(1+27y) (\theta x)^3}dx 
&&
\text{(see Example \ref{ex:Yukawa_conformal_limit})}, \\ 
& (1+27 y) \theta^3 x = - 6 y \theta  x - 27 y \theta^2 x  
&& \text{(since $x$ is a solution to \eqref{eq:PF_psi})}
\end{align*} 
where $Y_{\CY} = Y_{\CY}(x) (dx)^{\otimes 3}$. 
\end{proof} 

\begin{remark} 
The connection $\Nabla^{P_\cc}$ associated with the 
complex conjugate opposite line bundle 
(see \S\ref{subsec:curved}) respects the  
positive definite Hermitian metric $h$ 
on $\Theta \cong F^2_\vec \cong 
\bigcup_y H^{1,0}(E_y)$ given by 
$h(\lambda_1,\lambda_2) = \iu \int_{E_y} \lambda_1 
\cup \ov{\lambda_2}$ for $\lambda_i \in H^{1,0}(E_y)$. 
Therefore $\Nabla^{P_\cc}$ is the Chern connection 
associated with this Hermitian metric. Recall also from 
Lemma \ref{lem:propagator_calc_curved} that 
\[
Y_{\CY}(x) \Lambda_{\ov{x}} \, dx\wedge d\ov{x} 
= \frac{dx\wedge d\ov{x}}{(\tau-\ov\tau)^2 |1+27 y|^2 |\theta x|^6} 
= \frac{d\tau \wedge d\ov{\tau}}{(\tau-\ov\tau)^2}
\]
is the curvature of $\Nabla^{P_\cc}$ and gives the Poincar\'e 
metric on $\HH$. 
\end{remark} 

From Proposition \ref{prop:anomaly_equation} 
and Remark \ref{rem:anomaly_equation}, we obtain the following: 

\begin{proposition} 
\label{prop:HAE} 
The correlation functions $C^{(g)}_{\cc,n}dx^{\otimes n}$ 
associated with $P_\cc$ satisfy the following holomorphic 
anomaly equation:  
\begin{align*} 
C^{(g)}_{\cc,n+1} 
& = \Nabla^{P_\cc}_x C^{(g)}_{\cc,n} = 
\parfrac{C^{(g)}_{\cc,n}}{x} - 
\frac{n\tau_x}{(\tau-\ov{\tau})} C^{(g)}_{\cc,n}  \\ 
\parfrac{C^{(1)}_{\cc,1}}{\ov{x}}  
& = -\frac{1}{2} \frac{|\tau_x|^2}{(\tau-\ov{\tau})^2} \\ 
\parfrac{C^{(g)}_{\cc,0}}{\ov{x}} 
& = -\frac{1}{2} \sum_{h=1}^{g-1} 
\frac{3 \ov{\tau_x} }
{(\tau-\ov{\tau})^2}C^{(h)}_{\cc,1} C^{(g-h)}_{\cc,1}
- \frac{1}{2} \frac{3 \ov{\tau_x} }{(\tau-\ov{\tau})^2} 
C^{(g-1)}_{\cc,2}.  
\end{align*} 
\end{proposition} 

\begin{proposition} 
\label{prop:corr_alg} 
Let $\hC^{(g)}_{\alg,n} := C^{(g)}_{\alg,n} (\theta x)^n$ 
denote the correlation function associated with $P_\alg(a)$ 
written in the frame $(\frac{dy}{y})^{\otimes n}$. 
Then $\hC^{(g)}_{\alg,n}$ is a rational function of $y$ 
of the form: 
\begin{equation} 
\label{eq:corr_alg_form} 
\hC^{(g)}_{\alg,n} = \sum_{i=\ceil{n/3}}^{2g-2+n} 
\frac{c_i}{(1+27y)^i}, \qquad 
c_i\in \C, 
\end{equation} 
and satisfies the following anomaly equation: 
\begin{equation*} 
\hC^{(g)}_{\alg,n+1} 
= \theta \hC^{(g)}_{\alg,n} 
+ \frac{n(a+9y)}{1+27y}\hC^{(g)}_{\alg,n} 
+ \frac{9(1-6a)y-3a^2}{2(1+27y)} 
\left(\sum_{\substack{h+k=g \\ i+j=n}}
\binom{n}{i} \hC^{(h)}_{\alg,i+1} \hC^{(k)}_{\alg,j+1} 
+ \hC^{(g-1)}_{\alg,n+2} \right)   
\end{equation*} 
with $\theta = y\parfrac{}{y}$. 
$($In this proposition, we do not specialize $a$ to $\frac{1}{12}.)$
\end{proposition} 
\begin{proof} 
By Theorem \ref{thm:conifold_estimate} and 
Lemma \ref{lem:regularity_curved}, the $n$-tensor 
$\hC^{(g)}_{\alg,n} (\frac{dy}{y})^{\otimes n} 
= \hC^{(g)}_{\alg,n} (-3\frac{d\fry}{\fry})^{\otimes n}$ 
has poles of order at most $2g-2+n$ at the conifold point 
and is regular elsewhere. 
Thus $\hC^{(g)}_{\alg,n}$ is of the form \eqref{eq:corr_alg_form}. 
The anomaly equation follows from Propositions \ref{prop:anomaly_equation} 
and \ref{prop:torsion_alg_cc}: it suffices to note that 
\begin{equation} 
\label{eq:Nabla_alg} 
\Nabla^{P_\alg(a)} \left( \frac{dy}{y} \right) = \frac{a+9y}{1+27y} 
\frac{dy}{y} \otimes \frac{dy}{y},
\end{equation} 
which follows by combining Lemma~\ref{lem:propagator_calc_curved}, 
Corollaries~\ref{cor:propagator_with_alg} and~\ref{cor:E2}, and 
$\frac{dy}{y} = \frac{dx}{\theta x}$.  
\end{proof} 
\begin{remark} 
The above anomaly equation reconstructs 
$n$-point correlation functions 
$\hC^{(g)}_{\alg,n}$ from the base cases
$\hC^{(0)}_{\alg,3}$, $\hC^{(1)}_{\alg,1}$ and $\hC^{(h)}_{\alg,0}$, $h\ge 2$. The anomaly equation preserves 
the pole order condition at $y=-1/27$ and the vanishing condition at $y=\infty$. 
\end{remark} 

\begin{example}[genus-one potential] \label{ex:Y_genus_one}
In this example we set $a=-1/12$.  
The genus-one, one-point function $\hC^{(1)}_{\alg,1}$  
is of the form: 
\[
\hC^{(1)}_{\alg,1} =  \frac{c}{1+27y} 
\]
for some $c\in \C$. 
The transformation rule between $P_\alg$ and $P_\LR$ implies: 
\begin{equation} 
\label{eq:transf_LR_alg_genus_1}
\frac{\iu}{\sqrt{2\pi\iu}} \partial_t F_Y^1 = 
C^{(1)}_{Y,1} = C^{(1)}_{\alg,1} 
-\frac{1}{2} \frac{\pi\iu}{2} E_2(\tau) Y_{\CY}(x). 
\end{equation} 
Since $\partial_t F^1_Y|_{y=0} = -1/12$, it follows 
that $c = -1/24$. Using Corollary \ref{cor:E2}, we get: 
\[
(\theta x) \cdot C^{(1)}_{Y,1} = -\frac{1}{2} \frac{\theta^2x}{\theta x}
- \frac{1}{12} \left(1 +  \frac{27y}{1+27y} \right). 
\]
Integrating, we have 
\begin{align*} 
F^1_Y = 
\int C_{Y,1}^{(1)} dx & = \log \left( 
(\theta x)^{-\frac{1}{2}} y^{-\frac{1}{12}} (1+27y)^{-\frac{1}{12}} 
 \right) + \text{const} \\ 
 & = - \log\left(e^{\frac{\pi\iu}{24}} 
y^{\frac{1}{24}} (1+27y)^{-\frac{1}{24}} 
\eta(\tau)\right) \qquad 
\text{(by Corollary \ref{cor:eta})}
\end{align*} 
or equivalently, 
\[
e^{F^1_Y} \sqrt{dx} = \text{const}.  
\frac{\sqrt{\frac{dy}{y}}}{y^{\frac{1}{12}} (1+27y)^{\frac{1}{12}}}  
\]
which shows that $e^{F^1_Y}$ is modular of weight $-1/2$ 
(with automorphic factors). 
Equation~\eqref{eq:transf_LR_alg_genus_1} also implies that: 
\[
dF^1_Y = \left( 
- \frac{\pi\iu}{12} E_2(\tau) + \frac{1}{24} (\theta x)^2 \right) d\tau. 
\]
\end{example} 

\begin{remark}[solving the holomorphic anomaly equation] 
\label{rem:holomorphic_ambiguity} 
Bershadsky--Cecotti--Ooguri--Vafa 
introduced a Feynman diagram technique to solve the 
holomorphic anomaly equation~\cite{BCOV:HA,BCOV}. For example, in our case, 
a solution at genus one (see Proposition \ref{prop:HAE}) is
\begin{equation} 
\label{eq:genus_1_cc} 
C^{(1)}_{\cc,1} = - \frac{\tau_x}{2(\tau-\ov\tau)} + f(x) 
\end{equation} 
for some holomorphic function $f(x)$. On the other hand, the transformation 
rule from $P_\cc$ to $P_\alg(\frac{1}{12})$ gives: 
\[
C^{(1)}_{\alg,1} = C^{(1)}_{\cc,1} + \frac{\pi\iu}{4} 
\hE_2(\tau) Y_{\CY}(x) 
= C^{(1)}_{\cc,1} + \frac{\tau_x}{2(\tau-\ov{\tau})} 
+ \frac{d}{dx} \log(\eta(\tau));
\]
see Proposition \ref{prop:propagator_cc_alg}.
Therefore the holomorphic ambiguity $f(x) = 
C^{(1)}_{\alg,1} - \partial_x \log \eta(\tau)$ essentially 
corresponds to the algebraic potential $C^{(1)}_{\alg,1}$. 
More generally, for $g\ge 2$, the transformation rule gives: 
\[
C^{(g)}_{\alg,0} = C^{(g)}_{\cc,0} + \sum_{\Gamma} 
\frac{1}{|\Aut(\Gamma)|}
\Cont_\Gamma(\Delta(P_\cc,P_\alg), \{C^{(h)}_{\cc,\bullet}\}_{h<g})
\]
where we separate the Feynman rule \eqref{eq:Feynman_rule_fd} 
into the leading term $C^{(g)}_{\cc,0}$ and the lower-genus contribution. 
When viewed as an expression for $C^{(g)}_{\cc,0}$, 
this formula solves the holomorphic anomaly equation 
recursively in genus, with holomorphic ambiguity $C^{(g)}_{\alg,0}$. 
\end{remark} 
\begin{remark}
Integrating \eqref{eq:genus_1_cc}, 
we find that  
$F_\cc^1= \log(|\tau-\ov{\tau}|^{-\frac{1}{2}} 
|y|^{-\frac{1}{12}} |1+27y|^{\frac{1}{12}} 
|\eta(\tau)|^{-2})$ satisfies $\partial_x F_\cc^1 = C^{(1)}_{\cc,1}$. 
Then $F_\cc^1$ is $\Gamma_1(3)$-invariant and 
$e^{F_\cc^1}= |e^{F_Y^1}|^2 |\tau-\ov{\tau}|^{-\frac{1}{2}}$.  
We may view $e^{F_\cc^1}$ as a `norm' of the exponentiated 
genus-one potential, 
see \cite{BCOV:HA}, \cite[\S 9.4]{Coates--Iritani:Fock}. 
\end{remark} 

\subsection{Algebraic Opposite and Finite Generation} 

As we saw in Proposition \ref{prop:corr_alg}, correlation functions with 
respect to the algebraic opposite line bundle $P_\alg(a)$ belong to 
the polynomial ring $\C[(1+27y)^{-1}]$. 
Using the transformation rule from the algebraic opposite 
to other opposites, we conclude that correlation functions 
with respect to $P_\LR,P_\orb,P_\con$ 
belong to the polynomial ring generated by $(1+27y)^{-1}$ and 
the respective propagators, and that they 
are related to each other by `interchanging the propagators'. 
This recovers the results of Lho and Pandharipande 
\cite{Lho-Pandharipande:SQ_HAE}, and their version 
of the Crepant Resolution Conjecture \cite{Lho-Pandharipande:CRC}. 

Let us write the propagators between $P_\alg(a)$ 
and $P_\LR$, $P_\con$, $P_\orb$ 
in the frame $\theta \otimes \theta$, setting
\[
\Delta(P_\alg(a),P_\heartsuit) = \Delta_{\alg,\heartsuit} \, 
\theta \otimes \theta
\]
where $\heartsuit$ is one of `$\LR$', `$\con$', `$\orb$'. 

\begin{lemma}[cf.~Corollary \ref{cor:propagator_with_alg}]  
\label{lem:propagator_with_alg} 
Using Notation $\ref{nota:Darboux_cusps}$, we have 
\[
\Delta_{\alg,\heartsuit} = 3(1+27y) 
\frac{\theta^2 x_\heartsuit}{\theta x_\heartsuit} 
+ 27y + 3a, \qquad 
\heartsuit \in \{ \LR, \con, \orb\}. 
\]
\end{lemma} 
\begin{proof} 
We can deduce this from the previous computation 
(Corollary \ref{cor:propagator_with_alg}), but here 
we outline a simpler derivation. 
By definition -- see Definition \ref{def:Palg} -- $P_\alg$ is spanned by: 
\begin{align*} 
v_\alg & = (-27y-3a) \theta\zeta -3 (1+27y) \theta^2 \zeta. \\ 
\intertext{On the other hand $P_\heartsuit$ is cut out by $x_\heartsuit=0$ 
(see Notation~\ref{nota:Darboux_cusps}) and is therefore spanned by} 
v_\heartsuit & = 
3 (1+27y) \frac{\theta^2x_\heartsuit}{\theta x_\heartsuit} \theta \zeta 
- 3 (1+27y) \theta^2 \zeta. 
\end{align*} 
Since $\Omega(v_\alg,\theta\zeta) = \Omega(v_\heartsuit, 
\theta \zeta) =1$, it follows that 
$v_\heartsuit - v_\alg= \Delta_{\alg,\heartsuit}\, (\theta \zeta)$. 
The conclusion follows. 
\end{proof} 

\begin{lemma}[cf.~Lemma \ref{lem:propagator_calc_curved}]
\begin{align*}
\theta ((1+27y)^{-1}) & =  (1+27y)^{-2} -(1+27y)^{-1}, \\ 
\theta (\Delta_{\alg,\heartsuit}) & = 
- \frac{1}{3(1+27y)} (\Delta_{\alg,\heartsuit})^2
+ \frac{2(a+9y)}{1+27y} \Delta_{\alg,\heartsuit} 
+ \frac{9(1-6a)y - 3a^2}{1+27y}.   
\end{align*} 
\end{lemma} 
\begin{proof} 
The first equation is obvious. 
The second equation 
is an analogue of Lemma \ref{lem:propagator_calc_curved} 
(see also Proposition~\ref{prop:torsion_alg_cc} 
and equation~\eqref{eq:Nabla_alg})
and follows immediately from the Picard--Fuchs equation 
\eqref{eq:PF_psi} for $x_\heartsuit$. 
\end{proof} 

The lemma shows that the ring 
$\C[\Delta_{\alg,\heartsuit}, (1+27y)^{-1}]$ 
is closed under the differential $\theta = y\parfrac{}{y}$,
so that it is a differential ring. 
The following theorem shows that the Gromov--Witten potentials 
of $\cX$ and $Y$ belong respectively to the differential rings $\C[\Delta_{\alg,\orb},(1+27y)^{-1}]$ and
$\C[\Delta_{\alg,\LR},(1+27y)^{-1}]$.

\begin{theorem} 
\label{thm:finite_generation}
Let $C^{(g)}_{Y, n}, C^{(g)}_{\cX,n}, C^{(g)}_{\con,n}$ 
be as in Notation $\ref{nota:corr_wave_CY}$. 
For a pair $(g,n)$ of non-negative integers with $2g-2+n>0$, 
there exists a polynomial $f_{g,n}(\Delta,R) \in \C[\Delta,R]$ 
such that:
\begin{align*} 
C^{(g)}_{Y,n} & = 
f_{g,n}(\Delta_{\alg,\LR},(1+27y)^{-1}) (\theta x)^{-n}, \\ 
C^{(g)}_{\cX,n} & = f_{g,n}(\Delta_{\alg,\orb}, (1+27y)^{-1}) 
(\theta x_\orb)^{-n}, \\ 
C^{(g)}_{\con,n} & = f_{g,n}(\Delta_{\alg,\con},(1+27y)^{-1}) 
(\theta x_\con)^{-n}.     
\end{align*} 
Moreover, we have 
$\deg_{\Delta} f_{g,n} \le 3g-3+n$, $\deg_{R} f_{g,n} 
\le 2g-2+n$ and 
\begin{equation} 
\label{eq:diff_by_propagator} 
\parfrac{f_{g,n}}{\Delta} =\frac{1}{2} 
\sum_{\substack{n_1+n_2=n \\ g_1+g_2=g}} 
\binom{n}{n_1} 
f_{g_1,n_1+1} f_{g_2,n_2+1} + \frac{1}{2} f_{g-1,n+2}.  
\end{equation} 
\end{theorem} 
\begin{proof} 
This follows from the Feynman rule relating $\hC^{(g)}_{\alg,n}$ 
to each of $C^{(g)}_{Y,n}$, $C^{(g)}_{\cX,n}$, $C^{(g)}_{\con,n}$  
together with Proposition \ref{prop:corr_alg}. 
Note that the Feynman rule for $C^{(g)}_{Y,n}$, written in the frame 
$(\frac{dy}{y})^{\otimes n}$, is of the form: 
\[
C^{(g)}_{Y,n} (\theta x)^n = \sum_{\Gamma} 
\frac{1}{|\Aut(\Gamma)|} \Cont_{\Gamma}
\left (\Delta_{\alg,\LR}, 
\{\hC^{(h)}_{\alg,\bullet} \}_{h\le g} \right)
\]
and the Feynman rules for $C^{(g)}_{\cX,n}$ and $C^{(g)}_{\con,n}$ 
have the same shape. 
By Proposition \ref{prop:corr_alg}, $\hC^{(h)}_{\alg,m}$ is a polynomial 
in $(1+27y)^{-1}$ of degree $\le 2h-2+m$.  
Thus the right-hand side can be written as 
a polynomial $f_{g,n}(\Delta_{\alg,\LR}, (1+27y)^{-1})$ 
such that $f_{g,n}(\Delta,R)$ has degree $\le 2g-2+n$ 
in $R$. The degree of $f_{g,n}$ as a polynomial in $\Delta$ 
is bounded by $3g-3+n$, which is the maximum possible 
number of edges appearing in Feynman graphs. 
The differential equation for $f_{g,n}$ follows from the 
Feynman rule: the first term corresponds to separating edges 
and the second term corresponds to non-separating 
edges. 
\end{proof} 
\begin{remark} 
Lho--Pandharipande 
\cite[Theorems 1,2]{Lho-Pandharipande:SQ_HAE}, 
\cite[Theorems 1,2]{Lho-Pandharipande:CRC} 
showed essentially the same result for $F^g_Y, F^g_\cX$ 
using stable quotient invariants. 
The differential equation \eqref{eq:diff_by_propagator} 
together with $\partial_x C^{(g)}_{Y,n} = C^{(g)}_{Y,n+1}$ 
implies a ``holomorphic version'' of holomorphic anomaly equation 
proved by Lho--Pandharipande 
\cite[Theorem 2]{Lho-Pandharipande:SQ_HAE}, 
\cite[Theorem 2]{Lho-Pandharipande:CRC}; 
such equations are sometimes referred to as 
``modular anomaly equations''.
The above theorem implies that 
the Gromov--Witten potentials $F^{(g)}_Y = C^{(g)}_{Y,0}$, 
$F^{(g)}_\cX =C^{(g)}_{\cX,0}$  
are related by a change of generators: 
\[
\cF^{(g)}_Y = \left. \cF^{(g)}_\cX 
\right|_{\Delta_{\alg,\orb} \to \Delta_{\alg,\LR}}
\]
for $g\ge 2$. 
This is a formulation of Crepant Resolution Conjecture due to 
Lho and Pandharipande \cite{Lho-Pandharipande:CRC}. 
\end{remark}

\subsection{Calculation of Gromov--Witten Invariants 
and the Conifold Gap} \label{sec:calculation}

One can combine knowledge of the first $2g-1$ genus-$g$ Gromov--Witten invariants of $Y$ -- for instance from the topological vertex~\cite{AKMV,Li--Liu--Liu--Zhou} -- with the modularity results from \S\ref{sec:quasi-modularity} to determine all genus-$g$ Gromov--Witten invariants of $Y$ (and $\cX$), as we now explain.  This is essentially the calculation of Aganagic--Bouchard--Klemm~\cite[Section~6.4]{ABK}, placed in our rigorous mathematical setting.  

From Example~\ref{ex:Yukawa_conformal_limit} and Example~\ref{ex:Y_genus_one} we know the Yukawa coupling $Y_{\CY}$ and the genus-one data $d F^1_Y$ exactly.  
Let $g \geq 2$, suppose by induction that we know $F^h_Y$ 
exactly for $h<g$, and suppose that we know the first $2g-1$ 
genus-$g$ Gromov--Witten invariants: $n_{g,d}$, $0 \leq d \leq 2g-2$. 
Consider the transformation rule
\begin{equation}
  \label{eq:Feynman_sum_for_computation}
  C^{(g)}_{\alg,0} = C^{(g)}_{Y,0} + \sum_\Gamma \frac{1}{|\Aut(\Gamma)|} \Cont_{\Gamma}(\Delta, \{C^{(h)}_{Y,\bullet}\}_{h < g})
\end{equation}
between the correlation functions $C^{(g)}_{\alg,0}$ with respect to the algebraic opposite line bundle $P_{\alg}(a)$ and the correlation functions $C^{(g)}_{Y,\bullet}$ with respect to $P_{\LR}$ 
(see Notation \ref{nota:corr_wave_CY}).  
Here the precise value of $a$ is not important, but it is convenient to take $a = \frac{1}{12}$.  
In \eqref{eq:Feynman_sum_for_computation} we have divided the terms in the transformation rule \eqref{eq:Feynman_rule_fd} into the main term $C^{(g)}_{Y,0}$ and the sum over all Feyman graphs with a non-zero number of edges; also $\Delta = \Delta(P_{\LR},P_{\alg})$ from Corollary~\ref{cor:propagator_with_alg}.  Equation~\eqref{eq:Feynman_sum_for_computation} determines the first $2g-1$ terms of the Taylor series expansion of $C^{(g)}_{\alg,0}$ in $y$.  
On the other hand, we know that $C^{(g)}_{\alg,0}$ has a pole of order at most $2g-2$ at the conifold point, and is regular elsewhere, so (see Proposition \ref{prop:corr_alg}) 
\[
C^{(g)}_{\alg,0} = \sum_{i=0}^{2g-2} \frac{a_i}{(27 y + 1)^i}
\]
and the first $2g-1$ Taylor coefficients of $C^{(g)}_{\alg,0}$ determine $C^{(g)}_{\alg,0}$ exactly.  
Reading equation~\eqref{eq:Feynman_sum_for_computation} 
as an expression for $C^{(g)}_{Y,0}$ now determines $C^{(g)}_{Y,0}$, 
and hence all genus-$g$ Gromov--Witten invariants of $Y$, exactly.  

\begin{remark}
The correlation function $C^{(g)}_{Y,0}$ here is 
a `holomorphic ambiguity' of Bershadsky--Cecotti--Ooguri--Vafa 
\cite{BCOV}; Aganagic--Bouchard-Klemm \cite{ABK} 
denote it by $h_g^{(0)}$, see~Remark~\ref{rem:holomorphic_ambiguity}. 
\end{remark}

With the holomorphic ambiguities in hand, we can compute higher-genus Gromov--Witten invariants of $\cX$ too.  From Example~\ref{ex:Yukawa_conformal_limit} we have that the Yukawa coupling is
\[
Y_{\CY} = \frac{9}{\fry^3+27} \, d\fry^{\otimes 3},
\]
where $\fry = y^{-1/3}$, and by arguing as in Example~\ref{ex:Y_genus_one} we have that
\[
dF^1_{\cX} = \frac{1}{8} \frac{\fry^2}{\fry^3+27} \, d \fry + \frac{1}{2} \Delta_{\alg, \orb} Y_{\CY}
\]
where $\Delta_{\alg, \orb}$ is the propagator from Lemma~\ref{lem:propagator_with_alg}.  By inverting the mirror map $\frt = \frt(\fry)$ from Theorem~\ref{thm:mirrorsymmetryforXbar}, we can write $Y_{\CY}$ and $dF^1_{\cX}$ in terms of the orbifold flat co-ordinate $\frt$; thus we know both the Yukawa coupling and $dF^1_{\cX}$ near the orbifold point exactly.  Consider now the transformation rule
\begin{equation}
  \label{eq:Feynman_sum_for_orbifold}
  C^{(g)}_{\alg,0} = C^{(g)}_{\cX,0} + \sum_\Gamma \frac{1}{|\Aut(\Gamma)|} \Cont_{\Gamma}(\Delta, \{C^{(h)}_{\cX,\bullet}\}_{h < g})
\end{equation}
between the correlation functions $C^{(g)}_{\alg,0}$ and the correlation functions $C^{(g)}_{\cX,\bullet}$ with respect to $P_{\orb}$.  Here $\Delta = \Delta(P_{\orb},P_{\alg})$ from Corollary~\ref{cor:propagator_with_alg}.  Assuming by induction that we know the genus-$h$ orbifold Gromov--Witten invariants of $\cX$ for all $h<g$, or equivalently that we know $C^{(h)}_{\cX,n}$ for all $h< g$ and all $n$, equation \eqref{eq:Feynman_sum_for_orbifold} determines $C^{(g)}_{\cX,0}$, and hence all genus-$g$ orbifold Gromov--Witten invariants of $\cX$, exactly.

We can perform the same analysis at the conifold point, obtaining `conifold Gromov--Witten invariants'.  The conifold flat co-ordinate $x_\con$ from Notation~\ref{nota:Darboux_cusps} is the holomorphic solution to the Picard--Fuchs equations that satisfies $x_{\con} \sim 27y + 1$ near the conifold point (see Remark \ref{rem:x_con_asymptotics}).  
Thus we can find $x_{\con}$ by solving the Picard--Fuchs equations in power series:
\begin{align*}
  & x_{\con} = \textstyle 
             (27y+1) + \frac{11}{18}(27y+1)^{2} + \frac{109}{243}(27y+1)^{3} + \frac{9389}{26244}(27y+1)^{4} + \frac{88351}{295245}(27y+1)^{5} + \cdots \\
  \intertext{and}
  & 27y+1 = \textstyle
          x_{\con} - \frac{11}{18}x_{\con}^{2} + \frac{145}{486}x_{\con}^{3} - \frac{6733}{52488}x_{\con}^{4} + \frac{120127}{2361960}x_{\con}^{5} - \cdots 
\end{align*}
Example~\ref{ex:Yukawa_conformal_limit} determines the Yukawa coupling near the conifold point:
\begin{multline*}
  Y_{\CY}  = { - \frac{1}{3(1+27y)}} \Big( \frac{dy}{y} \Big)^{\otimes 3} = \\
  \Big( \textstyle \frac{1}{3}\frac{1}{x_{\con}}  -\frac{1}{54} + \frac{1}{2916}x_{\con} + \frac{7}{19683} \frac{x_{\con}^{2}}{2!} - \frac{529}{2361960}\frac{x_{\con}^{3}}{3!} + \frac{53}{531441}\frac{x_{\con}^{4}}{4!} - \frac{26093}{1205308188}\frac{x_{\con}^{5}}{5!} - \cdots \Big) dx_{\con}^{\otimes 3}
\end{multline*}
Arguing as in Example~\ref{ex:Y_genus_one} we have that
\[
dF^1_{\con} = {-\frac{1}{24}} \frac{1}{1+27y} \frac{dy}{y} + \frac{1}{2} \Delta_{\alg, \con} Y_{\CY}
\]
where $\Delta_{\alg, \con}$ is as in Lemma~\ref{lem:propagator_with_alg}, so that $dF^1_{\con}$ is
\[
\Big( \textstyle -\frac{1}{12}\frac{1}{x_{\con}} + \frac{5}{216} - \frac{1}{11664}x_{\mathit{con}} - \frac{5}{26244}\frac{x_{\mathit{con}}^{2}}{2!} + \frac{283}{3149280}\frac{x_{\mathit{con}}^{3}}{3!} - \frac{215}{6377292}\frac{x_{\mathit{con}}^{4}}{4!} + \frac{4517}{535692528}\frac{x_{\mathit{con}}^{5}}{5!} + \cdots \Big) d x_{\con}
\]
Note that both the Yukawa coupling and $dF^1_{\con}$ have a simple pole at the conifold point.  The transformation rule
\begin{equation}
  \label{eq:Feynman_sum_for_conifold}
  C^{(g)}_{\alg,0} = C^{(g)}_{\con,0} + \sum_\Gamma \frac{1}{|\Aut(\Gamma)|} \Cont_{\Gamma}(\Delta, \{C^{(h)}_{\con,\bullet}\}_{h < g})
\end{equation}
between the correlation functions $C^{(g)}_{\alg,0}$ and the correlation functions $C^{(g)}_{\con,\bullet}$ with respect to $P_{\con}$ inductively determines all the correlation functions $C^{(g)}_{\con,\bullet}$ from the holomorphic ambiguities $C^{(g)}_{\alg,0}$, exactly as above.  Here $\Delta = \Delta(P_{\con},P_{\alg})$ from Corollary~\ref{cor:propagator_with_alg}.  

The results of these calculations can be found in Appendix~\ref{sec:tables}.  We determine Gromov--Witten and Gopakumar--Vafa invariants of $Y$ up to genus~$7$ and degree~$15$, as well as orbifold Gromov--Witten invariants of $\cX$ up to genus~$7$ with up to~$27$ insertions of the orbifold class~$\unit_{\frac{1}{3}}$, and conifold Gromov--Witten invariants up to genus~$7$ and degree~$4$.  In particular we find that the genus-$g$ correlation function with respect to the conifold opposite line bundle $P_{\con}$, for $2 \leq g \leq 7$, has a pole of order of order $2g-2$ at the conifold point:
\begin{equation} 
\label{eq:gap}
C^{(g)}_{\con,0} = \frac{B_{2g}}{2g(2g-2)} 3^{g-1} x_{\con}^{2-2g} + \cdots
\end{equation} 
and that no other negative powers of $x_{\con}$ occur in the Laurent expansion of $C^{(g)}_{\con,0}$.  Thus we verify the ``conifold gap'' conjecture of Huang--Klemm~\cite{Huang--Klemm, Huang--Klemm--Quackenbush} up to genus~$7$.  

\subsection*{Source Code}

This paper is accompanied by fully-commented source code\footnote{\href{https://arxiv.org/src/1804.03292/anc}{\tt https://arxiv.org/src/1804.03292/anc}}, written in the computer algebra system Sage~\cite{Sage}.  This should allow the reader to verify the calculations presented here, and to perform similar calculations.  The source code, but not the text of this paper, is released under a Creative Commons~CC0 license~\cite{CC0}: see the included file {\tt LICENSE} for details.    If you make use of the source code in an academic or commercial context, please acknowledge this by including a reference or citation to this paper.  Part of the code, a Sage package for performing sums over Feynman graphs, makes use of data files produced by the program `boundary' by Stefano Maggiolo and Nicola Pagani~\cite{Maggiolo--Pagani}.

\raggedbottom

\pagebreak

\appendix

\section{Basic Facts About Connections With Logarithmic Singularities}

\begin{proposition}
\label{pro:fundamental_solution_logarithmic_connection}
Consider $\C^r \times\C^s$ with standard co-ordinates
$(x_1,\ldots,x_r,y_1,\ldots,y_s)$, 
a contractible open neighbourhood $U$ of $(0,0)$ in $\C^r \times \C^s$, 
and a trivial holomorphic vector bundle 
$E = \C^{N+1}\times (U\times \C)\to U\times \C$.  
Let $z$ denote the standard co-ordinate on the second factor
$\C$ of $U \times \C$. 
Suppose that $E$ has a ``partial" meromorphic flat connection 
$\nabla$  in the directions of $x$ and $y$: 
\begin{equation*}
    \nabla = d + \frac{1}{z} \Bigg(
    \sum_{i=1}^{r} A_i(x,y,z) \frac{dx_i}{x_i} +
    \sum_{j=1}^{s} B_j(x,y,z) dy_j 
    \Bigg)
\end{equation*}
where $A_1,\ldots,A_r$, $B_1,\ldots,B_s$ are matrix-valued 
holomorphic functions on $U \times \C$ such that, for $1 \leq i \leq r$,
$A_i(0,0,z)$ is nilpotent.  Then there
exists a unique matrix-valued function $L(x,y,z)$ of the form:
  \[
  L(x,y,z) = \tL(x,y,z) e^{-\sum_{i=1}^{r} A_i(0,0,z) \log x_i/z}
  \]
with $\tL$ regular along $U \times \C^\times$ and that
$\tL(0,0,z) = \id$, the identity matrix, such that:
\begin{equation}
    \label{eq:equations_defining_L}
    \begin{aligned}
      & \nabla_{x_k \partial_{x_k}} L(x,y,z) v = 0 & \quad & 1 \leq k \leq r \\
      & \nabla_{\partial_{y_k}} L(x,y,z) v= 0 && 1 \leq k \leq s
    \end{aligned}
  \end{equation}
for every $v\in \C^{N+1}$. 
\end{proposition}

\begin{proof} 
By the assumption, the residue endomorphism 
$N_i = A_i(0,0,z)/z$ along $x_i=0$ 
is non-resonant, i.e.~the eigenvalues of $N_i$ 
do not differ by positive integers. 
In this case, for each $z\in \C^\times$, 
the connection $\nabla|_{U\times \{z\}}$ 
is gauge equivalent to the connection 
$d + \sum_{i=1}^r N_i \frac{dx_i}{x_i}$ 
\cite{Deligne}*{5.4, 5.5}. 
The required gauge transformation is $\tL$ in the proposition. 
\end{proof} 

\begin{remark} 
The flat connection in the above proposition is only a ``partial" connection 
defined in the directions of $x$ and $y$. In what follows, 
we consider a ``full" flat connection extended in the direction of $z$: 
even in such a situation we still consider 
a matrix-valued function $L$ which solves the equations 
\eqref{eq:equations_defining_L} only in the directions of $x$ and $y$. 
Informally, we call such an $L$ a fundamental solution 
in the directions of $x$ and $y$.
\end{remark}

Let us recall Birkhoff factorization in the theory of loop groups
(see~\cite{Pressley--Segal}). 
A smooth loop $z\mapsto L(z)$ in $GL_{N+1}$ which is 
sufficiently close to the identity (in the ``big cell" of $LGL_{N+1}$) 
admits a unique factorization 
\[
L = L_+ L _-
\]
where $L_+$ is a holomorphic map from $\{z\in \C: |z|<1\}$ 
to $GL_{N+1}$ with smooth boundary values 
and $L_-$ is a holomorphic map from $\{z\in \Proj^1: |z|>1\}$ 
to $GL_{N+1}$ with smooth boundary values 
which equals the identity at $z=\infty$. 
In the following proposition, we regard the fundamental solution $L$ 
as an element of the loop group $LGL_{N+1}$ 
by restricting $z$ to lie in $S^1$ 
and consider its Birkhoff factorization. 
This method has been used in quantum cohomology in~\cite{Coates-Givental, Guest}.

\begin{proposition} 
\label{pro:Birkhoff_logarithmic_connection}
Suppose that the partial meromorphic flat connection $\nabla$ on $E$ in 
Proposition \ref{pro:fundamental_solution_logarithmic_connection} 
is extended in the $z$-direction to a meromorphic flat connection 
of the form: 
\begin{equation}
\label{eq:before_gauge}
\nabla = d + \frac{1}{z} 
\Bigg ( \sum_{i=1}^r A_i(x,y,z) \frac{dx_i}{x_i} 
+ \sum_{j=1}^s B_j(x,y,z) dy_j + 
C(x,y,z) \frac{dz}{z} \Bigg )
\end{equation} 
where $C(x,y,z)$ is a matrix-valued holomorphic function 
on $U\times \C$ such that $C(0,0,z)$ depends linearly 
on $z$, i.e.~$C(0,0,z) = C_0 + C_1 z$ for some 
constant matrices $C_0$ and $C_1$.  
Assume moreover that $A_i^{\circ} = A_i(0,0,z)$ is both nilpotent and 
independent of $z$. 
Let $L$ be the fundamental solution in the directions of 
$x$ and $y$ in Proposition \ref{pro:fundamental_solution_logarithmic_connection}.  
After shrinking $U$ if necessary, $L$ admits a Birkhoff factorization 
$L= L_+ L_-$ such that $L_+$ is holomorphic on $U\times \C$, 
and after gauge transformation by $L_+$, the connection 
$\nabla$ takes the form:
\begin{equation}
    \label{eq:after_gauge}
    L_+^{-1}\circ \nabla \circ L_+ = d + \frac{1}{z} 
\Bigg( \sum_{i=1}^{r} \tA_i(x,y) \frac{dx_i}{x_i} +
    \sum_{j=1}^{s} \tB_j(x,y) dy_j + 
    \tC(x,y,z) \frac{dz}{z}
    \Bigg)
\end{equation}
where $\tA_1,\ldots,\tA_r$, $\tB_1,\ldots,\tB_s, \tC$ are matrix-valued 
holomorphic functions on $U$ such that, for $1 \leq i \leq r$, 
$\tA_i|_{x_i=0}=A_i^{\circ}$ is
independent of $x_1,\dots,x_{i-1},x_{i+1},\dots,x_r,y_1,\dots,y_s$ 
and nilpotent and that $\tC(x,y,z)$ depends linearly on $z$: 
$\tC(x,y,z) = \tC_0(x,y) +C_1 z$
for some matrix-valued regular function $\tC_0$ on $U$ and some 
constant matrix $C_1$. 
\end{proposition} 

\begin{proof} 
Recall that the fundamental solution is of the form 
$L = \tL e^{-\sum_{i=1}^r A_i^\circ \log x_i /z}$ 
with $\tL$ holomorphic on $U\times \C^\times$. 
Because $\tL(0,0,z) = \id$, $\tL$ admits the Birkhofff 
factorization $\tL = \tL_+ \tL_-$, shrinking $U$ if necessary. 
Because $A_i^\circ = A_i(0,0,z)$ is independent of $z$, 
this gives the Birkhoff factorization of $L$: 
$L_+ := \tL_+$ and 
$L_- := \tL_- e^{-\sum_{i=1}^r A_i^\circ \log x_i/z}$. 
Note that $L_+$ is holomorphic on $U\times \{z\in \C: |z|<1\}$ 
with smooth boundary values. 
The fundamental solution $L$ transforms the connection $\nabla$ 
to $L^{-1}\circ \nabla \circ L = d + D dz/z$ 
with $D$ given by 
\begin{align*} 
D & = L^{-1} (z \partial_z L + z^{-1} C L) \\
& = e^{\sum_{i=1}^r A_i^{\circ} \log x_i/z} 
\tL^{-1} \left( z \partial_z \tL + \tL z^{-1} \sum_{i=1}^r A_i^\circ 
\log x_i + z^{-1} C \tL \right) e^{-\sum_{i=1}^r A_i^\circ \log x_i/z}.  
\end{align*} 
The flatness of $d + D\frac{dz}{z}$ implies that 
$D$ is independent of $x$ and $y$. 
The above expression for $D$ is polynomial in $\log x_1,\dots,\log x_r$ as 
$A_i^\circ$ is nilpotent. Therefore taking the constant term 
in $\log x$, $x$ and $y$, we obtain 
\[
D = z^{-1}C(0,0,z).  
\]
Substituting $L_+L_-$ for $L$ in the equation 
$L^{-1}\circ \nabla \circ L = d + D dz/z$, 
we obtain 
\begin{align*} 
z L_+^{-1} (x_i\partial_{x_i}L_+) + L_+^{-1} A_i L_+ & = 
z L_- (x_i\partial_{x_i} L_-^{-1} )  \\
z L_+^{-1} (\partial_{y_j} L_+) + L_+^{-1} B_j L_+ & 
=z L_- (\partial_{y_j} L_-^{-1}) \\
L_+^{-1} (z \partial_z L_+) + L_+^{-1}z^{-1}C L_+
& = L_-  (z\partial_z L_-^{-1}) +  L_- z^{-1}C(0,0,z) L_-^{-1}  
\end{align*} 
The first two equations show that the left-hand sides 
are analytically continued to $U\times \Proj^1$ and 
thus are independent of $z$. 
Therefore, the gauge transformation $L_+$ transforms 
the connection matrices $A_i$, $B_j$ into $z$-independent connection matrices.  
The third equation implies that 
\[
L_+^{-1} (z\partial_z L_+) + L_+^{-1} z^{-1} C L_+ 
= z^{-1} \left[L_+^{-1} C L_+\right]_{z=0} + C_1 
\]
where $C(0,0,z) = C_0 + C_1 z$. 
Therefore $L_+$ transforms the connection matrix $z^{-1} C$ 
into a connection matrix of the form 
$z^{-1}\tC =z^{-1} \tC_0(x,y) + C_1$. 
On the other hand, this equation can be viewed as a  
differential equation for $L_+$ in the $z$-direction. 
Since the differential equation has no singularities 
on $\C^\times$, $L_+$ is analytically continued 
to a holomorphic function on $U\times \C$. 

Now we know that $L_+$ is holomorphic on $U\times \C$, 
and after gauge transformation by $L_+$ the connection 
$\nabla$ remains flat and takes the form:
\[
\tnabla = L_+^{-1} \circ \nabla \circ L_+ = 
d + \frac{1}{z} \Bigg(
  \sum_{i=1}^{r} \tA_i(x,y) \frac{dx_i}{x_i} +
  \sum_{j=1}^{s} \tB_j(x,y) dy_j + 
  \tC(x,y,z) \frac{dz}{z}
  \Bigg)
\]
where $\tA_1,\ldots,\tA_r$, $\tB_1,\ldots,\tB_s$ are 
independent of $z$, $\tA_i(0,0)= A_i^\circ$ is nilpotent 
for $1\le i\le r$, and $\tC$ is linear in $z$: 
$\tC=\tC_0(x,y) + z C_1$. 
It remains to show that $\tA_i|_{x_i=0}$ is independent of 
$x_1,\dots,x_{i-1},x_{i+1},\dots,x_r, y_1,\dots,y_s$; 
then it coincides with the nilpotent matrix $A_i^\circ$.  
Flatness of $\tnabla$ yields: 
\begin{align*} 
\partial_{y_j} \tA_i  &= x_i \partial_{x_i} \tB_j + z^{-1} [\tA_i, \tB_j] \\ 
x_j\partial_{x_j} \tA_i & = x_i\partial_{x_i} \tA_j + z^{-1}[\tA_i,\tA_j]  
\end{align*} 
This implies that $\tA_i|_{x_i=0}$ is independent of $x_1,\dots,x_{i-1},
x_{i+1},\dots,x_r,y_1,\dots,y_s$. 
\end{proof} 

\begin{proposition}
\label{pro:constant_pairing} 
Let $(E, \nabla)$ be the meromorphic flat connection 
in Proposition~\ref{pro:Birkhoff_logarithmic_connection}. 
Suppose that $E$ is equipped 
with a holomorphic non-degenerate pairing 
\[
(\cdot,\cdot)_E \colon (-1)^*\cO(E) \otimes_{U\times \C} 
\cO(E) \to \cO_{U\times \C}  
\]
such that $\big(\cO(E), \nabla, (\cdot,\cdot)_E\big)$ is a \logDTEP structure with base 
$(U,D)$, where $(-) \colon U \times \C \to U\times \C$ is the map 
sending $(x,y,z)$ to $(x,y,-z)$ (see Definition \ref{def:logDTEP}) 
and $D = \{x_1\cdots x_r =0\}$ is the normal crossing divisor. 
Suppose also that the Gram matrix of the pairing $(\cdot,\cdot)_E$ 
is independent of $z$ along $\{(0,0)\} \times \C$. 
After the gauge transformation by $L_+$ in Proposition 
\ref{pro:Birkhoff_logarithmic_connection}, the Gram matrix of the 
pairing $(\cdot,\cdot)_E$ with respect to the new trivialization is constant on $U\times \C$. 
\end{proposition} 

\begin{proof} 
In the new trivialization after the gauge transformation by $L_+$, the connection takes the form \eqref{eq:after_gauge} 
and the pairing is flat with respect to it. 
Let $G$ be the Gram matrix of $(\cdot,\cdot)_E$ 
in the new trivialization. We expand 
$G = \sum_{n\ge 0} G^{(n)}(x,y) z^n$. 
The flatness of the pairing with respect to the 
connection \eqref{eq:after_gauge} implies that 
\begin{align*} 
x_i \parfrac{}{x_i}G^{(n)} & = 
- A_i^{\rm T} G^{(n+1)} + G^{(n+1)} A_i \\ 
\parfrac{}{y_i} G^{(n)} & = 
- B_i^{\rm T} G^{(n+1)} + G^{(n+1)} B_i  
\end{align*} 
By assumption, we have $G^{(n)}(0,0)=0$ for $n>0$. 
The second equation then implies that 
$G^{(n)}(0,y)$ is independent of $y$ and is zero 
for $n>0$. Expand: 
\[
A_i(x,y) = \sum_{I} A_i^I (y) x^I, \quad 
G^{(n)}(x,y) = \sum_{I} G^{n,I}(y) x^I
\]
where $I \in \N^r$ is a multi-index. 
We have from the first equation that 
\[
k_i G^{n,K} = \sum_{K = I + J} ( - (A_i^J)^{\rm T} G^{n+1,I} 
+ G^{n+1,I} A_i^J) 
\]
Suppose by induction that
$G^{n,K}=0$ for all $K$ with $0\le |K|\le m$ and all $n\ge 0$ 
except for the case $(n,K) = (0,0)$. 
For a multi-index $K$ with $|K| = m+1$, we have 
\[
k_i G^{n,K} = - (A_i^0)^{\rm T} G^{n+1,K}
+ G^{n+1,K} A_i^0.
\]
We can choose $1\le i\le r$ such that $k_i\neq 0$, because $|K| = m+1 >0$. 
Note that $A_i^0 = A_i(0,y) = A_i(0,0)$ is nilpotent. 
Using the above equation recursively, 
we find that $G^{n,K} =0$ using the nilpotence of $A_i^0$. 
This completes the induction step and we have 
that $G$ is constant. 
This completes the proof.
\end{proof} 

\section{Notation for Graphs}
\label{sec:graph_notation}

We fix terminology for graphs as follows.  A graph $\Gamma$ is given
by four finite sets $V(\Gamma)$, $E(\Gamma)$, $L(\Gamma)$, $F(\Gamma)$
called (the set of) \emph{vertices}, \emph{edges}, \emph{legs} and
\emph{flags} respectively, together with incidence maps
\begin{align*}
  \pi_V \colon F(\Gamma)  \to V(\Gamma) , &&
  \pi_E \colon F(\Gamma) \to E(\Gamma) \sqcup L(\Gamma) 
\end{align*}
such that $|\pi_E^{-1}(e)|=2$ for each $e\in E(\Gamma)$ and
$|\pi_E^{-1}(l)|=1$ for each $l \in L(\Gamma)$.  We assign to an edge
$e$ a closed interval $I_e \cong [0,1]$, to a leg $l$ a half-open
interval $H_l \cong [0,1)$ and to a vertex $v$ a point $p_v$, and fix
identifications $\pi_E^{-1}(e) \cong \partial I_e$, $\pi_L^{-1}(l)
\cong \partial H_l$.  By identifying $I_e$,~$H_l$,~$p_v$ via the map $\pi_V
\colon F(\Gamma) \cong \bigsqcup \partial I_e \sqcup
\bigsqcup \partial H_l \to V(\Gamma) \cong \{p_v\}$, we get a
topological realization $|\Gamma|$ of the graph $\Gamma$.  We say that
$\Gamma$ is connected if $|\Gamma|$ is connected, and write
$\chi(\Gamma)=\chi(|\Gamma|)$ for the topological Euler
characteristic of $|\Gamma|$.

\begin{landscape}
  \setlength{\textwidth}{20cm}

  \section{Tables of Gromov--Witten and Gopakumar--Vafa Invariants}
  \label{sec:tables}
  This section records the results of the calculations described in \S\ref{sec:calculation}.  Entries in {\bf bold face} are input to the calculation: everything else is derived from these.  Our results are in agreement with calculations and conjectures in the literature, except for a handful of cases where we correct typographical errors.  These are indicated in {\ttfamily typewriter font}.

  \begin{table}[h]
\centering
\resizebox{\textwidth}{!}{
\begin{tabular}{cccccccccccccccc}
\toprule
&\multicolumn{15}{c}{Degree} \\\cmidrule(r){2-16} Genus &
1 &
2 &
3 &
4 &
5 &
6 &
7 &
8 &
9 &
10 &
11 &
12 &
13 &
14 &
15 \\
\midrule
0 &
$ 3 $ &
$ -\frac{45}{8} $ &
$ \frac{244}{9} $ &
$ -\frac{12333}{64} $ &
$ \frac{211878}{125} $ &
$ -\frac{102365}{6} $ &
$ \frac{64639725}{343} $ &
$ -\frac{1140830253}{512} $ &
$ \frac{6742982701}{243} $ &
$ -\frac{36001193817}{100} $ &
$ \frac{6425982732150}{1331} $ &
$ -\frac{9581431054999}{144} $ &
$ \frac{2061386799232608}{2197} $ &
$ -\frac{37021055156692659}{2744} $ &
$ \frac{73982838271394248}{375} $\\ \addlinespace[0.6ex]
1 &
$ \frac{1}{4} $ &
$ -\frac{3}{8} $ &
$ -\frac{23}{3} $ &
$ \frac{3437}{16} $ &
$ -\frac{43107}{10} $ &
$ 79522 $ &
$ -\frac{39826681}{28} $ &
$ \frac{803703117}{32} $ &
$ -\frac{15878598203}{36} $ &
$ \frac{154610243281}{20} $ &
$ -\frac{2979940731399}{22} $ &
$ \frac{7124283102275}{3} $ &
$ -\frac{541814449674696}{13} $ &
$ \frac{41013714834701487}{56} $ &
$ -\frac{64436279290065616}{5} $\\ \addlinespace[0.6ex]
2 &
$ \frac{1}{80} $ &
$ 0 $ &
$ \frac{3}{20} $ &
$ -\frac{514}{5} $ &
$ \frac{43497}{8} $ &
$ -\frac{1552743}{8} $ &
$ \frac{92569957}{16} $ &
$ -\frac{776658618}{5} $ &
$ \frac{311565686229}{80} $ &
$ -\frac{186103710373}{2} $ &
$ \frac{17161329260151}{8} $ &
$ -\frac{962191179023583}{20} $ &
$ \frac{5278121482133523}{5} $ &
$ -\frac{910206655959750921}{40} $ &
$ \frac{966725384014894851}{2} $\\ \addlinespace[0.6ex]
3 &
$ \frac{1}{2016} $ &
$ \frac{1}{336} $ &
$ \frac{1}{56} $ &
$ \frac{1480}{63} $ &
$ -\frac{1385717}{336} $ &
$ \frac{34386105}{112} $ &
$ -\frac{4563656185}{288} $ &
$ \frac{27816690931}{42} $ &
$ -\frac{771022095237}{32} $ &
$ \frac{400254094073885}{504} $ &
$ -\frac{2722614157619637}{112} $ &
$ \frac{9834759858880697}{14} $ &
$ -\frac{1628439950424111871}{84} $ &
$ \frac{4121486387127690091}{8} $ &
$ -13266967197002009748 $\\ \addlinespace[0.6ex]
4 &
$ \frac{1}{57600} $ &
$ \frac{1}{1920} $ &
$ \frac{7}{1600} $ &
$ -\frac{2491}{900} $ &
$ \frac{3865243}{1920} $ &
$ -\frac{217225227}{640} $ &
$ \frac{364416184789}{11520} $ &
$ -\frac{316806697367}{150} $ &
$ \frac{726200060335821}{6400} $ &
$ -\frac{15051658211781731}{2880} $ &
$ \frac{137299697068139103}{640} $ &
$ -\frac{3220668414546452353}{400} $ &
$ \frac{84382375637970689569}{300} $ &
$ -\frac{14824230312581305514377}{1600} $ &
$ \frac{46493722208852997775773}{160} $\\ \addlinespace[0.6ex]
5 &
$ \frac{1}{1774080} $ &
$ \frac{1}{14080} $ &
$ \frac{61}{49280} $ &
$ \frac{4471}{22176} $ &
$ -\frac{65308319}{98560} $ &
$ \frac{5383395285}{19712} $ &
$ -\frac{17012987874515}{354816} $ &
$ \frac{64688948714407}{12320} $ &
$ -\frac{11945278310269797}{28160} $ &
$ \frac{350595910152610339}{12672} $ &
$ -\frac{13785482612596271967}{8960} $ &
$ \frac{465731911358273599411}{6160} $ &
$ -\frac{248785036687799870780761}{73920} $ &
$ \frac{6809369636793660022747587}{49280} $ &
$ -\frac{65332009871525107439528907}{12320} $\\ \addlinespace[0.6ex]
6 &
$ \frac{691}{39626496000} $ &
$ \frac{11747}{1320883200} $ &
$ \frac{377977}{1100736000} $ &
$ -\frac{4874687}{1238328000} $ &
$ \frac{202790371913}{1320883200} $ &
$ -\frac{24163714857019}{146764800} $ &
$ \frac{64139775474690313}{1132185600} $ &
$ -\frac{4310034999040379953}{412776000} $ &
$ \frac{991900415691784747}{768000} $ &
$ -\frac{239501070313053131971001}{1981324800} $ &
$ \frac{1350252537724641260419439}{146764800} $ &
$ -\frac{164221788876199036010533573}{275184000} $ &
$ \frac{2164019137440273660185654977}{63504000} $ &
$ -\frac{3029955595814315413860062951}{1728000} $ &
$ \frac{433647145446345870048459770393}{5241600} $\\ \addlinespace[0.6ex]
7 &
$ \frac{1}{1916006400} $ &
$ \frac{31}{29030400} $ &
$ \frac{703}{7603200} $ &
$ \frac{11705}{4790016} $ &
$ -\frac{8293308997}{319334400} $ &
$ \frac{540810103943}{7096320} $ &
$ -\frac{20390495664732131}{383201280} $ &
$ \frac{675146333220270311}{39916800} $ &
$ -\frac{77105305044973449611}{23654400} $ &
$ \frac{42491482875357032433349}{95800320} $ &
$ -\frac{1655992931289521245824679}{35481600} $ &
$ \frac{13400468324230111992071993}{3326400} $ &
$ -\frac{23713389495101796065291526451}{79833600} $ &
$ \frac{1025143548772512485242765294187}{53222400} $ &
$ -\frac{2489343025368827360553366826757}{2217600} $\\ \addlinespace[0.6ex]

\bottomrule\\
\end{tabular}
}
\caption{Some Gromov--Witten invariants of $Y = K_{\Proj^2}$}
\end{table}

  \bigskip

  \begin{table}[h]
\centering
\resizebox{\textwidth}{!}{
\begin{tabular}{cccccccccccccccc}
\toprule
&\multicolumn{15}{c}{Degree} \\\cmidrule(r){2-16} Genus &
1 &
2 &
3 &
4 &
5 &
6 &
7 &
8 &
9 &
10 &
11 &
12 &
13 &
14 &
15 \\
\midrule
0 &
$ 3 $ &
$ -6 $ &
$ 27 $ &
$ -192 $ &
$ 1695 $ &
$ -17064 $ &
$ 188454 $ &
$ -2228160 $ &
$ 27748899 $ &
$ -360012150 $ &
$ 4827935937 $ &
$ -66537713520 $ &
$ 938273463465 $ &
$ -13491638200194 $ &
$ 197287568723655 $\\ \addlinespace[0.6ex]
1 &
$ 0 $ &
$ 0 $ &
$ -10 $ &
$ 231 $ &
$ -4452 $ &
$ 80948 $ &
$ -1438086 $ &
$ 25301295 $ &
$ -443384578 $ &
$ 7760515332 $ &
$ -135854179422 $ &
$ 2380305803719 $ &
$ -41756224045650 $ &
$ 733512068799924 $ &
$ -12903696488738656 $\\ \addlinespace[0.6ex]
2 &
$ \bf 0 $ &
$ \bf 0 $ &
$ 0 $ &
$ -102 $ &
$ 5430 $ &
$ -194022 $ &
$ 5784837 $ &
$ -155322234 $ &
$ 3894455457 $ &
$ -93050366010 $ &
$ 2145146041119 $ &
$ -48109281322212 $ &
$ 1055620386953940 $ &
$ -22755110195405850 $ &
$ 483361869975894765 $\\ \addlinespace[0.6ex]
3 &
$ \bf 0 $ &
$ \bf 0 $ &
$ \bf 0 $ &
$ \bf 15 $ &
$ -3672 $ &
$ 290853 $ &
$ \tt -15363990 $ &
$ 649358826 $ &
$ -23769907110 $ &
$ 786400843911 $ &
$ -24130293606924 $ &
$ 698473748830878 $ &
$ -19298221675559646 $ &
$ 513289541565539286 $ &
$ -13226687073790872894 $\\ \addlinespace[0.6ex]
4 &
$ \bf 0 $ &
$ \bf 0 $ &
$ \bf 0 $ &
$ \bf 0 $ &
$ \bf 1386 $ &
$ \bf -290400 $ &
$ 29056614 $ &
$ -2003386626 $ &
$ 109496290149 $ &
$ -5094944994204 $ &
$ 210503102300868 $ &
$ -7935125096754762 $ &
$ 278055282896359878 $ &
$ -9179532480730484952 $ &
$ 288379973286696180135 $\\ \addlinespace[0.6ex]
5 &
$ \bf 0 $ &
$ \bf 0 $ &
$ \bf 0 $ &
$ \bf 0 $ &
$ \bf -270 $ &
$ \bf 196857 $ &
$ \bf -40492272 $ &
$ \bf 4741754985 $ &
$ -396521732268 $ &
$ 26383404443193 $ &
$ -1485630816648252 $ &
$ 73613315148586317 $ &
$ -3295843339183602162 $ &
$ 135875843241729533613 $ &
$ -5230662528295888702200 $\\ \addlinespace[0.6ex]
6 &
$ \bf 0 $ &
$ \bf 0 $ &
$ \bf 0 $ &
$ \bf 0 $ &
$ \bf 21 $ &
$ \bf -90390 $ &
$ \bf 42297741 $ &
$ \bf -8802201084 $ &
$ \bf 1156156082181 $ &
$ \bf -111935744536416 $ &
$ 8698748079113310 $ &
$ -572001241783007370 $ &
$ 32970159716836634586 $ &
$ -1707886552705077581628 $ &
$ 80979854504456065293006 $\\ \addlinespace[0.6ex]
7 &
$ \bf 0 $ &
$ \bf 0 $ &
$ \bf 0 $ &
$ \bf 0 $ &
$ \bf 0 $ &
$ \bf 27538 $ &
$ \bf -33388020 $ &
$ \bf 12991744968 $ &
$ \bf -2756768768616 $ &
$ \bf 395499033672279 $ &
$ \bf -42968546119317066 $ &
$ \bf 3786284014554551293 $ &
$ -283123099266200799858 $ &
$ 18542695412600660315361 $ &
$ -1088520963699453440916068 $\\ \addlinespace[0.6ex]

\bottomrule\\
\end{tabular}
}
\caption{Some Gopakumar--Vafa invariants of $Y = K_{\Proj^2}$}
\end{table}

  The input values for $2 \leq g \leq 4$ are taken from work of Klemm--Zaslow~\cite{Klemm--Zaslow}.  The input values for $5 \leq g \leq 7$ are taken from work of Haghighat--Klemm--Rauch~\cite{Haghighat--Klemm--Rauch}.  The $g=3$, $d=7$ Gopakumar--Vafa invariant corrects a typographical error in~\cite{Klemm--Zaslow}*{Figure~2}.

  \raggedbottom
  \pagebreak

  \begin{table}[h]
\centering
\resizebox{\textwidth}{!}{
\begin{tabular}{cccccccccccccc}
\toprule
&\multicolumn{13}{c}{$i$} \\\cmidrule(r){2-14} Genus &
0 &
1 &
2 &
3 &
4 &
5 &
6 &
7 &
8 &
9 &
10 &
11 &
12 \\
\midrule
2 &
$ -\frac{1}{2160} $ &
$ \frac{1}{4320} $ &
$ -\frac{1}{7680} $ &
$ 0 $ &
$ 0 $ &
$ 0 $ &
$ 0 $ &
$ 0 $ &
$ 0 $ &
$ 0 $ &
$ 0 $ &
$ 0 $ &
$ 0 $\\ \addlinespace[0.6ex]
3 &
$ \frac{1}{544320} $ &
$ -\frac{41}{1088640} $ &
$ \frac{673}{4354560} $ &
$ -\frac{809}{3870720} $ &
$ \frac{373}{4128768} $ &
$ 0 $ &
$ 0 $ &
$ 0 $ &
$ 0 $ &
$ 0 $ &
$ 0 $ &
$ 0 $ &
$ 0 $\\ \addlinespace[0.6ex]
4 &
$ -\frac{7}{41990400} $ &
$ \frac{863}{48988800} $ &
$ -\frac{37403}{209018880} $ &
$ \frac{727859}{1175731200} $ &
$ -\frac{5397557}{5573836800} $ &
$ \frac{10531379}{14863564800} $ &
$ -\frac{10450031}{52848230400} $ &
$ 0 $ &
$ 0 $ &
$ 0 $ &
$ 0 $ &
$ 0 $ &
$ 0 $\\ \addlinespace[0.6ex]
5 &
$ \frac{3161}{77598259200} $ &
$ -\frac{223829}{19399564800} $ &
$ \frac{11888603}{51732172800} $ &
$ -\frac{249278237}{165542952960} $ &
$ \frac{46781311369}{9932577177600} $ &
$ -\frac{140380033}{17517772800} $ &
$ \frac{20544043433}{2690729902080} $ &
$ -\frac{2921058683}{761014517760} $ &
$ \frac{117292943}{147639500800} $ &
$ 0 $ &
$ 0 $ &
$ 0 $ &
$ 0 $\\ \addlinespace[0.6ex]
6 &
$ -\frac{6261257}{317764871424000} $ &
$ \frac{13190237077}{1271059485696000} $ &
$ -\frac{7195198914889}{20336951771136000} $ &
$ \frac{19504344827507}{5084237942784000} $ &
$ -\frac{12996726530356181}{650782456676352000} $ &
$ \frac{1825708681106329}{30989640794112000} $ &
$ -\frac{24530495575515863}{231389317929369600} $ &
$ \frac{293773393899971387}{2468152724579942400} $ &
$ -\frac{81366182634551611}{997233424072704000} $ &
$ \frac{228929596808349497}{7313045109866496000} $ &
$ -\frac{26830970439753803}{5200387633682841600} $ &
$ 0 $ &
$ 0 $\\ \addlinespace[0.6ex]
7 &
$ \frac{263323}{17958454272000} $ &
$ -\frac{3447502789}{277322069606400} $ &
$ \frac{242936799515323}{366065131880448000} $ &
$ -\frac{5348522046750959}{488086842507264000} $ &
$ \frac{28861258637561351}{334688120576409600} $ &
$ -\frac{506419154595610651}{1301564913352704000} $ &
$ \frac{78865187127719825617}{71400132389634048000} $ &
$ -\frac{689067633277627125149}{333200617818292224000} $ &
$ \frac{128297351013478670161}{49363054491598848000} $ &
$ -\frac{23861155478166955057}{10969567664799744000} $ &
$ \frac{4906681264436796773633}{4212313983283101696000} $ &
$ -\frac{676965266387896064113}{1872139548125822976000} $ &
$ \frac{11779039621933858193}{237732006111215616000} $\\ \addlinespace[0.6ex]

\bottomrule\\
\end{tabular}
}
\caption{The coefficient of $(27y+1)^{-i}$ in the expansion of the genus-$g$ holomorphic ambiguity}
\end{table}

  This table records the Laurent expansion at the conifold point of the genus-$g$ holomorphic ambiguity $\hC^{(g)}_{\alg,0}$.  Here $a=\frac{1}{12}$.
  
  \bigskip

  \begin{table}[h]
\centering
\resizebox{\textwidth}{!}{
\begin{tabular}{cccccccccc}
\toprule
&\multicolumn{9}{c}{Number of insertions of $\unit_{\frac{1}{3}}$} \\\cmidrule(r){2-10} Genus &
3 &
6 &
9 &
12 &
15 &
18 &
21 &
24 &
27 \\
\midrule
0 &
$ \frac{1}{3} $ &
$ -\frac{1}{27} $ &
$ \frac{1}{9} $ &
$ -\frac{1093}{729} $ &
$ \frac{119401}{2187} $ &
$ -\frac{27428707}{6561} $ &
$ \frac{102777653467}{177147} $ &
$ -\frac{210755831694887}{1594323} $ &
$ \frac{24487690049215235}{531441} $\\ \addlinespace[0.6ex]
1 &
$ 0 $ &
$ \frac{1}{243} $ &
$ -\frac{14}{243} $ &
$ \frac{13007}{6561} $ &
$ -\frac{8354164}{59049} $ &
$ \frac{9730293415}{531441} $ &
$ -\frac{6226371315626}{1594323} $ &
$ \frac{18368708604854737}{14348907} $ &
$ -\frac{8774479580126347360}{14348907} $\\ \addlinespace[0.6ex]
2 &
$ \tt \frac{1}{19440} $ &
$ -\frac{13}{11664} $ &
$ \frac{20693}{524880} $ &
$ -\frac{12803923}{4723920} $ &
$ \tt \frac{314291111}{944784} $ &
$ -\frac{8557024202467}{127545840} $ &
$ \frac{7964469005139389}{382637520} $ &
$ -\frac{2165652826095222589}{229582512} $ &
$ \frac{12381768449768685533597}{2066242608} $\\ \addlinespace[0.6ex]
3 &
$ \tt -\frac{31}{2449440} $ &
$ \frac{11569}{22044960} $ &
$ -\frac{2429003}{66134880} $ &
$ \tt \frac{871749323}{198404640} $ &
$ -\frac{1520045984887}{1785641760} $ &
$ \frac{450933448038569}{1785641760} $ &
$ -\frac{39050288662607269}{357128352} $ &
$ \frac{9607109907927162237163}{144636982560} $ &
$ -\frac{14262745321381354134470275}{260346568608} $\\ \addlinespace[0.6ex]
4 &
$ \frac{313}{62985600} $ &
$ \tt -\frac{1889}{5038848} $ &
$ \frac{115647179}{2550916800} $ &
$ -\frac{29321809247}{3401222400} $ &
$ \frac{22766570703031}{9183300480} $ &
$ -\frac{855627159576453613}{826497043200} $ &
$ \frac{562917323177869058989}{929809173600} $ &
$ -\frac{2145113324078184246985223}{4463084033280} $ &
$ \frac{10131783225119510967855855547}{20083878149760} $\\ \addlinespace[0.6ex]
5 &
$ -\frac{519961}{174596083200} $ &
$ \frac{196898123}{523788249600} $ &
$ -\frac{339157983781}{4714094246400} $ &
$ \frac{78658947782147}{3856986201600} $ &
$ -\frac{1057430723091383537}{127280544652800} $ &
$ \frac{5402182862315780935057}{1145524901875200} $ &
$ -\frac{37382736073982969299220839}{10309724116876800} $ &
$ \frac{1518869184768410803025171219}{412388964675072} $ &
$ -\frac{17948119840640630466672146029249}{3711500682075648} $\\ \addlinespace[0.6ex]
6 &
$ \frac{14609730607}{5719767685632000} $ &
$ -\frac{258703053013}{514779091706880} $ &
$ \frac{2453678654644313}{17159303056896000} $ &
$ -\frac{40015774193969601803}{694951773804288000} $ &
$ \frac{5342470197951654213739}{166788425713029120} $ &
$ -\frac{37563825079969605317846887}{1563641491059648000} $ &
$ \frac{8019780014405254969486119183119}{337746562068883968000} $ &
$ -\frac{3070502316719712753650867310009617}{101323968620665190400} $ &
$ \frac{29810463687802263588732581882184525323}{607943811723991142400} $\\ \addlinespace[0.6ex]
7 &
$ -\frac{1122101011}{377127539712000} $ &
$ \frac{2196793414201}{2545610893056000} $ &
$ -\frac{2127526097369539}{6109466143334400} $ &
$ \frac{26373375124439869913}{137462988225024000} $ &
$ -\frac{350087626432381439100911}{2474333788050432000} $ &
$ \frac{1155026829373310723028673}{8435228822899200} $ &
$ -\frac{11403268061303112561993941625137}{66807012277361664000} $ &
$ \frac{16199636944299442178776584530770121}{60126311049625497600} $ &
$ -\frac{4251179144363528862957474967991845811}{8016841473283399680} $\\ \addlinespace[0.6ex]

\bottomrule\\
\end{tabular}
}
\caption{Some Gromov--Witten invariants of $\cX = \big[\C^3/\mu_3\big]$}
\end{table}

  The five entries in {\tt typewriter font} here correct typographical errors in the
  corresponding table on page 808 of~\cite{ABK}.

  \pagebreak

  \begin{table}[h]
\centering
\resizebox{\textwidth}{!}{
\begin{tabular}{cccccccccccccccccc}
\toprule
&\multicolumn{17}{c}{Degree} \\\cmidrule(r){2-18} Genus &
-12 &
-11 &
-10 &
-9 &
-8 &
-7 &
-6 &
-5 &
-4 &
-3 &
-2 &
-1 &
0 &
1 &
2 &
3 &
4 \\
\midrule
2 &
$ 0 $ &
$ 0 $ &
$ 0 $ &
$ 0 $ &
$ 0 $ &
$ 0 $ &
$ 0 $ &
$ 0 $ &
$ 0 $ &
$ 0 $ &
$ -\frac{1}{80} $ &
$ 0 $ &
$ -\frac{13}{25920} $ &
$ \frac{1}{19440} $ &
$ -\frac{3187}{377913600} $ &
$ \frac{239}{255091680} $ &
$ -\frac{19151}{257132413440} $\\ \addlinespace[0.6ex]
3 &
$ 0 $ &
$ 0 $ &
$ 0 $ &
$ 0 $ &
$ 0 $ &
$ 0 $ &
$ 0 $ &
$ 0 $ &
$ \frac{1}{112} $ &
$ 0 $ &
$ 0 $ &
$ 0 $ &
$ \frac{121}{58786560} $ &
$ -\frac{1}{1469664} $ &
$ \frac{23855}{179992689408} $ &
$ -\frac{557}{24794911296} $ &
$ \frac{15575867}{4498788705546240} $\\ \addlinespace[0.6ex]
4 &
$ 0 $ &
$ 0 $ &
$ 0 $ &
$ 0 $ &
$ 0 $ &
$ 0 $ &
$ -\frac{3}{160} $ &
$ 0 $ &
$ 0 $ &
$ 0 $ &
$ 0 $ &
$ 0 $ &
$ -\frac{1093}{31744742400} $ &
$ \frac{7}{377913600} $ &
$ -\frac{6830569}{1190155742208000} $ &
$ \frac{1561279}{1205032688985600} $ &
$ -\frac{1444309519}{5726315338059571200} $\\ \addlinespace[0.6ex]
5 &
$ 0 $ &
$ 0 $ &
$ 0 $ &
$ 0 $ &
$ \frac{27}{352} $ &
$ 0 $ &
$ 0 $ &
$ 0 $ &
$ 0 $ &
$ 0 $ &
$ 0 $ &
$ 0 $ &
$ \frac{9841}{8380611993600} $ &
$ -\frac{809}{942818849280} $ &
$ \frac{118418785}{326612060022657024} $ &
$ -\frac{113975899}{1002105184160424960} $ &
$ \frac{11188464609233}{393947589997146260275200} $\\ \addlinespace[0.6ex]
6 &
$ 0 $ &
$ 0 $ &
$ -\frac{18657}{36400} $ &
$ 0 $ &
$ 0 $ &
$ 0 $ &
$ 0 $ &
$ 0 $ &
$ 0 $ &
$ 0 $ &
$ 0 $ &
$ 0 $ &
$ -\frac{61203943}{926602365072384000} $ &
$ \frac{1276277}{21059144660736000} $ &
$ -\frac{279842720162009}{9052836032762704465920000} $ &
$ \frac{984486511}{81990424158580224000} $ &
$ -\frac{338480893523407}{87026603972096855678976000} $\\ \addlinespace[0.6ex]
7 &
$ \frac{81}{16} $ &
$ 0 $ &
$ 0 $ &
$ 0 $ &
$ 0 $ &
$ 0 $ &
$ 0 $ &
$ 0 $ &
$ 0 $ &
$ 0 $ &
$ 0 $ &
$ 0 $ &
$ \frac{550838251}{100073055427817472000} $ &
$ -\frac{7943}{1309171316428800} $ &
$ \frac{27776712091}{7792369912031464488960} $ &
$ -\frac{1177971963811}{788977453593185779507200} $ &
$ \frac{11286380576743987}{21070001524321295871639552000} $\\ \addlinespace[0.6ex]

\bottomrule\\
\end{tabular}
}
\caption{Some conifold Gromov--Witten invariants}
\end{table}

  This table records the expansion coefficients of the genus-$g$ conifold correlation function $C^{(g)}_{\con,0}$ as a Laurent series in the conifold flat co-ordinate:
  \[
  C^{(g)}_{\con,0} = \sum_{d \in \Z} n^{\con}_{g,d} \, x_{\con}^d
  \]
  See Notation~\ref{nota:Darboux_cusps} and Notation~\ref{nota:corr_wave_CY} for precise definitions.  The genus-$g$ conifold correlation function has a pole of order $2g-2$ at the conifold point, so $n_{g,d}^{\con}$ vanishes for $d<2-2g$ and $n_{g,2-2g}^{\con}$ is non-zero.  The leading term
\[
C^{(g)}_{\con,0} = \frac{B_{2g}}{2g(2g-2)} 3^{g-1} x_{\con}^{2-2g} + \cdots
\]
agrees with predictions in the literature, up to rescaling $x_{\con}$ by $\sqrt{3}$ to match with~\cite{BCOV,Ghoshal--Vafa} or $\sqrt{-3}$ to match with~\cite{Huang--Klemm--Quackenbush}.  Note that no other negative powers of $x_{\con}$ occur, as predicted by~\cite{Huang--Klemm, Huang--Klemm--Quackenbush}: this is the ``conifold gap''.

 \end{landscape}

\end{document}